\documentclass[letterpaper, oneside]{report}

\usepackage[english]{babel}
\usepackage[utf8]{inputenc}
\usepackage[T1]{fontenc}
\usepackage{fancyhdr}
\usepackage{csquotes}
\usepackage{cite}

\usepackage{amssymb}
\usepackage{amsmath}
\usepackage{nicefrac}
\usepackage{bbm}
\usepackage{amsfonts}
\usepackage{tensor}
\usepackage{amsthm}
\usepackage{mathrsfs} 
\usepackage{mathtools}
\usepackage{extpfeil}
\usepackage{tikz-cd}
\usepackage{bm}
\usepackage{bbold}
\usepackage[mathscr]{euscript}

\usepackage{color}
\usepackage{graphicx}
\usepackage{wrapfig}
\usepackage{subfigure}
\usepackage{placeins}

\usepackage{multicol}
\usepackage{ltxtable}
\usepackage{url}
\usepackage{paralist}
\usepackage{verbatim}
\usepackage{hyperref}
\usepackage{glossaries}
\usepackage{float}
\usepackage{nameref}
\usepackage{array}
\usepackage{marginnote}
\usepackage{adjustbox}
\usepackage{blindtext}
\usepackage{todonotes}
\usetikzlibrary{decorations.pathmorphing} 
\usepackage{quiver}
\usepackage{yfonts}
\usepackage{geometry}

\newcommand{\field}{\mathbb{k}}
\DeclareMathOperator{\colim}{colim }

\DeclareMathOperator{\map}{Map}
\DeclareMathOperator{\mor}{Mor}
\DeclareMathOperator{\alg}{Alg}
\DeclareMathOperator{\e}{End}
\DeclareMathOperator{\lmod}{LMod}
\DeclareMathOperator{\rmod}{RMod}

\DeclareMathOperator{\dk}{D(\field)}
\DeclareMathOperator{\ch}{Ch}

\DeclareMathOperator{\sym}{\textnormal{Sym}}
\DeclareMathOperator{\an}{\mathscr{S}}
\DeclareMathOperator{\fin}{\mathscr{F}in_{\ast}}
\DeclareMathOperator{\spec}{\textnormal{Spec}}
\DeclareMathOperator{\defm}{\textnormal{Def}}
\DeclareMathOperator{\pr}{\textnormal{Pr}^{\textnormal{L}}}
\newcommand{\obar}{\textnormal{Bar}}
\newcommand{\cobar}{\textnormal{Cobar}}
\newcommand{\ccoalg}{\textnormal{CoAlg}^{\textnormal{conil}}}
\newcommand{\coalg}{\textnormal{CoAlg}}
\newcommand{\lie}{\mathbf{Lie}}
\newcommand{\art}{\alg^{\textnormal{Art}}_{\mathscr{P}}}
\newcommand{\faut}{\widehat{\mathcal{A}\textnormal{ut}}}
\newcommand{\indcoh}{\textnormal{IndCoh}}
\newcommand{\qcoh}{\textnormal{QCoh}}
\renewcommand{\o}{\mathcal{O}^{\otimes}}
\newcommand{\lm}{\mathscr{LM}^{\otimes}}
\newcommand{\alm}{\mathscr{LM}}
\newcommand{\triv}{\mathscr{T}riv^{\otimes}}
\newcommand{\com}{\mathscr{C}om^{\otimes}}

\newcommand{\ass}{{\mathscr{A}ssoc}^{\otimes}}
\newcommand{\aass}{\mathscr{A}ssoc}
\newcommand{\nccoassoc}{\mathbf{coAssoc}^{\textnormal{nc}}}
\newcommand{\coassoc}{\mathbf{coAssoc}}
\newcommand{\assoc}{\mathbf{Assoc}}
\newcommand{\nuassoc}{\mathbf{Assoc}^{\textnormal{nu}}}
\newcommand{\cai}{\mathbb{A}^{\textnormal{u}}_{\infty}}

\newcommand{\braces}{\mathbf{Braces}}
\newcommand{\ger}{\mathbf{Ger}}
\newcommand{\dgcomm}{\mathbf{Com}}
\newcommand{\nucomm}{\mathbf{Com}^{\textnormal{nu}}}
\newcommand{\dgcocomm}{\mathbf{coCom}}
\newcommand{\nccocomm}{\mathbf{coCom}^{\textnormal{nc}}}
\newcommand{\td}{\textnormal{Td}}
\newcommand{\at}{\textnormal{At}}
\newcommand{\gt}{\textnormal{GT}}
\newcommand{\grt}{\textnormal{GRT}}
\newcommand{\aut}{\textnormal{Aut}}
\newcommand{\gal}{\textnormal{Gal}(\bar{\mathbb{Q}}/\mathbb{Q})}
\renewcommand{\hom}{\textnormal{Hom}}

\newcommand{\eone}{C_{\ast}(\mathbb{E}^T_1)}

\newcommand{\dgpsh}[1]{\textnormal{dgPSh}(#1)}
\newcommand{\dgsh}[1]{\textnormal{dgSh}(#1)}
\newcommand{\affpsh}[1]{\textnormal{dgPSh}^{\textnormal{aff}}(#1)}
\newcommand{\affsh}[1]{\textnormal{dgSh}^{\textnormal{aff}}(#1)}
\newcommand{\diag}{\textnormal{Diag}}
\DeclareMathOperator{\n}{\langle n\rangle}
\DeclareMathOperator{\M}{\langle m\rangle}
\DeclareMathOperator{\otimesC}{\otimes_{\mathscr{P}^{\text{!`}}}}

\newcounter{intro}

\newtheorem{theorem}{Theorem}
\newtheorem{Itheorem}[intro]{Theorem}
\newtheorem{lemma}[theorem]{Lemma}
\newtheorem{proposition}[theorem]{Proposition}
\newtheorem{corollary}[theorem]{Corollary}

\newtheorem{conjecture}[theorem]{Conjecture}
\newtheorem{defprop}[theorem]{Definition/Proposition}

\theoremstyle{definition}
\newtheorem{definition}[theorem]{Definition}
\newtheorem{notation}[theorem]{Notation}

\newtheorem{example}[theorem]{Example}

\theoremstyle{remark}
\newtheorem{remark}[theorem]{Remark}
\newtheorem*{remark*}{Remark}

\renewcommand{\thetheorem}{\thesection.\arabic{theorem}}
\renewcommand{\thetheorem}{\thesection.\arabic{theorem}}

\newcommand{\displaypar}{%
  \par\addvspace{\abovedisplayskip}%
}
\begin{document}
\begin{titlepage}
\thispagestyle{empty}
\centering
\vspace*{2cm}
{\huge \bfseries {Derived operadic centers in algebraic geometry and deformation quantization} \par}
\vspace{2cm}
{\large Sonja M. Farr \par}
\vfill
\normalsize{This is the author's Ph.D. thesis, submitted to the University of Nevada, Reno, in August 2026.}
\end{titlepage}
\begin{center}
\section*{Abstract}
\end{center}
\vspace*{1cm}
In algebraic geometry, it is well known that Hochschild cohomology and, in particular, the algebraic structure it carries, plays a central role in studying the infinitesimal noncommutative deformations of geometric spaces. This thesis provides, for the first time, an explicit interface between J. Lurie's work on higher centers and the Hochschild cohomology of an algebraic $\field$-scheme within the framework of formal deformation quantization. Our motivation stems from the mysterious appearance of the square root of the Todd genus in Kontsevich's formality theorem for algebraic varieties, as well as the conjectural relationships between these objects and the motivic Galois group.\par
Our main result is a canonical solution to Deligne's conjecture for Hochschild cochains in the affine and global cases, even for singular schemes, by exhibiting the Hochschild complex as an $\infty$-operadic center. We show that this equips the Hochschild complex with a universal $\mathbb{E}_2$-algebra structure that precisely agrees with the classical Gerstenhaber bracket and cup product on cohomology in the affine and smooth cases. Finally, we interpret this universal $\mathbb{E}_2$-algebra structure in terms of operadic formal moduli problems by exhibiting a relationship between $\infty$-operadic centers and formal automorphism groups. 
\newpage
\setcounter{tocdepth}{2}
\tableofcontents
\newpage
\chapter{Introduction}
Objects equipped with a fixed algebraic structure often assemble into a geometric object called a {moduli space}, rather than forming a discrete set. Such a moduli space is an algebro-geometric object that can be modeled by a ``functor of points'' from a subcategory of differential graded (dg) commutative algebras to homotopy types satisfying some additional properties, as described in the introduction of \cite{DAGX}. A moduli space encodes the relationship between algebraic objects of the given type. For example, given a fixed algebraic object $A$ viewed as a point in the given moduli space, the space of loops at the point $A$ encodes the automorphism group of $A$. {Deformation theory} is the study of the infinitesimal neighborhood of a point $A$ in a moduli space $\mathcal{M}$. Such an infinitesimal neighborhood is called a formal moduli problem. In this thesis, we are interested in the deformation theory of algebras of a given type. Classical examples include associative algebras, commutative algebras, Lie algebras, and Poisson algebras. Such a type of algebra is encoded by an {algebraic operad}, as defined for example in \cite[Chapter 5]{LV}. Operads are given by sequences of objects $\{\mathcal{O}(n)\}_n$, indexed by natural numbers, where we think of $\mathcal{O}(n)$ as the collection of universal $n$-ary operations of a given type of algebra, together with composition operations, which can be represented by grafting of trees 
\begin{center}
    \begin{tikzpicture}[scale=.7,
    line cap=round,
    line join=round,
    strand/.style={black,line width=0.68pt},
    labelnode/.style={
        circle,
        draw,
        fill=white,
        inner sep=1.2pt,
        minimum size=16pt,
        font=\large
    }
]

\coordinate (mu) at (0,0);
\coordinate (nu) at (0,1.52);

\draw[strand] (-0.34,0.10) -- (-2.95,1.98);
\draw[strand] (-0.31,0.20) -- (-2.70,2.05);
\draw[strand] (-0.23,0.30) -- (-2.36,1.94);
\draw[strand] (0.34,0.10) -- (2.95,1.98);
\draw[strand] (0.31,0.20) -- (2.70,2.05);
\draw[strand] (0.23,0.30) -- (2.36,1.94);

\draw[strand] (-0.34,1.62) -- (-1.64,2.17);
\draw[strand] (-0.30,1.72) -- (-1.48,2.25);
\draw[strand] (-0.22,1.82) -- (-1.18,2.25);
\draw[strand] (0.34,1.62) -- (1.64,2.17);
\draw[strand] (0.30,1.72) -- (1.48,2.25);
\draw[strand] (0.22,1.82) -- (1.18,2.25);

\draw[strand] (0,0.33) -- (0,1.19);
\draw[strand] (0,-0.33) -- (0,-0.82);

\node[labelnode] at (mu) {$\mu$};
\node[labelnode] at (nu) {$\nu$};
\node[font=\large] at (0.31,0.76) {};

\end{tikzpicture}
\end{center}
In the 1980s, P. Deligne, V. Drinfeld, and other mathematicians working in deformation theory noticed that, in characteristic 0, any formal moduli problem is controlled by a dg Lie algebra. In particular, the ``shifted'' tangent space of a formal moduli problem admits a dg Lie algebra structure, from which the original formal moduli problem can be reconstructed. At the time, this correspondence was primarily used as a heuristic. Later, it was sharpened into a theorem and rigorously proven by J. Pridham in \cite{Pr} and J. Lurie in \cite{DAGX}.\par
M. Kontsevich used this correspondence in 1997 to solve the problem of formal deformation quantization, which asks whether any Poisson manifold can be perturbatively quantized, i.e. equipped with an associative star product. One can rephrase this question in terms of moduli spaces. On the one hand, consider the formal moduli problem of the trivial Poisson structure on the algebra of smooth functions in the moduli space of Poisson structures. On the other hand, consider the formal moduli problem of the commutative algebra structure on the smooth functions in the moduli space of associative algebra structures. In \cite{K}, Kontsevich exhibits an explicit isomorphism in the homotopy category of dg Lie algebras between the dg Lie algebra controlling the Poisson formal moduli problem, and the dg Lie algebra controlling the associative formal moduli problem. By the Lurie-Pridham correspondence, the existence of this isomorphism implies that any Poisson deformation of the trivial Poisson structure yields an associative deformation of the commutative algebra structure, and hence any Poisson structure can be quantized. \par
The dg Lie algebra controlling the Poisson formal moduli problem is given by the so-called polyvector fields. This is the exterior algebra on vector fields, equipped with the Schouten bracket, which is the unique extension of the commutator bracket to a Lie bracket acting by derivations on the exterior algebra multiplication. The dg Lie algebra corresponding to the associative formal moduli problem is the complex of polydifferential operators, which is a smooth version of the Hochschild cochain complex.\par 
In particular, for a general associative algebra $A$, the dg Lie algebra corresponding to the infinitesimal neighborhood of $A$ in the moduli space of associative algebras is the shifted Hochschild cochain complex $C^{\ast}(A,A)[1]$. The dg Lie algebra structure on this complex was first described by M. Gerstenhaber in his 1963 paper \cite{G}. In addition to the shifted dg Lie structure, the Hochschild cochain complex also admits a cup product, which is only associative on the cochain level, but commutative in cohomology. Gerstenhaber also showed that on Hochschild cohomology, the shifted Lie bracket acts by derivations with respect to the cup product, making Hochschild cohomology $\text{HH}^{\ast}(A,A)$ into a shifted Poisson algebra. Such an algebra is now also called a Gerstenhaber algebra, and there is an operad $\ger$ whose algebras are precisely Gerstenhaber algebras. In a later paper \cite{GV}, Gerstenhaber and A. Voronov identified further canonical algebraic structure on the cochain level in the form of the so-called brace operations. These brace operations encode the cup product and the shifted dg Lie structure, but also a homotopy version of the compatibility of these two structures which corresponds to the Poisson structure in cohomology. It was later proved in \cite{GJ} that the brace operations are also controlled by an operad, called $\braces$. \par
This raises the question: precisely which algebraic structure is present on the Hochschild cochain complex. To this end, F. Cohen proved in \cite{Coh} that the operad $\ger$ governing Gerstenhaber algebras is isomorphic to the homology of the topological operad $D_2$ of little 2-disks, which is the topological operad encoding double loop spaces. This in particular implies that the cohomology of any algebra over the singular chains on little 2-disks operad $C_{\ast}(D_2)$ inherits the structure of a Gerstenhaber algebra. This prompted Deligne to conjecture in a 1993 letter that the Gerstenhaber algebra structure on Hochschild cohomology should lift to the structure of an algebra over the singular chains on little 2-disks operad on the Hochschild cochain complex. Deligne further explained that this $C_{\ast}(D_2)$-algebra structure is expected to come from the fact that the Hochschild cochain complex can be viewed as an ``algebraic double loop space''. This is because it can be identified with the derived automorphisms of $A$ as a bimodule over itself. It therefore inherits two multiplications, one from composition of automorphisms, and one given by the convolution product, which exists because $A$ is the monoidal unit on the monoidal category of $A$-bimodules. Deligne's conjecture has now been proven multiple times, for example in \cite{MCS}, \cite{Vor} and \cite{Tam}. However, all of these solutions to Deligne's conjecture require certain choices that are used to identify $C_{\ast}(D_2)$ with other operads that can be shown to act on the Hochschild complex.\displaypar
This discussion invites the following questions:
\begin{enumerate}
    \item Is there a universal $C_{\ast}(D_2)$-algebra structure on $C^{\ast}(A,A)$ which lifts the cup product and Gerstenhaber bracket in cohomology? If so, how does it compare to the ones constructed in \cite{MCS}, \cite{Vor}, and \cite{Tam}?
    \item How does Kontsevich's formality isomorphism interact with this additional algebraic structure?
    \item What is the deformation-theoretic meaning of the additional algebraic structure? 
\end{enumerate}
The second question was already addressed by Kontsevich in \cite{K}. He proves that a certain linearization of his quantization morphism is compatible with the products on both sides, but only in cohomology. In particular, it induces an isomorphism of commutative algebras in cohomology. \par
A partial answer to the first question can be found in D. Tamarkin's solution to Deligne's conjecture. In \cite{Tam}, Tamarkin constructs a morphism of operads from the homotopy Gerstenhaber algebra operad $\mathbf{Ger}_{\infty}$, which is the canonical resolution of the operad of Gerstenhaber algebras, to the operad $\braces$ encoding the brace operations found in \cite{GV}. This produces the structure of a homotopy Gerstenhaber algebra on the Hochschild cochain complex. This in fact solves Deligne's conjecture, since the dg operad $C_{\ast}(D_2)$ is formal in the sense of rational homotopy theory, i.e. it is equivalent to its homology $H_{\ast}(D_2) \cong \ger$. This was proven by Tamarkin in \cite{Tam2}. \par
However, neither the morphism from the homotopy Gerstenhaber operad to the braces operad, nor the equivalence between the chains on little 2-disks operad and its homology are canonical. Both instead depend on the choice of a so-called Drinfeld associator. These were first defined by Drinfeld in \cite{Dr} in the context of quantization of quasi-triangular quasi-Hopf algebras, or equivalently infinitesimally braided symmetric monoidal categories. Curiously, Drinfeld associators seem to appear in any type of formal quantization problem where one tries to extend an infinitesimal deformation to a formal deformation. \par
Associators form a torsor over an infinite dimensional algebraic group, the (pro-unipotent) Grothendieck-Teichmüller group $\text{GT}(\mathbb{Q})$. This group was also introduced by Drinfeld, building on A. Grothendieck's \cite{Gro2} and Y. Ihara's \cite{Ih} work on geometric actions of the absolute Galois group of the rationals. Drinfeld's pro-unipotent Grothendieck-Teichmüller group contains the motivic analogue of the absolute Galois group as a subgroup. In particular, this implies a direct connection between the homotopy theory of the little 2-disk operad, and the absolute Galois group $\text{Gal}(\bar{\mathbb{Q}}/\mathbb{Q})$. Just like the absolute Galois group, the pro-unipotent Grothendieck-Teichmüller group remains a mysterious object about which little is known. \par
Tamarkin's solution to Deligne's conjecture can be used to construct another formality morphism between polyvector fields and polydifferential operators, as explained in \cite{Tam} and \cite{K2}. By construction, Tamarkin's formality morphism is an $\infty$-quasi-isomorphism of $\ger_{\infty}$-algebras, not just of dg Lie algebras. In particular, regarding the second question, it is compatible with products on the complex level, not just in cohomology. In contrast to Kontsevich's formality morphism, the existence of Tamarkin's formality morphism is proven using obstruction theory, and the proof is thus non-constructive. In particular, while there is an explicit formula for computing the morphism on hypercohomology of Kontsevich's formality morphism at a fixed deformation, it is not known how the linearization at a fixed deformation of Tamarkin's formality morphism acts on hypercohomology. This is especially important for the global version of the comparison between polyvector fields and polydifferential operators, which was proven by D. Calaque and M. Van den Bergh for Kontsevich's formality morphism and Tamarkin's formality morphism in \cite{CVdB} and \cite{CVdB2} respectively.\par
Due to the dependence of Tamarkin's homotopy Gerstenhaber structure on the Hochschild complex on a choice of associator, changing this choice is expected to result in a different $\ger_{\infty}$-structure, and consequently a different $\ger_{\infty}$ formality morphism. Specifically, Kontsevich conjectured in \cite{K2} that the Grothendieck-Teichmüller group should act on the space of such formality quasi-isomorphisms. This conjecture was confirmed by V. Dolgushev, C. Rogers and T. Willwacher in \cite{DRW}. They proved that the Lie algebra of the Grothendieck-Teichmüller group acts on the Gerstenhaber algebra of polyvector fields by derivations in the derived 1-category of complexes of sheaves of $\field$-vector spaces, and therefore on $\ger_{\infty}$ formality maps by pre-composition with this action. This inspires another question:
\begin{enumerate}
    \item[4.] What can we learn about $\text{GT}(\mathbb{Q})$, and consequently about the absolute Galois group $\text{Gal}(\bar{\mathbb{Q}}/\mathbb{Q})$, from its actions on formality morphisms for Hochschild cochains?
\end{enumerate}
\section*{What we do}
In this thesis, we will give a complete answer to the first question by constructing a canonical $C_{\ast}(D_2)$-algebra structure on the Hochschild cochain complex which solves Deligne's conjecture. To do this, we show that the Hochschild complex of an associative algebra $A$ is a derived center of $A$ in the $\infty$-category of homotopy associative algebras.\par 
The theory of derived centers was developed by Lurie in \cite[Section 5.3]{HA}. Given a monoidal $\infty$-category $\mathcal{C}$, the center of an object $X\in \mathcal{C}$ is defined as the universal homotopy associative algebra acting on it. In particular, the center $\mathfrak{Z}(X) \in \alg_{\mathscr{A}ssoc}(\mathcal{C})$ is an associative algebra object in the underlying $\infty$-category. Taking $\mathcal{C}$ to be the symmetric monoidal $\infty$-category of homotopy associative algebras, the derived center of an associative algebra $A$ is a so-called 2-algebra object
\begin{align*}
    \mathfrak{Z}(A) \in \alg_{\mathscr{A}ssoc}(\alg_{\mathscr{A}ssoc}(\dk))
\end{align*}
in the derived $\infty$-category of the field $\field$. By the Dunn additivity theorem \cite[Theorem 5.1.2.2]{HA}, the $\infty$-category of 2-algebra objects is equivalent to the $\infty$-category of algebras over the little 2-disk operad. This is analogous to classical centers in the 1-category of associative algebras, where the center of an associative algebra forms a commutative algebra.\par
By showing that the Hochschild complex can be identified with the derived center of an associative algebra, we therefore obtain a $C_{\ast}(D_2)$-algebra structure on $C^{\ast}(A,A)$. In other words, the Hochschild cochain complex carries the structure of an algebra over the chains on little 2-disks operad because it is a center in the $\infty$-category of homotopy associative algebras. In contrast to the homotopy Gerstenhaber algebra structure constructed by Tamarkin, this $C_{\ast}(D_2)$-algebra structure is canonical, because it is constructed by means of a universal property in some $\infty$-category.\par
On the level of underlying complexes, we prove the identification of the Hochschild complex with the derived center using \cite[Theorem 5.3.1.30]{HA}, which says that, under certain circumstances, the derived center of an algebra $A$ over an operad $\mathcal{O}$ can be computed as the endomorphism object of $A$ viewed as an $\mathcal{O}$-operadic $A$-module. However, to obtain a true solution to Deligne's conjecture, one needs to show that the center $C_{\ast}(D_2)$-algebra structure on $C^{\ast}(A,A)$ lifts the shifted bracket and cup product on cohomology discovered by Gerstenhaber. This was not previously known to hold for the center $C_{\ast}(D_2)$-structure. To prove that it does, we build on the proof of \cite[Theorem 5.3.1.30]{HA} and prove that, as predicted by Deligne, the two underlying products on the center can be identified with the composition and convolution product respectively. We then further show that, in a precise sense, this already uniquely determines the 2-algebra structure on a center up to homotopy. \par
We then develop an $\infty$-categorical Eckmann-Hilton Argument to compute the Gerstenhaber bracket on cohomology of the center of an associative algebra. The 1-cycle in the topological operad $D_2$ corresponding to the Gerstenhaber bracket is the ``double twist''
\begin{center}
\begin{tikzpicture}[
  scale=.9,
  line join=round,
  strand/.style={draw, line width=1.1pt, line cap=round},
  over/.style={
    draw,
    line width=1.1pt,
    line cap=round,
    preaction={draw=white, line width=5pt, line cap=butt}
  }
]

\def\a{0.72}

\foreach \y in {1,0,-1}
  \draw[dashed] (-1.25,\y) -- (1.25,\y);

\draw[strand]
  plot[domain=0:360,samples=250,smooth]
    ({ \a*cos(\x)},{ 2-\x/90 });

\draw[strand]
  plot[domain=0:360,samples=250,smooth]
    ({-\a*cos(\x)},{ 2-\x/90 });

\draw[over]
  plot[domain=78:102,samples=60,smooth]
    ({ \a*cos(\x)},{ 2-\x/90 });

\draw[over]
  plot[domain=258:282,samples=60,smooth]
    ({-\a*cos(\x)},{ 2-\x/90 });

\end{tikzpicture}
\end{center}
While the Dunn additivity theorem states that the data of a homotopy associative algebra in the $\infty$-category of homotopy associative algebras is equivalent to the data of a $D_2$-algebra, it is highly non-trivial to express the data of the corresponding $D_2$-algebra from a 2-algebra. In particular, it is difficult to directly express the double twist in terms of coherence data of a 2-algebra. Our $\infty$-categorical Eckmann-Hilton Argument therefore decomposes the double twist into four half-twists as shown in the above figure. The half-twists can easily be expressed in terms of the 2-algebra structure as the compatibility data showing that the second multiplication is a morphism of algebras, where the algebra structure is given by the first multiplication.\par
The same techniques can be used for the global case. Similar to how Kontsevich defined the complex of polydifferential operators for a smooth manifold, one can also define a sheaf of polydifferential operators for a smooth variety. The hypercohomology of this sheaf computes the Hochschild cohomology of the variety. The braces operations glue together to make the sheaf of polydifferential operators into a sheaf of $\braces$-algebras. Therefore, the Hochschild cohomology of a smooth variety again inherits a cup product and Gerstenhaber bracket. This definition does however not extend to singular varieties. Instead, we solve this problem by defining the Hochschild cochain complex of a general scheme $X$ over $\field$ as the derived center of the structure sheaf $\mathcal{O}_X$, viewed as an associative algebra in the $\infty$-category of dg sheaves via the forgetful functor from commutative to associative algebras. By construction, this center carries a $C_{\ast}(D_2)$-algebra structure in the $\infty$-category of dg sheaves, regardless of whether the scheme is singular. Again using the $\infty$-categorical Eckmann-Hilton Argument, we show that, for smooth varieties, this $C_{\ast}(D_2)$-algebra structure recovers the Gerstenhaber algebra structure on hypercohomology.\displaypar
We also give a partial answer to the third question on the meaning of the additional algebraic structure on the Hochschild complex in terms of the associated formal moduli space. Given a dg operad $\mathcal{O}$, the formal moduli problem of deforming a fixed $\mathcal{O}$-algebra $A$ is controlled by the so-called operadic tangent complex of $A$. This is the dg Lie algebra of derived $\mathcal{O}$-operadic derivations of $A$. For the (non-unital) associative operad, this dg Lie algebra agrees with the shifted Hochschild cochain complex equipped with the Gerstenhaber bracket (up to the degree 0 component).\par
To understand the meaning of the cup product and more generally the $C_{\ast}(D_2)$-structure on the Hochschild complex in terms of the operadic tangent complex, one has to consider a generalization of formal moduli problems in noncommutative geometry. Given a dg operad $\mathcal{O}$, one can consider the notion of a formal neighborhood of a point in an affine $\mathcal{O}$-scheme, where the category of affine $\mathcal{O}$-schemes is given by the opposite category of coconnective (i.e. non-positively cochain graded) $\mathcal{O}$-algebras. For the little k-disk operads $D_k$, this was introduced by Lurie in \cite{DAGX}, and for more general dg operads by Calaque, R. Campos and J. Nuiten in \cite{CCN}. In particular we have the notion of a $D_2$-formal moduli problem, which is an infinitesimal space modeled on $D_2$-algebras.\par
There is a generalization of the Lurie-Pridham correspondence to operadic formal moduli problems, which was proven in \cite{CCN}. This generalization asserts that, for sufficiently nice operads $\mathcal{O}$, the $\infty$-category of $\mathcal{O}$-operadic formal moduli problems is equivalent to the $\infty$-category of algebras over a dual operad $\mathcal{O}^!$. Since $D_2$ is self-dual (up to a shift), this implies that $D_2$-algebras control $D_2$-formal moduli problems, similar to how Lie algebras control commutative formal moduli problems. We can thus ask which $D_2$-formal moduli problem corresponds to the center of a given associative algebra. The answer to this was given by Lurie in \cite[Section 5.3]{DAGX}, or in slightly different terms by J. Francis in \cite{FR}. Namely, the center of $A$ controls the $D_2$-deformation theory of the $\field$-linear $\infty$-category of right modules over $A$, which is very closely related to the $D_2$-deformation theory of $A$ as an associative algebra. \par
The dg Lie algebra controlling a commutative formal moduli problem can geometrically be expressed as the Lie algebra of the formal group given by the loop space of the formal moduli space. Given an algebra $A$ over a dg operad $\mathcal{O}$, the loop space of the formal neighborhood of $A$ in the moduli space of $\mathcal{O}$-algebras is given by the formal automorphism group of $A$. This is the formal neighborhood of the identity on $A$ in the algebraic group of automorphisms of $A$ as an $\mathcal{O}$-algebra. The center of an $\mathcal{O}$-algebra $A$ can in turn be interpreted as the linearization of a would-be $\mathcal{O}$-algebra of endomorphisms of $A$. Specifically, by \cite[Theorem 5.3.1.30]{HA}, it can be identified with the $\mathcal{O}$-algebra of endomorphisms of $A$ as an $\mathcal{O}$-operadic $A$-module. Noting that the $\infty$-category of $\mathcal{O}$-operadic $A$-modules is the tangent category at $A$ of the $\infty$-category of $\mathcal{O}$-algebras, this provides an abstract relation between the center and the formal moduli space of an $\mathcal{O}$-algebra.\displaypar
The final question on the implications of this theory for the Grothendieck-Teichmüller group was the initial motivation for this project. It is not answered in this thesis, but is instead joint work in progress with my advisor Chris Rogers.
\section*{Overview and main results}
We give an overview of the main results and the structure of this thesis. The main Theorems \ref{ithmA}, \ref{ithmB}, \ref{ithmC}, and \ref{ithmD} appear in my recent preprint \cite{FARR}, which was submitted to \textit{Advances in Mathematics} in June 2025. \displaypar
In Chapter \ref{Kontsevich_formality_and_deformation_quantization}, we give an exposition on Kontsevich's formality theorem and its generalizations. We define Hochschild cohomology for algebras and smooth varieties, and survey the different algebraic structures it carries. We discuss the relation of Kontsevich's formality morphism to the Duflo isomorphism in Lie algebra theory. Finally, we introduce the Grothendieck-Teichmüller group and its relation to the (motivic) absolute Galois group, and give an account of its action on formality morphisms as in \cite{DRW}.\par
Chapters \ref{infty_operads_and_centers} and \ref{The_hochschild_complex_as_a_center} contain the main new results contained in this thesis. In Chapter \ref{infty_operads_and_centers}, we review Lurie's theory of $\infty$-operadic centers. We define the little $k$-disks operads in the form of the $\mathbb{E}_k$-operads, and describe Lurie's version of the Dunn additivity theorem. The main result of Chapter \ref{infty_operads_and_centers} is the following description of the 2-algebra structure of the center of an $\mathbb{E}_1$-algebra.
\begin{Itheorem}[Corollary \ref{cor7}]\label{ithmA}
Let $\mathcal{C}$ be a symmetric monoidal $\infty$-category, and let $A\in \alg_{\mathbb{E}_1}(\mathcal{C})$. The underlying object of the center $\mathfrak{Z}(A)\in \alg_{\mathbb{E}_1}(\alg_{\mathbb{E}_1}(\mathcal{C}))$ can be identified with the endomorphism object of the $A$-bimodule $A$. The outer multiplication is given by the composition product $\circ$, and the inner multiplication is given by the convolution product $\ast$. Further, there is a contractible choice of fillings of the compatibility square
\begin{center}
\begin{tikzcd}[ampersand replacement=\&, column sep=large]
\mathfrak{Z}(A)^{\otimes 4} \arrow[d, "(\ast \otimes \ast) (\textnormal{id}\otimes \tau\otimes \textnormal{id})"'] \arrow[r, "\circ \otimes \circ"] \& \mathfrak{Z}(A)^{\otimes 2} \arrow[d, "\ast"] \\
\mathfrak{Z}(A)^{\otimes 2} \arrow[r, "\circ"']                                                                                           \& \mathfrak{Z}(A)                               
\end{tikzcd}
\end{center}
in $\mathcal{C} \times_{{}_A\textnormal{BiMod}_A(\mathcal{C})} {}_A\textnormal{BiMod}_A(\mathcal{C})/A$. 
\end{Itheorem}
In Chapter \ref{The_hochschild_complex_as_a_center}, we prove that the $\mathbb{E}_1$-center indeed recovers the classical Gerstenhaber bracket and cup product on Hochschild cohomology, both for associative algebras and smooth varieties. In Section \ref{The_bracket_on_a_2_algebra}, we define the Gerstenhaber bracket on an $\mathbb{E}_2$-algebra using the double twist, and show how to use an $\infty$-categorical Eckmann-Hilton argument to compute the Gerstenhaber bracket of a 2-algebra. The main result of this section is Corollary \ref{cor1}. Then, in Section \ref{Recovering_the_Hochschild_complex_as_e1_center}, we use this result to prove that the Gerstenhaber bracket on the center of an associative algebra agrees with the classical one. For this, we use techniques from dg category theory, such as the dg nerve defined in \cite[Section 1.3.1]{HA}. The main result of this section is
\begin{Itheorem}[Theorem \ref{thm1}, Corollary \ref{cor9}]\label{ithmB}
    Let $A$ be an associative $\field$-algebra. The Hochschild complex
    \begin{align*}
        C^{\ast}(A,A) = \hom_{\field}(A^{\otimes \ast},A)
    \end{align*}
    together with the evaluation map is a center for $A\in \alg_{\mathbb{E}_1}(\dk)$. In particular, it is an object of $\alg_{\mathbb{E}_1}(\alg_{\mathbb{E}_1}(\dk))\simeq \alg_{\mathbb{E}_2}(\dk)$. Its underlying Gerstenhaber bracket and cup product in cohomology agree with the classical Gerstenhaber algebra structure obtained from the $\braces$-algebra structure.
\end{Itheorem}
Section \ref{The_Hochschild_complex_of_a_scheme} is dedicated to the global case of the Hochschild cohomology of a scheme over $\field$. We define the Hochschild complex of a scheme as the $\mathbb{E}_1$-center of its structure sheaf, and characterize some basic properties of this object. The main result of this section is that the center of the structure sheaf behaves well with respect to the local structure on the scheme, namely
\begin{Itheorem}[Theorem \ref{thm5}]\label{ithmC}
Let $X$ be a quasi-compact separated scheme over $\field$, and denote by $\text{Sh}_{\infty}(X)$  the $\infty$-category of dg sheaves on $X$. Let $U = \textnormal{Spec}(A)\subseteq X$ be an affine open. The functor $\mathbb{R}\Gamma_U: \textnormal{Sh}_{\infty}(X) \rightarrow \dk$ is lax symmetric monoidal and hence induces a functor $\mathbb{R}\Gamma_U: \alg_{\mathbb{E}_2}(\textnormal{Sh}_{\infty}(X)) \rightarrow \alg_{\mathbb{E}_2}(\dk)$. We have
\begin{align*}
    \mathbb{R}\Gamma_U(\mathfrak{Z}_{\mathbb{E}_1}(\mathcal{O}_X)) \simeq \mathfrak{Z}_{\mathbb{E}_1}(A).
\end{align*}
\end{Itheorem}
We then compare the induced $\mathbb{E}_2$-algebra structure to the homotopy Gerstenhaber structure on the sheaf of polydifferential operators of a smooth variety. This yields the global comparison theorem
\begin{Itheorem}[Theorem \ref{thm6}, Corollary \ref{thm9}]\label{ithmD}
For a smooth quasi-compact separated scheme $X$ of finite type over $\field$, the sheaf of polydifferential operators $\mathcal{D}_{\textnormal{poly}}(X)$ is a center of $\mathcal{O}_X$:
\begin{align*}
    \mathcal{D}_{\textnormal{poly}}(X) \simeq \mathfrak{Z}_{\mathbb{E}_1}(\mathcal{O}_X).
\end{align*}
This equips $\mathcal{D}_{\textnormal{poly}}(X)$ with the structure of an $\mathbb{E}_2$-algebra. The corresponding Gerstenhaber algebra in the $\field$-linear derived 1-category agrees with the classical one coming from the $\braces$-algebra structure.
\end{Itheorem}
In the final Chapter \ref{formal_moduli_problems_and_deformations}, we discuss the relation between the center of an algebra over an operad and its deformation theory. We describe the structure of the moduli space of algebras over an operad, and how this relates to the deformation theory of such an algebra. We then give an exposition on commutative and operadic formal moduli problems, following \cite{GR2} and \cite{CCN} respectively. We examine how the center of an operadic algebra corresponds to the formal moduli problem of the algebra using the generalized Lurie-Pridham correspondence, and finally relate this back to Kontsevich's solution of deformation quantization. While parts of this chapter are only conjectural, it is the first compilation of this information into a coherent picture. 
\section*{Conventions} 
Throughout this thesis, $\field$ will be a field of characteristic 0. In Chapters \ref{infty_operads_and_centers} and \ref{The_hochschild_complex_as_a_center}, complexes are generally chain graded unless stated otherwise, and we view non-negatively graded cochain complexes as non-positive chain complexes. In Chapters \ref{Kontsevich_formality_and_deformation_quantization} and \ref{formal_moduli_problems_and_deformations}, complexes are generally cochain graded to match the usual convention in the literature. For a cochain complex $C^{\ast}$, we denote by $C^{\ast}[n]$ the cochain complex with $C^k[n] = C^{k+n}$. We denote the 1-category of $\field$-vector spaces by $\text{Vect}_{\field}$, the dg category of (co)chain complexes of $\field$-vector spaces by $\ch(\field)$, and the derived $\infty$-category of $\field$-vector spaces by $\dk$. Singular (co)chains of a topological space are always taken with $\field$-coefficients.\par
The term ``operad'' refers to not necessarily reduced unital symmetric operads. We denote $\infty$-operads by calligraphic letters (e.g. $\ass$) and dg (co)operads as well as dg $\infty$-(co)operads by bold letters (e.g. $\assoc$). By ``Koszul operad'' we mean a binary quadratic Koszul operad. In particular, Koszul operads are reduced and canonically augmented. We always assume dg cooperads to be conilpotent. For a dg operad $\mathscr{P}$, we denote by $\mathscr{P}\{k\}$ its shift by $k$; in particular, $A$ is a $\mathscr{P}$-algebra if and only if $A[-k]$ is a $\mathscr{P}\{k\}$-algebra.\par
We try to use the term ``$\infty$-category'' for $\infty$-categories, but if nothing else is stated, ``category'' refers to $\infty$-category and ``1-category'' refers to ordinary categories. The $\infty$-category of spaces is denoted by $\an$. If $C$ is a model category, we denote by $C^c$ the full subcategory of cofibrant objects, and by $C^{\circ}$ the full subcategory of fibrant-cofibrant objects. We denote by $R$ and $Q$ the fibrant and cofibrant replacement functor respectively. \par
We denote the presheaf tensor product simply by ``$\otimes$'', and we decorate symbols with ``$(-)^a$'' to indicate sheafification.
\section*{Acknowledgments}
I want to thank my advisor, Christopher Rogers, for his guidance and friendship, and for sharing his insights with me. His way of thinking about mathematics and mathematicians has had a great impact on me. I would also like to thank Jonathan Beardsley, Christopher Herald, Bruce Blackadar, and Matthew Tucker for serving on my committee.\displaypar
Special thanks go to Marc Levine and Damien Calaque, who each spent a great deal of time and effort answering my questions and sharing their ideas. I would like to thank J. Beardsley, D. Borisov, A. Căldăraru, M. Hopkins, G. Horel, A. Joyal, J. Pridham, N. Rozenblyum, and M. Spitzweck for helpful discussions and for sharing their perspectives on this project.\displaypar
I am very grateful to my family, and especially to my parents Andrea Rahn-Farr and Karsten Farr, for their continuous support and encouragement. \par
Finally, I want to thank my wonderful husband, Dominik Farr, for going on this adventure with me, for reminding me to have fun along the way, and for giving me strength during the difficult times. I could not have done this without him. \displaypar
This work was supported by NSF Grant DMS-2305407.
\newpage
\chapter{Kontsevich formality and deformation quantization}\label{Kontsevich_formality_and_deformation_quantization} 
The motivation for our work on higher centers is to realize the program developed by Kontsevich in \cite{K2} on the connection between Hochschild cohomology, deformation theory, motives and the Grothendieck–Teichmüller group. Based on his original formality theorem in deformation quantization, and Tamarkin's generalization of it, Kontsevich conjectures that the pro-unipotent Grothendieck–Teichmüller group, or rather its subgroup given by the motivic Galois group of $\mathbb{Q}$, acts on formality morphisms for the smooth Hochschild cochain complex, as well as on the cohomology of polyvector fields themselves. In this chapter, we give some background on Kontsevich's program, and review the current state of these conjectures.\displaypar
In the same paper \cite{K2}, Kontsevich also states a ``generalized Deligne conjecture'', claiming that every $n$-algebra admits a universal $(n+1)$-algebra acting on it. He calls this universal object the Hochschild complex of the $n$-algebra, and he claims that it can be computed as a quotient of the deformation complex. The higher centers considered in Chapters \ref{infty_operads_and_centers} and \ref{The_hochschild_complex_as_a_center} are a rigorous formalization of this idea, and we therefore expect them to be instrumental in resolving the remaining questions in Kontsevich's program.
\section{Kontsevich's formality morphism}\label{Kontsevich_formality_morphism} 
Let $M$ be a smooth manifold. A Poisson structure on $M$ is given by a Poisson bracket on the commutative algebra $C^{\infty}(M)$ of smooth functions on $M$. Given such a Poisson bracket $\{-,-\}$ on $M$, one can try to extend it to an associative product $\star$ on $C^{\infty}(M)[[\hbar]]$ of the form
\begin{align*}
    f \star g  = fg + \{f,g\}\hbar + \text{higher order terms in }\hbar.
\end{align*}
Such an associative product is called a \textbf{star product} on $M$. The process of finding a star product is called \textbf{formal quantization} of the Poisson structure. In 1997, Kontsevich proved \cite{K} that every Poisson manifold can be formally quantized, and that its quantization is canonical up to gauge action. To prove this, Kontsevich introduced two dg Lie algebras $T_{\text{poly}}^{\ast}(M)[1]$ and $D_{\text{poly}}^{\ast}(M)[1]$, such that Maurer-Cartan elements in $T_{\text{poly}}^{\ast}(M)[1]$ correspond to Poisson structures on $M$, while Maurer-Cartan elements in $D_{\text{poly}}^{\ast}(M)[1] \otimes_{\mathbb{R}} \hbar \mathbb{R}[[\hbar]]$ correspond to star products. Note that in particular, a Maurer-Cartan element in $T_{\text{poly}}^{\ast}(M)[1]$ yields a Maurer-Cartan element in $T_{\text{poly}}^{\ast}(M)[1]\otimes_{\mathbb{R}} \hbar\mathbb{R}[[\hbar]]$. Kontsevich then exhibits an explicit $L_{\infty}$-quasi-isomorphism 
\begin{align*}
\mathscr{U}: T_{\text{poly}}^{\ast}(M)[1] \rightarrow D_{\text{poly}}^{\ast}(M)[1],
\end{align*}
which lifts the Hochschild-Kostant-Rosenberg quasi-isomorphism of the underlying cochain complexes. Consequently, we get a correspondence between the respective Maurer-Cartan elements, proving that every Poisson manifold can be formally quantized. \displaypar 
In particular, define the dg Lie algebra of \textbf{polyvector fields} as the complex
\begin{align*}
T_{\text{poly}}^{\ast}(M)[1] := \Gamma(M, \sym^{\ast+1}(T_M[-1]))
\end{align*}
equipped with the zero differential. The Lie bracket is given by the Schouten bracket, which agrees with the commutator of vector fields in homogeneous degree 0. The dg Lie algebra of \textbf{polydifferential operators} $D_{\text{poly}}^{\ast}(M)[1]$ is a sub dg Lie algebra of the shifted Hochschild cochain complex $C^{\ast}(C^{\infty}(M), C^{\infty}(M))[1]$, which in degree $n \geq -1$ is given by all $f\in \hom_{\mathbb{R}}(C^{\infty}(M)^{\otimes (n+1)}, C^{\infty}(M))$ which are differential operators in each variable separately. It is equipped with the shifted Gerstenhaber bracket. 
\section{Deligne's conjecture on Hochschild cochains}\label{delignes_conjecture_on_hochschild_cochains}
Both polyvector fields and polydifferential operators support additional algebraic structure. The dg Lie algebra of polyvector fields admits a commutative product of degree 1 coming from the product on the symmetric algebra $\sym^{\ast}(T_M[-1])$, and this commutative product makes the unshifted version $T_{\text{poly}}^{\ast}(M)$ into a \textbf{Gerstenhaber algebra}\footnote{Note that a Gerstenhaber algebra is the same as a $\mathbb{P}_2$-algebra, i.e. a shifted Poisson algebra with Poisson bracket of cochain degree -1.}. The dg Lie algebra of polydifferential operators also admits a product, which is inherited from the cup product on the Hochschild complex. This product is noncommutative, but on cocycles it is commutative up to a boundary. It was shown by Gerstenhaber \cite{G} that it makes the cohomology of $D_{\text{poly}}^{\ast}(M)$ into a Gerstenhaber algebra as well.\displaypar
By a result of Cohen \cite{Coh}, the Gerstenhaber operad is the singular homology of the topological operad $D_2$ of little 2-disks. In particular, if $A$ is an algebra over $C_{\ast}(D_2)$, then $H_{\ast}(A)$ is an algebra over the operad $\ger$ of Gerstenhaber algebras. This inspired Deligne in a 1993 letter to make the following conjecture.
\begin{conjecture}[Deligne's Conjecture]
    The Hochschild cochain complex of an associative $\field$-algebra $A$ is an algebra over the chains on little 2-disks operad in such a way that the induced Gerstenhaber algebra structure on cohomology recovers Gerstenhaber's original one.
\end{conjecture}
Note that the category of $A$-bimodules admits a monoidal structure by taking the tensor product over $A$. In his letter, Deligne explained that the $C_{\ast}(D_2)$-algebra structure on the cochain level should come from the two multiplications on $C^{\ast}(A,A) \simeq \mathbb{R}\hom_{A\otimes A^{\text{op}}}(A,A)$ induced by $A$ being a monoidal unit in this category. We have an ``inner multiplication'' coming from the bialgebra structure of $A$ in the bimodule category: Let $\e(A) := \hom_{A\otimes A^{\text{op}}}(A,A)$. Then we get a map
\begin{align*}
    \e(A) \otimes \e(A) \otimes A \xrightarrow{\cong} \e(A) \otimes \e(A) \otimes (A \otimes_A A) \\\cong \e(A) \otimes A \otimes_A \e(A) \otimes A \xrightarrow{\text{ev}\otimes \text{ev}} A \otimes_A A \xrightarrow{\cong} A
\end{align*}
inducing a multiplication $\e(A) \otimes \e(A) \rightarrow \e(A)$. We also have an ``outer multiplication'' given by composition. \displaypar
Multiple different proofs have been given for Deligne's Conjecture, see for example \cite{Tam}, \cite{Vor}, \cite{MCS}. Tamarkin \cite{Tam} solves the Deligne Conjecture by constructing a map $\Psi_T: \ger_{\infty} \rightarrow \braces$ from the operad $\ger_{\infty}$ of homotopy Gerstenhaber algebras to the Braces-operad defined in \cite{GV}, which canonically acts on the Hochschild cochain complex. This map $\Psi_T$ is compatible with the respective morphisms from the $L_{\infty}$-operad into $\ger_{\infty}$, and the Lie operad into $\braces$. In addition, Tamarkin proves a formality result for the little 2-disks operad, stating that $C_{\ast}(D_2)$ is quasi-isomorphic to its homology $H_{\ast}(D_2) \simeq \ger$, and therefore  $C_{\ast}(D_2)$ is also quasi-isomorphic to $\ger_{\infty}$. This equips $C^{\ast}(A,A)$ with a $C_{\ast}(D_2)$-algebra structure. \displaypar
Tamarkin's results in \cite{Tam} also give a new proof of Kontsevich's formality theorem\footnote{Tamarkin's argument works for $M= \mathbb{R}^n$, so $A = C^{\infty}(M) = \mathbb{R}[x_1,\dots,x_n]$.  It was shown in \cite{DTT} that Tamarkin's formality theorem remains true for any regular commutative algebra $A$.}. This new proof needs as only input Tamarkin's solution of Deligne's conjecture in the form of the morphism $\Psi_T$. From there, one can show that $T_{\text{poly}}^{\ast}(A) \simeq D_{\text{poly}}^{\ast}(A)$ as $\ger_{\infty}$-algebras by computing certain obstructions. This highlights the dependencies among these results. By construction, Tamarkin's version of the formality morphism is not only a $L_{\infty}$-quasi-isomorphism, but a $\mathbf{Ger}_{\infty}$-quasi-isomorphism.
\section{Formality for non-affine smooth varieties}
The classical formality theorem applies to smooth manifolds (which are affine $C^{\infty}$-schemes by \cite{J}) and smooth affine varieties. To extend it to the case of non-affine smooth varieties, one first needs a global version for polyvector fields and polydifferential operators. \displaypar
To this end, let $X$ be a smooth variety over $\field$. The \textbf{tangent sheaf of} $X$ is given by the $\mathcal{O}_X$-module $\mathcal{T}_X = \mathcal{H}om_{\mathcal{O}_X}(\Omega_{X/\field}, \mathcal{O}_X)$. 
\begin{definition}
The \textbf{sheaf of polyvector fields on} $X$ is given by 
\begin{align*}
    \mathcal{T}_{\textnormal{poly}}^{\ast}(X) := \sym_{\mathcal{O}_X} \mathcal{T}_X[-1].
\end{align*}
This is a sheaf of Gerstenhaber algebras with the symmetric algebra product and the Schouten bracket.
\end{definition}
Note however that the Schouten bracket is not $\mathcal{O}_X$-linear, so this is a Gerstenhaber algebra only as a sheaf of $\field$-cochain complexes. \displaypar
The situation for the Hochschild complex is somewhat more complicated, since $C^{\ast}(A,A)$ is not functorial in $A$, and hence does not glue together to a sheaf. For general varieties $X$ over $\field$, there have been multiple proposed definitions, for example by Grothendieck \cite{Gr} and J.-L. Loday \cite{Lo}, Gerstenhaber-D. Schack \cite{GS} and G. Swan \cite{S}. The perhaps most straightforward generalization from the affine case was given by Swan, who defined the Hochschild cohomology of $X$ to be
\begin{align*}
    \text{HH}^{\ast}(X) := \text{Ext}^{\ast}_{\mathcal{O}_{X\times_{\field} X}}(\Delta_{\ast}\mathcal{O}_X,\Delta_{\ast} \mathcal{O}_X).
\end{align*}
Unfortunately, this does not carry a Gerstenhaber bracket. For the case that $X$ is smooth, Kontsevich \cite{K} gave an alternative definition, which does carry the structure of a Gerstenhaber algebra.
\begin{defprop}[\cite{K}]
There exists a complex of quasi-coherent $\mathcal{O}_X$-modules $\mathcal{D}_{\textnormal{poly}}^{\ast}(X)$ such that on affine opens $U = \spec A$, 
\begin{align*}
    \mathcal{D}^{\ast}_{\textnormal{poly}}(X)(\spec(A)) \cong D^{\ast}_{\textnormal{poly}}(A).
\end{align*}
This is the \textbf{sheaf of polydifferential operators on} $X$. It is a sheaf of $\textbf{Braces}$-algebras, but again this algebraic structure is not $\mathcal{O}_X$-linear.
\end{defprop}
Using Tamarkin's map $\Psi_T$, the sheaf of polydifferential operators becomes a sheaf of $\mathbf{Ger}_{\infty}$-algebras. In particular, the Hochschild cohomology of $X$, defined as the hypercohomology of the sheaf of polydifferential operators
\begin{align*}
    \text{HH}^{\ast}(X) := \mathbb{H}^{\ast}(X,\mathcal{D}^{\ast}_{\text{poly}}(X)),
\end{align*}
obtains a Gerstenhaber algebra structure.\displaypar
In Section \ref{the_infty_category_of_dg_sheaves}, we give another definition for the Hochschild complex of a not necessarily smooth, separated, quasi-compact scheme $X$ over $\field$, that is given in terms of the $\mathbb{E}_1$-center of the structure sheaf in an $\infty$-category of dg sheaves on $X$. This definition automatically comes equipped with a canonical little 2-disks algebra structure, even in the singular case. The difficulty instead lies in proving that this is a sensible definition for the Hochschild complex. This is done in Theorem \ref{thm1} and Theorem \ref{thm6}. \displaypar
Using these definitions for the sheaf of polyvector fields and the sheaf of polydifferential operators, Calaque and Van den Bergh successfully extended Kontsevich's formality theorem to smooth varieties.
\begin{theorem}[Theorem 1.3 \cite{CVdB}]\label{thm12}
Let $X$ be a smooth variety over $\field$, and let $\td(X)$ denote the Todd class of $X$. The morphism
\begin{align*}
    \mathcal{T}_{\textnormal{poly}}^{\ast}(X) \xrightarrow{I_\textnormal{HKR}\circ \td(X)^{1/2}\wedge -}\mathcal{D}_{\textnormal{poly}}^{\ast}(X)
\end{align*}
is an isomorphism of Gerstenhaber algebras in the derived 1-category of sheaves of $\field$-cochain complexes on $X$.
\end{theorem}
In contrast to the affine case, in order for the formality isomorphism to respect the product and Gerstenhaber bracket, one needs to correct the Hochschild-Kostant-Rosenberg map $I_\text{HKR}$ by the contraction with the formal square root of the Todd class. The Todd class is defined as
\begin{align*}
    \td(X) := \det(q(\at(X))) \in \bigoplus_{n\geq 0} H^n(X,\Omega^n_X),
\end{align*}
where $\at(X) \in \text{Ext}^1(\mathcal{T}_X, \Omega^{\ast}_X \otimes_{\mathcal{O}_X} \mathcal{T}_X)$ is the Atiyah class of the tangent sheaf of $X$, and $q$ is the formal power series
\begin{align*}
    q(x) = \frac{x}{1-e^{-x}}.
\end{align*}
In particular, the Todd class can be formally expanded in terms of the scalar Atiyah classes
\begin{align*}
    a_i(X) \in H^i\left(X, (\mathcal{T}_{\textnormal{poly}}^i(X))^{\vee}\right),
\end{align*}
and therefore acts by contraction on $\mathcal{T}_{\text{poly}}^{\ast}(X)$. \displaypar
Notably, the above formality morphism is defined only in the derived 1-category. This is a consequence of the choice of local formality morphism it is built from. Forgetting their respective products, $\mathcal{T}_{\text{poly}}^{\ast}(X)[1]$ and $\mathcal{D}_{\text{poly}}^{\ast}(X)[1]$ are sheaves of dg Lie algebras, and Calaque and Van den Bergh construct their formality morphism essentially by gluing together Kontsevich's local formality morphisms described in Section \ref{Kontsevich_formality_morphism}. This gluing yields a zig-zag of quasi-isomorphisms of sheaves of dg Lie algebras, and in particular an isomorphism of Lie algebras in the derived 1-category of sheaves of $\field$-cochain complexes. Kontsevich argued in \cite{K} that his local formality morphism is compatible with cup products up to homotopy, and hence the global isomorphism of Lie algebras constructed above lifts to an isomorphism of Gerstenhaber algebras. The drawback of this strategy is that we do not obtain a zig-zag of quasi-isomorphisms of sheaves of $\ger_{\infty}$-algebras, since it is unknown if the cup product admits higher compatibility with the dg Lie algebra structure. The advantage, however, is that we can explicitly write down the obtained morphism in the derived 1-category of sheaves of $\field$-cochain complexes. In a second paper \cite{CVdB2}, Calaque and Van den Bergh instead use Tamarkin's local formality morphism to obtain a zig-zag of quasi-isomorphisms of sheaves of $\ger_{\infty}$-algebras directly, which then again implies the existence of an isomorphism of Gerstenhaber algebras in the derived 1-category. Since Tamarkin's local formality morphism is obtained by obstruction theoretic arguments, rather than an explicit construction, they are not able to further determine this induced isomorphism; in particular, it is not clear if it agrees with the one constructed from Kontsevich's formality morphism.
\section{The formality theorem and the Duflo isomorphism}
The correction term $\det(q(\at(X)))^{1/2}$ in the global formality theorem \ref{thm12} looks very similar to the correction term in the \textbf{Duflo isomorphism theorem}.
\begin{theorem}[\cite{D}, \cite{PT}]
Let $\mathfrak{g}$ be a finite-dimensional Lie algebra over $\field$. The morphism
\begin{align*}
    H^{\ast}(\mathfrak{g},\sym(\mathfrak{g})) \xrightarrow{I_{\textnormal{PBW}}\circ J^{1/2}} H^{\ast}(\mathfrak{g}, U(\mathfrak{g}))
\end{align*}
is an isomorphism of graded commutative algebras. Here $J\in \widehat{\sym}(\mathfrak{g}^{\vee})$ is the \textbf{Duflo element}, with $J(x) = \det(q'(\textnormal{ad}_x))$ and
\begin{align*}
    q'(x) = \frac{1-e^{-x}}{x}.
\end{align*}
\end{theorem}
This is, of course, no coincidence. In fact, there are multiple related ways to link these two theorems together, see for example \cite[Section 8.3.4]{K} and \cite[Section 5.2]{CR}. In this section, we highlight two of these connections. \displaypar
The first one links the Duflo isomorphism theorem to the local formality morphism constructed by Kontsevich. Let $\mathfrak{g}$ be a finite-dimensional Lie algebra over $\field$. The Lie algebra structure on $\mathfrak{g}$ makes the dual space $\mathfrak{g}^{\vee}$ into a \textbf{Poisson variety} by uniquely extending the bracket on $\mathfrak{g}$ to a Poisson bracket on the symmetric algebra $\mathcal{O}_{\mathfrak{g}^{\vee}} \cong \sym(\mathfrak{g})$. This Poisson structure corresponds to a Maurer-Cartan element $\pi$ in the dg Lie algebra $T^{\ast}_{\text{poly}}(\mathfrak{g}^{\vee})[1]$. Tensoring with $\hbar \field[[\hbar]]$, this yields a Maurer-Cartan element $\hbar\pi$ in $T^{\ast}_{\text{poly}}(\mathfrak{g}^{\vee})[1] \otimes \hbar \field[[\hbar]]$. Since $\mathscr{U}$ is an $L_{\infty}$-quasi-isomorphism, it determines a corresponding Maurer-Cartan element $\tilde{\pi}$ of $D_{\text{poly}}^{\ast}(\mathfrak{g}^{\vee})[1]\otimes \hbar \field[[\hbar]]$, and hence an associative product $\star$ on $\mathcal{O}_{\mathfrak{g}^{\vee}}[[\hbar]] \cong \sym(\mathfrak{g})[[\hbar]]$. 
\begin{proposition}[Theorem 8.2 \cite{K}]
Setting $\hbar = 1$ yields a well-defined product $\star$ on $\sym(\mathfrak{g})$ such that for $x,y\in \mathfrak{g}$, 
\begin{align*}
    x\star y - y\star x = [x,y].
\end{align*}
We therefore have a unique isomorphism of algebras
\begin{align*}
    I_{\textnormal{alg}}: U(\mathfrak{g}) \xrightarrow{\cong} (\sym(\mathfrak{g}),\star).
\end{align*}
\end{proposition}
We have an induced $L_{\infty}$-quasi-isomorphism between the twisted complexes 
\begin{align*}
    \mathscr{U}^{\pi}: (T^{\ast}_{\text{poly}}(\mathfrak{g}^{\vee})[1] \otimes \hbar \field[[\hbar]])^{\hbar \pi} \rightarrow (D_{\text{poly}}^{\ast}(\mathfrak{g}^{\vee})[1]\otimes \hbar \field[[\hbar]])^{\tilde{\pi}}.
\end{align*}
Its tangent map $(\mathscr{U}^{\pi})_1$ therefore yields a quasi-isomorphism of the underlying twisted complexes. One can again define this map for $\hbar = 1$, thus obtaining a quasi-isomorphism
\begin{align*}
    (\mathscr{U}^{\pi})_1: (T^{\ast}_{\text{poly}}(\mathfrak{g}^{\vee})[1])^{\pi} \rightarrow (D_{\text{poly}}^{\ast}(\mathfrak{g}^{\vee})[1])^{\tilde{\pi}}.
\end{align*}
Note in particular that the right hand side is just the shifted Hochschild cochain complex of the associative algebra $(\sym(\mathfrak{g}), \star)$. On degree 0 cochains, this yields an endomorphism $I_T$ of $\sym(\mathfrak{g})$.
\begin{theorem}[Theorem 8.3.4 \cite{K}]
The following diagram commutes
\begin{center}
\begin{tikzcd}[column sep=large]
{H^{\ast}(\mathfrak{g}, \sym(\mathfrak{g}))} \arrow[r, "I_{\text{PBW}}\circ J^{1/2}"'] \arrow[rr, "I_T", bend left] & {H^{\ast}(\mathfrak{g},U(\mathfrak{g}))} \arrow[r, "I_{\text{alg}}"'] & {H^{\ast}(\mathfrak{g}, (\sym(\mathfrak{g}), \star))}.
\end{tikzcd}
\end{center}
In particular, the Duflo morphism $I_{\text{PBW}}\circ J^{1/2}$ is given by $I_{\text{alg}}^{-1}\circ I_T$.
\end{theorem}
This means that, up to the identification of the quantized symmetric algebra with the universal enveloping algebra, the Duflo isomorphism is given by the tangent map of the local Kontsevich formality morphism at the Maurer-Cartan element defining the Lie algebra structure on $\mathfrak{g}$. The Hochschild-Kostant-Rosenberg map, which is the tangent map of the local formality morphism at the trivial Poisson structure, for the affine variety $\mathfrak{g}^{\vee}$ corresponds to the PBW morphism, so the above discussion shows that the Duflo correction term $J^{1/2}$ can be identified with the change in tangent map induced by twisting with the Maurer-Cartan element encoding the Lie algebra structure on $\mathfrak{g}$. \displaypar
In \cite{CVdB}, Calaque and Van den Bergh employ a very similar strategy to prove the global formality theorem. Given a smooth variety $X$, they construct a particular resolution of the structure sheaf $\mathcal{O}_X$, which carries a canonical connection. This connection can be identified with a Maurer-Cartan element in the dg Lie algebra of polyvector fields on this resolution, which is then used to twist the Kontsevich formality map.\displaypar
While this approach shows how to obtain the Lie theoretic Duflo morphism from the formality theorem, one can also understand the formality theorem through the lens of the Duflo isomorphism theorem. \displaypar
To this end, it was proven by M. Kapranov in \cite{Ka} that the shifted tangent sheaf $\mathcal{T}_X[-1]$ of a smooth variety $X$ is endowed with the structure of a Lie algebra in the derived 1-category of quasi-coherent sheaves on $X$. This Lie bracket, commonly called the \textbf{Atiyah Lie bracket}, is given by the Atiyah class of the tangent sheaf, which can be viewed as a morphism in the derived 1-category
\begin{align*}
    \alpha_X: \mathcal{T}_X \otimes \mathcal{T}_X \rightarrow \mathcal{T}_X[1].
\end{align*}
In addition, any quasi-coherent sheaf on $X$ admits an action by this Lie algebra via its Atiyah class. Therefore, the derived 1-category of quasi-coherent sheaves on $X$ can be identified with the category of representations of the Lie algebra $\mathcal{T}_X[-1]$. The adjoint action of this Lie algebra corresponds to the Atiyah class of the tangent sheaf, and therefore the global formality theorem \ref{thm12} can be identified with the Duflo isomorphism theorem\footnote{Note that the Duflo isomorphism theorem has not been proven for general tensor categories other than $\text{Vect}_{\mathbb{R}}$.} for $\mathfrak{g} = \mathcal{T}_X[-1]$. \displaypar
To make this analogy precise, we work in the context of derived algebraic geometry. In particular, in this section, all constructions are assumed to be ``derived''. We use notation and results from \cite[Chapters 5, 6, 7, and 8]{GR2}. For partial accounts of definitions used, see Section \ref{formal_moduli_problems_and_operadic_formal_moduli_problems} and Appendix \ref{appendixB}.\displaypar
Let $X$ be a (derived) scheme. The $\infty$-category $\textbf{FormGrpd}(X)$ of formal groupoids on $X$, as defined in \cite[Chapter 5, Section 2.2]{GR2}, has a final object $\widehat{X\times X}$, given by the completion of the pair groupoid of $X$ along the diagonal. Its Lie algebroid is given by the \textbf{tangent Lie algebroid} $\mathbb{T}_{X} \xrightarrow{\text{id}} \mathbb{T}_{X}$, with $\mathbb{T}_X$ the tangent complex of $X$ equipped with the commutator bracket. There is a functor 
\begin{align*}
    \Omega^{\text{fake}}: \textbf{FormGrpd}(X) \rightarrow \text{Grp}(\textbf{FMP}_{/X})
\end{align*}
given by viewing a formal groupoid $\mathcal{R}$ over $X$ as a pointed formal moduli problem over $X$ via the maps $\text{unit}: X\rightleftarrows \mathcal{R}: p_s$, and then taking loops. Applying this to the formal pair groupoid, we obtain the \textbf{formal inertia group} of $X$
\begin{align*}
    \widehat{\mathcal{L}X} := \Omega^{\text{fake}}(\widehat{X\times X}) \simeq X\times_{\widehat{X \times X}} X.
\end{align*}
This agrees with the formal completion of the free loop space $\mathcal{L}X$ at the unit section. The Lie algebra of the formal group $\widehat{\mathcal{L}X}$ is an object $\textnormal{Lie}_X(\widehat{\mathcal{L}X}) \in \alg_{\lie}(\indcoh(X))$, and by \cite[Chapter 7, Corollary 3.2.8]{GR2} has underlying ind-coherent sheaf
\begin{align*}
    \text{forget}_{\lie}(\textnormal{Lie}_X(\widehat{\mathcal{L}X})) \simeq \mathbb{T}_X[-1].
\end{align*}
The Lie algebra $\textnormal{Lie}_X(\widehat{\mathcal{L}X})$ canonically acts on any $\mathcal{F}\in \indcoh(X)$ as described in \cite[Chapter 8, Section 6.1]{GR2}. By \cite[Chapter 8, Section 6.1.4]{GR2}, if $\mathcal{F} = \Upsilon_X(\mathcal{E})$ for some $\mathcal{E}\in \qcoh(X)$, where $\Upsilon_X: \qcoh(X) \rightarrow \indcoh(X)$ as described in (\ref{eq4}), this action can be identified with the one induced by the canonical map $\textnormal{Lie}_X(\widehat{\mathcal{L}X}) \rightarrow \text{ker-anch}(\text{At}(\mathcal{E})) \simeq \Upsilon_X(\text{End}(\mathcal{E}))$ where $\text{At}(\mathcal{E})$ is the Lie algebroid $\mathbb{T}_{X/\text{Perf}}$ induced by the map $X \rightarrow \text{Perf}$ classifying $\mathcal{E}$. By \cite[Proposition 3.40]{He}, if $X$ is a smooth variety and $\mathcal{E}$ a perfect complex on $X$, $\text{At}(\mathcal{E})$ agrees with the classical Atiyah class of $\mathcal{E}$. By \cite[Proposition 3.41]{He}, the canonical action of $\textnormal{Lie}_X(\widehat{\mathcal{L}X})$ on the object $\mathbb{T}_X[-1]\in \indcoh(X)$ is given by the adjoint action. In particular, if $X$ is a smooth variety, the Lie bracket on $\mathbb{T}_X[-1]$ is the classical Atiyah bracket. \displaypar

The symmetric algebra on $\mathbb{T}_X[-1] \in \indcoh(X)$ is the universal enveloping algebra if we equip the shifted tangent complex with the \textbf{trivial Lie structure}. This trivial Lie algebra structure corresponds to viewing $\mathbb{T}_X[-1]$ as a vector prestack\footnote{This is the derived algebraic geometry version of viewing a classical vector space as an affine space. For a precise definition, see \cite[Chapter 7, Section 1.4.1]{GR2}.} equipped with the (formal) abelian group structure given by addition. In particular, using \cite[Chapter 6, Theorem 6.1.2]{GR2},
\begin{align*}
    \sym_{\indcoh(X)}(\mathbb{T}_X[-1]) &\simeq U^{\text{Hopf}}(\mathbb{T}_X[-1]^{\text{triv}})\\ &\simeq U^{\text{Hopf}}(\textnormal{Lie}_X \text{Vect}_X(\mathbb{T}_X[-1])) \\&\simeq \text{Grp}(\text{Distr}^{\text{Cocom}^{\text{aug}}})(\text{Vect}_X(\mathbb{T}_X[-1])).
\end{align*}
Similarly, the universal enveloping algebra of $\mathbb{T}_X[-1]$ equipped with the Atiyah bracket can be written in terms of the formal inertia group:
\begin{align*}
    U^{\text{Hopf}}(\mathbb{T}_X[-1]^{\text{Atiyah}}) &\simeq U^{\text{Hopf}}(\textnormal{Lie}_X \widehat{\mathcal{L}X}) \\&\simeq \text{Grp}(\text{Distr}^{\text{Cocom}^{\text{aug}}})(\widehat{\mathcal{L}X}).
\end{align*}
These identifications of the symmetric and universal enveloping algebra with distributions on the Lie algebra and the Lie group respectively are well-known in the classical case. \displaypar
By \cite[Chapter 7, Corollary 3.2.2]{GR2}, we have an equivalence of pointed formal moduli problems over $X$
\begin{align*}
    \text{exp}_{\widehat{\mathcal{L}X}}: \text{Vect}_X(\mathbb{T}_X[-1]) \rightarrow \widehat{\mathcal{L}X}.
\end{align*}
This equivalence is given by applying the formal spectrum functor $\spec^{\text{inf}}$ defined in \cite[Chapter 7, Section 1.3]{GR2} to the PBW isomorphism of cocommutative coalgebras $\sym_{\indcoh(X)}(\mathbb{T}_X[-1]) \rightarrow U^{\text{Hopf}}(\mathbb{T}_X[-1]^{\text{Atiyah}})$, and noting that $\spec^{\text{inf}}\circ U^{\text{Hopf}} \simeq \exp_X$. However, since the PBW morphism is not an isomorphism of associative algebras, the map $\exp_{\widehat{\mathcal{L}X}}$ does not respect the group structures\footnote{In particular, for smooth schemes, the $\qcoh$-pullback along $\exp_{\widehat{\mathcal{L}X}}$ can be identified with the Hochschild-Kostant-Rosenberg map.}. In particular, the induced map on distributions of $\text{Vect}_X(\mathbb{T}_X[-1])$ and $\widehat{\mathcal{L}X}$ respectively is not a morphism of associative algebras. The correction term in the Duflo isomorphism and the global formality morphism can be understood precisely as a measure of the failure of this map between distributions to be an associative algebra morphism. This is explained for example in \cite[Example 4.0.1]{KP}. \displaypar 
The Lie algebra $\mathbb{T}_X[-1] \in \alg_{\lie}(\indcoh(X))$ corresponds to the Lie algebra $\mathfrak{g}$ in the classical Duflo theorem. Recall that in the classical case, the Duflo isomorphism goes between the derived invariants under the respective $\mathfrak{g}$-actions on the symmetric algebra and the universal enveloping algebra of $\mathfrak{g}$. In the geometric case, taking derived invariants in $\text{Rep}_{\mathbb{T}_X[-1]}(\indcoh(X))$ corresponds to taking global sections of the underlying sheaves. For a classical smooth algebraic variety, N. Markarian proves in \cite{M} that $\mathcal{D}^{\ast}_{\text{poly}}(X)$ is the universal enveloping algebra of $\mathcal{T}_X[-1]$ in the derived 1-category of $\mathcal{O}_X$-modules. There are also results identifying Hochschild cohomology with classical distributions on the loop space of $X$, see for example \cite[2.9]{CH}.\displaypar
For a prestack $\mathcal{Y}$, let $\omega_{\mathcal{Y}} = p^!(\field)\in \indcoh(\mathcal{Y})$ be the \textbf{dualizing object} of $\mathcal{Y}$, where $p: \mathcal{Y}\rightarrow \spec(\field)$ is the projection. The distributions functor $\text{Distr}: \textbf{FMP}_{/X} \rightarrow \indcoh(X)$ is explicitly given by $(\mathcal{Y} \xrightarrow{\pi} X) \mapsto \pi^{\indcoh}_{\ast}(\omega_{\mathcal{Y}})$. We will explain how the induced morphism of $\exp_{\widehat{\mathcal{L}X}}$ on distributions corresponds to an identification of dualizing objects.\displaypar
For the remainder of this section, let $X$ be a smooth proper scheme. By \cite{GR2}, we have a symmetric monoidal equivalence of $\infty$-categories
\begin{align}\label{eq4}
    \Upsilon_X: \qcoh(X) &\rightarrow \indcoh(X)\\
    \mathcal{F} \mapsto \mathcal{F} \otimes \omega_X.
\end{align}
For a quasi-compact scheme, there also exists a $\qcoh$-dualizing object $\omega^{\qcoh}_X\in \qcoh(X)$. For a smooth $Z$, this is given by $\Omega^{\dim Z}_Z[\dim Z]$.
\begin{definition}[Definition 3.0.1 \cite{KP}]
For an almost finite type scheme $Z$, an \textbf{orientation} on $Z$ is a choice of equivalence $\mathcal{O}_Z \simeq \omega_Z^{\qcoh}$ in $\qcoh(Z)$.
\end{definition}
Let $(\mathcal{H}\xrightarrow{\pi} X)\in \text{Grp}(\textbf{FMP}_{/X})$ be a formal group over $X$ such that $\textnormal{Lie}_X(\mathcal{H})$ lies in $\text{Coh}(X)^{<0}$. Then the group structure on $\mathcal{H}$ yields a trivialization of the tangent complex of $\mathcal{H}$ over $X$ given by identifying it with ``left-invariant vector fields'':
\begin{align*}
    \mathbb{T}_{\mathcal{H}/X} \simeq \pi^{\ast}\mathbb{T}_{X/B_X\mathcal{H}} \simeq  \pi^{\ast} e^{\ast} \pi^{\ast} \mathbb{T}_{X/B_X\mathcal{H}} \simeq \pi^{\ast} e^{\ast} \mathbb{T}_{\mathcal{H}/X} \simeq \pi^{\ast} \textnormal{Lie}_X(\mathcal{H}).
\end{align*}
The vector prestack $\text{Vect}_X(\mathbb{T}_X[-1])$ is equipped with two formal group structures over $X$. One is just given by addition on the shifted tangent complex, and the second one is the loop group structure on $\widehat{\mathcal{L}X}$ pulled back via the exponential map. Applying the above procedure yields two different trivializations of the relative tangent complex $\mathbb{T}_{\text{Vect}_X(\mathbb{T}_X[-1])/X}$, where for the additive group structure we have the Lie algebra $\mathbb{T}_X[-1]^{\text{triv}}$, and for the loop group structure we have the Lie algebra $\mathbb{T}_X[-1]^{\text{Atiyah}}$. 
\begin{theorem}[Theorem 4.4.1 \cite{KP}]
The composition of the trivialization coming from the additive group structure with the inverse of the trivialization coming from the loop group structure is given by 
\begin{align*}
    d\exp_{\widehat{\mathcal{L}X}} = \frac{1-e^{-\textnormal{ad}_{\mathbb{T}_X[-1]}}}{\textnormal{ad}_{\mathbb{T}_X[-1]}}: \pi^{\ast} \mathbb{T}_X[-1] \simeq \mathbb{T}_{\textnormal{Vect}_X(\mathbb{T}_X[-1])/X} \simeq \pi^{\ast} \mathbb{T}_X[-1]
\end{align*}
\end{theorem}
In \cite[Section 3.3]{KP}, G. Kondyrev and A. Prikhodko develop a method to obtain orientations on $\text{Vect}_X(\mathbb{T}_X[-1])$ from the trivializations of the relative tangent complex coming from formal group structures on this space. The endomorphism of $\mathcal{O}_{\widehat{\mathcal{L}X}}$ given by composing the additive group orientation with the inverse of the loop group orientation is then given by the determinant of $d\exp_{\widehat{\mathcal{L}X}}$, which precisely agrees with the classical description of the Todd class.\displaypar
The morphism $\Upsilon_{\widehat{\mathcal{L}X}}$ sends $\mathcal{O}_{\widehat{\mathcal{L}X}}$ to $\mathcal{O}_{\widehat{\mathcal{L}X}} \otimes \omega_{\widehat{\mathcal{L}X}} \simeq \omega_{\widehat{\mathcal{L}X}} \in \indcoh(\widehat{\mathcal{L}X})$, and this endomorphism of $\mathcal{O}_{\widehat{\mathcal{L}X}}$ therefore corresponds to an endomorphism of $\omega_{\widehat{\mathcal{L}X}}$. Applying $\pi_{\ast}^{\indcoh}$, this then yields an endomorphism of $\text{Distr}(\widehat{\mathcal{L}X})$. We expect this endomorphism to be a ``correction term'' for $\text{Distr}^{\text{Cocom}^{\text{aug}}}(\exp_{\widehat{\mathcal{L}X}})$. 
\section{The Grothendieck-Teichmüller group}
Recall that Tamarkin \cite{Tam} constructs a morphism $\Psi_T: \ger_{\infty} \rightarrow \braces$, 
\begin{center}
\begin{tikzcd}
\ger_{\infty} \arrow[r, "\Psi_T"'] \arrow[rr, "\simeq", bend left] & \braces \arrow[r] & \ger
\end{tikzcd},
\end{center}
which shows that any $\braces$-algebra can be lifted to a $\ger_{\infty}$-algebra. To construct this factorization of the resolution of $\ger$, Tamarkin uses Etingof-Kazhdan quantization of Lie bialgebras \cite{EK}, which requires a choice of a \textbf{Drinfeld associator}. In particular, the map $\Psi_T$ is non-canonical and non-unique. Hence, the $\ger_{\infty}$-algebra structure on polydifferential operators also depends on the choice of a Drinfeld associator, and consequently so does the construction of Tamarkin's formality morphism. \displaypar
The goal of this section is to explain the appearance and consequences of Drinfeld associators in the formality theorem.\displaypar
As a general slogan, Drinfeld associators appear in any kind of formal quantization problem where one tries to extend an infinitesimal deformation to a formal deformation. They were first defined by Drinfeld in \cite{Dr} in the context of finding formal deformations extending first-order deformations of quasi-triangular quasi-Hopf algebras. Via Tannaka duality, quasi-triangular quasi-Hopf algebras correspond to (non-strict) braided monoidal 1-categories, and so this question can be equivalently posed as finding formal deformations of ``infinitesimally braided'' symmetric monoidal 1-categories as braided monoidal 1-categories. Drinfeld showed that indeed any first-order deformation of a symmetric monoidal 1-category can be formally quantized, and further that such a quantization is determined by a tuple $(\lambda, f)$ with $\lambda \in \field^{\times}$ and $f \in \text{Grp}(\field\langle \langle x,y\rangle \rangle)$ a group-like element in the noncommutative formal power series ring in two variables, satisfying certain relations, see \cite[Proposition 10.2.7]{F1}. These tuples are now called \textbf{Drinfeld associators}. Drinfeld also proved the existence of such an object over any $\field$. He explicitly constructed an associator over $\mathbb{C}$, the so-called \textbf{Knizhnik-Zamolodchikov associator} $\Phi_{\text{KZ}}$, which comes from a solution of the Knizhnik–Zamolodchikov differential equations. \displaypar
It has subsequently been shown that one can describe the set of Drinfeld associators in terms of operads. To this end, let $\mathcal{P}a\mathcal{B}$ be the operad in groupoids of \textbf{parenthesized braids}, and denote by $\widehat{\mathcal{P}a\mathcal{B}}$ its Malcev completion as defined in \cite[Section 9.2]{F1}. Further, let $\mathcal{P}a\mathcal{C}D$ be the operad in categories enriched in filtered cocommutative coalgebras of \textbf{parenthesized chord diagrams}, and denote by $\widehat{\mathcal{P}a\mathcal{C}D}$ its completion with respect to the filtration. 
\begin{theorem}[Theorem 10.2.9 \cite{F1}]
The set of Drinfeld associators is given by isomorphisms of operads in pro-unipotent groupoids
\begin{align*}
    \textnormal{Iso}_{\mathcal{O}p}(\widehat{\mathcal{P}a\mathcal{B}}, \textnormal{Grp}(\widehat{\mathcal{P}a\mathcal{C}D}))
\end{align*}
which restrict to the identity on objects.
\end{theorem}
In the same paper \cite{Dr}, Drinfeld introduced two groups $\gt(\field)$ and $\grt(\field)$ which canonically act on the set of Drinfeld associators. In particular, Drinfeld associators form a torsor over these two groups. In operadic language, 
\begin{align*}
    \gt(\field) = \aut_{\mathcal{O}p}(\widehat{\mathcal{P}a\mathcal{B}})
\end{align*}
is the group of operad automorphisms which are the identity on objects of the completed parenthesized braids operad. Similarly, the group $\grt(\field)$ is the operad automorphism group of the group-like elements in the completed parenthesized chord diagrams operad. Both of these groups come equipped with a split surjection to $\field^{\times}$, and we denote the kernels of these by $\gt(\field)_1$ and $\grt(\field)_1$ respectively. Both of these are pro-unipotent groups.\displaypar
The set of Drinfeld associators, as well as the two automorphism groups $\gt(\field)$ and $\grt(\field)$, play a central role in deformation quantization. We have already seen that Tamarkin's local formality morphism requires a choice of associator. While Kontsevich's formality morphism does not require a lift of the $\braces$-algebra structure on polydifferential operators to a homotopy Gerstenhaber algebra structure, Kontsevich still uses specific integrals closely resembling the coefficients in the KZ-associator to build his map $\mathscr{U}$. \displaypar
Drinfeld named the group $\gt(\field)$ the (pro-unipotent) \textbf{Grothendieck–Teichmüller group}. This name comes from the fact that the equations satisfied by an element $(\lambda, f)\in \gt(\field)$ precisely resemble the equations found by Grothendieck in \cite{Gro2} while studying the action of the absolute Galois group $\gal$ on the so-called Teichmüller groupoids. For an algebraic variety $X$ over $\mathbb{Q}$, and the choice of a point $p: \spec \mathbb{Q} \rightarrow X$, there is an action
\begin{align*}
    \rho_{X,p}: \gal \rightarrow \aut(\pi_1^{\text{et}}(\overline{X},p)),
\end{align*}
where $\overline{X} = X\times_{\spec \mathbb{Q}} \spec \overline{\mathbb{Q}}$. Note that we have $\pi_1^{\text{et}}(\overline{X},p) \cong \widehat{\pi}_1(X(\mathbb{C}),p)$, the profinite completion of the topological fundamental group of the $\mathbb{C}$-points of ${X}$.\displaypar
Grothendieck's idea was to study the group $\gal$ through these actions for suitable $X$. Specifically, he was interested in $X= \mathcal{M}_{g,\nu}$ the moduli spaces of smooth curves of genus $g$ with $\nu$ marked points. Of specific interest is the moduli space
\begin{align*}
    \mathcal{M}_{0,4} \simeq \mathbb{P}^1-\{0,1,\infty\}.
\end{align*}
Note that $\mathbb{P}^1(\mathbb{C})-\{0,1,\infty\} \simeq \mathbb{C}-\{0,1\}$, and therefore 
\begin{align*}
    \pi_1(\mathbb{P}^1(\mathbb{C})-\{0,1,\infty\}, p) \cong \pi_1(\mathbb{C}-\{0,1\},p) \cong \mathbb{F}(x,y)
\end{align*}
with $\mathbb{F}(x,y)$ the free group on two generators $x,y$. It is a direct corollary of G. Belyi’s theorem on complex smooth projective curves \cite{Be} that the resulting action
\begin{align*}
    \rho_{\mathcal{M}_{0,4}}: \gal \rightarrow \aut(\widehat{\mathbb{F}}(x,y))
\end{align*}
is in fact faithful. This shows that the absolute Galois group can be realized as a subgroup of a very concrete group, and Ihara gave an equally concrete description of the map $\rho_{\mathcal{M}_{0,4}}$
\begin{theorem}[\cite{Ih}]
An element $\sigma \in \gal$ acts on the generators $x$ and $y$ of $\widehat{\mathbb{F}}(x,y)$ as
\begin{align*}
    \sigma(x) = x^{\chi(\sigma)}, \quad \sigma(y) = f_{\sigma}^{-1} y^{\chi(\sigma)} f_{\sigma},
\end{align*}
where $\chi: \gal \rightarrow \widehat{\mathbb{Z}}^{\times}$ is the cyclotomic character and $f_{\sigma}$ is a uniquely determined element of $[\widehat{\mathbb{F}}(x,y),\widehat{\mathbb{F}}(x,y)]$.
\end{theorem}
In particular, this yields an embedding
\begin{align*}
    \gal &\hookrightarrow \widehat{\mathbb{Z}}^{\times} \times \widehat{\mathbb{F}}(x,y),\\
    \sigma &\mapsto (\chi(\sigma), f_{\sigma}).
\end{align*}
Thus, an element of the absolute Galois group can be described as a tuple $(\lambda,f) \in \widehat{\mathbb{Z}}^{\times} \times \widehat{\mathbb{F}}(x,y)$, similarly to how an object of the pro-unipotent Grothendieck–Teichmüller group $\gt(\mathbb{Q})$ can be described as a tuple $(\lambda, f) \in \mathbb{Q}^{\times} \times \text{Grp}(\mathbb{Q}\langle \langle x,y\rangle \rangle)$. To understand the absolute Galois group of the rationals, one is now reduced to understanding the image of this embedding. In \cite{Ih}, Ihara identifies three equations which are satisfied by any tuple in the image of $\gal$ by considering the action of $\gal$ on other $\mathcal{M}_{0,\nu}$. 
\begin{theorem}[\cite{Ih}]\label{thm13}
If $(1 + 2\nu,f) \in \widehat{\mathbb{Z}}^{\times} \times \widehat{\mathbb{F}}(x,y)$ is in the image of $\gal \hookrightarrow \widehat{\mathbb{Z}}^{\times} \times \widehat{\mathbb{F}}(x,y)$, then $f(x,1) = 1 = f(1,y)$, and the following equations hold
\begin{enumerate}
    \item $f(x,y)f(y,x) =1$.
    \item $f(x,y)x^{\nu}f(z,x)z^{\nu}f(y,z)y^{\nu} = 1$ whenever $zyx = 1$.
    \item $f(x_{23},x_{34})f(x_{12}x_{12},x_{34}x_{24})f(x_{12},x_{23}) = f(x_{12},x_{24}x_{23})f(x_{23}x_{13},x_{34})$ for $x_{ij}$ the generators of the pure braid group $P_4$ on four strands.
\end{enumerate}
\end{theorem}
\begin{definition}[\cite{Ih2}]
The \textbf{profinite Grothendieck–Teichmüller group} $\widehat{\textnormal{GT}}$ is the subgroup of  $\widehat{\mathbb{Z}}^{\times} \times \widehat{\mathbb{F}}(x,y)$ of tuples $(\lambda,f)$ satisfying the equations in Theorem \ref{thm13}.
\end{definition}
By construction, we have an embedding
\begin{align*}
    \gal \hookrightarrow \widehat{\gt}.
\end{align*}
It is still unknown if this is a bijection, or if there are further equations satisfied by tuples in the image.\displaypar
The equations in Theorem \ref{thm13} are precisely the same equations Drinfeld found for tuples $(\lambda, f) \in \field^{\times} \times \text{Grp}(\field\langle \langle x,y \rangle \rangle)$ in his pro-unipotent Grothendieck–Teichmüller group $\gt(\field)$. This raises the question whether there also is a pro-unipotent analogue of the (profinite) absolute Galois group. It turns out that the appropriate object to replace $\gal$ in the pro-unipotent world is the \textbf{motivic Galois group}.\displaypar
Let $\text{DM}(\mathbb{Q})_{\mathbb{Q}}$ be the triangulated category of mixed motives over $\mathbb{Q}$. Let $\text{DMT}(\mathbb{Q})_{\mathbb{Q}}$ be the triangulated subcategory generated by the Tate motives $\mathbb{Q}(n)$. M. Levine showed in \cite{L} that $\text{DMT}(\mathbb{Q})_{\mathbb{Q}}$ carries a canonical $t$-structure, and that the heart of this $t$-structure is the category $\text{MT}(\mathbb{Q})$ of mixed Tate motives over $\mathbb{Q}$. In \cite{DG}, Deligne and A. Goncharov introduced the full subcategory $\text{MT}(\mathbb{Z})$ of \textbf{mixed Tate motives unramified over} $\mathbb{Z}$, which is shown to be a neutral Tannakian category over $\mathbb{Q}$. In fact, $\text{MT}(\mathbb{Z})$ is equipped with multiple fiber functors corresponding to different cohomology theories. In particular, we get fiber functors
\begin{align*}
    \omega_{\bullet}:\text{MT}(\mathbb{Z}) \rightarrow \text{Vect}_{\mathbb{Q}}
\end{align*}
with $\bullet\in \{\text{dR}, B, \ell\}$, for the de Rham, Betti, and $\ell$-adic cohomology respectively, and their corresponding affine group schemes
\begin{align*}
    \mathcal{G}_{\bullet}^{\text{MT}(\mathbb{Z})}:= \underline{\aut}^{\otimes}(\omega_{\bullet}).
\end{align*}
These groups are called the \textbf{motivic Galois groups}. By general Tannaka theory, we have an equivalence
\begin{align*}
    \text{MT}(\mathbb{Z}) \simeq \text{Rep}\left(\mathcal{G}_{\bullet}^{\text{MT}(\mathbb{Z})}\right),
\end{align*}
and in particular, all of the motivic Galois groups are equivalent. The action of $\mathcal{G}_{\bullet}^{\text{MT}(\mathbb{Z})}$ on the Tate motive $\mathbb{Q}(1)$ yields a surjection $\mathcal{G}_{\bullet}^{\text{MT}(\mathbb{Z})}\rightarrow \mathbb{G}_m$, and therefore we have a short exact sequence
\begin{align*}
    1 \rightarrow \mathcal{U}_{\bullet}^{\text{MT}(\mathbb{Z})} \rightarrow \mathcal{G}_{\bullet}^{\text{MT}(\mathbb{Z})} \rightarrow \mathbb{G}_m \rightarrow 1,
\end{align*}
with $\mathcal{U}_{\bullet}^{\text{MT}(\mathbb{Z})}$ a pro-unipotent group. In $\text{MT}(\mathbb{Z})$, 
\begin{align*}
    \dim_{\mathbb{Q}} \text{Ext}^1_{\text{MT}(\mathbb{Z})}(\mathbb{Q}(0),\mathbb{Q}(m)) = \begin{cases}
        1 & m = 3,5,7,\dots\\ 0 & \text{ otherwise}
    \end{cases}.
\end{align*}
From this it follows that the Lie algebra of $\mathcal{U}_{\bullet}^{\text{MT}(\mathbb{Z})}$ is the free Lie algebra on generators $\sigma_3, \sigma_5,\sigma_7,\dots$, where $\sigma_{2n+1}$ sits in degree $-2n-1$. These generators are called the \textbf{Deligne-Drinfeld elements}. Note that they are only well-defined up to commutators.\displaypar
To understand the connection of the theory of motives to the Grothendieck–Teichmüller group, consider again the \'etale fundamental group of the projective line minus three points. Deligne realized in \cite{De} that it is part of a larger object, namely the \textbf{motivic fundamental groupoid}\footnote{To define this, one needs to use tangential basepoints as developed in \cite{De}.} $\pi_1^{\text{mot}}(\mathbb{P}^1-\{0,1,\infty\}, \vec{1}_0, -\vec{1}_1)$. This motivic fundamental groupoid was rigorously constructed in \cite{DG}, and it is a pro-object in the abelian category $\text{MT}(\mathbb{Z})$. Its Betti realization is given by the unipotent completion of the topological fundamental groupoid $\pi_1(\mathbb{P}^1-\{0,1,\infty\}, \vec{1}_0, -\vec{1}_1)$, so its $\mathbb{Q}$-points are given precisely by the group-like elements in $\mathbb{Q}\langle \langle x,y\rangle \rangle$. The affine ring of the de Rham realization is generated by the elements $\omega_0 = \frac{dz}{z}$ and $\omega_1 = \frac{dz}{1-z}$, and therefore its $\mathbb{Q}$-points are again given by group-like noncommutative formal power series in two generators dual to the $\omega_i$. The comparison isomorphism is given by the pairing of integrating a path with respect to the forms $\omega_i$. In particular, consider the path $\mathbf{dch}$ which goes from $0$ to $1$ in a straight line. Then the periods
\begin{align*}
    \int_{\mathbf{dch}} \omega_{i_1}\dots \omega_{i_n}
\end{align*}
with $i_k\in \{0,1\}$, and $i_1 = 1$, $i_n = 0$ are precisely the \textbf{multiple zeta values} $\zeta(n_1,\dots,n_r)$. These are closely related with the theory of Drinfeld associators. In fact by \cite[Lemma 2.2]{Bro}, the KZ-associator can be recovered as
\begin{align*}
    \Phi_{\text{KZ}} = \text{comp}_{\text{B},\text{dR}}(\mathbf{dch}) \in \pi_1^{\text{dR}}(\mathbb{P}^1-\{0,1,\infty\})(\mathbb{C}),
\end{align*}
and all its coefficients can be expressed in terms of multiple zeta values.\displaypar
By construction, the de Rham realization ${}_x\Pi_y := \pi_1^{\text{dR}}(\mathbb{P}^1-\{0,1,\infty\}, x,y)$, for base points $x,y\in \{\vec{1}_0,-\vec{1}_1\}$, receives an action of the de Rham motivic Galois group $\mathcal{G}_{\text{dR}}^{\text{MT}(\mathbb{Z})}$. It turns out that these actions are compatible with the groupoid structure on ${}_{\bullet}\Pi_{\bullet}$. From this, F. Brown \cite[Section 2.4]{Bro} obtains a description of the action of the pro-unipotent subgroup $\mathcal{U}_{\text{dR}}^{\text{MT}(\mathbb{Z})}$ on ${}_0\Pi_1$, in particular we get a map
\begin{align*}
    \rho: \mathcal{U}_{\text{dR}}^{\text{MT}(\mathbb{Z})} \rightarrow \aut(\pi_1^{\text{dR}}(\mathbb{P}^1-\{0,1,\infty\}))_1 \cong \text{Grp}(\mathbb{Q}\langle \langle x,y\rangle \rangle).
\end{align*}
Brown was able to prove a motivic analogue of Ihara's theorem:
\begin{theorem}[\cite{Bro2}]
The action of $\mathcal{G}_{\textnormal{dR}}^{\textnormal{MT}(\mathbb{Z})}$ on $\pi_1^{\textnormal{dR}}(\mathbb{P}^1-\{0,1,\infty\})$ is fully faithful. In particular, the map
\begin{align*}
    \rho: \mathcal{G}_{\textnormal{dR}}^{\textnormal{MT}(\mathbb{Z})} \simeq \mathbb{G}_m \rtimes \mathcal{U}_{\textnormal{dR}}^{\textnormal{MT}(\mathbb{Z})} \rightarrow \mathbb{Q}^{\times} \times \textnormal{Grp}(\mathbb{Q}\langle \langle x,y\rangle \rangle)
\end{align*}
is injective.
\end{theorem}
Using the explicit description of $\rho$, one can again directly see that this action factors through the pro-unipotent Grothendieck–Teichmüller group. In particular, we have an embedding
\begin{align*}
    \mathcal{G}_{\textnormal{dR}}^{\textnormal{MT}(\mathbb{Z})} \hookrightarrow \gt(\mathbb{Q}).
\end{align*}
We also obtain an embedding of pro-unipotent subgroups $\mathcal{U}_{\textnormal{dR}}^{\textnormal{MT}(\mathbb{Z})} \hookrightarrow \gt(\mathbb{Q})_1$, and by the description of the Lie algebra of $\mathcal{U}_{\textnormal{dR}}^{\textnormal{MT}(\mathbb{Z})}$, this yields an inclusion
\begin{align*}
    \text{free}_{\text{Lie}}(\sigma_3,\sigma_5,\sigma_7,\dots) \hookrightarrow \mathfrak{gt},
\end{align*}
where $\mathfrak{g}\mathfrak{t}$ is the Lie algebra of $\gt(\mathbb{Q})_1$. The images of $\sigma_{2n+1}$ in $\mathfrak{g}\mathfrak{t}$ are again called Deligne-Drinfeld elements.\displaypar
In \cite{K}, Kontsevich builds his local formality morphism $\mathscr{U}: T_{\text{poly}}^{\ast}(M)[1] \rightarrow D_{\text{poly}}^{\ast}(M)[1]$ by introducing a certain operad of \textbf{graph complexes}. In particular, the component $\mathscr{U}_n$ is constructed from maps $\mathscr{U}_{\Gamma}$, which depend on a labeled graph $\Gamma$ with $n$ vertices. Kontsevich's graph complexes are related to the Grothendieck–Teichmüller group by the following result proved by Willwacher.
\begin{theorem}[\cite{Wi}]
Let $\mathbf{GC}$ denote Kontsevich's graph complex, and let $\mathfrak{grt}$ be the Lie algebra of the pro-unipotent part of the graded Grothendieck–Teichmüller group $\grt(\mathbb{Q})$. Then there is an isomorphism of Lie algebras
\begin{align*}
    H^0(\mathbf{GC}) \cong \mathfrak{grt}.
\end{align*}
If $\sigma_{2n+1}$ is a Deligne-Drinfeld element in $\mathfrak{grt}$, then each representative of the class $\tilde{\sigma}_{2n+1}$ in $H^0(\mathbf{GC})$ corresponding to $\sigma_{2n+1}$ has a nonzero component given by the $m$th ``spoked wheel'' for $m=2n+1$
\begin{center}
\begin{tikzpicture}[
    node/.style={circle, fill=black, inner sep=1.5pt},
    every node/.style={font=\small}
]

  \node[node, label=below:{$m+1$}] (center) at (0,0) {};

  \foreach \i/\label/\angle in {1/1/140, 2/2/110, 3/3/80, 4/n/170} {
    \node[node, label=\angle:{\label}] (v\i) at (\angle:2cm) {};
    \draw (center) -- (v\i);
  }

  \draw (v4) -- (v1) -- (v2) -- (v3);

  \draw[thick, dotted] (v3) -- (50:2cm);
  \draw[thick, dotted] (v4) -- (200:2cm);

\end{tikzpicture}
\end{center}
\end{theorem}
This identification of the Grothendieck–Teichmüller Lie algebra with the cohomology of the graph complex enabled Dolgushev, Rogers and Willwacher in \cite{DRW} to exhibit an action of $\mathfrak{grt}$ on the cohomology of polyvector fields, where a given graph $\Gamma$ with $n$ vertices acts on $n$ polyvectors $v_1,\dots,v_n$ by contraction with certain differential operators, which are determined by the graph. 
\begin{theorem}[Theorem 7.1, Theorem 8.1 \cite{DRW}]
    Let $X$ be a smooth variety over $\mathbb{C}$. For any cocycle $\gamma \in \textnormal{GC}$ one can construct a map
    \begin{align*}
        D_{\gamma}: H^{\ast}(X,\mathcal{T}^{\ast}_{\textnormal{poly}}(X)) \rightarrow  H^{\ast}(X,\mathcal{T}^{\ast}_{\textnormal{poly}}(X)),
    \end{align*}
    which is a derivation of Gerstenhaber algebras. If $\gamma = \tilde{\sigma}_{2n+1}$ corresponds to a Deligne-Drinfeld element, then the action of the corresponding degree 0 derivation $D_{\tilde{\sigma}_{2n+1}}$ on $H^{\ast}(X,\mathcal{T}^{\ast}_{\textnormal{poly}}(X))$ is a non-zero scalar multiple of the contraction with the $(2n+1)$th component of the Chern character $\textnormal{ch}_{2n+1}(X)$ of the tangent sheaf of $X$.
\end{theorem}
In particular, this theorem yields an action of the graded Grothendieck–Teichmüller Lie algebra $\mathfrak{grt}$ on the Gerstenhaber algebra $H^{\ast}(X,\mathcal{T}^{\ast}_{\textnormal{poly}}(X))$ by derivations. Dolgushev, Rogers, and Willwacher also found classes of varieties for which this action is non-trivial, and can therefore be used to infer information about $\mathfrak{grt}$. To the author's knowledge, this is the only known non-trivial non-torsor action of the Grothendieck–Teichmüller group.\displaypar
Recall that Tamarkin's local formality morphism requires the choice of a Drinfeld associator, and so does the choice of a $\ger_{\infty}$-algebra structure on polydifferential operators. It is therefore expected that the Grothendieck–Teichmüller group, which is the automorphism group of associators, acts on these objects. This was conjectured by Kontsevich in \cite{K2}, and then proved in \cite{DRW} using the above action on the cohomology of polyvector fields. Let $X$ be as above. Then precomposing the global formality morphism on hypercohomology in Theorem \ref{thm12} with the derivation $D_{\gamma}$ induced from a 0-cocycle of the graph complex yields another isomorphism of Gerstenhaber algebras. This induces an action of $\mathfrak{grt}$ on the collection of global formality morphisms. 
\begin{theorem}[Theorem 9.2 \cite{DRW}]
In this action, a Deligne-Drinfeld element $\sigma_{2n+1}$ acts by changing the correction term $\td(X)^{1/2}$ by a factor of $\exp(\textnormal{ch}_{2n+1}(X))$.
\end{theorem}

\newpage
\setcounter{theorem}{0}
\chapter{$\infty$-operads and centers}\label{infty_operads_and_centers}
The theory of $\infty$-operads as developed in \cite{HA} is a way to formalize topological operads in the language of quasi-categories. Just like in classical operad theory, one can define the notion of an algebra over an $\infty$-operad, as well as modules over such an operadic algebra. By working in the formalism of $\infty$-categories, one can exploit derived universal properties in operadic constructions. This enables the definition of the center of an algebra over an $\infty$-operad as a universal object. In good situations, one can compute this center in terms of operadic modules over the algebra.\displaypar
Of special interest are the little $n$-disks $\infty$-operads $\mathbb{E}_n$, whose algebras are equipped with $n$ compatible multiplications. By the Dunn additivity theorem, the center of an $\mathbb{E}_n$-algebra is the universal $\mathbb{E}_{n+1}$-algebra acting on it.\displaypar
In this chapter, we review Lurie's theory of $\infty$-operads and higher centers, and prove the technical results necessary for the comparison of the center to the classical Hochschild complex in Chapter \ref{The_hochschild_complex_as_a_center}. The main result is Corollary \ref{cor7}, which states that the two compatible multiplications on the center of an $\mathbb{E}_1$-algebra are given by the convolution product and the composition product.  
\section{$\infty$-operads as $\infty$-categories of operators}
We follow Lurie's formalism for $\infty$-operads as developed in \cite{HA}. We briefly recollect some important notions. \displaypar
Let $p: \mathcal{O}^{\otimes} \rightarrow \fin$ be a functor. If $\langle n\rangle \in \fin$, we denote by $\mathcal{O}^{\otimes}_{\langle n\rangle}$ the fiber of $p$ over $\langle n\rangle$. For $n=1$, we denote $\o_{\langle 1\rangle}$ simply by $\mathcal{O}$. If $f: \langle m\rangle \rightarrow \langle n\rangle$ is a morphism in $\fin$, we denote by $\map_{\mathcal{O}^{\otimes}}^f(C,C')$ the union of the connected components of $\map_{\mathcal{O}^{\otimes}}(C,C')$ lying over $f$. For $1\leq i \leq n$, we denote by $\rho^i: \n \rightarrow \langle 1 \rangle$ the map sending $i$ to $1$, and $j\neq i$ to $\ast$.
\begin{definition}
An $\infty$\textbf{-operad} is a functor $p: \mathcal{O}^{\otimes} \rightarrow \fin$ such that the following hold.
\begin{enumerate}
    \item For every inert morphism $f: \langle m\rangle \rightarrow \langle n \rangle$ in $\fin$ and every $C\in \mathcal{O}^{\otimes}_{\langle m\rangle}$, there exists a $p$-coCartesian lift $\bar{f}: C \rightarrow C'$ in $\mathcal{O}^{\otimes}$ of $f$.
    \item  Let $C\in \o_{\M}$, $C'\in \o_{\n}$, and $f: \M \rightarrow \n$ in $\fin$. For $1\leq i\leq n$, fix $p$-coCartesian lifts $C' \rightarrow C'_i$ of $\rho^i$. The induced map 
    \begin{align*}
        \map_{\o}^{f}(C,C') \rightarrow \prod_{1\leq i\leq n} \map_{\o}^{\rho^i \circ f}(C,C'_i)
    \end{align*}
    is a homotopy equivalence.
    \item For $n\geq 0$, the collection of functors $\{\rho^i_!: \o \rightarrow \mathcal{O}\}_{1\leq i\leq n}$ given by sending $C\in \o$ to the target of the $p$-coCartesian lift of $\rho^i$ at $C$ induces an equivalence $\o_{\langle n\rangle} \xrightarrow{\simeq} \mathcal{O}^{ n}$.
\end{enumerate}
\end{definition}
Some classical examples include
\begin{example}
\begin{enumerate}
\item The identity functor $\fin \rightarrow \fin$ is called the commutative $\infty$-operad and denoted by $\com$.
\item The inclusion of the wide subcategory of $\fin$ whose morphisms are only the inert morphisms is called the trivial $\infty$-operad $\triv$.
\item Let $\mathcal{O}$ be a topological (colored) operad. Form the topological category $\mathcal{O}^{\otimes} \rightarrow \fin$ of operators. Then the homotopy coherent nerve of the corresponding simplicial category $N_{\textnormal{hc}}(\mathcal{O}_{\Delta}^{\otimes}) \rightarrow \fin$ is an $\infty$-operad by \cite[Proposition 2.1.1.27]{HA}.
\item Consider the classical associative operad \text{Assoc} in the category of sets. This can be viewed as a (discrete) topological operad, and applying the procedure above hence yields an $\infty$-operad. This is called the associative $\infty$-operad, and is denoted by $\ass$. 
\item Define a colored operad $\text{LM}$ in the category of sets as follows. The set of colors has two elements denoted $\mathfrak{a}$ and $\mathfrak{m}$. Let $\{c_i\}_{i\in I}$ be a finite collection of colors in $\text{LM}$, and let $d$ be another color. Then
\begin{align*}
    \text{LM}(\{c_i\}_{i\in I}, d) = \begin{cases}
        \text{Assoc}(I) & d = \mathfrak{a} \text{ and } \forall \, i\in I: c_i = \mathfrak{a} \\
        \text{Linear orderings }\{i_1 < \dots < i_n\} \text{ of }I    & d = \mathfrak{m}\\
         \text{such that }c_{i_n} = \mathfrak{m} \text{ and } c_{i_j} = \mathfrak{a} \text{ for } j < n
    \end{cases}.
\end{align*}
The composition law on $\text{LM}$ is determined by the composition of linear orderings. The corresponding $\infty$-operad is denoted by $\lm$.
\end{enumerate}
\end{example}
A morphism of $\infty$-operads $\o \rightarrow \mathcal{O}'^{\otimes}$ is a morphism $f\in \text{Fun}_{\fin}(\o, \mathcal{O}'^{\otimes})$ which preserves inert morphisms\footnote{A morphism of an $\infty$-operad $p:\o \rightarrow \fin$ is called inert if it is $p$-coCartesian, and its image in $\fin$ is inert.}. The $\infty$-category of morphisms from $\o$ to $\mathcal{O}'^{\otimes}$ is denoted by $\alg_{\mathcal{O}}(\mathcal{O}')$, and objects of this are frequently called $\mathcal{O}$\textbf{-algebras in }$\mathcal{O}'$.
\begin{definition}
\begin{enumerate}
    \item A morphism $f: \o \rightarrow \mathcal{O}'^{\otimes}$ of $\infty$-operads is a \textbf{fibration of }$\infty$\textbf{-operads} if $f$ is a categorical fibration.
    \item Let $p:\o \rightarrow \fin$ be an $\infty$-operad. A morphism $f: \mathcal{C}^{\otimes} \rightarrow \o$ of $\infty$-categories is a \textbf{coCartesian fibration of }$\infty$\textbf{-operads} if $f$ is a coCartesian fibration of $\infty$-categories, and $p\circ f: \mathcal{C}^{\otimes} \rightarrow \fin$ exhibits $\mathcal{C}^{\otimes}$ as an $\infty$-operad. In this case, we also call $\mathcal{C}$ an \textbf{$\mathcal{O}$-monoidal $\infty$-category}.
\end{enumerate}
\end{definition}
For example, a symmetric monoidal $\infty$-category is a coCartesian fibration of $\infty$-operads $\mathcal{C}^{\otimes} \rightarrow \com$.\displaypar
From now on, we will freely use notation from \cite{HA} regarding $\infty$-operads, algebras and modules over these.
\section{Morphism objects and operadic centers}
We recall the definition of endomorphism objects, centers and centralizers, and we show how these notions are related. The definitions we use can be found in \cite[Section 4.2, 4.7 and 5.3]{HA}.\displaypar
Recall that $\lm$ is the 2-colored $\infty$-operad whose algebras are given by pairs of homotopy associative algebras and modules over them as defined in \cite[Definition 4.2.1.7]{HA}. Denote by $\mathfrak{a}$ and $\mathfrak{m}$ the ``algebra'' and ``module'' color respectively.
\begin{definition}
Let $q: \mathcal{C}^{\otimes} \rightarrow \lm$ be a coCartesian fibration of $\infty$-operads. We then say that $q$ exhibits the $\infty$-category $\mathcal{C}_{\mathfrak{m}}$ as \textbf{left tensored} over the monoidal $\infty$-category $\mathcal{C}_{\mathfrak{a}}^{\otimes} := \mathcal{C}^{\otimes} \times_{\lm} \ass$. 
\end{definition}
In particular, if $q$ exhibits $\mathcal{C}_{\mathfrak{m}}$ as left tensored over $\mathcal{C}^{\otimes}_{\mathfrak{a}}$, then it determines a tensoring
\begin{align*}
    \otimes: \mathcal{C}_{\mathfrak{a}} \times \mathcal{C}_{\mathfrak{m}} \rightarrow \mathcal{C}_{\mathfrak{m}},
\end{align*}
well-defined up to homotopy, that is compatible with the monoidal structure on $\mathcal{C}_{\mathfrak{a}}$ up to homotopy. This is obtained as a coCartesian lift of the map $\{\mathfrak{a},\mathfrak{m}\} \rightarrow \mathfrak{m}$ over the active map $\langle 2\rangle \rightarrow \langle 1 \rangle$ in $\lm$. Note that any monoidal $\infty$-category $\mathcal{C}^{\otimes}\rightarrow \ass$ is left tensored over itself by considering the fiber product $\mathcal{C}^{\otimes} \times_{\ass} \lm \rightarrow \lm$.
\begin{notation}
Let $q: \mathcal{C}^{\otimes} \rightarrow \lm$ exhibit $\mathcal{C}_{\mathfrak{m}}$ as left tensored over $\mathcal{C}_{\mathfrak{a}}^{\otimes}$. Then we denote by 
\begin{align*}
    \lmod(\mathcal{C}) := \alg_{/\lm}(\mathcal{C})
\end{align*}
the $\infty$-category of pairs of associative algebras in $\mathcal{C}_{\mathfrak{a}}$ and left modules over them.
\end{notation}
Recall that for ordinary categories, internal homs and more generally enrichments are right adjoint to a tensoring. Similarly, one makes the following definition.
\begin{definition}\label{def3}
    Let $\mathcal{C}^{\otimes}\rightarrow \lm$ be a coCartesian fibration of $\infty$-operads exhibiting $\mathcal{C}_{\mathfrak{m}}$ as left tensored over $\mathcal{C}^{\otimes}_{\mathfrak{a}}$. If $M,N\in \mathcal{C}_{\mathfrak{m}}$, a \textbf{morphism object} for $M$ and $N$ is an object $\mor_{\mathcal{C}_{\mathfrak{m}}}(M,N)\in \mathcal{C}_{\mathfrak{a}}$ together with a map $\alpha\in \map_{\mathcal{C}_{\mathfrak{m}}}(\mor_{\mathcal{C}_{\mathfrak{m}}}(M,N)\otimes M,N)$ such that for each $C\in \mathcal{C}_{\mathfrak{a}}$, post-composition with $\alpha$ induces a homotopy equivalence
    \begin{align}
        \map_{\mathcal{C}_{\mathfrak{a}}}(C, \mor_{\mathcal{C}_{\mathfrak{m}}}(M,N))\xrightarrow{\simeq} \map_{\mathcal{C}_{\mathfrak{m}}}(C\otimes M,N).
    \end{align}
    We call $\mathcal{C}_{\mathfrak{m}}$ \textbf{enriched} over $\mathcal{C}^{\otimes}_{\mathfrak{a}}$ if morphism objects exist for all $M,N \in \mathcal{C}_{\mathfrak{m}}$.
\end{definition}
The following result shows that we can think of a morphism object as the classifying object of maps $A\otimes M \rightarrow N$ in $\mathcal{C}_{\mathfrak{m}}$ with $A\in \mathcal{C}_{\mathfrak{a}}$. Denote by $\mathcal{C}_{\mathfrak{a}} \times_{\mathcal{C}_{\mathfrak{m}}} {\mathcal{C}_{\mathfrak{m}}}_{/N}$ the pullback of the forgetful functor ${\mathcal{C}_{\mathfrak{m}}}_{/N} \rightarrow \mathcal{C}_{\mathfrak{m}}$ and the map
\begin{align}\label{eq2}
    -\otimes M: \mathcal{C}_{\mathfrak{a}} \rightarrow \mathcal{C}_{\mathfrak{m}}.
\end{align}
\begin{proposition}\label{prop6}
    Let $q: \mathcal{C}^{\otimes} \rightarrow \lm$ be a coCartesian fibration of $\infty$-operads. Let $M,N\in\mathcal{C}_{\mathfrak{m}}$. Then an object $\mor_{\mathcal{C}_{\mathfrak{m}}}(M,N)\in \mathcal{C}_{\mathfrak{a}}$ together with a map $\alpha: \mor_{\mathcal{C}_{\mathfrak{m}}}(M,N)\otimes M \rightarrow N$ is a morphism object of $M$ and $N$ if and only if $(\mor_{\mathcal{C}_{\mathfrak{m}}}(M,N),\alpha)\in \mathcal{C}_{\mathfrak{a}}\times_{\mathcal{C}_{\mathfrak{m}}}{\mathcal{C}_{\mathfrak{m}}}_{/N}$ is final.
\end{proposition}
\begin{proof}
    Note that ${\mathcal{C}_{\mathfrak{m}}}_{/N}\rightarrow \mathcal{C}_{\mathfrak{m}}$ is a right fibration, and since these are stable under base change, so is $f:\mathcal{C}_{\mathfrak{a}}\times_{\mathcal{C}_{\mathfrak{m}}}{\mathcal{C}_{\mathfrak{m}}}_{/N}\rightarrow \mathcal{C}_{\mathfrak{a}}$. Consider the functor $F: \mathcal{C}_{\mathfrak{a}}^{\textnormal{op}}\rightarrow \an$ classifying $f$. Then by \cite[Lemma 2.2.2.4]{HTT}, its underlying functor $hF: h\mathcal{C}_{\mathfrak{a}}^{\textnormal{op}}\rightarrow h\an$ can be recovered as follows. On objects, $X\in \mathcal{C}_{\mathfrak{a}}$ is sent to its fiber
    \begin{align*}
(\mathcal{C}_{\mathfrak{a}}\times_{\mathcal{C}_{\mathfrak{m}}}{\mathcal{C}_{\mathfrak{m}}}_{/N})\times_{\mathcal{C}_{\mathfrak{a}}}\{X\} \simeq \{X\otimes M\}\times_{\mathcal{C}_{\mathfrak{m}}}{\mathcal{C}_{\mathfrak{m}}}_{/N} \simeq \map_{\mathcal{C}_{\mathfrak{m}}}(X\otimes M,N).
    \end{align*}
    Given a morphism $e: Y\rightarrow X$ in $h\mathcal{C}_{\mathfrak{a}}$, the induced map between the fibers comes from solving the lifting problem
    \begin{center}
        \begin{tikzcd}
{\{1\}\times \text{Map}_{\mathcal{C}_{\mathfrak{m}}}(X\otimes M,N)} \arrow[d, hook] \arrow[rr, hook]  &                           & \mathcal{C}_{\mathfrak{a}}\times_{\mathcal{C}_{\mathfrak{m}}}{\mathcal{C}_{\mathfrak{m}}}_{/N} \arrow[d,"f"] \\
{\Delta^1\times \text{Map}_{\mathcal{C}_{\mathfrak{m}}}(X\otimes M,N)} \arrow[rd] \arrow[rru, dashed] &                           & \mathcal{C}_{\mathfrak{a}}                                               \\
                                                                                      & \Delta^1 \arrow[ru, "e"'] &                                                                         
\end{tikzcd},
    \end{center}
    and restricting the lift to $\{0\}\times \map_{\mathcal{C}_{\mathfrak{m}}}(X\otimes M,N)$. Since $f$ is a pullback of the right fibration ${\mathcal{C}_{\mathfrak{m}}}_{/N}\rightarrow \mathcal{C}_{\mathfrak{m}}$, the lift above is induced by the solution to 
    \begin{center}
        \begin{tikzcd}
{\{1\}\times \map_{\mathcal{C}_{\mathfrak{m}}}(X\otimes M,N)} \arrow[r] \arrow[d]                                      & {\mathcal{C}_{\mathfrak{m}}}_{/N} \arrow[d] \\
{\Delta^1\times \map_{\mathcal{C}_{\mathfrak{m}}}(X\otimes M,N)} \arrow[r, "e\otimes \text{id}_M"'] \arrow[ru, dashed] & \mathcal{C}_{\mathfrak{m}}               
\end{tikzcd}.
    \end{center}
    But the restriction of this lift to $\{0\}\times \map_{\mathcal{C}_{\mathfrak{m}}}(X\otimes M,N)$ is given by pre-composition with $e\otimes \text{id}_M$. Therefore, we see that $hF$ is given by the composition of $-\otimes M$ and $\map_{\mathcal{C}_{\mathfrak{m}}}(-,N)$. Recall that by \cite[Proposition 4.4.4.5]{HTT}, an object $(X,X\otimes M\xrightarrow{\eta} N)$ is final in $\mathcal{C}_{\mathfrak{a}}\times_{\mathcal{C}_{\mathfrak{m}}}{\mathcal{C}_{\mathfrak{m}}}_{/N}$ if and only if the pair $(X,\eta\in hF(X))$ represents $hF$. Then we are done after noting that by definition, $(\mor_{\mathcal{C}_{\mathfrak{m}}}(M,N),\alpha)$ is a morphism object exactly if it represents the functor $X\mapsto \map_{\mathcal{C}_{\mathfrak{m}}}(X\otimes M,N)$.
\end{proof}
Now put $N = M$. Then by the above proposition, $\text{End}(M):= \mor_{\mathcal{C}_{\mathfrak{m}}}(M,M)\in \mathcal{C}_{\mathfrak{a}}$ classifies maps $A \otimes M \rightarrow M$. Here, the object $A\in \mathcal{C}_{\mathfrak{a}}$ need not carry the structure of an algebra object. However, we do expect $\text{End}(M)$ to carry the structure of an associative algebra coming from composition, and $M$ to be a module over it. 
\begin{definition}\label{def1}
    Let $q: \mathcal{C}^{\otimes}\rightarrow \lm$ be a coCartesian fibration of $\infty$-operads. Define the \textbf{endomorphism} $\infty$\textbf{-category} of $M\in \mathcal{C}_{\mathfrak{m}}$ as
    \begin{align*}
        \mathcal{C}_{\mathfrak{a}}[M] := \mathcal{C}_{\mathfrak{a}} \times_{\mathcal{C}_{\mathfrak{m}}} {\mathcal{C}_{\mathfrak{m}}}_{/M}
    \end{align*}
    where the pullback is given by the forgetful functor and equation (\ref{eq2}).
\end{definition}
\begin{proposition}\label{prop9}
The $\infty$-category $\mathcal{C}_{\mathfrak{a}}[M]$ is the underlying $\infty$-category of a monoidal $\infty$-category with tensor product given up to homotopy by 
\begin{align*}
    (A, A\otimes M \xrightarrow{\alpha} M) \otimes (B, B\otimes M \xrightarrow{\beta} M) = (A \otimes B, A\otimes B \otimes M \xrightarrow{\text{id}\otimes \beta} A \otimes M \xrightarrow{\alpha} M).
\end{align*}
\end{proposition}
\begin{proof}
We show in Appendix \ref{appendixA} that $\mathcal{C}_{\mathfrak{a}}[M]$ agrees with Lurie's endomorphism $\infty$-category as in \cite[Definition 4.7.1.1]{HA}. It is shown in \cite[Proposition 4.7.1.30]{HA} that this admits the structure of a monoidal $\infty$-category with the described underlying tensor product.
\end{proof}
Then by \cite[Corollary 3.2.2.5]{HA} with $K= \emptyset$ and $\mathcal{O}^{\otimes} = \ass$ we automatically get the following.
\begin{proposition}\label{cor6}
    Assume that the $\infty$-category $\mathcal{C}_{\mathfrak{a}}[M]$ has a final object $(\e(M),\alpha)$. Then the $\infty$-category $\alg_{\aass}(\mathcal{C}_{\mathfrak{a}}[M])$ also admits a final object, and if $X$ is final in $\alg_{\aass}(\mathcal{C}_{\mathfrak{a}}[M])$ then we have an equivalence $X(\mathfrak{a})\simeq (\e(M),\alpha)$.
\end{proposition}
Roughly speaking, Proposition \ref{cor6} implies that $\e(M)$ can be ``promoted'' to an algebra object in an essentially unique way. By abuse of notation, from here on, we denote by $\e(M)$ a final object of $\alg_{\aass}(\mathcal{C}_{\mathfrak{a}}[M])$ with underlying object $\e(M)\in \mathcal{C}_{\mathfrak{a}}[M]$. \displaypar
By Proposition \ref{prop9}, the tensor product of $\e(M)$ with itself in $\mathcal{C}_{\mathfrak{a}}[M]$ is given by
\begin{align*}
    (\e(M) \otimes \e(M))\otimes M \simeq \e(M)\otimes (\e(M)\otimes M) \xrightarrow{\text{id}_{\e(M)}\otimes \alpha} \e(M) \otimes M \xrightarrow{\alpha} M,
\end{align*}
and hence by the proof of \cite[Proposition 3.2.2.1]{HA}, the algebra structure on $\e(M)$ has a multiplication $\circ$ making the following diagram commute in $\mathcal{C}_{\mathfrak{m}}$
\begin{center}
    \begin{tikzcd}
                                                     & \e(M) \otimes M \arrow[rd, "\alpha"] &   \\
(\e(M)\otimes \e(M)) \otimes M \arrow[rr, "\alpha  (\text{id}\otimes \alpha)"'] \arrow[ru, "\circ \otimes \text{id}"] &                            & M
\end{tikzcd}.
\end{center}
The above diagram implies that $M$ is automatically a module over the algebra $\e(M)$, and this action $\e(M) \otimes M \xrightarrow{\alpha} M$ is universal among algebra actions making $M$ into a module. In particular, the fiber over $M$ of the functor $\lmod(\mathcal{C}) \rightarrow \mathcal{C}_{\mathfrak{m}}$ is non-trivial. Moreover, we have
\begin{proposition}
There is an equivalence of $\infty$-categories 
\begin{align*}
    \alg_{\aass}(\mathcal{C}_{\mathfrak{a}}[M]) \xrightarrow{\simeq} \lmod(\mathcal{C})\times_{\mathcal{C}_{\mathfrak{m}}}\{M\}.
\end{align*}
\end{proposition}
\begin{proof}
By Appendix \ref{appendixA}, $\mathcal{C}_{\mathfrak{a}}[M]$ agrees with Lurie's endomorphism $\infty$-category. Then using \cite[Theorem 4.7.1.34]{HA} and \cite[Remark 4.7.1.35]{HA}, we get an equivalence of $\infty$-categories 
\begin{align*}
    \alg_{\mathbb{A}_{\infty}}(\mathcal{C}_{\mathfrak{a}}[M]) \xrightarrow{\simeq} \lmod(\mathcal{C})\times_{\mathcal{C}_{\mathfrak{m}}} \{M\}.
\end{align*}
Using \cite[Proposition 4.1.3.19]{HA}, we also have an equivalence of $\infty$-categories $\alg_{\aass}(\mathcal{C}_{\mathfrak{a}}[M]) \xrightarrow{\simeq} \alg_{\mathbb{A}_{\infty}}(\mathcal{C}_{\mathfrak{a}}[M])$, and the statement follows by composing these.
\end{proof}
There are a variety of interesting situations in which such an (endo)morphism object fails to exist, in particular if we consider $\infty$-categories arising as categories of algebra objects. The archetypal example is the following.
\begin{example}\label{ex1}
    Let $\field$ be a field and consider the (symmetric) monoidal category $\mathcal{C} = \alg_{\field}$ as left tensored over itself. Let $M\in \mathcal{C}$ be a $\field$-algebra. Then for any endomorphism $\varphi \in \hom_{\mathcal{C}}(M,M)$, the pair $(\field,\field\otimes M \cong M \xrightarrow{\varphi}M)$ is an object of $\mathcal{C}[M]$. Therefore, if $(A,\alpha)\in \mathcal{C}[M]$ is a final object, then $\alpha\circ (u_A\otimes \text{id}_M) = \varphi$ for any endomorphism $\varphi: M\rightarrow M$, where $u_A$ is the unit of $A$. But this is only possible if $M$ only admits a single endomorphism. This difficulty was to be expected, since we know that the monoidal category of $\field$-algebras is not closed.
\end{example}
Lurie's solution to this problem is to relax the expectations on the morphism object. In the above discussion, we start out requiring that $\mor_{\mathcal{C}_{\mathfrak{m}}}(M,N)$ classify all morphisms $C\otimes M\rightarrow N$ in $\mathcal{C}_{\mathfrak{m}}$, and then in the case $N=M$ get for free that $\e(M)$ also classifies algebra actions of algebras on $M$. Instead, we now consider objects that only classify the algebra actions. 
\begin{definition}\label{def5}
    Let $\mathcal{C}^{\otimes}\rightarrow \lm$ be a coCartesian fibration of $\infty$-operads, and let $M\in \mathcal{C}_{\mathfrak{m}}$. A \textbf{center} $\mathfrak{Z}(M)$ of $M$ is a final object of $\lmod(\mathcal{C})\times_{\mathcal{C}_{\mathfrak{m}}}\{M\}$. We generally identify $\mathfrak{Z}(M)$ with its image in $\alg_{\aass}(\mathcal{C}_{\mathfrak{a}})$.
\end{definition}
Clearly if $M$ admits an endomorphism object, then this endomorphism object is also a center of $M$. The converse does not hold: The category $\alg_{\aass}(\mathcal{C}_{\mathfrak{a}}[M])$ might have final objects although $\mathcal{C}_{\mathfrak{a}}[M]$ does not.
\begin{example}[Example \ref{ex1} continued]
The ordinary center $Z(A)$ of an associative $\field$-algebra $A$ is indeed the universal algebra object acting on $A$. To see this, note first that the center is a commutative algebra, and therefore an algebra object in the category of associative algebras. It comes with a natural action on $A$ given by multiplication in $A$. Now suppose that $B$ is a commutative algebra with action $\beta: B\otimes A \rightarrow A$ making $A$ into a $B$-module. Then the restriction of $\beta$ to $A$ yields the identity on $A$ and $\beta$ must be an algebra morphism. Hence 
\begin{align*}
    \beta(b\otimes a) &= \beta(b\otimes 1)\cdot \beta(1\otimes a) = \beta(b\otimes 1) \cdot a \quad \text{and} \\
    \beta(b\otimes a) &= \beta(1\otimes a) \cdot \beta(b\otimes 1) = a\cdot \beta(b\otimes 1),
\end{align*}
showing that $\beta$ sends $B$ to $Z(A)$.
\end{example}
There also is a relative version of the center. 
\begin{definition}\label{def6}
Let $\mathcal{C}^{\otimes} \rightarrow \lm$ be a coCartesian fibration of $\infty$-operads, let $\mathbb{1}$ denote the monoidal unit of $\mathcal{C}_{\mathfrak{a}}^{\otimes}$, and let $f: M \rightarrow N$ be a morphism in $\mathcal{C}_{\mathfrak{m}}$. A \textbf{centralizer} $\mathfrak{Z}(f)$ of $f$ is a final object in
\begin{align*}
    \textnormal{Act}(f) := (\mathcal{C}_{\mathfrak{a}})_{\mathbb{1}/} \times_{{\mathcal{C}_{\mathfrak{m}}}_{M/}} ({\mathcal{C}_{\mathfrak{m}}}_{M/})_{/f}.
\end{align*}
We generally identify $\mathfrak{Z}(f)$ with its image in $\mathcal{C}_{\mathfrak{a}}$.
\end{definition}
The objects of this $\infty$-category are given by commuting triangles in $\mathcal{C}_{\mathfrak{m}}$
\begin{center}
    \begin{tikzcd}
                                          & C\otimes M \arrow[rd] &   \\
\mathbb{1}\otimes M \arrow[rr,"f"'] \arrow[ru] &                       & N
\end{tikzcd}.
\end{center}
In particular, the centralizer is equipped with an action $\mathfrak{Z}(f) \otimes M \rightarrow N$ making the above diagram commute. 
\begin{lemma}\label{lem5}
Let $\mathcal{C}^{\otimes} \rightarrow \lm$ be a coCartesian fibration of $\infty$-operads. Let $f: M\rightarrow N$ be a morphism in $\mathcal{C}_{\mathfrak{m}}$. Let $\overline{M}\in \lmod_{\mathbb{1}}(\mathcal{C}_{\mathfrak{m}})$ be a lift of $M$ as module over the trivial algebra. Let $\mathcal{C}^{\otimes}_{\overline{M}_{{\alm}/}}\rightarrow \lm$ be defined as in \cite[Notation 2.2.2.3]{HA}. Then centralizers of $f$ can be identified with morphism objects
\begin{align*}
    \mor_{{\mathcal{C}_{\mathfrak{m}}}_{M/}}(\textnormal{id}_M,f)\in (\mathcal{C}_{\mathfrak{a}})_{\mathbb{1}/}.
\end{align*}
\end{lemma}
\begin{proof}
Let $\mathcal{C}'^{\otimes} := \mathcal{C}^{\otimes}_{\overline{M}_{\lm}}$. By Proposition \ref{prop6}, it suffices to show that we have an equivalence of $\infty$-categories $\text{Act}(f)\simeq (\mathcal{C}')_{\mathfrak{a}} \times_{\mathcal{C}_{\mathfrak{m}}'} {\mathcal{C}_{\mathfrak{m}}'}_{/f}$. But we have $(\mathcal{C}')_{\mathfrak{a}} \times_{\mathcal{C}_{\mathfrak{m}}'} {\mathcal{C}_{\mathfrak{m}}'}_{/f} \simeq (\mathcal{C}_{\mathfrak{a}})_{\mathbb{1}/} \times_{{\mathcal{C}_{\mathfrak{m}}}_{M/}} ({\mathcal{C}_{\mathfrak{m}}}_{M/})_{/f}$, so this is clear.
\end{proof}
We would like to see that these notions are compatible, in the sense that the centralizer of an identity morphism recovers the center. 
\begin{proposition}[Proposition 5.3.1.8 \cite{HA}]\label{prop7}
Let $M \in \mathcal{C}_{\mathfrak{m}}$, and suppose there exists a centralizer $\mathfrak{Z}(\textnormal{id}_M)\in \mathcal{C}_{\mathfrak{a}}$. Then there exists a center $\mathfrak{Z}(M)\in \alg_{\aass}(\mathcal{C}_{\mathfrak{a}})$. Further, a lift of $M$ to a module over an algebra $A\in \alg_{\aass}(\mathcal{C}_{\mathfrak{a}})$ exhibits $A$ as a center of $M$ if and only if the action map $A\otimes M \rightarrow M$ exhibits $A$ as a centralizer of $\textnormal{id}_M$.
\end{proposition}
\begin{proof}[Proof]
By Lemma \ref{lem5}, the centralizer of the identity is a morphism object $\mor_{{\mathcal{C}_{\mathfrak{m}}}_{M/}}(\text{id}_M,\text{id}_M)$. By Proposition \ref{cor6}, this morphism object admits an essentially unique structure of an algebra object in $(\mathcal{C}_{\mathfrak{a}})_{\mathbb{1}/}$, and $\text{id}_M$ lifts to a module over this algebra structure. In particular, $\mathfrak{Z}(\text{id}_M)$ admits a canonical algebra structure making it into the center of $\text{id}_M$ in ${\mathcal{C}_{\mathfrak{m}}}_{M/}$. Now use \cite[Lemma 5.3.1.10]{HA} to see that the forgetful functor $\lmod({\mathcal{C}_{\mathfrak{m}}}_{M/})\times_{{\mathcal{C}_{\mathfrak{m}}}_{M/}}\{\text{id}_M\} \rightarrow \lmod(\mathcal{C})\times_{\mathcal{C}_{\mathfrak{m}}}\{M\}$ preserves final objects.
\end{proof}
\section{Tensor product of $\infty$-operads}
The Boardman-Vogt tensor product on ordinary operads is designed such that algebras over the tensor product $\mathcal{P}\boxtimes_{\text{BV}} \mathcal{O}$ are given by $\mathcal{P}$-algebras in the category of $\mathcal{O}$-algebras. However, it is well-known that this tensor product does not make the category of (reduced) operads into a monoidal model category. We briefly review the corresponding construction for $\infty$-operads following \cite[Section 2.5.5]{HA}. \displaypar
We want to capture bilinearity of a map between $\infty$-operads. To this end, define a functor $\wedge: \fin \times \fin \rightarrow \fin$ by sending $(\langle m\rangle,\langle n\rangle)$ to the pointed set $(\langle m\rangle^{\circ}\times \langle n\rangle^{\circ})_{+}\cong \langle mn\rangle$, where the isomorphism is given by the lexicographic ordering, and by sending $(f: \langle m\rangle \rightarrow \langle n\rangle,g: \langle m'\rangle \rightarrow \langle n'\rangle)$ to 
\begin{align*}
    \langle mm'\rangle \xrightarrow{\cong} (\langle m\rangle^{\circ}\times \langle m'\rangle^{\circ})_+ \xrightarrow{f\times g} (\langle n\rangle^{\circ}\times \langle n'\rangle^{\circ})_+ \xrightarrow{\cong} \langle nn'\rangle.
\end{align*}
\begin{definition}
We call a map of simplicial sets $F: \mathcal{O}^{\otimes} \times \mathcal{O}'^{\otimes} \rightarrow \mathcal{O}''^{\otimes}$ a \textbf{bifunctor of $\infty$-operads} if the diagram below commutes, and if $F$ sends pairs of inert maps to an inert map in $\mathcal{O}''^{\otimes}$.
\begin{center}
\begin{tikzcd}
\mathcal{O}^{\otimes}\times \mathcal{O}'^{\otimes} \arrow[d] \arrow[r, "F"] & \mathcal{O}''^{\otimes} \arrow[d] \\
\fin \times \fin \arrow[r, "\wedge"]                                        & \fin                             
\end{tikzcd}
\end{center}
Define $\text{Bil}(\mathcal{O}^{\otimes},\mathcal{O}'^{\otimes};\mathcal{O}''^{\otimes})$ to be the full subcategory of $\text{Fun}_{\fin}(\mathcal{O}^{\otimes}\times \mathcal{O}'^{\otimes},\mathcal{O}''^{\otimes})$ spanned by the bifunctors.
\end{definition}
\begin{proposition}[Proposition 3.2.4.3, \cite{HA}]\label{prop10}
    Let $\mathcal{C}^{\otimes}$ be a symmetric monoidal $\infty$-category and $\mathcal{O}^{\otimes}$ an $\infty$-operad. Then the functor 
    \begin{align*}
        \mathbf{sSet}_{/\fin} \rightarrow \mathbf{Set}
    \end{align*}
    sending $K \rightarrow \fin$ to the set of diagrams
    \begin{center}
 \begin{tikzcd}
K \times \mathcal{O}^{\otimes} \arrow[d] \arrow[r, "F"] & \mathcal{C}^{\otimes} \arrow[d] \\
\fin\times \fin \arrow[r, "\wedge"]                      & \fin                           
\end{tikzcd}
\end{center}
such that for $v\in K$ a vertex and $f$ an inert morphism in $\mathcal{O}^{\otimes}$, the map $F(s_0(v),f)$ is inert in $\mathcal{C}^{\otimes}$, is representable. The representing object, which we denote by $\alg_{\mathcal{O}}(\mathcal{C})^{\otimes}\rightarrow \fin$, is a coCartesian fibration of $\infty$-operads. 
\end{proposition}
The fiber $\alg_{\mathcal{O}}(\mathcal{C})^{\otimes}_{\langle 1\rangle}$ over $\langle 1\rangle\in \fin$ is given by the full subcategory of $\text{Fun}_{\fin}(\mathcal{O}^{\otimes},\mathcal{C}^{\otimes})$ of maps that preserve inert morphisms, and hence can be identified with the $\infty$-category $\text{Alg}_{\mathcal{O}}(\mathcal{C})$. Hence, this proposition equips the $\infty$-category $\alg_{\mathcal{O}}(\mathcal{C})$ with the structure of a symmetric monoidal $\infty$-category. For $X\in \mathcal{O}$, the evaluation map $\alg_{\mathcal{O}}(\mathcal{C})^{\otimes} \rightarrow \mathcal{C}^{\otimes}$ is a morphism of $\infty$-operads, hence a lax symmetric monoidal functor, and we see that the symmetric monoidal structure on $\alg_{\mathcal{O}}(\mathcal{C})$ is given by the pointwise tensor product in $\mathcal{C}^{\otimes}$.
\begin{corollary}\label{cor11}
There is an equivalence of $\infty$-categories 
\begin{align*}
     \textnormal{Bil}(\mathcal{O}^{\otimes},\mathcal{O}'^{\otimes};\mathcal{O}''^{\otimes}) \simeq \alg_{\mathcal{O}}(\alg_{\mathcal{O}'}(\mathcal{C})),
\end{align*}
where the category $\alg_{\mathcal{O}'}(\mathcal{C})$ is viewed as a symmetric monoidal $\infty$-category via Proposition \ref{prop10}.
\end{corollary}
\begin{proof}
An $\mathcal{O}$-algebra in $\alg_{\mathcal{O}'}(\mathcal{C})^{\otimes}$ is given by a morphism of simplicial sets $\mathcal{O}^{\otimes} \rightarrow \alg_{\mathcal{O}'}(\mathcal{C})^{\otimes}$ over $\fin$ sending inert morphisms to inert morphisms. In particular, by the construction of $\alg_{\mathcal{O}'}(\mathcal{C})^{\otimes}$, such an $\mathcal{O}$-algebra is given by a diagram
\begin{center}
    \begin{tikzcd}
\mathcal{O}^{\otimes}\times \mathcal{O}'^{\otimes} \arrow[d] \arrow[r] & \mathcal{C}^{\otimes} \arrow[d] \\
\fin\times \fin \arrow[r, "\wedge"]                                    & \fin                           
\end{tikzcd}
\end{center}
such that for every $X\in \mathcal{O}^{\otimes}$ and every inert map $f$ in $\mathcal{O}'^{\otimes}$, the tuple $(\text{id}_X,f)$ is sent to an inert map in $\mathcal{C}^{\otimes}$. The condition that inert morphisms in $\mathcal{O}^{\otimes}$ are sent to inert maps in $\alg_{\mathcal{O}'}(\mathcal{C})^{\otimes}$ translates to the fact that for inert maps $f$ in $\mathcal{O}^{\otimes}$ and $X\in \mathcal{O}'$, the tuple $(f,\text{id}_X)$ is sent to an inert map in $\mathcal{C}^{\otimes}$; and together those two conditions say exactly that tuples of inert maps are sent to an inert map. This is true if and only if $F$ is a bifunctor $\mathcal{O}^{\otimes} \times \mathcal{O}'^{\otimes}\rightarrow \mathcal{C}^{\otimes}$.
\end{proof}
\begin{definition}
We say that a bifunctor $F:\mathcal{O}^{\otimes}\times \mathcal{O}'^{\otimes} \rightarrow \mathcal{O}''^{\otimes}$ \textbf{exhibits $\mathcal{O}''^{\otimes}$ as a tensor product} of $\mathcal{O}^{\otimes}$ and $\mathcal{O}'^{\otimes}$ if for every $\infty$-operad $\mathcal{C}^{\otimes}$, pre-composition with $F$ determines an equivalence of $\infty$-categories 
\begin{align*}
    \alg_{\mathcal{O}''}(\mathcal{C}) \rightarrow \text{Bil}(\mathcal{O}^{\otimes},\mathcal{O}'^{\otimes};\mathcal{C}^{\otimes}).
\end{align*}
We say that $\mathcal{O}''^{\otimes}$ is a tensor product of $\mathcal{O}^{\otimes}$ and $\mathcal{O}'^{\otimes}$ if there exists a bifunctor with this property.
\end{definition}
In particular, in this case, if $\mathcal{C}^{\otimes}$ is a symmetric monoidal $\infty$-category, the above discussion shows that we have an equivalence of $\infty$-categories
\begin{align*}
    \alg_{\mathcal{O}''}(\mathcal{C}) \rightarrow \alg_{\mathcal{O}}(\alg_{\mathcal{O}'}(\mathcal{C})).
\end{align*}
\par In this sense, the tensor product of $\infty$-operads is a derived version of the Boardman-Vogt tensor products of operads.
\section{The little cubes operads}\label{the_little_cubes_operads}
We are mainly interested in algebras over the little $k$-disks operads. These are topological operads, and we present a construction to view them as $\infty$-operads. For ease of construction, we consider ``little $k$-cubes'' instead of ``little $k$-disks''. \displaypar
Let $\square^k := [0,1]^k$ be the $k$-unit cube. Consider the topological one-colored operad $\mathbb{E}_k^T$ with $\mathbb{E}_k^T(n) = \text{Rect}(\square^k \times \{1,\dots,n\},\square^k)$ the space of rectilinear embeddings. Let $\mathbb{E}_k^{T,\otimes}$ denote its topological category of operators, and consider the corresponding simplicial category $\text{Sing}_{\bullet}(\mathbb{E}_k^{T,\otimes})$. Taking the homotopy coherent nerve we obtain an $\infty$-category $\mathbb{E}_k^{\otimes}$, which is an $\infty$-operad since the underlying simplicial operad is fibrant. In particular, objects in $\mathbb{E}_k^{\otimes}$ are given by $\langle n\rangle$ for $n\in \mathbb{N}$, morphisms are given by points in the space
\begin{align*}
    \map_{\mathbb{E}_k^{T,\otimes}}(\langle n\rangle,\langle m\rangle) = \coprod_{f: \langle n\rangle \rightarrow\langle m\rangle} \prod_{j\in \langle m\rangle^{\circ}}  \text{Rect}(\square^k \times \{1,\dots,n\},\square^k),
\end{align*}
and to give a 2-simplex with boundary as shown below is equivalent to giving a path from $F\circ E$ to $G$ in the space $\map_{\mathbb{E}_k^{T,\otimes}}(\langle m\rangle,\langle k\rangle)$.
\begin{center}
    \begin{tikzcd}
                                                  & \langle m\rangle \arrow[rd, "F"] &                  \\
\langle n\rangle \arrow[rr, "G"'] \arrow[ru, "E"] &                                  & \langle k\rangle
\end{tikzcd}
\end{center}
We sometimes also denote the object $\langle n\rangle$ of $\mathbb{E}_k^{\otimes}$ by $\{\underbrace{\mathfrak{a},\dots,\mathfrak{a}}_{n \text{ times}}\}$, where we view $\mathfrak{a}$ as the single color of $\mathbb{E}_k^T$.\displaypar 
Following \cite[Notation 2.1.1.16]{HA}, we denote by $\text{Mul}_{\mathbb{E}_k}(\langle n\rangle, \langle 1\rangle)\in h\an$ the fiber of the Kan complex $\map_{\mathbb{E}_k^{\otimes}}(\langle n\rangle, \langle 1\rangle)$ over the active map $\langle n \rangle \rightarrow \langle 1 \rangle$. Note that the mapping spaces of the $\infty$-category $\mathbb{E}_k^{\otimes}$ are weakly equivalent to the mapping spaces in the topological category $\mathbb{E}_k^{T,\otimes}$. We hence have a weak equivalence $\text{Mul}_{\mathbb{E}_k}(\langle n\rangle,\langle 1\rangle)\simeq \text{Rect}(\square^k\times \{1,\dots,n\},\square^k) = \mathbb{E}_k^T(n)$, and we will implicitly use the space $\mathbb{E}_k^T(n)$ as a representative for the weak homotopy type $\text{Mul}_{\mathbb{E}_k}(\langle n\rangle, \langle 1\rangle)$.\displaypar
For $k=1$ we obtain the $\mathbb{E}_1^{\otimes}$-operad, which is equivalent to the $\infty$-operad $\ass$ and governs homotopy associative algebras. We will generally identify algebras over these two $\infty$-operads. For $k=2$, we instead recover the $\infty$-operadic version of the little 2-disks operad $D_2$. In particular, an element in $\text{Mul}_{\mathbb{E}_1}(\langle n\rangle,\langle 1\rangle)$ is given by a rectangular embedding of $n$ copies of the interval $[0,1]$ into the interval $[0,1]$. An element in $\text{Mul}_{\mathbb{E}_2}(\langle n\rangle,\langle 1\rangle)$ is given by a rectangular embedding of $n$ copies of the square $[0,1]\times[0,1]$ into the square $[0,1]\times [0,1]$. 
\begin{figure}[h!]
    \centering
    \tikzset{every picture/.style={line width=0.75pt}} 
\begin{tikzpicture}[x=0.75pt,y=0.75pt,yscale=-1,xscale=1]

\draw   (230,109.75) .. controls (230,82.27) and (252.27,60) .. (279.75,60) .. controls (307.23,60) and (329.5,82.27) .. (329.5,109.75) .. controls (329.5,137.23) and (307.23,159.5) .. (279.75,159.5) .. controls (252.27,159.5) and (230,137.23) .. (230,109.75) -- cycle ;
\draw   (245.14,75.14) -- (314.36,75.14) -- (314.36,144.36) -- (245.14,144.36) -- cycle ;
\draw    (280.5,75) -- (281,144.5) ;
\draw  [fill={rgb, 255:red, 0; green, 0; blue, 0 }  ,fill opacity=1 ] (328.17,109.75) .. controls (328.17,109.01) and (328.76,108.42) .. (329.5,108.42) .. controls (330.24,108.42) and (330.83,109.01) .. (330.83,109.75) .. controls (330.83,110.49) and (330.24,111.08) .. (329.5,111.08) .. controls (328.76,111.08) and (328.17,110.49) .. (328.17,109.75) -- cycle ;

\draw (292.83,101.5) node [anchor=north west][inner sep=0.75pt]   [align=left] {1};
\draw (257.5,101.5) node [anchor=north west][inner sep=0.75pt]   [align=left] {2};
\draw (332,104) node [anchor=north west][inner sep=0.75pt]  [font=\footnotesize] [align=left] {0};

\end{tikzpicture}
    \caption{The element $\mu_0\in \map_{\mathbb{E}_2}(\langle 2\rangle,\langle 1\rangle)_0$.}
\end{figure}
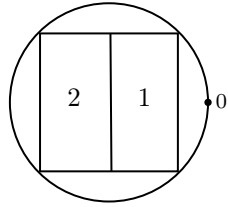
\FloatBarrier
Note that for $k\geq 1$, we have a homotopy equivalence $\text{Mul}_{\mathbb{E}_k}(\langle 2\rangle,\langle 1\rangle)\simeq S^{k-1}$. To construct this, fix a circumscribed $(k-1)$-sphere about $[0,1]^k$. Then to each point in $\text{Mul}_{\mathbb{E}_k}(\langle 2\rangle,\langle 1\rangle)\simeq \mathbb{E}_k^T(2)$ given by a rectangular embedding of two $k$-squares into a $k$-square, assign the intersection $S^{k-1}\cap r_{\vec{21}}$, where $r_{\vec{21}}$ is the ray through the centers of the two embedded copies of $[0,1]^k$ which starts at the center of the second copy. We frequently use this homotopy equivalence as a convenient method to label morphisms in the little cubes operads. In particular, fix a homotopy inverse $S^0 \rightarrow \text{Mul}_{\mathbb{E}_1}(\langle 2\rangle,\langle 1\rangle)$ and a homotopy inverse $S^1 \rightarrow \text{Mul}_{\mathbb{E}_2}(\langle 2\rangle,\langle 1\rangle)$. Then we get two points in $\text{Mul}_{\mathbb{E}_1}(\langle 2\rangle,\langle 1\rangle)$ named $\mu_0$ and $\mu_{1}$, and for every $t\in [0,2\pi)$ we get a point $\mu_t\in \text{Mul}_{\mathbb{E}_2}(\langle 2\rangle,\langle 1\rangle)$. We denote by the same letters representatives of these operations as 0-simplices in the respective mapping spaces.\displaypar
 Recall that 2-morphisms\footnote{We use ``2-morphism'' and ``2-simplex'' in a $\infty$-category interchangeably. However, we will often favor ``2-morphism'' in a context where one edge of the 2-simplex is degenerate to emphasize the globular nature of topological and dg categories.} in $\mathbb{E}^{\otimes}_k$ are given by paths in the relevant mapping spaces. There are two such 2-morphisms in $\mathbb{E}^{\otimes}_2$ that will play a special role in the subsequent discussion. On the one hand, for each $t\in [0,2\pi)$, there is a 2-morphism $\sigma_t\in \map_{\mathbb{E}_2^{\otimes}}(\langle 2\rangle,\langle 1\rangle)_{1}$ with boundary given by $\mu_t$ and $\mu_{t+\pi (\text{mod }2\pi)}$ that is represented by the braid
\begin{figure}[h!]
    \centering
    \tikzset{every picture/.style={line width=0.75pt}} 

\begin{tikzpicture}[x=0.75pt,y=0.75pt,yscale=-1,xscale=1]

\draw    (170.33,70.67) -- (220.67,170.67) ;
\draw    (220.67,70.67) -- (198.67,114.67) ;
\draw    (192,126.67) -- (170,170.67) ;

\draw (223,57.4) node [anchor=north west][inner sep=0.75pt]    {$1$};
\draw (157.33,165.07) node [anchor=north west][inner sep=0.75pt]    {$1$};
\draw (224,166.73) node [anchor=north west][inner sep=0.75pt]    {$2$};
\draw (155,60.07) node [anchor=north west][inner sep=0.75pt]    {$2$};

\end{tikzpicture}.
\end{figure}
\FloatBarrier
On the other hand, for each $t\in [0,2\pi)$, there is a non-trivial 2-morphism $\gamma_t\in  \map_{\mathbb{E}_2^{\otimes}}(\langle 2\rangle,\langle 1\rangle)_{1}$ between $\mu_t$ and itself represented by the double-braid
\begin{figure}[h!]
    \centering
   \tikzset{every picture/.style={line width=0.75pt}} 

\begin{tikzpicture}[x=0.75pt,y=0.75pt,yscale=-1,xscale=1]

\draw    (169.67,50.67) -- (220,150.67) ;
\draw    (220,50.67) -- (198,94.67) ;
\draw    (191.33,106.67) -- (169.33,150.67) ;
\draw    (169.33,150.67) -- (219.67,250.67) ;
\draw    (220,150.67) -- (198,194.67) ;
\draw    (190.67,206) -- (168.67,250) ;

\draw (222.33,37.4) node [anchor=north west][inner sep=0.75pt]    {$1$};
\draw (222.67,249.73) node [anchor=north west][inner sep=0.75pt]    {$1$};
\draw (158.33,250.73) node [anchor=north west][inner sep=0.75pt]    {$2$};
\draw (154.33,40.07) node [anchor=north west][inner sep=0.75pt]    {$2$};

\end{tikzpicture}.
\end{figure}
\FloatBarrier
 The classical Eckmann-Hilton argument shows that, in the 1-categorical case, algebra objects in the category of algebra objects yield commutative algebra objects. In particular, if $(A,\cdot)$ is an associative algebra in a symmetric monoidal category ${C}$ and $\ast: (A,\cdot)\otimes(A,\cdot)\rightarrow (A,\cdot)$ endows $(A,\cdot)$ with the structure of an associative algebra object in $\alg({C})$, then both operations $\cdot$ and $\ast$ agree and they are commutative. In this sense, an $\mathbb{E}_1$-algebra inside the symmetric monoidal category of $\mathbb{E}_1$-algebras in ${C}$ is the same as a commutative algebra inside ${C}$, which in the 1-categorical case is the same as an $\mathbb{E}_2$-algebra. In fact, this pattern continues for all the little $k$-cubes operads, as was shown by Dunn for topological operads and later by Lurie for $\infty$-operads. To explain this, for $k,k'\geq 0$, define a topological functor
\begin{align*}
    \rho: \mathbb{E}_k^{T,\otimes}\times \mathbb{E}_{k'}^{T,\otimes} \rightarrow \mathbb{E}_{k+k'}^{T,\otimes}
\end{align*}
given on objects by $\rho(\langle m\rangle,\langle n\rangle) = \langle m\rangle \wedge \langle n\rangle$, and sending a pair of morphisms $(\alpha,\{f_j: \square^k\times \alpha^{-1}(\{j\})\rightarrow \square^k\}_{j\in \langle n\rangle^{\circ}})$ and $(\beta, \{g_i: \square^{k'}\times \beta^{-1}(\{i\})\rightarrow \square^{k'}\}_{i\in\langle n'\rangle^{\circ}})$ to 
\begin{align*}
    (\alpha \wedge \beta, \{f_j\times g_i:\square^{k+k'}\times \alpha^{-1}(\{j\})\times\beta^{-1}(\{i\})\rightarrow \square^{k+k'}\}_{j\in \langle n\rangle^{\circ},i\in \langle n'\rangle^{\circ}}).
\end{align*}
In order for this to make sense, we note that viewing a tuple $(j,i)\in \langle n\rangle^{\circ}\times \langle n'\rangle^{\circ}$ as an element of $\langle nn'\rangle^{\circ}$, we have $(\alpha\wedge \beta)^{-1}((j,i)) = \alpha^{-1}(\{j)\}\times \beta^{-1}(\{i\})$. This descends to a simplicial functor, which after applying the homotopy coherent nerve yields a map of $\infty$-categories $\rho: \mathbb{E}_k^{\otimes}\times \mathbb{E}_{k'}^{\otimes} \rightarrow\mathbb{E}_{k+k'}^{\otimes}$. By construction, the diagram
\begin{center}
    \begin{tikzcd}
\mathbb{E}_k^{\otimes}\times \mathbb{E}_{k'}^{\otimes} \arrow[d] \arrow[r, "\rho"] & \mathbb{E}_{k+k'}^{\otimes} \arrow[d] \\
\fin\times \fin \arrow[r, "\wedge"]                                                & \fin                                 
\end{tikzcd}
\end{center}
commutes, and clearly $\rho$ sends pairs of inert morphisms to inert morphisms. Thus, $\rho$ is a bifunctor of $\infty$-operads.
\begin{theorem}[Dunn Additivity Theorem, Theorem 5.1.2.2\cite{HA}]\label{thm7}
    The bifunctor $\rho: \mathbb{E}_k^{\otimes} \times \mathbb{E}_{k'}^{\otimes} \rightarrow \mathbb{E}_{k+k'}^{\otimes}$ exhibits the $\infty$-operad $\mathbb{E}_{k+k'}^{\otimes}$ as a tensor product of $\mathbb{E}_k^{\otimes}$ and $\mathbb{E}_{k'}^{\otimes}$.
\end{theorem}
This means that for every symmetric monoidal $\infty$-category $\mathcal{C}^{\otimes}$, precomposition with $\rho$ determines an equivalence of $\infty$-categories
\begin{align*}
    \rho^{\ast}: \text{Alg}_{\mathbb{E}_{k+k'}}(\mathcal{C}) \rightarrow \text{Alg}_{\mathbb{E}_{k}}(\text{Alg}_{\mathbb{E}_{k'}}(\mathcal{C})).
\end{align*}
In particular, for every $\mathbb{E}_k$-algebra $A$ in $\mathbb{E}_{k'}$-algebras in $\mathcal{C}^{\otimes}$, there exists a $\mathbb{E}_{k+k'}$-algebra $\tilde{A}$ in $\mathcal{C}^{\otimes}$ such that $\tilde{A}\circ \rho$ is equivalent to $A$ in the $\infty$-category $\text{Alg}_{\mathbb{E}_{k}}(\text{Alg}_{\mathbb{E}_{k'}}(\mathcal{C}))$.\displaypar
Let $\mathcal{C}^{\otimes}$ be a symmetric monoidal $\infty$-category and consider the symmetric monoidal $\infty$-category $\alg_{\mathbb{E}_k}(\mathcal{C})^{\otimes}$ constructed in \ref{prop10}. Recall from Definition \ref{def5} that the center $\mathfrak{Z}(A)$ of an $\mathbb{E}_k$-algebra $A\in \alg_{\mathbb{E}_k}(\mathcal{C})$ carries the structure of an $\aass$-algebra in $\alg_{\mathbb{E}_k}(\mathcal{C})^{\otimes}$. Identifying $\ass$ with $\mathbb{E}_1^{\otimes}$, the Dunn additivity theorem then implies that the center is in fact an $\mathbb{E}_{k+1}$-algebra
\begin{align*}
    \mathfrak{Z}(A) \in \alg_{\mathbb{E}_1}(\alg_{\mathbb{E}_k}(\mathcal{C})) \simeq \alg_{\mathbb{E}_{k+1}}(\mathcal{C}).
\end{align*}
\section{The center as endomorphism object of bimodules} \label{the_center_as_endomorphism_object_of_bimodules}
We have a particularly nice description for the center in symmetric monoidal $\infty$-categories that arise as the category of $\mathcal{O}$-algebras for a coherent $\infty$-operad $\mathcal{O}$.\displaypar 
First consider the 1-categorical case. There is an equivalence of 1-categories between associative algebra objects in the monoidal category of $A$-bimodules and associative algebras under $A$
\begin{align*}
    \alg({}_A\text{BiMod}_A) \xrightarrow{\simeq} \alg(\text{Vect}_{\field})^{A/}.
\end{align*}
The equivalence is given by sending an algebra in bimodules to its unit morphism. Note that $A$ is the monoidal unit in $\text{Bimod}_A$ and hence $\alg(\text{Bimod}_A)$. The 1-centralizer of an algebra morphism $f: A\rightarrow B$ is defined as 
    \begin{align*}
        Z(f) = \{b\in B: \forall a\in A: f(a)b=bf(a)\}.
    \end{align*}
Viewing the data of $f$ as an $A$-bimodule structure on $B$, we can see that this agrees with the set of $A$-bimodule maps from $A$ to $B$:
    \begin{align*}
        {Z}(f) \cong \hom_{{}_A\text{BiMod}_A}(A,B).
    \end{align*}
We explain how this picture generalizes to the $\infty$-categorical case. Let $\mathcal{C}^{\otimes} \rightarrow \fin$ be a symmetric monoidal $\infty$-category and $\mathcal{O}^{\otimes}$ be a coherent $\infty$-operad; we will mostly be interested in the case $\mathcal{O}^{\otimes} = \mathbb{E}_1^{\otimes}$. Consider the unique bifunctor of $\infty$-operads $\mathcal{O}^{\otimes} \times \lm \rightarrow \fin \times \fin \xrightarrow{\wedge} \fin$. We have a coCartesian fibration $\alg_{\mathcal{O}}(\mathcal{C})^{\otimes} \rightarrow \lm$ with fibers over $\mathfrak{a}$ and $\mathfrak{m}$ respectively both equivalent to $\alg_{\mathcal{O}}(\mathcal{C})$ coming from the fact that $\alg_{\mathcal{O}}(\mathcal{C})$ admits the structure of a symmetric monoidal $\infty$-category as constructed in Proposition \ref{prop10}, and is hence left tensored over itself.\displaypar 
Let $A\in \alg_{\mathcal{O}}(\mathcal{C})_{\mathfrak{m}}$. To generalize the relationship between a centralizer and bimodule morphisms to $\infty$-operads, one needs to first recover the statement about algebra objects in the monoidal category of bimodules. Let $\overline{A}\in \lmod_{\mathbb{1}}(\alg_{\mathcal{O}}(\mathcal{C}))$ be a lift of $A$ as a module over the trivial algebra $\mathbb{1}\in \alg_{\mathbb{E}_1}(\alg_{\mathcal{O}}(\mathcal{C})_{\mathfrak{a}})$. We obtain a coCartesian $\lm$-family of $\mathcal{O}$-operads as defined in \cite[Definition 5.3.1.19]{HA} 
\begin{align*}
    \mathcal{C}^{\otimes} \times_{\fin} (\mathcal{O}^{\otimes} \times \lm) \rightarrow \mathcal{O}^{\otimes} \times \lm,
\end{align*}
and we can regard $\overline{A}$ as a coCartesian $\lm$-family of $\mathcal{O}$-algebras as defined in \cite[Remark 5.3.1.22]{HA} by noting that there is a bijection
\begin{align*}
    \text{Fun}_{\lm}(\lm, \alg_{/\mathcal{O}}^{\lm}(\mathcal{C}^{\otimes} \times_{\fin} (\mathcal{O}^{\otimes} \times \lm))) \cong \text{Fun}_{\lm}(\lm, \alg_{\mathcal{O}}(\mathcal{C})^{\otimes}).
\end{align*}
We can view $A$ as an $\mathcal{O}$-module over itself, and hence $\overline{A}$ also determines a coCartesian $\lm$-family of $\mathcal{O}$-algebras in the coCartesian $\lm$-family of $\mathcal{O}$-operads\footnote{See \cite[Definition 5.3.1.23]{HA}.}
\begin{align*}
    \overline{C}^{\otimes} := \text{Mod}^{\mathcal{O},\lm}_{\overline{A}}(\mathcal{C}^{\otimes} \times_{\fin} (\mathcal{O}^{\otimes} \times \lm))^{\otimes} \rightarrow \mathcal{O}^{\otimes} \times \lm.
\end{align*}
This allows us to identify algebra objects in the category of $A$-bimodules with algebras under $A$:
\begin{proposition}[Proposition 5.3.1.27 \cite{HA}]
The forgetful functor 
\begin{align*}
    \theta: \alg_{/\mathcal{O}}^{\lm}(\textnormal{Mod}^{\mathcal{O},\lm}_{\overline{A}}(\mathcal{C}^{\otimes} \times_{\fin} (\mathcal{O}^{\otimes} \times \lm))^{\overline{A}_{\lm/}} \rightarrow \alg_{/\mathcal{O}}^{\lm}(\mathcal{C}^{\otimes} \times_{\fin} (\mathcal{O}^{\otimes}\times \lm))^{\overline{A}_{\lm/}}
\end{align*}
is an equivalence of $\infty$-categories. 
\end{proposition}
\begin{remark}
Note that for all $s\in\lm$, the algebra $\overline{A}_s\in \alg_{/\mathcal{O}}(\textnormal{Mod}^{\mathcal{O}}_{\overline{A}_s}(\mathcal{C}^{\otimes} \times_{\fin} (\mathcal{O}^{\otimes} \times \{s\})))$ is a trivial algebra, so in particular for $s= \mathfrak{m}$ we get an equivalence 
\begin{align*}
    \theta_{\mathfrak{m}}: \alg_{/\mathcal{O}}(\textnormal{Mod}_{A}^{\mathcal{O}}(\mathcal{C}^{\otimes} \times_{\fin} (\mathcal{O}^{\otimes} \times \{\mathfrak{m}\}))) \rightarrow \alg_{/\mathcal{O}}(\mathcal{C}^{\otimes} \times_{\fin} (\mathcal{O}^{\otimes} \times \{\mathfrak{m}\}))^{A/} \simeq \alg_{\mathcal{O}}(\mathcal{C})_{\mathfrak{m}}^{A/}.
\end{align*}
\end{remark}
Since $\overline{A}(\mathfrak{a})$ is the trivial algebra in $\alg_{\mathcal{O}}(\mathcal{C})_{\mathfrak{a}}$, we also have an equivalence $\overline{\mathcal{C}}^{\otimes}_{\mathfrak{a}} \simeq \mathcal{C}^{\otimes} \times_{\fin} (\mathcal{O}^{\otimes} \times \{\mathfrak{a}\})$, and therefore $\alg^{\lm}_{/\mathcal{O}}(\overline{\mathcal{C}})_{\mathfrak{a}} \simeq \alg_{\mathcal{O}}(\mathcal{C})_{\mathfrak{a}}$. To find the centralizer of $\text{id}_A$ in $\alg_{\mathcal{O}}(\mathcal{C})_{\mathfrak{m}}$, it hence suffices to find the centralizer of $\text{id}_A$ in $\alg_{/\mathcal{O}}(\textnormal{Mod}^{\mathcal{O}}_{A}(\mathcal{C}^{\otimes} \times_{\fin} (\mathcal{O}^{\otimes} \times \{\mathfrak{m}\})))$, in which $A$ is the trivial algebra. We can then use Proposition \ref{prop7} to upgrade this to a center $\mathfrak{Z}(A) \in \alg_{\mathbb{E}_1}(\alg_{\mathcal{O}}(\mathcal{C})_{\mathfrak{a}})$ of $A$. 
\begin{theorem}[Proposition 5.3.1.29 \cite{HA}]\label{thm8}
Suppose that for all $X\in \mathcal{O}$, there exists a morphism object $\mor_{\overline{\mathcal{C}}_{X,\mathfrak{m}}}({A}(X),{A}(X))\in \overline{ \mathcal{C}}_{X,\mathfrak{a}}$. Then there exists a centralizer $\mathfrak{Z}(\textnormal{id}_{A})\in \alg_{/\mathcal{O}}^{\lm}(\overline{\mathcal{C}})_{\mathfrak{a}}$. Furthermore, if $Z\in \alg_{/\mathcal{O}}^{\lm}(\overline{\mathcal{C}})_{\mathfrak{a}}$, then a commutative diagram 
\begin{center}
    \begin{tikzcd}
                                              & Z\otimes A \arrow[rd] &   \\
A \arrow[rr, "\textnormal{id}_A"'] \arrow[ru] &                       & A
\end{tikzcd}
\end{center}
exhibits $Z$ as the centralizer of $\textnormal{id}_A$ if and only if for all $X\in \mathcal{O}$, the induced map $Z(X)\otimes A(X) \rightarrow A(X)$ exhibits $Z(X)$ as a morphism object of $A(X)$ and $A(X)$.
\end{theorem}
\begin{proof}
By definition, the centralizer is a final object of the $\infty$-category
\begin{align*}
    \mathcal{A} := (\alg_{/\mathcal{O}}^{\lm}(\overline{\mathcal{C}})_{\mathfrak{a}})_{\mathbb{1}} \times_{(\alg_{/\mathcal{O}}^{\lm}(\overline{\mathcal{C}})_{\mathfrak{m}})_{A/}} (\alg_{/\mathcal{O}}^{\lm}(\overline{\mathcal{C}})_{\mathfrak{m}})_{A//A}.
\end{align*}
Since $\overline{A}_s$ is the trivial algebra in $\overline{\mathcal{C}}_s^{\otimes}$ for $s\in \{\mathfrak{a},\mathfrak{m}\}$, we can use \cite[Theorem 2.2.2.4]{HA} to get an $\mathcal{O}^{\otimes}$-monoidal $\infty$-category 
\begin{align*}
    \mathcal{E}^{\otimes} := (\overline{\mathcal{C}}_{\mathfrak{a}}^{\otimes})_{\mathbb{1}_{\mathcal{O}/}} \times_{(\overline{\mathcal{C}}_{\mathfrak{m}}^{\otimes})_{A_{\mathcal{O}/}}} (\overline{\mathcal{C}}_{\mathfrak{m}}^{\otimes})_{A_{\mathcal{O}//A_{\mathcal{O}}}} \rightarrow \mathcal{O}^{\otimes}
\end{align*}
such that $\alg_{/\mathcal{O}}(\mathcal{E}) \simeq \mathcal{A}$. Finally, use that limits in algebra categories are computed objectwise by \cite[Corollary 3.2.2.5]{HA} to argue that we are reduced to showing that for each $X\in \mathcal{O}$, the fiber $\mathcal{E}_X$ admits a final object. But a final object in 
\begin{align*}
    \mathcal{E}_X \simeq (\overline{\mathcal{C}}_{X,\mathfrak{a}})_{\mathbb{1}(X)/} \times_{(\overline{\mathcal{C}}_{X,\mathfrak{m}})_{A(X)/}}(\overline{\mathcal{C}}_{X,\mathfrak{m}})_{A(X)//A(X)}
\end{align*}
is equivalent to a morphism object $\mor_{\overline{\mathcal{C}}_{X,\mathfrak{m}}}(A(X),A(X))$ by Proposition \ref{prop6}.
\end{proof}
\begin{corollary}\label{cor8}
Let $\mathcal{O}^{\otimes}= \mathbb{E}_1^{\otimes}$, and let $A\in \alg_{\mathbb{E}_1}(\mathcal{C})$ be an $\mathbb{E}_1$-algebra in a symmetric monoidal $\infty$-category $\mathcal{C}^{\otimes}$. Assume that the morphism object $\mor_{\textnormal{Mod}_A^{\mathbb{E}_1}(\mathcal{C}^{\otimes} \times_{\fin} \mathbb{E}_1^{\otimes})_{\mathfrak{a}}}(A,A)\in \mathcal{C}$ exists\footnote{Note that here we implicitly view $A$ as a module over itself, i.e. as an object $A\in \textnormal{Mod}_A^{\mathbb{E}_1}(\mathcal{C}^{\otimes} \times_{\fin} (\mathbb{E}_1^{\otimes}\times \{\mathfrak{m}\}))_{\mathfrak{a}}$.}. Then there exists a centralizer $\mathfrak{Z}(\textnormal{id}_A) \in \alg_{\mathbb{E}_1}(\mathcal{C})$ with underlying object 
\begin{align*}
\mathfrak{Z}(\textnormal{id}_A)(\mathfrak{a}) \simeq \mor_{\textnormal{Mod}_A^{\mathbb{E}_1}(\mathcal{C}^{\otimes} \times_{\fin} \mathbb{E}_1^{\otimes})_{\mathfrak{a}}}(A,A),
\end{align*}
and the action of the centralizer has underlying map given by the evaluation $\alpha$ of the morphism object on $A$. Further, the multiplication of the $\mathbb{E}_1$-algebra structure on $\mathfrak{Z}(\textnormal{id}_A)$ is induced by the action of the tensor product $\mathfrak{Z}(\textnormal{id}_A)(\mathfrak{a}) \otimes \mathfrak{Z}(\textnormal{id}_A)(\mathfrak{a})$ on $A$ given by the tensor product $\alpha \otimes \alpha$ in $\mathcal{E}_{\mathfrak{a}}$:
\begin{align*}
    \ast: A \xrightarrow{\simeq} A\otimes A \rightarrow (\mathfrak{Z}(\textnormal{id}_A)(\mathfrak{a}) \otimes A) \otimes (\mathfrak{Z}(\textnormal{id}_A)(\mathfrak{a}) \otimes A) \rightarrow A \otimes A \xrightarrow{\textnormal{mult.}} A.
\end{align*}
\end{corollary}
\begin{proof}
    The first part follows directly from Theorem \ref{thm8}. For the claim about the multiplication, note that by the proof of Theorem \ref{thm8} the algebra structure on the centralizer is induced from $\bigl(\mathfrak{Z}(\text{id}_A)(\mathfrak{a}), \alpha\bigr)$ being final in $\mathcal{E}_{\mathfrak{a}}$. In particular, in $\mathcal{E}_{\mathfrak{a}}$ we have a unique-up-to-contractible-choice map 
    \begin{align*}
        \mathfrak{Z}(\text{id}_A)(\mathfrak{a}) \otimes \mathfrak{Z}(\text{id}_A)(\mathfrak{a}) \rightarrow \mathfrak{Z}(\text{id}_A)(\mathfrak{a}).
    \end{align*}
    Now recall the construction of the tensor product in operadic slice categories in \cite[Theorem 2.2.2.4]{HA}. This shows that up to homotopy, the monoidal product in $\mathcal{E}_{\mathfrak{a}}$ is given in the first component by 
    \begin{center}
        \begin{tikzcd}
\mathbb{1} \arrow[dd] &          & \mathbb{1} \arrow[dd] &   & \mathbb{1} \arrow[d, "\simeq"']         \\
                      & \otimes  &                       & = & \mathbb{1} \otimes \mathbb{1} \arrow[d] \\
X                     &          & Y                     &   & X\otimes Y                             
\end{tikzcd}
    \end{center}
    and in the second component by 
    \begin{center}
       \begin{tikzcd}
A \arrow[d, "\simeq"'] \arrow[rdd, "\textnormal{id}_{A}"] &                 &                                                                                            & A \arrow[d, "\simeq"'] \arrow[rdd, "\textnormal{id}_{A}"] &                 \\
\mathbb{1} \otimes A \arrow[d]                                          &                 & \otimes                                                                                    & \mathbb{1} \otimes A \arrow[d]                                          &                 \\
X \otimes A \arrow[r]                                                   & A &                                                                                            & Y\otimes A \arrow[r]                                                    & A \\
                                                                                      &                 & A \arrow[d, "\simeq"'] \arrow[rrddd, "\textnormal{id}_{A}"]    &                                                                                       &                 \\
                                                                                      &                 & A\otimes A \arrow[d, "\simeq"']                                &                                                                                       &                 \\
                                                                                      & =               & (\mathbb{1} \otimes A) \otimes (\mathbb{1}\otimes A) \arrow[d] &                                                                                       &                 \\
                                                                                      &                 & (X\otimes A) \otimes (Y \otimes A) \arrow[r]                   & A\otimes A \arrow[r, "\textnormal{mult}"']                & A
\end{tikzcd}
    \end{center}
    This corresponds to the tensor product action given in the statement. 
\end{proof}
\begin{corollary}\label{cor7}
    Let $\mathcal{O}^{\otimes} = \mathbb{E}_1^{\otimes}$ and $A\in \alg_{\mathbb{E}_1}(\mathcal{C})$, and assume that the morphism object exists. Then there exists a center $\mathfrak{Z}(A)\in \alg_{\mathbb{E}_1}(\alg_{\mathbb{E}_1}(\mathcal{C}))$ with underlying object $\mathfrak{Z}(A)(\mathfrak{a})(\mathfrak{a}) \simeq \mor_{\textnormal{Mod}_A^{\mathbb{E}_1}(\mathcal{C}^{\otimes} \times_{\fin} \mathbb{E}_1^{\otimes})_{\mathfrak{a}}}(A,A)$. The outer multiplication is given by the composition product
    \begin{align*}
       \circ:  (\mathfrak{Z}(A)(\mathfrak{a})(\mathfrak{a}) \otimes \mathfrak{Z}(A)(\mathfrak{a})(\mathfrak{a})) \otimes A \xrightarrow{\textnormal{id}\otimes \alpha} \mathfrak{Z}(A)(\mathfrak{a})(\mathfrak{a}) \otimes A \xrightarrow{\alpha} A.
    \end{align*}
    The inner multiplication is given by the convolution product $\ast$ described in Corollary \ref{cor8}. View the center as an object in $\alg_{\mathbb{E}_1}(\alg_{/\mathbb{E}_1}(\mathcal{E}))$ with the monoidal structure on $\alg_{/\mathbb{E}_1}(\mathcal{E})$ given by Proposition \ref{prop9}. The two multiplications assemble into a square 
    \begin{center}
        \begin{tikzcd}
\Bigl(\mathfrak{Z}(A)(\mathfrak{a})\otimes \mathfrak{Z}(A)(\mathfrak{a})\Bigr)(\mathfrak{a})\otimes \Bigl(\mathfrak{Z}(A)(\mathfrak{a})\otimes \mathfrak{Z}(A)(\mathfrak{a})\Bigr)(\mathfrak{a}) \arrow[d] \arrow[rr] &  & \mathfrak{Z}(A)(\mathfrak{a})(\mathfrak{a}) \otimes \mathfrak{Z}(A)(\mathfrak{a})(\mathfrak{a}) \arrow[d, "\ast"] \\
\Bigl(\mathfrak{Z}(A)(\mathfrak{a})\otimes \mathfrak{Z}(A)(\mathfrak{a})\Bigr)(\mathfrak{a}) \arrow[rr, "\circ"']                                                                                           &  & \mathfrak{Z}(A)(\mathfrak{a})(\mathfrak{a})                                                                      
\end{tikzcd}
    \end{center}
    in $\mathcal{E}_{\mathfrak{a}}$. Since $\mathfrak{Z}(A)(\mathfrak{a})(\mathfrak{a})$ is the morphism object and thus final, there is a contractible choice of 2-simplices filling this square.
\end{corollary}
\begin{proof}
    By Proposition \ref{prop7}, the action $\circ$ above is the one induced by the monoidal structure on $\alg_{/\mathbb{E}_1}(\mathcal{E})$. The square is given by evaluating $\mathfrak{Z}(A)$ at the square 
    \begin{equation}\label{dig2}
          \begin{tikzcd}[column sep = large]
{(\langle 2\rangle,\langle 2\rangle)} \arrow[d, "{(\mu_0,\textnormal{id}_{\langle 2\rangle})}"'] \arrow[r, "{(\textnormal{id}_{\langle 2\rangle},\mu_0)}"] & {(\langle 2\rangle,\langle 1\rangle)} \arrow[d, "{(\mu_0,\textnormal{id}_{\langle 1\rangle})}"] \\
{(\langle 1\rangle,\langle 2\rangle)} \arrow[r, "{(\textnormal{id}_{\langle 1\rangle},\mu_0)}"']                                                 & {(\langle 1\rangle,\langle 1\rangle)}                                                
\end{tikzcd}
     \end{equation}
in $\mathbb{E}_1^{\otimes} \times \mathbb{E}_1^{\otimes}$, i.e. it is given by 
    \begin{center}
        \begin{tikzcd}
{\mathfrak{Z}(A)(\{\mathfrak{a},\mathfrak{a}\})(\{\mathfrak{a},\mathfrak{a}\})} \arrow[d] \arrow[r] & {\mathfrak{Z}(A)(\mathfrak{a})(\{\mathfrak{a},\mathfrak{a}\})} \arrow[d] \\
{\mathfrak{Z}(A)(\{\mathfrak{a},\mathfrak{a}\})(\mathfrak{a})} \arrow[r]                            & \mathfrak{Z}(A)(\mathfrak{a})(\mathfrak{a})                             
\end{tikzcd}.
    \end{center}
\end{proof}
\begin{remark}
    Note that Corollary \ref{cor7} combined with Theorem \ref{thm7} implies that the morphism object completely determines the $\mathbb{E}_2$-algebra structure of the center.
\end{remark}
\newpage
\setcounter{theorem}{0}
\chapter{The Hochschild complex as a center}\label{The_hochschild_complex_as_a_center}
As described in Section \ref{the_little_cubes_operads}, the center of an $\mathbb{E}_1$-algebra has the structure of an $\mathbb{E}_1$-algebra in $\mathbb{E}_1$-algebras, which via the Dunn additivity theorem is equivalent to the structure of an $\mathbb{E}_2$-algebra. In the first part of this chapter, we develop a method to compute the Gerstenhaber bracket of such an $\mathbb{E}_2$-algebra coming from an $\mathbb{E}_1$-algebra in $\mathbb{E}_1$-algebras.\displaypar 
We use this result to prove our main result Theorem \ref{thm1}, stating that the underlying Gerstenhaber algebra structure of the $\mathbb{E}_1$-center of an associative $\field$-algebra recovers the classical Hochschild cohomology. \displaypar
Globalizing, we give a definition of the Hochschild cochain complex of a general scheme over $\field$ in terms of the $\mathbb{E}_1$-center of its structure sheaf, which by construction is again an $\mathbb{E}_2$-algebra. For smooth varieties, Kontsevich defined the Hochschild cochain complex as a quasi-coherent sheaf glued from the local Hochschild complexes. We show in Theorem \ref{thm6} that the center definition recovers this sheaf, including the algebraic structure in hypercohomology.
\section{The bracket operation on 2-algebras}\label{The_bracket_on_a_2_algebra}
The classical Eckmann-Hilton argument shows that two compatible associative algebra structures on a set are equivalent to a commutative algebra structure. If the two associative multiplications are labeled $\ast$ and $\cdot$, the argument first identifies the multiplication $\cdot$ with the opposite multiplication of $\ast$
\begin{align*}
    a\cdot b = (1\ast a)\cdot (b\ast 1) = (1\cdot b)\ast(a \cdot 1) = b \ast a,
\end{align*}
and then further identifies the opposite multiplication of $\ast$ with the opposite multiplication of $\cdot$
\begin{align*}
    b\ast a = (b\cdot 1)\ast(1\cdot a) = (b\ast 1)\cdot (1\ast a) = b \cdot a.
\end{align*}
Together, this yields $a\cdot b = b\cdot a$. The interchange law $(a\ast b) \cdot (c \ast d) = (a\cdot c) \ast( b\cdot d)$ can be expressed as the commutativity of a certain square\footnote{This is the first diagram in Theorem \ref{thm11}.} in $\mathbb{E}_1^{\otimes} \times \mathbb{E}_1^{\otimes}$. The identities of the form $a\cdot b = (1\ast a)\cdot (b\ast 1)$ correspond to certain semi-inert restriction maps\footnote{These are the triangle diagrams in Theorem \ref{thm11}.} in $\mathbb{E}_1^{\otimes} \times \mathbb{E}_1^{\otimes}$. In particular, in the $\infty$-categorical setting, all the identities used in the Eckmann-Hilton argument are witnessed by 2-simplices in the product $\mathbb{E}_1^{\otimes} \times \mathbb{E}_1^{\otimes}$. This yields a non-trivial homotopy $a\cdot b \simeq b\cdot a$. Repeating this argument with the respective opposite multiplications, i.e. with the roles of $a$ and $b$ reversed, yields another homotopy $b\cdot a \simeq a\cdot b$. Composing these two homotopies, we obtain a non-trivial homotopy $a\cdot b \simeq a\cdot b$. In this section, we show that this corresponds to the non-trivial 2-morphism in $\mathbb{E}_2^{\otimes}$ given by the double twist $\gamma_0$.
\subsection{The bracket operation of an $\mathbb{E}_2$-algebra}\label{the_bracket_operation_of_an_e2_algebra}
Consider the topological operad $\mathbb{E}^T_2$ of little 2-squares defined in Section \ref{the_little_cubes_operads} and its corresponding dg operad $C_{\ast}(\mathbb{E}^T_2)$. If $C_{\ast}(\mathbb{E}^T_2) \rightarrow \e_A$ is an operad algebra in $\ch(\field)$, we have an action of the 2-ary operation space
    \begin{align*}
        C_{\ast}(\mathbb{E}^T_2(2)) \otimes A^{\otimes 2} \rightarrow A,
    \end{align*}
and recalling that $\mathbb{E}^T_2(2) \simeq S^1$, taking homology yields a map
    \begin{align*}
        H_{\ast}(S^1) \otimes H_{\ast}(A)^{\otimes 2} \rightarrow H_{\ast}(A).
    \end{align*}
Since $H_{\ast}(S^1)\cong \mathbb{Z}[p] \oplus \mathbb{Z}[\gamma]$ for some choice of basepoint $p\in S^1$ and generating loop $\gamma: [0,1] \rightarrow S^1$, this yields two 2-ary operations on $H_{\ast}(A)$; one of degree 0 induced by $[p]$
    \begin{align*}
        \smile: H_{\ast}(A)\otimes H_{\ast}(A) \rightarrow H_{\ast}(A)
    \end{align*}
and one of degree 1 induced by $[\gamma]$
    \begin{align*}
        [\cdot,\cdot]: H_{\ast}(A) \otimes H_{\ast}(A) \rightarrow H_{\ast}(A)[-1].
    \end{align*}
Cohen showed in \cite{Coh} that these two operations make $H_{\ast}(A)$ into a Gerstenhaber algebra. In particular, the bracket operation on $H_{\ast}(A)$ is induced by the chain level operation $A^{\otimes 2} \rightarrow A$ corresponding to a choice of generating loop $\gamma$ of the homology of $S^1$. 
\begin{definition}
    Let $\mathcal{C}^{\otimes}$ be a symmetric monoidal $\infty$-category and let $A: \mathbb{E}_2^{\otimes} \rightarrow \mathcal{C}^{\otimes}$ be an $\mathbb{E}_2$-algebra in $\mathcal{C}^{\otimes}$. For $t\in [0,2\pi)$, we call the image under $A$ of $\gamma_t\in \map_{\mathbb{E}_2^{\otimes}}(\langle 2\rangle,\langle 1\rangle)_{1}$ the \textbf{bracket operation} of $A$ at $\mu_t\in \map_{\mathbb{E}_2^{\otimes}}(\langle 2\rangle, \langle 1 \rangle)_0$.
\end{definition}
\subsection{The bracket operation of a 2-algebra}\label{the_bracket_operation_of_a_2_algebra}
We now consider the special case of $\mathbb{E}_2$-algebras obtained via the Dunn additivity theorem \ref{thm7}. We have seen at the end of Section \ref{the_little_cubes_operads} that these are precisely the types of $\mathbb{E}_2$-algebras obtained as the center of an $\mathbb{E}_1$-algebra in a symmetric monoidal $\infty$-category. The main result of this section, Corollary \ref{cor1}, will enable us to recover the classical Gerstenhaber bracket in cohomology when defining the Hochschild complex as a center.
\begin{definition}
Let $\mathcal{C}^{\otimes}$ be a symmetric monoidal $\infty$-category. A \textbf{2-algebra} in $\mathcal{C}^{\otimes}$ is a bifunctor of $\infty$-operads $A: \mathbb{E}_1^{\otimes} \times \mathbb{E}_1^{\otimes} \rightarrow \mathcal{C}^{\otimes}$.
\end{definition}
\begin{remark}
Note that by Corollary \ref{cor11}, a 2-algebra is equivalently an object $A\in\alg_{\mathbb{E}_1}(\alg_{\mathbb{E}_1}(\mathcal{C}))$.
\end{remark}
Fix a 2-algebra $A: \mathbb{E}_1^{\otimes} \times \mathbb{E}_1^{\otimes} \rightarrow \mathcal{C}^{\otimes}$. The Dunn additivity theorem \ref{thm7} tells us that there exists an $\mathbb{E}_2$-algebra $\tilde{A}: \mathbb{E}_2^{\otimes} \rightarrow \mathcal{C}^{\otimes}$ such that the restriction of $\tilde{A}$ along $\rho: \mathbb{E}_1^{\otimes} \times \mathbb{E}_1^{\otimes} \rightarrow \mathbb{E}_2^{\otimes}$ is equivalent to $A$ in the category of bifunctors. Fixing such an $\mathbb{E}_2^{\otimes}$-algebra $\tilde{A}$, we can ask whether it is possible to express the bracket operations $\tilde{A}(\gamma_t)$ in terms of the original 2-algebra $A$. 
\begin{notation}
\begin{itemize}
    \item By abuse of notation, we denote by $A\in \mathcal{C}$ the image $A(\langle 1\rangle,\langle1 \rangle)$ and similarly by $\tilde{A} \in \mathcal{C}$ the image $\tilde{A}(\langle 1\rangle)$.
    \item We commonly label a morphism $f:\langle m\rangle \rightarrow \langle n\rangle$ in $\fin$ by its sequence of images $(f(1),\dots,f(m))$. We denote by $\tau: \langle 2\rangle \rightarrow \langle 2\rangle$ the switch map $\tau = (2,1)$, and by $\iota_{i,j}: \langle 2\rangle \rightarrow \langle 4\rangle$ the inclusion $1\mapsto i$, $2\mapsto j$.
    \item A map $F\in \map_{\mathbb{E}_k^{\otimes}}(\langle m\rangle,\langle n\rangle)_0$ is given by specifying a map $f: \langle m\rangle \rightarrow \langle n\rangle$ in $\fin$ and, for every $j\in \langle n\rangle^{\circ}$, specifying an element $F_j\in \mathbb{E}_k^T(f^{-1}(j))$. If the underlying map $f$ is clear, we denote such a morphism by the sequence $(F_1,\dots, F_n)$ of rectangular embeddings.
    \item Let $f: \langle m\rangle \rightarrow \langle n\rangle$ be semi-inert\footnote{See \cite[Definition 3.3.1.1]{HA}.} in $\fin$. We fix a coCartesian lift $F$ of $f$ in $\mathbb{E}_k^{\otimes}$ by choosing $F_j \in \mathbb{E}_k^T(f^{-1}(j))$ to be the identity map $\square^k \times \{f^{-1}(j)\} \rightarrow \square^k$ for each $j\in \langle n\rangle^{\circ}$ such that $f^{-1}(j)$ is non-empty. We denote those maps in $\mathbb{E}_k^{\otimes}$ by the semi-inert map they lift. 
\end{itemize}
\end{notation}
Recall that $\mu := \mu_0\in \map_{\mathbb{E}_1^{\otimes}}(\langle 2\rangle, \langle 1\rangle)_0$ corresponds to the following embedding of little intervals
\begin{center}
    \begin{tikzpicture}[x=0.75pt,y=0.75pt,yscale=-.8,xscale=.8]
    \draw    (220,100.5) -- (341,100.5) ;
    \draw [shift={(341,100.5)}, rotate = 180] [color={rgb, 255:red, 0; green, 0; blue, 0 }  ][line width=0.75]    (0,5.59) -- (0,-5.59)   ;
    \draw [shift={(280.5,100.5)}, rotate = 180] [color={rgb, 255:red, 0; green, 0; blue, 0 }  ][line width=0.75]    (0,5.59) -- (0,-5.59)   ;
    \draw [shift={(220,100.5)}, rotate = 180] [color={rgb, 255:red, 0; green, 0; blue, 0 }  ][line width=0.75]    (0,5.59) -- (0,-5.59)   ;
    \draw (306,82) node [anchor=north west][inner sep=0.75pt]   [align=left] {1};
    \draw (244,80.5) node [anchor=north west][inner sep=0.75pt]   [align=left] {2};
\end{tikzpicture}
\end{center}
The key observation in expressing the bracket operations in terms of the original 2-algebra is given by the following theorem.
\begin{theorem}\label{thm11}
     The images under $A$ of the 2-simplices in $\mathbb{E}^{\otimes}_1\times \mathbb{E}^{\otimes}_1$ 
     \begin{center}
          \begin{tikzcd}[column sep = large]
{(\langle 2\rangle,\langle 2\rangle)} \arrow[d, "{(\mu,\textnormal{id}_{\langle 2\rangle})}"'] \arrow[r, "{(\textnormal{id}_{\langle 2\rangle},\mu)}"] & {(\langle 2\rangle,\langle 1\rangle)} \arrow[d, "{(\mu,\textnormal{id}_{\langle 1\rangle})}"] \\
{(\langle 1\rangle,\langle 2\rangle)} \arrow[r, "{(\textnormal{id}_{\langle 1\rangle},\mu)}"']                                                 & {(\langle 1\rangle,\langle 1\rangle)}                                                
\end{tikzcd}
     \end{center}
     
     \begin{center}
          \begin{tikzcd}[column sep = large]
{(\langle 1\rangle,\langle 2\rangle)} \arrow[d, "{(\textnormal{id}_{\langle 1\rangle},(2,3))}"'] \arrow[rd, "{(\textnormal{id}_{\langle 1\rangle},\textnormal{id}_{\langle 2\rangle})}"] &                                       \\
{(\langle 1\rangle,\langle 4\rangle)} \arrow[r, "{(\textnormal{id}_{\langle 1\rangle},(\mu,\mu))}"']                                                                             & {(\langle 1\rangle,\langle 2\rangle)}
\end{tikzcd}
     \end{center}

     \begin{center}
         \begin{tikzcd}[column sep = large]
{(\langle 2\rangle,\langle 1\rangle)} \arrow[rd, "{(\tau,\textnormal{id}_{\langle 1\rangle})}"'] \arrow[r, "{((2,3),\textnormal{id}_{\langle 1\rangle})}"] & {(\langle 4\rangle,\langle 1\rangle)} \arrow[d, "{((\mu,\mu),\textnormal{id}_{\langle 1\rangle})}"] \\
                                                                                                                                               & {(\langle 2\rangle,\langle 1\rangle)}                                                    
\end{tikzcd}
     \end{center}
     can be composed to yield a 2-simplex in $\mathcal{C}^{\otimes}$ with boundary 
     \begin{center}
         \begin{tikzcd}[column sep = large]
{(A,A)} \arrow[rd, "{\iota_{2,3}}"] \arrow[rrd, "{\textnormal{id}_{(A,A)}}", bend left] \arrow[rdd, "\tau"', bend right] &                                                                               &                          \\
                                                                                                                   & {(A,A,A,A)} \arrow[d, "{(m_1,m_1)\circ \tau_{2,3}}"] \arrow[r, "{(m_2,m_2)}"] & {(A,A)} \arrow[d, "m_1"] \\
                                                                                                                   & {(A,A)} \arrow[r, "m_2"']                                                     & A                       
\end{tikzcd}
     \end{center}
     whose homotopy class is identified under the equivalence $\tilde{A}\circ \rho \simeq A$ with the image under $\tilde{A}$ of the half-twist $\sigma_0$ between $\mu_0$ and $\mu_\pi$ in $\map_{\mathbb{E}_2^{\otimes}}(\langle 2\rangle,\langle 1\rangle)_1$.
\end{theorem}
\begin{proof}
    First check that the 2-simplices in $\mathbb{E}_1^{\otimes}\times \mathbb{E}_1^{\otimes}$ indeed induce composable 2-simplices in $\mathcal{C}^{\otimes}$ with the depicted boundaries. Recall that a bifunctor of $\infty$-operads sends pairs of inert maps to inert maps. Therefore, if $\rho^i: \langle m\rangle \rightarrow \langle 1 \rangle$ is the inert map $i\mapsto 1$, $i\neq j \mapsto \ast$, the map $A(\rho^i,\text{id}): A(\langle m\rangle, \langle n\rangle) \rightarrow A(\langle 1\rangle, \langle n\rangle)$ must be an inert lift in $\mathcal{C}^{\otimes}$ of $\rho^i: \langle mn\rangle \rightarrow \langle 1\rangle$. A similar argument holds for $\rho^j$ in the second component. This shows that under the identification $\mathcal{C}^{\otimes}_{\langle n\rangle} \simeq \mathcal{C}^{\times n}$, the object $A(\langle m\rangle, \langle n\rangle)$ is equivalent to the tuple $(A,\dots, A)$ with $m$ entries. Applying $A$ to the diagrams in $\mathbb{E}_1^{\otimes} \times \mathbb{E}_1^{\otimes}$, we hence get diagrams
    \begin{center}
       \begin{tikzcd}
                                                                                  & {(A,A,A,A)} \arrow[rr] \arrow[dd] &                                                                        & {(A,A)} \arrow[dd] \\
{A(\langle 2\rangle,\langle 2\rangle)} \arrow[dd] \arrow[rr] \arrow[ru, "\simeq"] &                                   & {A(\langle 2\rangle,\langle 1\rangle)} \arrow[dd] \arrow[ru, "\simeq"] &                    \\
                                                                                  & {(A,A)} \arrow[rr]                &                                                                        & A                  \\
{A(\langle 1\rangle,\langle 2\rangle)} \arrow[rr] \arrow[ru, "\simeq"]            &                                   & {A(\langle 1\rangle,\langle 1\rangle)} \arrow[ru, "\simeq"]            &                   
\end{tikzcd}
    \end{center}

    \begin{center}
        \begin{tikzcd}
                                                                                 & {(A,A)} \arrow[d] \arrow[rd]                                &         \\
{A(\langle 1\rangle,\langle 2\rangle)} \arrow[d] \arrow[rd] \arrow[ru, "\simeq"] & {(A,A,A,A)} \arrow[r]                                       & {(A,A)} \\
{A(\langle 1\rangle,\langle 4\rangle)} \arrow[r] \arrow[ru, "\simeq" near end]            & {A(\langle 1\rangle,\langle 2\rangle)} \arrow[ru, "\simeq"] &        
\end{tikzcd}
    \end{center}

    \begin{center}
       \begin{tikzcd}
                                                                                 & {(A,A)} \arrow[r] \arrow[rd]                                          & {(A,A,A,A)} \arrow[d] \\
{A(\langle 2\rangle,\langle 1\rangle)} \arrow[r] \arrow[rd] \arrow[ru, "\simeq"] & {A(\langle 4\rangle,\langle 1\rangle)} \arrow[d] \arrow[ru, "\simeq" near start] & {(A,A)}               \\
                                                                                 & {A(\langle 2\rangle,\langle 1\rangle)} \arrow[ru, "\simeq"]           &                      
\end{tikzcd}
    \end{center}
    It hence suffices to show that the respective back sides of the diagrams fit together into the depicted 2-simplex. To this end we check that the maps $(A,A,A,A) \rightarrow (A,A)$ in the square agree with the respective maps $(A,A,A,A)\rightarrow (A,A)$ in the triangles, and similarly for the two maps $(A,A) \rightarrow (A,A,A,A)$ in the different triangles. Note that by point (2) of the definition \cite[Definition 2.1.1.10]{HA} of $\infty$-operads, it suffices to show that each of those pairs of maps agrees after post-composition with coCartesian lifts of the $\rho^i$. Consider first the maps induced by $(\langle 2\rangle,\langle 2\rangle) \xrightarrow{\textnormal{id}_{\langle 2\rangle},\mu} (\langle 2\rangle,\langle 1\rangle)$ and $(\langle 1\rangle,\langle 4\rangle) \xrightarrow{\textnormal{id}_{\langle 1\rangle},(\mu,\mu)} (\langle 1\rangle,\langle 2\rangle)$. We have a factorization 
    \begin{center}
        \begin{tikzcd}
{(\langle 2\rangle,\langle 2\rangle)} \arrow[r, "{\textnormal{id},\mu}"] \arrow[rd, "{\rho^i,\textnormal{id}}"'] & {(\langle 2\rangle,\langle 1\rangle)} \arrow[r, "{\rho^i,\textnormal{id}}"] & {(\langle 1\rangle,\langle 1\rangle)} \\
                                                                                                                 & {(\langle 1\rangle,\langle 2\rangle)} \arrow[ru, "{\textnormal{id},\mu}"']  &                                      
\end{tikzcd}
    \end{center}
    in $\mathbb{E}_1^{\otimes}\times \mathbb{E}_1^{\otimes}$, where the unique bifunctor $\mathbb{E}_1^{\otimes} \times \mathbb{E}_1^{\otimes} \rightarrow \fin$ sends the lower left hand side map to the inert map $f_1= 1,2,\ast,\ast: \langle 4\rangle \rightarrow \langle 2\rangle$ if $i=1$ and to the inert map $f_2 =\ast, \ast, 1,2$ if $i = 2$. Since inert pairs are sent to inert maps by $A$, this diagram maps to 
    \begin{center}
        \begin{tikzcd}
{(A,A,A,A)} \arrow[r, "{A(\textnormal{id},\mu)}"] \arrow[rd] & {(A,A)} \arrow[r, "\rho^i"] & A \\
                                                            & {(A,A)} \arrow[ru, "m_2"']  &  
\end{tikzcd}.
    \end{center}
    where $(A,A,A,A)\rightarrow (A,A)$ is an inert lift of $f_1$ or $f_2$ respectively. On the other hand, we also have a factorization 
    \begin{center}
        \begin{tikzcd}
{(\langle 1\rangle,\langle 4\rangle)} \arrow[r, "{\textnormal{id},(\mu,\mu)}"] \arrow[rd, "{\textnormal{id},f_i}"'] & {(\langle 1\rangle,\langle 2\rangle)} \arrow[r, "{\textnormal{id},\rho^i}"] & {(\langle 1\rangle,\langle 1\rangle)} \\
                                                                                                                    & {(\langle 1\rangle,\langle 2\rangle)} \arrow[ru, "{\textnormal{id},\mu}"']  &                                      
\end{tikzcd}
    \end{center}
    This again is sent by $A$ to 
    \begin{center}
        \begin{tikzcd}[column sep = large]
{(A,A,A,A)} \arrow[rd] \arrow[r, "{A(\textnormal{id},(\mu,\mu))}"] & {(A,A)} \arrow[r, "\rho^i"] & A \\
                                                                   & {(A,A)} \arrow[ru, "m_2"']  &  
\end{tikzcd},
    \end{center}
    showing that our two maps $(A,A,A,A)\rightarrow (A,A)$ indeed agree up to homotopy. An analogous analysis can be carried out with the other two maps $(A,A,A,A)\rightarrow (A,A)$.\par 
    For the two inclusions $(\langle 1\rangle,\langle 2\rangle) \rightarrow (\langle 1\rangle,\langle 4\rangle)$ and $(\langle 2\rangle,\langle 1\rangle)\rightarrow (\langle 4\rangle,\langle 1\rangle)$, it suffices to show that the two unit maps coming from $(\langle 1\rangle,\langle 0\rangle) \rightarrow (\langle 1\rangle,\langle 1\rangle)$ and $(\langle 0\rangle,\langle 1\rangle)\rightarrow (\langle 1\rangle,\langle 1\rangle)$ agree. To see this, note that the composition
    \begin{align*}
        (\langle 1\rangle, \langle 2\rangle) \rightarrow (\langle 1\rangle, \langle 4\rangle) \xrightarrow{\rho^i} (\langle 1\rangle, \langle 1\rangle)
    \end{align*}
    is either inert (for $i=2,3$), or it factors as 
    \begin{align*}
        (\langle 1\rangle, \langle 2\rangle) \xrightarrow{\text{inert}} (\langle 1\rangle, \langle 0\rangle) \rightarrow (\langle 1\rangle, \langle 1\rangle)
    \end{align*}
    for $i=1,4$. The same holds for the other inclusion with the entries switched. To show that the two unit maps agree, note that both $(\langle 0\rangle,\langle 0\rangle) \rightarrow (\langle 1\rangle,\langle 0\rangle)$ and $(\langle 0\rangle,\langle 0\rangle) \rightarrow (\langle 0\rangle,\langle 1\rangle)$ are sent to the identity on the empty tuple by $A$, up to homotopy, since $\mathcal{C}^{\otimes}_{\langle 0\rangle}$ is contractible. Then the image of the diagram
    \begin{center}
        \begin{tikzcd}
                                                                       & {(\langle 1\rangle,\langle 0\rangle)} \arrow[rd] &                                       \\
{(\langle 0\rangle,\langle 0\rangle)} \arrow[rr] \arrow[ru] \arrow[rd] &                                                  & {(\langle 1\rangle,\langle 1\rangle)} \\
                                                                       & {(\langle 0\rangle,\langle 1\rangle)} \arrow[ru] &                                      
\end{tikzcd}
    \end{center}
    under $A$ then shows that the two unit maps are homotopic.\par
    We now show that the depicted 2-simplex in $\mathcal{C}^{\otimes}$ indeed corresponds to the image of the half-twist under $\tilde{A}$. To this end, examine the images of the three 2-simplices in $\mathbb{E}_1^{\otimes} \times \mathbb{E}_1^{\otimes}$ under $\rho: \mathbb{E}_1^{\otimes}\times \mathbb{E}_1^{\otimes} \rightarrow \mathbb{E}_2^{\otimes}$. We get 
    \begin{center}
        \includegraphics[scale=.7]{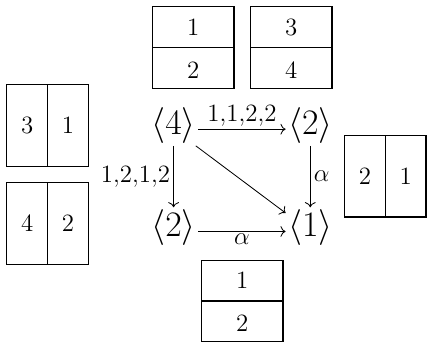}
    \end{center}

    \begin{center}
        \includegraphics[scale=.7]{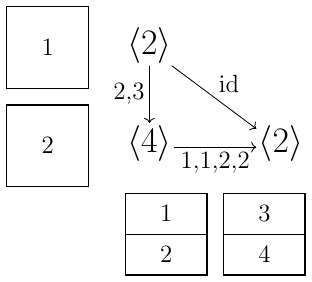}
    \end{center}

    \begin{center}
       \includegraphics[scale=.7]{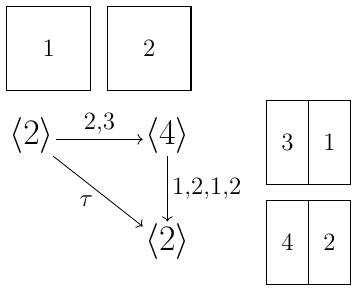}
    \end{center}
    with the paths forming the fillings of the two triangles in the square diagram both giving the constant path at
    \begin{center}
    \begin{tikzpicture}[x=1cm,y=1cm]
  \draw (0,0) rectangle (2,2);
  \draw (1,0) -- (1,2);
  \draw (0,1) -- (2,1);

  \node at (0.5,1.5) {3};
  \node at (1.5,1.5) {1};
  \node at (0.5,0.5) {4};
  \node at (1.5,0.5) {2};
\end{tikzpicture}
    \end{center}
    while the fillings of the triangle diagrams are given by continuously enlarging the respective rectangles. Composing those (as paths) in the topological category $\mathbb{E}_2^{T,\otimes}$ as depicted here
    \begin{equation}\label{dig3}
        \begin{tikzcd}
	{\langle 2\rangle} \\
	& {\langle 4\rangle} & {\langle 2\rangle} \\
	& {\langle 2\rangle} & {\langle 1\rangle}
	\arrow["{2,3}", from=1-1, to=2-2]
	\arrow["{\text{id}}", curve={height=-18pt}, from=1-1, to=2-3]
	\arrow["\tau"', curve={height=18pt}, from=1-1, to=3-2]
	\arrow["{1,1,2,2}", from=2-2, to=2-3]
	\arrow["{1,2,1,2}"', from=2-2, to=3-2]
	\arrow["\alpha", from=2-3, to=3-3]
	\arrow["\alpha"', from=3-2, to=3-3]
\end{tikzcd}
    \end{equation}
    we hence get the half-twist 
    \begin{center}
        \includegraphics[scale=.3]{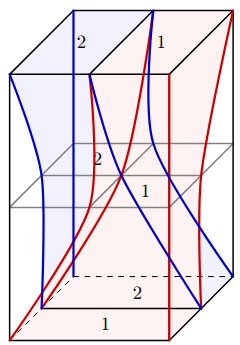}.
    \end{center}
    Applying $\tilde{A}$ to the above simplices (\ref{dig3}) induces a diagram of 2-simplices in $\mathcal{C}^{\otimes}$
    \begin{center}
        \begin{tikzcd}
{(\tilde{A},\tilde{A})} \arrow[rd] \arrow[rrd, bend left] \arrow[rdd, bend right] &                                                                  &                                   \\
                                                                                  & {(\tilde{A},\tilde{A},\tilde{A}, \tilde{A})} \arrow[r] \arrow[d] & {(\tilde{A},\tilde{A})} \arrow[d] \\
                                                                                  & {(\tilde{A},\tilde{A})} \arrow[r]                                & \tilde{A}                        
\end{tikzcd}
    \end{center}
    and by construction the isomorphism between $\tilde{A}\circ \rho$ and $A$ identifies this diagram with the one in the statement. Since maps between $\infty$-categories respect composition, this proves the claim.
\end{proof}
We can repeat this analysis for the other ``parts'' of the classical Eckmann-Hilton argument; using the inclusion $(1,4):\langle 2\rangle \rightarrow \langle 4\rangle$ and the opposite multiplication $\mu_1\in \map_{\mathbb{E}_1^{\otimes}}(\langle 2\rangle,\langle 1\rangle)_0$. In particular, we get representations of the image under $\tilde{A}$ of all four different parts of the double twist in terms of compositions of 2-simplices in the image of $A$. As a corollary, we obtain
\begin{corollary}\label{cor1}
    The homotopy class of the bracket on $\tilde{A}$ at $t=0$ can be produced as a composition of the following 2-simplices in $\mathcal{C}^{\otimes}$
    \begin{center}
         \begin{tikzcd}
{(A,A)} \arrow[rd, "{\iota_{2,3}}"] \arrow[rrd, "{\textnormal{id}}", bend left] \arrow[rdd, "\tau"', bend right] &                                                                               &                          \\
                                                                                                                   & {(A,A,A,A)} \arrow[d, "{(m_1,m_1)\circ \tau_{2,3}}"] \arrow[r, "{(m_2,m_2)}"] & {(A,A)} \arrow[d, "m_1"] \\
                                                                                                                   & {(A,A)} \arrow[r, "m_2"']                                                     & A                       
\end{tikzcd},
    \end{center}

    \begin{center}
        \begin{tikzcd}
{(A,A)} \arrow[rd, "{\iota_{1,4}}"] \arrow[rrd, "\textnormal{id}", bend left] \arrow[rdd, "\textnormal{id}"', bend right] &                                                                                                                                                        &                                            \\
                                                                                                              & {(A,A,A,A)} \arrow[d, "{(m_1^{\textnormal{op}},m_1^{\textnormal{op}})\circ \tau_{2,3}}"] \arrow[r, "{(m_2^{\textnormal{op}},m_2^{\textnormal{op}})}"] & {(A,A)} \arrow[d, "m_1^{\textnormal{op}}"] \\
                                                                                                              & {(A,A)} \arrow[r, "m_2^{\textnormal{op}}"']                                                                                                            & A                                         
\end{tikzcd},
    \end{center}

    \begin{center}
        \begin{tikzcd}
{(A,A)} \arrow[rd, "{\iota_{2,3}}"] \arrow[rrd, "\textnormal{id}", bend left] \arrow[rdd, "\tau"', bend right] &                                                                                                                                                        &                                            \\
                                                                                                         & {(A,A,A,A)} \arrow[d, "{(m_1^{\textnormal{op}},m_1^{\textnormal{op}})\circ \tau_{2,3}}"] \arrow[r, "{(m_2^{\textnormal{op}},m_2^{\textnormal{op}})}"] & {(A,A)} \arrow[d, "m_1^{\textnormal{op}}"] \\
                                                                                                         & {(A,A)} \arrow[r, "m_2^{\textnormal{op}}"']                                                                                                            & A                                         
\end{tikzcd},
    \end{center}

    \begin{center}
        \begin{tikzcd}
{(A,A)} \arrow[rd, "{\iota_{1,4}}"] \arrow[rrd, "\textnormal{id}", bend left] \arrow[rdd, "\textnormal{id}"', bend right] &                                                                               &                          \\
                                                                                                                          & {(A,A,A,A)} \arrow[d, "{(m_1,m_1)\circ \tau_{2,3}}"] \arrow[r, "{(m_2,m_2)}"] & {(A,A)} \arrow[d, "m_1"] \\
                                                                                                                          & {(A,A)} \arrow[r, "m_2"']                                                     & A                       
\end{tikzcd}.
    \end{center}
\end{corollary}

\section{Recovering the Hochschild complex as $\mathbb{E}_1$-center}\label{Recovering_the_Hochschild_complex_as_e1_center}
Often the Hochschild cohomology of a $\field$-algebra is called its ``derived center''. In this section we will show that this terminology is appropriate in a very precise sense. Namely, the Hochschild complex is the $\mathbb{E}_1$-center in the derived $\infty$-category of chain complexes. \par
In contrast to \cite[Definition 5.1]{BZFN}, we use the definition of the derived center via its universal property as in \ref{def5}, and then show, using Lurie's theory of higher centers, that we can reduce to considering the derived endomorphism complex of $A$ as an $A$-bimodule. As a consequence, we get a direct proof that the $\mathbb{E}_2$-algebra structure is obtained via the higher Eckmann-Hilton argument \ref{cor1} from the Yoneda and convolution products, and we can use this to show that this definition of the Hochschild cochain complex yields the correct Gerstenhaber algebra in cohomology.
\subsection{Symmetric monoidal dg model categories}\label{symmetric_monoidal_dg_model_categories}
We first examine how a symmetric monoidal model category with a compatible structure of a dg category yields a symmetric monoidal $\infty$-category. If $\mathbf{S}$ is a monoidal model category, one can define the notion of an $\mathbf{S}$-enriched model category as in \cite[Definition A.3.1.5]{HTT}. In this section, we consider $\mathbf{S}$ to be the (symmetric) monoidal model category $\ch(\field)$ of chain complexes of $\field$-modules with the projective model structure as in \cite[Proposition 4.2.13]{H}. We call $\ch(\field)$-enriched model categories \textbf{dg model categories}. Analogously to simplicial model categories, they are tensored and cotensored over $\ch(\field)$, and they satisfy an (SM7) property for the cotensoring.\displaypar
If $C$ is a dg category, denote by $C_0$ the underlying 1-category. If $C$ is a dg model category, we view $C_0$ as an ordinary model category.
\begin{theorem}
Let $C$ be a dg category that is tensored over the category $\ch(\field)$ of chain complexes of $\field$-modules, let $C'$ be a full dg subcategory, and let $W$ be a collection of morphisms in $C'$ that are isomorphisms in the homotopy category of the dg category $C$. Assume that the following conditions are satisfied:
\begin{itemize}
    \item Every isomorphism in $C'$ belongs to $W$.
    \item The set $W$ satisfies the 2-out-of-3 property.
    \item For all $X\in C'$, we also have $N_{\ast}(\Delta^1)\otimes X \in C'$.
    \item For each $X\in C'$, the map $N_{\ast}(\Delta^1)\otimes X \rightarrow X$ induced by the map $[1]\rightarrow [0]$ belongs to $W$.
\end{itemize}
Then the canonical map $\theta: N(C'_0) \rightarrow N_{\text{dg}}(C')$ constructed in \cite[Remark 1.3.1.9]{HA} induces an equivalence of $\infty$-categories $\theta': N(C'_0)[W^{-1}]$ $\simeq N_{\text{dg}}(C')$.
\end{theorem}
\begin{proof}
The above conditions are exactly what is needed to repeat the proof of \cite[Proposition 1.3.4.5]{HA} replacing $\ch(\mathcal{A})$ with ${C}$ and $\ch(\mathcal{A})'$ with ${C}'$.
\end{proof}
\begin{corollary}\label{cor2}
Let $C$ be a dg model category, and let $C^{\circ}$ be the full dg subcategory on bifibrant objects. Then the map $N(C^{c}_0) \rightarrow N_{\text{dg}}(C^{\circ})$ exhibits the dg nerve as the underlying $\infty$-category of $C_0$.
\end{corollary}
\begin{proof}
Let $C_{\Delta}$ be the simplicial category obtained from ${C}$ via the Dold-Kan correspondence. Since $N({C}_0^{c})[W^{-1}] \simeq N({C}_0^{\circ})[W^{-1}]$ it suffices to take ${C}' = {C}^{\circ}$ above and $W$ the set of homotopy equivalences. In a dg model category, left homotopy between bifibrant objects agrees with chain homotopy in the dg sense, i.e. a homotopy is a map $h: N(\Delta^1) \otimes X \rightarrow Y$ with the correct restrictions to $\{0\}$ and $\{1\}$. In particular,
\begin{align*}
    \hom_{{C}_0}(N_{\ast}(\Delta^1)\otimes X,Y)  &\cong \hom_{\ch(\field)}(N_{\ast}(\Delta^1),\map_{{C}}(X,Y)) \\&\cong \hom_{\ch_{\geq 0}(\field)}(N_{\ast}(\Delta^1),\tau_{\geq 0} \map_{{C}}(X,Y)) \\&\cong \hom_{\text{sSet}}(\Delta^1, \text{DK}_{\bullet}\tau_{\geq 0}\map_{{C}}(X,Y)) \\&\cong \map_{{C}_{\Delta}}(X,Y)_1
\end{align*}
so left homotopies correspond to 1-chains of the mapping complex of ${C}$. One checks directly that the diagram making $h: N_{\ast}(\Delta^1)\otimes X \rightarrow Y$ into a homotopy between $f,g\in \hom_{{C}_0}(X,Y)$ forces the corresponding 1-chain $z\in \map_{{C}}(X,Y)_1$ to satisfy $dz = f-g$. This shows that homotopy equivalences in ${C}^{\circ}_0$ become isomorphisms in the homotopy category $h{C}_{\Delta}$, which is isomorphic to the dg homotopy category $h{C}$. Clearly every isomorphism is a homotopy equivalence and the set of homotopy equivalences satisfies 2-out-of-3. By assumption, the map $\otimes: \ch(\field) \times {C} \rightarrow {C}$ is a left Quillen bifunctor. The complex $N_{\ast}(\Delta^1)$ is cofibrant, and hence $N_{\ast}(\Delta^1)\otimes -$ preserves cofibrant objects. It also preserves fibrant objects, since $N_{\ast}(\Delta^1)$ is dualizable and so $N_{\ast}(\Delta^1) \otimes - \cong \map_{C}(N_{\ast}(\Delta^1)^{\vee},-)$. This shows that $N_{\ast}(\Delta^1)\otimes X\in {C}^{\circ}$. Finally, note that $d_0: \field\rightarrow N_{\ast}(\Delta^1)$ is a trivial cofibration in $\ch(\field)$, and thus if $X\in {C}^{\circ}$, the map $X \cong \field\otimes X \rightarrow N_{\ast}(\Delta^1) \otimes X$ is again a trivial cofibration. Now the map $N_{\ast}(\Delta^1)\rightarrow \field$ is a left inverse to $d_0$ and in particular 
\begin{align*}
    X \rightarrow N_{\ast}(\Delta^1) \otimes X \rightarrow X
\end{align*}
is the identity on $X$ and thus a weak equivalence. By 2-out-of-3, this means that $N_{\ast}(\Delta^1)\otimes X \rightarrow X$ must be a weak equivalence.
\end{proof}
\begin{remark}
Note that for any dg category ${C}$, we have an equivalence of $\infty$-categories $N_{\text{hc}}({C}_{\Delta}) \rightarrow N_{\text{dg}}({C})$. For simplicial model categories, the homotopy coherent nerve of the bifibrant objects is always equivalent to the $\infty$-category underlying the model category. In contrast, $C_{\Delta}$ is only \textbf{weakly simplicially enriched}, i.e. it is neither tensored nor cotensored over $\mathbf{sSet}$, and therefore does not satisfy the requirements of this theorem. The above corollary then shows that we get this relationship between the homotopy coherent nerve and the model category regardless.
\end{remark}
If ${C}$ is a (symmetric) monoidal model category, then \cite[Example 4.1.7.6]{HA} shows that $N({C}^c)[W^{-1}]$ is a (symmetric) monoidal $\infty$-category. If ${C}$ is also a simplicial model category and the (symmetric) monoidal structure is compatible with the simplicial enrichment, then \cite[Corollary 4.1.7.16]{HA} shows that the (symmetric) monoidal structure on this $\infty$-category is given by $N_{\text{hc}}(({C}^{\circ})^{\otimes})$, and in fact one readily checks that the same holds if ${C}$ is just weakly simplicially enriched. Since the Dold-Kan functor $\mathrm{DK}_{\bullet}$ is not symmetric, Lurie's result does not imply that the dg nerve of a (symmetric) monoidal dg model category $C$ presents the (symmetric) monoidal structure of $N({C}^c)[W^{-1}]$. Nevertheless, $\mathrm{DK}_{\bullet}$ is homotopy symmetric lax monoidal, and in fact V. Hinich proved the following.
\begin{proposition}[\cite{Hin}, Theorem 3.2.3]\label{prop8}
The dg nerve  
\begin{align*}
    N_{\textnormal{dg}}: N(\textnormal{Cat}_{\textnormal{dg}})[W_{\textnormal{dg}}^{-1}] \rightarrow N(\textnormal{Cat}_{\Delta})[W_{\Delta}^{-1}]\simeq \textnormal{Cat}_{\infty}
\end{align*}
from the symmetric monoidal $\infty$-category of dg categories to the symmetric monoidal $\infty$-category of $\infty$-categories is lax symmetric monoidal. In particular, it is a morphism of $\infty$-operads. It thus induces a map from the $\infty$-category of symmetric monoidal dg categories to the $\infty$-category of symmetric monoidal $\infty$-categories.
\end{proposition}
\subsection{Rectification of algebras over an $\infty$-operad}\label{Rectification_of_algebras_over_an_infty_operad}
Let $\mathcal{O}$ be a topological operad. Applying the singular chain functor with $\field$-coefficients produces a dg operad $C_{\ast}(\mathcal{O})$. In this section, we will generalize results of Hinich \cite{Hin} and D. Pavlov and J. Scholbach \cite{PS}, and show that $C_{\ast}(\mathcal{O})$-algebras in a symmetric monoidal dg model category $C$ correspond to algebras over the $\infty$-operad $N^{\otimes}(\text{Sing}_{\bullet}(\mathcal{O}))$ in the symmetric monoidal $\infty$-category $N_{\text{dg}}(C^{\circ})$.\displaypar
Let ${C}$ be a symmetric monoidal dg model category. Suppose further that ${C}$ is cofibrantly generated and symmetrically flat as defined in \cite[Definition 2.1. (vii)]{PS}, and that $C_{\ast}(\mathcal{O})$ is admissible \cite[Definition 5.1.]{PS} and well-pointed \cite[Definition 6.1]{PS} in ${C}$, and that $C$ admits a lax symmetric monoidal fibrant replacement functor. Let $\alg_{C_{\ast}(\mathcal{O})}(C)$ denote the model category of strict $C_{\ast}(\mathcal{O})$-algebras in $C$ with the transfer model structure from the forgetful functor to $C$. Note that this model structure exists by the assumption that $C_{\ast}(\mathcal{O})$ is admissible in $C$. The construction in \cite[Section 4.2]{Hin} generalizes directly to give a functor
\begin{align*}
    \phi: N(\alg_{C_{\ast}({\mathcal{O}})}({C})^{c}) \rightarrow \alg_{N^{\otimes}(\text{Sing}_{\bullet}(\mathcal{O}))}(N_{\text{dg}}({C}^{\circ})).
\end{align*}
Roughly, a cofibrant $C_{\ast}(\mathcal{O})$-algebra in $C$ is given by a symmetric monoidal functor of dg categories from the PROP $P_{C_{\ast}(\mathcal{O})}$ generated by $C_{\ast}(\mathcal{O})$ to $C^c$, and applying the dg nerve as in Proposition \ref{prop8} yields a symmetric monoidal functor $P_{\mathcal{O}} \rightarrow \text{N}_{\text{dg}}(C^{\circ})$. Here, $P_{\mathcal{O}}$ is the symmetric monoidal envelope of the $\infty$-operad $N^{\otimes}(\text{Sing}_{\bullet}(\mathcal{O}))$ as defined in \cite[Construction 2.2.4.1.]{HA}, and hence this construction yields an $N^{\otimes}(\text{Sing}_{\bullet}(\mathcal{O}))$-algebra in $\text{N}_{\text{dg}}(C^{\circ})$.\displaypar
This functor $\phi$ carries weak equivalences to equivalences as shown in \cite[Diagram 49]{Hin}, and therefore yields a comparison map
\begin{align*}
    \Phi:  N(\alg_{C_{\ast}(\mathcal{O})}({C})^{c})[W_{\alg_{C_{\ast}(\mathcal{O})}({C})}^{-1}] \rightarrow \alg_{N^{\otimes}(\text{Sing}_{\bullet}(\mathcal{O}))}(N_{\text{dg}}({C}^{\circ})).
\end{align*}
\begin{theorem}\label{thm3}
    Let $\mathcal{O}$ and ${C}$ as above. Then $\Phi$ is an equivalence of $\infty$-categories.
\end{theorem}
\begin{proof}
    We use Corollary 4.7.3.16 in \cite{HA}, which is an application of the $\infty$-categorical Barr-Beck Theorem \cite[Theorem 4.7.3.5]{HA}. Consider the diagram
    \begin{center}
        \begin{tikzcd}
{ N(\alg_{C_{\ast}(\mathcal{O})}({C})^{c})[W_{\alg_{C_{\ast}(\mathcal{O})}({C})}^{-1}]} \arrow[rr, "\Phi"] \arrow[rd, "G"'] &                                              & { \alg_{N^{\otimes}(\text{Sing}_{\bullet}(\mathcal{O}))}(N({C}^{c})[W^{-1}])} \arrow[ld, "G'"] \\
                                                                                                                                                & {(N({C}^c)[W^{-1}])^{[\mathcal{O}]}} &                                                                                
\end{tikzcd}
    \end{center}
    where $(N(C^c)[W^{-1}])^{[\mathcal{O}]}$ denotes the coproduct of $N(C^c)[W^{-1}]$ over the collection of colors $[\mathcal{O}]$ in $\mathcal{O}$, and $G$ and $G'$ are the forgetful functors given by evaluation at the colors of $\mathcal{O}$. To show that $\Phi$ is an equivalence of $\infty$-categories, it then suffices to show that $G$ and $G'$ satisfy the assumptions (1)-(5) of \cite[Corollary 4.7.3.16]{HA}, and that in addition $G$ is conservative. To verify these assumptions, we follow the proof of \cite[Theorem 4.5.4.7]{HA}; in particular we will refer to ``steps (a)-(e)'' from this proof.\par
    Steps (a)-(c) can be proven exactly like in \cite[Theorem 4.5.4.7]{HA}, by just replacing the commutative operad by $\mathcal{O}$. For step (d), it is clear that $G$ is conservative since the weak equivalences in $\alg_{C_{\ast}(\mathcal{O})}({C})$ are transferred from the ones in ${C}$ via the forgetful functor. To show that $G$ preserves geometric realization of simplicial objects, it suffices to show that it preserves homotopy sifted colimits. This is shown in \cite[Proposition 7.9]{PS}. Finally, for step (e), we need to show that the canonical transformation $G'\circ F' \rightarrow G\circ F$ is an equivalence, where $F$ and $F'$ are the left adjoints of $G$ and $G'$ respectively. This boils down to showing that for any cofibrant object $X\in {C}^{[\mathcal{O}]}$, the strict free $C_{\ast}(\mathcal{O})$-algebra generated by $X$ is also a free $N^{\otimes}(\text{Sing}_{\bullet}(\mathcal{O}))$-algebra in the sense of \cite[Definition 3.1.3.1]{HA}. To this end, Hinich \cite[Lemma 4.3.4]{Hin} proved an analogue of \cite[Proposition 3.1.3.13]{HA} for the setting of free algebras generated by objects of different colors. We have a map of $\infty$-operads\footnote{Recall that $\triv$ is the trivial $\infty$-operad as defined in \cite[Example 2.1.3.5.]{HA}} $(\triv)^{[\mathcal{O}]} \rightarrow {N}^{\otimes}(\text{Sing}_{\bullet}(\mathcal{O}))$, and a collection $X:=\{X_i\}_{i\in [\mathcal{O}]}$ of objects in ${N}_{\text{dg}}(C^{\circ})$ induces a $(\triv)^{[\mathcal{O}]}$-algebra $\overline{X}$ in $N_{\text{dg}}(C^{\circ})^{\otimes} \times_{\fin} N^{\otimes}(\text{Sing}_{\bullet}(\mathcal{O}))$ over $N^{\otimes}(\text{Sing}_{\bullet}(\mathcal{O}))$. Let $\mathbb{F}_{C_{\ast}(\mathcal{O})}(X)$ be the strict free $C_{\ast}(\mathcal{O})$ algebra generated by $X$, and let $\mathbb{F}= \phi(\mathbb{F}_{C_{\ast}(\mathcal{O})}(X))$. For a color $c\in [\mathcal{O}]$, Hinich \cite[Lemma 4.3.4]{Hin} constructs a map $\text{Sym}_{\mathcal{O}}(\overline{X})_c \rightarrow \mathbb{F}_c$ which is an equivalence for all colors if and only if $\mathbb{F}$ is indeed the free $N^{\otimes}(\text{Sing}_{\bullet}(\mathcal{O}))$-algebra. Hinich checks this for the case of $C = \ch(\field)$, but one readily sees that all his arguments still work for any symmetric monoidal dg model category $C$.
\end{proof}
\begin{remark}
We call a homotopy preimage of an $N^{\otimes}(\text{Sing}_{\bullet}(\mathcal{O}))$-algebra $A$ under $\Phi$ a \textbf{strictification} of $A$.
\end{remark}
\begin{corollary}
The dg nerve induces a map $\alg_{\lm}(N(\textnormal{Cat}_{\textnormal{dg}})[W_{\textnormal{dg}}^{-1}]) \rightarrow \alg_{\lm}(\textnormal{Cat}_{\infty})$. By \ref{thm3}, this implies that a dg model category left tensored over a monoidal dg model category yields an $\infty$-category left tensored over a monoidal $\infty$-category.
\end{corollary}
We can now show that if $M$ is a dg category left tensored over a monoidal dg model category $C$, a dg morphism object yields an $\infty$-categorical morphism object in the sense of Definition \ref{def3}.
\begin{lemma}\label{lem1}
Let ${C}$ be a monoidal dg model category with underlying monoidal product $\otimes: {C}_0\times {C}_0 \rightarrow {C}_0$. Then the induced monoidal product $N_{\text{dg}}({C}^{\circ}) \times  N_{\text{dg}}({C}^{\circ})\rightarrow N_{\text{dg}}({C}^{\circ})$ sends $A,B\in {C}^{\circ}$ to an object equivalent to $R(A\otimes B)$. A similar statement holds for dg model categories left tensored over a monoidal dg model category.
\end{lemma}
\begin{proof}
    This follows directly from the description of the monoidal structure in Proposition \ref{prop8}.
\end{proof}
\begin{theorem}\label{thm4}
Let ${C}$ be a monoidal dg model category and let ${M}$ be a dg model category that is left tensored over ${C}$. In particular, we have a dg functor\footnote{Here $\boxtimes$ denotes the tensor product of dg categories.} $\otimes: {C}\boxtimes {M} \rightarrow {M}$ whose underlying functor is a left Quillen bifunctor. Assume that for $A,B\in {M}^{\circ}$ we have a dg morphism object $\mor_{{M}}(A,B)\in {C}$ together with map $\alpha: \mor_{{M}}(A,B)\otimes A \rightarrow B$ in ${M}$ such that composition with $\alpha$ induces an isomorphism 
\begin{align*}
    \map_{{C}}(C,\mor_{{M}}(A,B)) \cong \map_{{M}}(C\otimes A,B) \in \ch(\field).
\end{align*}
Then
\begin{itemize}
    \item[(1)] The induced map $\tilde{\alpha}\in \map_{N_{\textnormal{dg}}({M}^{\circ})}(R(Q\mor_{{M}}(A,B)\otimes A),B)$  makes $Q\mor_{{M}}(A,B)\in N_{\textnormal{dg}}({C}^{\circ})$ into a morphism object for $A,B\in N_{\textnormal{dg}}({M}^{\circ})$ in the sense of Definition \ref{def3}.
    \item[(2)] If $\beta: R(X\otimes A) \rightarrow B$ is another morphism object for $A$ and $B$ in $N_{\textnormal{dg}}({M}^{\circ})$, then the induced map $f: X\xrightarrow{\simeq} Q\mor_{{M}}(A,B)$ is a weak equivalence in ${C}$, and $\tilde{\alpha}\circ R(f\otimes \textnormal{id}_A) \simeq \beta$ are homotopic.
\end{itemize}
\end{theorem}
\begin{proof}
For (1), note that if $C\in N_{\text{dg}}({C}^{\circ})$ is bifibrant, $Q\mor_{{M}}(A,B) \xtwoheadrightarrow{\simeq} \mor_{{M}}(A,B)$ is the cofibrant replacement map, and $C\otimes A\xhookrightarrow{\simeq} R(C\otimes A)$ is the fibrant replacement map, we get a weak equivalence
\begin{align*}
    \map_{{C}}(C,Q\mor_{{M}}(A,B)) \xrightarrow{\simeq} \map_{{C}}(C,\mor_{{M}}(A,B)) \cong \map_{{M}}(C\otimes A,B) \xrightarrow{\simeq} \map_{{M}}(R(C\otimes A),B)
\end{align*}
of chain complexes. Applying $\textnormal{DK}_{\bullet}\tau_{\geq 0}$, we get
\begin{align*}
    \map_{N_{\text{dg}}({C}^{\circ})}(C,Q\mor_{{M}}(A,B)) \simeq \map_{N_{\text{dg}}({M}^{\circ})}(R(C\otimes A),B).
\end{align*}
Together with Lemma \ref{lem1} this yields the result.\displaypar
For (2), we automatically get $M\simeq \mor_{{M}}(A,B)$ in the $\infty$-category $N_{\textnormal{dg}}({C}^{\circ})$ since morphism objects are unique up to equivalence. Now recall that $N_{\textnormal{dg}}({C}^{\circ}) \simeq N({C}^c)[W^{-1}]$, and since model categories are saturated this implies the result.
\end{proof}
\subsection{Gerstenhaber algebras in the homotopy category}
Recall that in Section \ref{the_bracket_operation_of_an_e2_algebra}, we defined the bracket operations of an $\mathbb{E}_2$-algebra as the image of the 2-morphisms $\gamma_t \in \map_{\mathbb{E}_2^{\otimes}}(\langle 2\rangle,\langle 1\rangle)_1$. We explain how to explicitly describe the bracket operations of a 2-algebra in a symmetric monoidal dg model category, and how the bracket operations of an $\mathbb{E}_2$-algebra in this setting correspond to actual Gerstenhaber brackets. \displaypar
The following result will make it possible for us to compute the Gerstenhaber bracket on the center of an $\mathbb{E}_1$-algebra in the dg nerve of a symmetric monoidal dg model category.
\begin{corollary}\label{cor5}
If $\mathcal{C}^{\otimes}$ is the dg nerve of a symmetric monoidal dg model category $C$ and $A$ is a 2-algebra with associated $\mathbb{E}_2$-algebra $\tilde{A}$, then the  bracket operation at $t=0$ on $\tilde{A}$ is given by the chain homotopy 
\begin{align*}
    \tilde{A}(\gamma_0) \simeq h\iota_{2,3} + h^{\text{op}}\iota_{1,4} +  h^{\text{op}}\iota_{2,3} + h \iota_{1,4}.
\end{align*}
Here $h$ and $h^{\text{op}}$ are the chain homotopies corresponding to the image of the square diagram (\ref{dig2}) in $\mathbb{E}^{\otimes}_1\times \mathbb{E}^{\otimes}_1$ for the multiplications and their opposite multiplications respectively.
\end{corollary}
\begin{proof}
Note first that by construction of the symmetric monoidal structure on ${C}$, a morphism $\underbrace{(A,\dots,A)}_{m \text{ times}} \rightarrow \underbrace{(A,\dots,A)}_{n \text{ times}}$ in $\mathcal{C}^{\otimes}$ corresponds to a map $A^{\otimes m} \rightarrow A^{\otimes n}$ in $C$, and this map is unique up to chain homotopy. Similarly, a 2-morphism between such maps in $C^{\otimes}$ corresponds to a chain homotopy between the corresponding maps in $C$. Fix such maps in $C$ corresponding to all the involved diagrams. Horizontal composition of maps is strictly defined in dg categories, and hence we have well-defined whiskering compositions $h\iota_{2,3}$, $h\iota_{1,4}$, $h^{\text{op}}\iota_{2,3}$ and $h^{\text{op}}\iota_{1,4}$. Finally, note that horizontal composition of chain homotopies is given by addition. Then the result follows from Corollary \ref{cor1}.
\end{proof}
We now want to argue that in the case of a symmetric monoidal dg model category, the abstract bracket operations correspond precisely to the Gerstenhaber bracket in the homotopy category. \displaypar
Let $A: \mathbb{E}_2^{\otimes} \rightarrow N_{\text{dg}}(C^{\circ})^{\otimes}$ be an $\mathbb{E}_2$-algebra in the dg nerve of a symmetric monoidal dg model category $C$. Then we get an induced map
\begin{align*}
    \mathbb{E}^T_2(2) \simeq \map_{\mathbb{E}_2^{\otimes}}^{\alpha}(\langle 2\rangle, \langle 1\rangle) \rightarrow \map_{N_{\text{dg}}(C^{\circ})^{\otimes}}(A(\langle 2\rangle),A(\langle 1\rangle)).
\end{align*}
Let $A= A(\langle 1\rangle)\in C^{\circ}$. Then we have a homotopy equivalence
\begin{align*}
    \map_{N_{\text{dg}}(C^{\circ})^{\otimes}}(A(\langle 2\rangle),A(\langle 1\rangle)) \simeq \map_{N_{\text{dg}}(C^{\circ})}(A^{\otimes 2},A).
\end{align*}
Hence we get a map (well-defined up to homotopy)
\begin{align*}
    \mathbb{E}^T_2(2) \rightarrow \map_{N_{\text{dg}}(C^{\circ})}(A^{\otimes 2},A) \simeq \text{DK}\tau_{\geq 0} \map_C(A^{\otimes 2},A),
\end{align*}
and therefore taking homology 
\begin{align*}
    (\ger(2))_n \cong H_n(\mathbb{E}^T_2(2)) \rightarrow \hom_{hC}(A^{\otimes 2},A[n]).
\end{align*}
This procedure yields a Gerstenhaber algebra structure on $A$ in the dg homotopy category of $C$ whose bracket is indeed given by the image of $\gamma_0$.\displaypar
Given an operation $\rho \in \map_{\mathbb{E}_2^{\otimes}}^{\alpha}(\langle n\rangle, \langle 1\rangle)_0$, the algebra $A$ yields a map $A(\rho)\in \map_{N_{\textnormal{dg}}(C^{\circ})^{\otimes}}(A(\langle n\rangle),A(\langle 1\rangle))_0$, which by the construction of the symmetric monoidal structure \ref{prop8} on $N_{\textnormal{dg}}(C^{\circ})^{\otimes}$ corresponds to a morphism $A(\langle 1\rangle)^{\otimes n} \rightarrow A(\langle 1\rangle)$ in $C$. By abuse of notation, we often denote the image of $\rho$ as a map $A(\rho): A^{\otimes n} \rightarrow A$.
\begin{proposition}
Let $A: \mathbb{E}_2^{\otimes} \rightarrow N_{\textnormal{dg}}(C^{\circ})^{\otimes}$ be an $\mathbb{E}_2$-algebra in the dg nerve of a symmetric monoidal dg model category. Let $A^{\textnormal{str}}$ be a homotopy preimage of $A$ under $\Phi$. Without loss of generality, assume that the underlying object of $A^{\textnormal{str}}$ is fibrant. Note that $\Phi(A^{\textnormal{str}})(\langle 1\rangle) \simeq A^{\textnormal{str}}$ agree in $C$ by construction. There is a chain homotopy $h\in \map_C({A^{\textnormal{str}}}^{\otimes 2}, A)_1$
\begin{center}
\begin{tikzcd}
	{{A^{\textnormal{str}}}^{\otimes 2}} && {A^{\textnormal{str}}} \\
	{A^{\otimes 2}} && A
	\arrow[""{name=0, anchor=center, inner sep=0}, "{A^{\textnormal{str}}(\mu_0)}", from=1-1, to=1-3]
	\arrow["\simeq"', from=1-1, to=2-1]
	\arrow["\simeq", from=1-3, to=2-3]
	\arrow[""{name=1, anchor=center, inner sep=0}, "{A(\mu_0)}"', from=2-1, to=2-3]
	\arrow[shorten <=4pt, shorten >=4pt, Rightarrow, from=0, to=1, "h"]
\end{tikzcd}
\end{center}
and a degree 2 map $K\in \map_C({A^{\textnormal{str}}}^{\otimes 2}, A)_2$ filling the cylinder
\begin{center}
\begin{tikzcd}[column sep = large, row sep = large]
	{{A^{\textnormal{str}}}^{\otimes 2}} && {A^{\textnormal{str}}} \\
	\\
	{A^{\otimes 2}} && A
	\arrow[""{name=0, anchor=center, inner sep=0}, "{{A^{\textnormal{str}}(\mu_0)}}", curve={height=-18pt}, from=1-1, to=1-3]
	\arrow[""{name=1, anchor=center, inner sep=0}, "{{A^{\textnormal{str}}(\mu_0)}}"', curve={height=18pt}, from=1-1, to=1-3]
	\arrow["\simeq"', from=1-1, to=3-1]
	\arrow["\simeq", from=1-3, to=3-3]
	\arrow[""{name=2, anchor=center, inner sep=0}, "{{A(\mu_0)}}", curve={height=-18pt}, from=3-1, to=3-3]
	\arrow[""{name=3, anchor=center, inner sep=0}, "{{A(\mu_0)}}"', curve={height=18pt}, from=3-1, to=3-3]
	\arrow["{A^{\textnormal{str}}(\gamma_0)}", shorten <=5pt, shorten >=5pt, Rightarrow, from=0, to=1]
	\arrow["{A(\gamma_0)}", shorten <=5pt, shorten >=5pt, Rightarrow, from=2, to=3]
\end{tikzcd}.
\end{center}
These (higher) homotopies identify the product and bracket of $A$ and $A^{\textnormal{str}}$. 
\end{proposition}
\begin{proof}
There is a natural equivalence $\eta: \Delta^1 \times \mathbb{E}_2^{\otimes} \rightarrow N_{\text{dg}}(C^{\circ})^{\otimes}$ between $\Phi(A^{\text{str}})$ and $A$. Evaluating at the 1-simplex $(e, s_0\langle n\rangle)\in \Delta^1 \times \mathbb{E}_2^{\otimes}$, where $e$ denotes the non-degenerate 1-simplex of $\Delta^1$, yields equivalences $\Phi(A^{\text{str}})^{\otimes n} \xrightarrow{\simeq} A^{\otimes n}$ in $C$. Evaluating at the 1-simplex $(e,\mu_0)$ yields a map $\Phi(A^{\textnormal{str}})^{\otimes 2} \rightarrow A$, and the 2-simplex $(s_0e, s_1\mu_0)$ yields a chain homotopy $h_1$ between $\eta(e,\mu_0)$ and $\eta(e,s_0\langle 1 \rangle)\circ \Phi(A^{\textnormal{str}})(\mu_0)$. Similarly, the 2-simplex $(s_1e, s_0\langle 2\rangle)$ yields a chain homotopy $h_2$ between $\eta(e,\mu_0)$ and $A(\mu_0) \circ \eta(e,s_0 \langle 2\rangle)$. Composing these, we get the chain homotopy $h$ above. A similar analysis with 3-simplices yields the degree 2 map $K$.
\end{proof}
\begin{corollary}\label{cor10}
The induced Gerstenhaber algebra of $A^{\textnormal{str}}$ in the homotopy category agrees with the one constructed directly from $A$.
\end{corollary}
\subsection{The category of $\mathbb{E}_1$-modules in chain complexes}\label{the_category_of_e1_modules_in_chain_complexes}
We want to use the description of the center in Corollary \ref{cor7} to compute the center of an associative $\field$-algebra. Recall that this description involves the morphism object in the category of $\mathbb{E}_1$-modules over the algebra. In this section, we examine this module category.\displaypar
Let $\mathcal{C}^{\otimes} = N_{\text{dg}}(\ch(\field))^{\otimes}$ be the symmetric monoidal $\infty$-category corresponding to the symmetric monoidal dg model category $\ch(\field)$. Let $A$ be an $\eone$-algebra in $\ch(\field)$. By abuse of notation, we denote $\Phi(A)\in \alg_{\mathbb{E}_1}(\mathcal{C})$ again by $A$, where $\Phi$ is the equivalence constructed in Section \ref{Rectification_of_algebras_over_an_infty_operad}. We use the notation from Section \ref{the_center_as_endomorphism_object_of_bimodules}.\displaypar
Recall that $\mathbb{E}_1^{\otimes}$ only has a single color $\mathfrak{a}$, and by Corollary \ref{cor7} the underlying object of the center $\mathfrak{Z}({A})$ at this color is equivalent to the morphism object $\text{Mor}_{\bar{\mathcal{C}}_{\mathfrak{a},\mathfrak{m}}}({A},{A})\in \bar{\mathcal{C}}_{\mathfrak{a},\mathfrak{a}}$. Here ${A}$ is viewed as a module over itself, i.e. as an object of
\begin{align*}
        \overline{\mathcal{C}}_{\mathfrak{a},\mathfrak{m}}=\text{Mod}_{{A}}^{\mathbb{E}_1}(\mathcal{C}\times_{\fin}\mathbb{E}_1)_{\mathfrak{a}}.
\end{align*}
Note that $\overline{\mathcal{C}}_{\mathfrak{a},\mathfrak{a}} = \text{Mod}_{\mathbb{1}}^{\mathbb{E}_1}(\mathcal{C}\times_{N(\text{Fin}_{\ast})}\mathbb{E}_1)_{\mathfrak{a}} \simeq (\mathcal{C}^{\otimes}\times_{\fin}\mathbb{E}_1^{\otimes})_{\mathfrak{a}} \simeq \mathcal{C}$, so 
\begin{align*}
    \mathfrak{Z}({A})(\mathfrak{a})\simeq \mor_{\text{Mod}_{{A}}^{\mathbb{E}_1}(\mathcal{C}\times_{\fin}\mathbb{E}_1)_{\mathfrak{a}}}({A},{A}) \in \mathcal{C}.
\end{align*}
By \cite[Proposition B.1.2]{Hin}, we have an equivalence of $\infty$-categories
\begin{align*}
    \text{Mod}_{{A}}^{\mathbb{E}_1}(N_{\text{dg}}(\ch(\field))\times_{\fin}\mathbb{E}_1)_{\mathfrak{a}} \simeq \alg_{M\mathbb{E}_1}(N_{\text{dg}}(\ch(\field))) \times_{\alg_{\mathbb{E}_1}(N_{\text{dg}}(\ch(\field)))} \{{A}\}.
\end{align*}
Here $M\mathbb{E}_1$ is the $\infty$-operad defined as in \cite[5.2]{Hin} that governs pairs of $\mathbb{E}_1$-algebras and bimodules over them. We can now use Hinich's Rectification Theorem for modules \cite[Theorem 5.2.3]{Hin} to get an equivalence of $\infty$-categories
\begin{align*}
    N(\text{Mod}_{{A}}^{\eone}(\ch(\field))^c)[W_{\text{Mod}}^{-1}] \xrightarrow{\simeq} \text{Mod}^{\mathbb{E}_1}_{{A}}(N_{\text{dg}}(\ch(\field))\times_{\fin}\mathbb{E}_1)_{\mathfrak{a}}.
\end{align*}
Note here that since ${A}$ is cofibrant, by \cite[Theorem 2.6]{BM} the module category $\text{Mod}_{{A}}^{\eone}(\ch(\field))$ indeed carries a model category structure transferred via the forgetful functor to $\ch(\field)$. 
\begin{notation}
    Let $\mathcal{O}$ be a dg operad and $B$ an $\mathcal{O}$-algebra. We denote by $U_{\mathcal{O}}(B)$ the enveloping algebra of $B$ as defined in \cite[Section 12.3.4]{LV}.
\end{notation}
By \cite[Proposition 2.7]{BM}, the category $\lmod_{U_{\eone}({A})}(\ch(\field))$ can also be made into a model category via transfer from the forgetful functor. By \cite[Theorem 1.10]{BM} we have an isomorphism of categories making the following diagram commute 
\begin{equation}\label{dig1}
    \begin{tikzcd}
	{\text{Mod}_{{A}}^{\eone}(\ch(\field))} && {\lmod_{U_{\eone}({A})}(\ch(\field))} \\
	& {\ch(\field)}
	\arrow["\cong", from=1-1, to=1-3]
	\arrow[from=1-1, to=2-2]
	\arrow[from=1-3, to=2-2]
\end{tikzcd}.
\end{equation}
In particular, this isomorphism yields a Quillen equivalence between these two model categories. Consequently, the model category $\lmod_{U_{\eone}(A)}(\ch(\field))$ has an underlying $\infty$-category equivalent to $\text{Mod}_A^{\mathbb{E}_1}(N_{\text{dg}}(\ch(\field))\times_{\fin} \mathbb{E}_1)_{\mathfrak{a}}$.
\subsection{The Hochschild complex as a center}\label{the_hochschild_complex_as_a_center}
Fix an associative $\field$-algebra $A$. In particular, $A$ is an $\assoc$-algebra in the category $\ch(\field)$. In order to exhibit the Hochschild cochain complex of $A$ as an $\mathbb{E}_1$-center, we use the equivalence $\Phi$ constructed in Section \ref{Rectification_of_algebras_over_an_infty_operad} to view $A$ as an $\mathbb{E}_1$-algebra in the symmetric monoidal $\infty$-category $\mathcal{C}^{\otimes} := N_{\text{dg}}(\ch(\field))^{\otimes}$.\\\par
To this end, let $\phi: C_{\ast}(\mathbb{E}^T_1) \xrightarrow{\simeq} \assoc$ be the projection map. Then we get $\phi^{\ast}A \in \alg_{C_{\ast}(\mathbb{E}^T_1)}(\ch(\field))$. If $\tilde{A}\xtwoheadrightarrow{\simeq} \phi^{\ast}A$ is a cofibrant replacement, we can then use Theorem \ref{thm3} to get an object $\Phi(\tilde{A})\in \alg_{\mathbb{E}_1}(N_{\text{dg}}(\ch(\field)))$, which we by abuse of notation again denote by $\tilde{A}$. We arrive at our next main result:
\begin{theorem}\label{thm1}
    For any projective resolution $P\xtwoheadrightarrow{\simeq}A$ of $A$ as an $A^e$-module, the evaluation map 
    \begin{align*}
        \textnormal{ev}: \map_{\ch(A^e)}(P,P) \otimes P \rightarrow P
    \end{align*}
    makes the Hochschild complex of $A$ into a center of $\tilde{A}\in \alg_{\mathbb{E}_1}(\ch(\field))$. In particular, this makes $\map_{\ch(A^e)}(P,P)$ into an object of $\alg_{\mathbb{E}_1}(\alg_{\mathbb{E}_1}(N_{\textnormal{dg}}(\ch(\field))))\simeq \alg_{\mathbb{E}_2}(N_{\textnormal{dg}}(\ch(\field)))$.
\end{theorem}
Clearly, to prove the theorem we will need to apply our results in Section \ref{the_center_as_endomorphism_object_of_bimodules} and \ref{symmetric_monoidal_dg_model_categories} on endomorphism objects in bimodules, and dg mapping objects. Hence, the first step is to identify derived modules over $A^e$ with modules over $\mathbb{E}_1$, or equivalently via the Quillen equivalence (\ref{dig1}), modules over the enveloping algebra $U_{\eone}(\tilde{A})$.
\begin{lemma}\label{prop2}
There exists a zig-zag of quasi-isomorphisms between $U_{\eone}(\tilde{A})$ and $A\otimes A^{\textnormal{op}}$.
\end{lemma}
\begin{proof}
Let $\theta: \cai \xtwoheadrightarrow{\simeq} \assoc$ be a cofibrant replacement of dg operads, and note that we have a diagram
\begin{center}
    \begin{tikzcd}
	& \eone \\
	{\cai} & {\assoc}
	\arrow["{\phi,\simeq}", two heads, from=1-2, to=2-2]
	\arrow["{\psi, \simeq}", from=2-1, to=1-2]
	\arrow["{\theta,\simeq}"', two heads, from=2-1, to=2-2]
\end{tikzcd}.
\end{center}
Since we work in characteristic zero, all of the involved operads are automatically admissible and $\Sigma$-cofibrant. In particular, all the above weak equivalences are strong equivalences of operads, and thus induce a Quillen equivalence between their respective algebra categories. 
Let $A'\xtwoheadrightarrow{\simeq} A$ be a cofibrant replacement in associative algebras, and let $\hat{A} \xtwoheadrightarrow{\simeq} \theta^{\ast}A'$ be a cofibrant replacement in $\cai$-algebras. Then in particular, the unit map $\hat{A} \rightarrow \psi^{\ast}\psi_! \hat{A}$ is a weak equivalence, and hence using \cite[Theorem 17.4.A, 17.4.B]{F} we get the following diagram of weak equivalences of dg algebras
\begin{center}
    \begin{tikzcd}[sep=small]
	&&& {U_{\cai}(\hat{A})} \\
	&& {U_{\cai}(\theta^{\ast}A')} && {U_{\cai}(\psi^{\ast}\psi_!\hat{A})} \\
	& {A'\otimes A'^{\text{op}} = U_{\assoc}(A')} &&&& {U_{\eone}(\psi_!\hat{A})} \\
	{A\otimes A^{\text{op}}}
	\arrow["\simeq"', from=1-4, to=2-3]
	\arrow["\simeq", from=1-4, to=2-5]
	\arrow["{\theta_{\flat}}"', from=2-3, to=3-2]
	\arrow["{\psi_{\flat}}", from=2-5, to=3-6]
	\arrow["\simeq"', from=3-2, to=4-1]
\end{tikzcd}
\end{center}
Hence it suffices to show that $U_{\eone}(\psi_!\hat{A})$ is quasi-isomorphic to $U_{\eone}(\tilde{A})$. To this end, let $A_1 \xtwoheadrightarrow{\simeq} \phi^{\ast}A'$ be a cofibrant replacement. Note that $\psi_! \hat{A}$ is still cofibrant, and we hence have a lift $f:\psi_! \hat{A} \rightarrow A_1$ in the diagram
\begin{center}
    \begin{tikzcd}
\emptyset \arrow[d] \arrow[r]                                & A_1 \arrow[d, "\simeq", two heads] \\
\psi_!\tilde{A} \arrow[r, "\simeq"'] \arrow[ru, "f", dashed] & \phi^{\ast}A'                     
\end{tikzcd}.
\end{center}
By 2-out-of-3, $f$ must be a weak equivalence. This is a weak equivalence between cofibrant objects, so again by \cite[Theorem 17.4.A]{F}, we get a quasi-isomorphism $U_{\eone}(\psi_!\hat{A}) \xrightarrow{\simeq} U_{\eone}(A_1)$. The map $\phi^{\ast} A' \rightarrow \phi^{\ast}A$ is a trivial fibration as $\phi^{\ast}$ is right Quillen, and in particular the composition
\begin{align*}
    A_1 \xtwoheadrightarrow{\simeq} \phi^{\ast}A' \xtwoheadrightarrow{\simeq} \phi^{\ast}A
\end{align*}
is again a cofibrant replacement. We now again find a lift $g: A_1 \rightarrow \tilde{A}$ in the diagram
\begin{center}
    \begin{tikzcd}
\emptyset \arrow[d] \arrow[r]                               & \tilde{A} \arrow[d, "\simeq", two heads] \\
A_1 \arrow[r, "\simeq"', two heads] \arrow[ru, "g", dashed] & \phi^{\ast}A                            
\end{tikzcd}
\end{center}
which again is a weak equivalence, and thus finally induces a quasi-isomorphism
\begin{align*}
    U_{\eone}(A_1) \xrightarrow{\simeq} U_{\eone}(\tilde{A}).
\end{align*}
Summarizing, we get the following zig-zag
\begin{align*}
    A\otimes A^{\text{op}} \xleftarrow{\simeq} U_{\mathbb{A}_{\infty}}(\hat{A}) \xrightarrow{\simeq} U_{\eone}(\tilde{A}).
\end{align*}
\end{proof}
\begin{proof}[Proof of Theorem \ref{thm1}]
By Corollary \ref{cor8} and Corollary \ref{cor7}, the center of $\tilde{A}$ is given by the morphism object $\mor_{\textnormal{Mod}^{\mathbb{E}_1}_{\tilde{A}}(\mathcal{C}^{\otimes} \times_{\fin} \mathbb{E}_1^{\otimes})_{\mathfrak{a}}}(\tilde{A},\tilde{A})$ together with its evaluation map 
\begin{align*}
\alpha: \mor_{\textnormal{Mod}^{\mathbb{E}_1}_{\tilde{A}}(\mathcal{C}^{\otimes} \times_{\fin} \mathbb{E}_1^{\otimes})_{\mathfrak{a}}}(\tilde{A},\tilde{A}) \otimes \tilde{A} \rightarrow \tilde{A}.
\end{align*}
It therefore suffices to exhibit $\text{ev}: \map_{\ch(A^e)}(P,P) \otimes P \rightarrow P$ as such a morphism object. By the rectification results in Section \ref{the_category_of_e1_modules_in_chain_complexes} and the Quillen equivalence (\ref{dig1}), the $\mathbb{E}_1$-module category $\textnormal{Mod}^{\mathbb{E}_1}_{\tilde{A}}(\mathcal{C}^{\otimes} \times_{\fin} \mathbb{E}_1^{\otimes})_{\mathfrak{a}}$ is the underlying $\infty$-category of the model category $\lmod_{U_{\eone}(\tilde{A})}(\ch(\field))$. By Lemma \ref{prop2}, we further have a Quillen equivalence 
\begin{align*}
    \lmod_{U_{\eone}(\tilde{A})}(\ch(\field)) \simeq \lmod_{A\otimes A^{\text{op}}}(\ch(\field)) \cong \ch(A^e),
\end{align*}
and $\ch(A^e)$ is a dg model category; in particular it is left tensored over the symmetric monoidal dg model category $\ch(\field)$. We can hence apply Theorem \ref{thm4}: $P\in \ch(A^e)^{\circ}$ is bifibrant by assumption, and we already know that $\map_{\ch(A^e)}(P,P)$ together with the evaluation map is a dg endomorphism object for $P$. Therefore, it also is an endomorphism object for $P\in N_{\text{dg}}(\ch(A^e)^{\circ})$, and the equivalence of $\infty$-categories $N_{\text{dg}}(\ch(A^e)^{\circ})\simeq \text{Mod}^{\mathbb{E}_1}_{\tilde{A}}(\mathcal{C}^{\otimes}\times_{\fin} \mathbb{E}_1^{\otimes})_{\mathfrak{a}}$ identifies $P$ with $\tilde{A}$. This proves the claim.
\end{proof}
Now that we have established the Hochschild complex as a center, we can compute its bracket as an $\mathbb{E}_2$-algebra, as in Section \ref{the_bracket_operation_of_a_2_algebra}, and compare to the classical Gerstenhaber bracket. Recall that we view the Hochschild cochain complex 
\begin{align*}
    C^{\ast}(A,A) \simeq \hom_{\field}(A^{\otimes \ast},A)
\end{align*}
as a chain complex concentrated in non-positive degrees with component in degree $-n$ given by $\hom_{\field}(A^{\otimes n},A)$. We will follow the sign conventions in \cite{Wit}.
\begin{definition}
The \textbf{signed Gerstenhaber bracket} is the degree 1 map
\begin{align*}
    \hom_{\field}(A^p,A) \otimes \hom_{\field}(A^q,A) &\rightarrow \hom_{\field}(A^{p+q-1},A)\\
    f\otimes g &\mapsto (-1)^{p+1} (f\{g\} - (-1)^{(p-1)(q-1)} g\{f\})
\end{align*}
with $-\{-\}$ the circle product as defined in \cite[Definition 1.4.1]{Wit}
\begin{align*}
    f\{g\}(a_1\otimes \dots \otimes a_{p+q-1}) = \sum_{i=1}^p (-1)^{\epsilon_i} f(a_1\otimes \dots \otimes a_{i-1} \otimes g(a_i \otimes \dots \otimes a_{i+q})\otimes a_{i+q+1} \otimes \dots \otimes a_{p+q-1})
\end{align*}
and $\epsilon = (q-1)(i-1)$. We call the Gerstenhaber algebra structure on the Hochschild cohomology with the signed cup product and the signed Gerstenhaber bracket the \textbf{signed classical Gerstenhaber algebra structure}.
\end{definition}
\begin{corollary}\label{cor9}
Take $P= B(A) \xtwoheadrightarrow{\simeq}A$ to be the bar resolution. The strictification of the Hochschild complex ${\hom}_{\field}(A^{\otimes \ast},A)\in \ch(\field)$ as the center of $\tilde{A}$ naturally carries the structure of a $C_{\ast}(\mathbb{E}^T_2)$-algebra. This $C_{\ast}(\mathbb{E}^T_2)$-algebra structure recovers the signed classical Gerstenhaber algebra structure in cohomology.
\end{corollary}
\begin{proof}
By Theorem \ref{thm1}, the classical Hochschild complex is a center of $\tilde{A}$ and thus inherits a $\mathbb{E}_2$-algebra structure in the derived $\infty$-category. The first part of the corollary follows directly from the Rectification Theorem \ref{thm3}. For the second part, note that by Corollary \ref{cor10}, the underlying Gerstenhaber structure on cohomology of the strictification agrees with the Gerstenhaber structure computed directly from the $\mathbb{E}_2$-algebra $\hom_{\field}(A^{\otimes \ast},A)$.\par 
By Corollary \ref{cor7}, the two multiplications of the 2-algebra structure on the Hochschild complex as center are given by composition and convolution respectively. But both of these recover the formula for the signed cup product for the bar resolution. Hence, we get the classical signed cup product in cohomology. \par
Again by Corollary \ref{cor7}, the compatibility square has a contractible choice of fillings, and it therefore suffices to find a chain homotopy $h: \hom_{\field}(A^{\otimes \ast},A)^{\otimes 4} \rightarrow \hom_{\field}(A^{\otimes (\ast -1)}, A)$ in the square
\begin{center}
    \begin{tikzcd}
	{{\hom}_{\field}(A^{\otimes \ast},A)^{\otimes 4}} & {{\hom}_{\field}(A^{\otimes \ast},A)^{\otimes 2}} \\
	{{\hom}_{\field}(A^{\otimes \ast},A)^{\otimes 2}} & {{\hom}_{\field}(A^{\otimes \ast},A)}
	\arrow["{\smile \otimes \smile}", from=1-1, to=1-2]
	\arrow["{(\smile \otimes \smile) \circ (\text{id}\otimes \tau \otimes \text{id})}"', from=1-1, to=2-1]
	\arrow["h"', shorten <=14pt, shorten >=11pt, Rightarrow, from=1-2, to=2-1]
	\arrow["\smile", from=1-2, to=2-2]
	\arrow["\smile"', from=2-1, to=2-2]
\end{tikzcd}.
\end{center}
Recall that the circle product $f\otimes g \mapsto f\{g\}$ witnesses the homotopy commutativity of the signed cup product $f\otimes g \mapsto f\smile g$ via the following equality \cite[Lemma 1.4.5]{Wit}:
\begin{align*}
    f\smile g - (-1)^{pq} g\smile f = (-1)^{(p+1)q} \partial g\{f\} + (-1)^{(p+1)q-1} \partial(g\{f\}) + (-1)^{pq+1} g\{\partial f\}.
\end{align*}
with $|f| = -p$ and $|g| = -q$, and $\partial f = (-1)^p df$ with $d$ the usual Hochschild codifferential. This in turn yields a homotopy for the square by setting 
\begin{align*}
    h: f_1\otimes g_1 \otimes f_2 \otimes g_2 \mapsto (-1)^{|f_1|+|f_1|+|f_2||g_1|-1} f_1 \smile f_2\{g_1\}\smile g_2,
\end{align*}
with codifferential $\partial$ on the Hochschild complex. Restricting along $\iota_{1,4}$ yields the trivial homotopy between the cup product and itself, and restricting along $\iota_{2,3}$ yields the homotopy 
\begin{align*}
    H: {\hom}_{\field}(A^{\otimes p},A) \otimes {\hom}_{\field}(A^{\otimes q},A) &\rightarrow {\hom}_{\field}(A^{\otimes p+q-1},A)\\
    f\otimes g &\mapsto (-1)^{pq+q-1} g\{f\}
\end{align*}
between $\smile$ and $\smile \circ \tau$. For the opposite multiplications, we get 
\begin{align*}
    h^{\text{op}}: f_1\otimes g_1\otimes f_2\otimes g_2 \mapsto (-1)^{\alpha + |f_2||g_1| + |f_2| -1} g_2 \smile g_1\{f_2\} \smile f_1
\end{align*}
with $\alpha = |f_1||f_2| + |g_1||g_2| + |f_1||g_1| + |f_2||g_2| + |f_1||g_2| + |f_2||g_1|$. Restricting along $\iota_{1,4}$ again yields the trivial homotopy, and restricting along $\iota_{2,3}$ yields 
\begin{align*}
    H^{\text{op}}: f\otimes g \mapsto (-1)^{p+1} f\{g\}.
\end{align*}
Hence by the remark after Corollary \ref{cor5}, the bracket is given by
\begin{align*}
    [f,g] = H(f\otimes g) + H^{\text{op}}(f\otimes g) &= (-1)^{|f||g|+|g|-1} g\{f\} + (-1)^{|f|+1} f\{g\} \\&= (-1)^{|f|+1}(f\{g\} - (-1)^{(|f|+1)(|g|+1)} g\{f\}) \\&= (-1)^{|f|+1} [f,g]_G.
\end{align*}
This agrees with the signed classical Gerstenhaber bracket, proving the claim. 
\end{proof}

\section{The Hochschild complex of a scheme} \label{The_Hochschild_complex_of_a_scheme}
As explained in Section \ref{delignes_conjecture_on_hochschild_cochains}, we want to globalize the above results to a quasi-compact separated scheme $X$ over $\field$. We follow the same reasoning as in the previous section. The structure sheaf $\mathcal{O}_X$ is an associative algebra object in the category of (pre)sheaves of $\field$-modules on $X$. We consider the dg category $\dgpsh{X}$ of complexes of presheaves of $\field$-modules on $X$, and show that it can be equipped with a dg model category structure that presents the $\infty$-category of dg sheaves on $X$. We can make $\mathcal{O}_X$ into a $C_{\ast}(\mathbb{E}_1^T)$-algebra in $\dgpsh{X}$, which by Theorem \ref{thm3} produces an $\mathbb{E}_1$-algebra in the associated $\infty$-category of dg sheaves. We then argue that the center of this $\mathbb{E}_1$-algebra is a good model for the Hochschild cochain complex of $X$. Note that this does not require $X$ to be smooth. 
\subsection{The $\infty$-category of dg sheaves}\label{the_infty_category_of_dg_sheaves}
Let $X$ be a quasi-compact separated scheme over $\field$. Since the category of presheaves of complexes of $\field$-modules on $X$ is a functor category, it admits an injective and a projective model structure. These model structures do not know anything about the geometry of $X$; in particular the bifibrant objects are not ``homotopy sheaves'' in any way. In order to present the $\infty$-category of dg sheaves, one needs to localize these model structures at Čech nerves of covering families. In this section, we will recall how to construct the ``local projective'' model structure, and prove some basic properties of this model category. Note that we use presheaves instead of sheaves since the category of dg sheaves does not have enough projectives, and the injective model structure does not behave well with respect to tensor products.
\begin{proposition}[\cite{Hin2}, Theorem 1.3.1]
Let $S$ be a site. There is a cofibrantly generated model structure on the category $\dgpsh{S}$ of presheaves of $\field$-module complexes on $S$ with 
\begin{itemize}
    \item weak equivalences the maps $f: \mathcal{F} \rightarrow \mathcal{G}$ such that the degreewise sheafification $f^a: \mathcal{F}^a \rightarrow \mathcal{G}^a$ is a quasi-isomorphism of complexes of sheaves,
    \item cofibrations generated by maps $f: \mathcal{F} \rightarrow \mathcal{F}\langle x;dx = z\in \mathcal{F}(U)\rangle$ corresponding to adding a section to kill a cycle z over $U\in S$, and
    \item fibrations the maps $f: \mathcal{F} \rightarrow \mathcal{G}$ such that $f(U): \mathcal{F}(U) \rightarrow \mathcal{G}(U)$ is surjective for all $U\in S$ and for any hypercover $\epsilon: V_{\bullet} \rightarrow U$ the diagram
    \begin{center}
        \begin{tikzcd}
\mathcal{F}(U) \arrow[d, "f(U)"'] \arrow[r] & {\check{C}(V_{\bullet},\mathcal{F})} \arrow[d] \\
\mathcal{G}(U) \arrow[r]                    & {\check{C}(V_{\bullet},\mathcal{G})}          
\end{tikzcd}
    \end{center}
    is a homotopy pullback. 
\end{itemize}
\end{proposition}
This is the left Bousfield localization of the projective model structure on $\dgpsh{S}$ with respect to the Čech complexes of hypercoverings. We therefore call it the \textbf{local projective model structure}. In particular, acyclic fibrations in the local projective model structure are just acyclic fibrations in the unlocalized projective model structure. Note that if $S$ has enough points, weak equivalences can be detected at stalks, i.e. a map $f: \mathcal{F} \rightarrow \mathcal{G}$ is a weak equivalence if and only if for each point $x$ of $S$, the induced map $f_x: \mathcal{F}_x \rightarrow \mathcal{G}_x$ is a quasi-isomorphism.
\begin{notation}
Let $X$ be a scheme over $\field$. Let $\textnormal{Aff}(X)$ be the site of affine open subsets of $X$, and let $\textnormal{Open}(X)$ be the site of all open subsets of $X$. We call the associated categories of dg presheaves $\affpsh{X}$ and $\dgpsh{X}$ respectively. We have a natural inclusion $\iota: \textnormal{Aff}(X) \rightarrow \textnormal{Open}(X)$ that induces a restriction functor
\begin{align*}
    \iota_{\ast}: \dgpsh{X} \rightarrow\affpsh{X}.
\end{align*}
\end{notation}
\begin{proposition}\label{prop3}
The restriction functor $\iota_{\ast}$ admits a left adjoint $\iota^{-1}$, and the pair $\iota^{-1}\dashv \iota_{\ast}$ forms a Quillen equivalence. Both $\iota_{\ast}$ and $\iota^{-1}$ preserve weak equivalences, and $\iota^{-1}$ preserves acyclic fibrations. The unit $\textnormal{id} \Rightarrow \iota_{\ast} \iota^{-1}$ is an isomorphism, and the counit $\iota^{-1}\iota_{\ast} \Rightarrow \textnormal{id}$ is a componentwise weak equivalence.
\end{proposition}
\begin{proof}
The left adjoint $\iota^{-1}$ is given by $(\iota^{-1}\mathcal{F})(V) = \colim_{V\subseteq U \in {\textnormal{Aff}(X)}}\mathcal{F}(U)$. The direct image $\iota_{\ast}$ clearly preserves acyclic fibrations, since these are pointwise. The sites of all opens and of affine opens have the same points, namely points in the topological space $X$. This follows because affine opens form a basis of the Zariski topology. Even more, $\iota_{\ast}$ and $\iota^{-1}$ preserve stalks at these points. Taking stalks is a left adjoint, so this follows trivially for the inverse image, and for the direct image we note that small enough neighborhoods of a point $x\in X$ always contain an affine open neighborhood of $x$. This shows that both adjoints preserve weak equivalences, and in particular $\iota^{-1}$ preserves acyclic cofibrations. The fact that $\iota^{-1}$ preserves acyclic fibrations follows from the fact that filtered colimits are exact in Grothendieck categories. Finally, note that if $U$ is affine, then $\colim_{U\subseteq W \in {\textnormal{Aff}(X)}}\mathcal{F}(W) \cong F(U)$ since $U$ is final in the index category. This shows that the unit is an isomorphism. The fact that the counit is a componentwise weak equivalence again follows from the fact that both adjoints preserve weak equivalences.
\end{proof}
\begin{definition}
    We call the underlying $\infty$-category of the local projective model structure on $\dgpsh{X}$ the \textbf{$\infty$-category of dg sheaves on $X$}
    \begin{align*}
        \textnormal{Sh}_{\infty}(X) := N(\dgpsh{X}^c)[W^{-1}].
    \end{align*}
\end{definition}
By Proposition \ref{prop3}, we have an equivalence of $\infty$-categories
\begin{align*}
    \iota_{\ast}: \textnormal{Sh}_{\infty}(X)\rightarrow N(\affpsh{X}^c)[W^{-1}] 
\end{align*}
with quasi-inverse $\iota^{-1}: N(\affpsh{X}^c)[W^{-1}]\rightarrow \textnormal{Sh}_{\infty}(X)$.
\begin{remark}
Even though the model categories of presheaves on affine opens and general opens yield the same $\infty$-category, the above model category structure depends on the choice of site. On affine open subsets, all quasi-coherent sheaves on $X$ are automatically fibrant, which is not true for general opens. In particular, on affine opens the structure sheaf $\mathcal{O}_X$ itself is fibrant.
\end{remark}
\begin{proposition}
    If the topos on $S$ has enough points and $S$ admits finite products, then the local projective model structure yields a closed symmetric monoidal model category. If in addition $S$ admits a final object, then $\dgpsh{S}$ is a symmetric monoidal dg model category.
\end{proposition}
\begin{proof}
By \cite[Proposition 7.9]{PS}, the global projective model structure on $\dgpsh{S}$ inherits the structure of a symmetric monoidal model category since $S$ admits finite products. By \cite[Theorem 4.6]{W}, to show that this symmetric monoidal model structure descends to the local projective model structure, it suffices to argue that for $f$ a local weak equivalence and $\mathcal{F}$ a cofibrant object, the map $f\otimes \textnormal{id}_{\mathcal{F}}$ is again a local weak equivalence. But this is clear if the topos has enough points, since we can then check local weak equivalences on stalks. Since the presheaf category admits an internal hom given by
\begin{align*}
    \mathcal{H}om(\mathcal{F},\mathcal{G})(U) = \map_{\dgpsh{S_U}}(\mathcal{F}|_U, \mathcal{G}|_U),
\end{align*}
this proves the first part. For the second part, note that presheaves of chain complexes are the same as chain complexes of presheaves of $\field$-modules.  Since the latter is an abelian category, this category automatically admits a dg enrichment. Recall that if $\ast\in S$ is terminal, we have the constant presheaf functor
\begin{align*}
    \textnormal{const}_{\ast}: \ch(\field) &\rightarrow \dgpsh{S}\\
    C &\mapsto (U \mapsto C).
\end{align*}
We define a tensoring
\begin{align*}
    \ch(\field) \times \dgpsh{S} \rightarrow \dgpsh{S}, \quad (C,\mathcal{F}) \mapsto \textnormal{const}_{\ast}(C) \otimes \mathcal{F}
\end{align*}
as well as a powering
\begin{align*}
    \ch(\field)^{\textnormal{op}}\times \dgpsh{S} \rightarrow \dgpsh{S}, \quad (C,\mathcal{F}) \mapsto \mathcal{H}om(\textnormal{const}_{\ast}(C),\mathcal{F}).
\end{align*}
One easily checks that these indeed satisfy the correct adjointness properties. It hence suffices to check the pushout-product axiom. The argument below shows that $\textnormal{const}_{\ast}$ preserves cofibrations. Hence if $i: C\rightarrow D$ is a cofibration in $\ch(\field)$, then $\textnormal{const}_{\ast}(i): \textnormal{const}_{\ast}(C) \rightarrow \textnormal{const}_{\ast}(D)$ is a cofibration in $\dgpsh{S}$, and therefore the pushout-product axiom follows directly from the pushout-product axiom in $\dgpsh{S}$.
\end{proof}
\begin{corollary}
Let $X$ be a separated quasi-compact scheme over $\field$. The category $\dgpsh{X}$ admits the structure of a symmetric monoidal dg model category. In particular, the $\infty$-category $\textnormal{Sh}_{\infty}(X)$ admits the structure of a symmetric monoidal $\infty$-category $\textnormal{Sh}_{\infty}(X)^{\otimes}$.
\end{corollary}
Let $U\subseteq X$ be an affine open. Then we have adjoint functors
\begin{center}
    \begin{tikzcd}
            \ch(\field)\arrow[r, shift left=1ex, "C_U"{name=G}] & \affpsh{X}\arrow[l, shift left=.5ex, "\Gamma_U"{name=F}]
            \arrow[phantom, from=F, to=G, , "\scriptscriptstyle\boldsymbol{\dashv}" rotate=-90]
        \end{tikzcd}
\end{center}
where $\Gamma_U$ sends a complex of presheaves $\mathcal{F}$ to $\mathcal{F}(U)$ and $C_U$ is the constant presheaf functor sending $C$ to the presheaf 
\begin{align*}
    V\mapsto \begin{cases}
        C & \text{if } V\subseteq U\\
        0 & \text{otherwise}
    \end{cases}.
\end{align*}
If we equip $\affpsh{X}$ with the projective model structure, then $\Gamma_U$ preserves fibrations and weak equivalences by construction. Therefore we obtain a Quillen adjunction. We can compose this with the Quillen adjunction
\begin{center}
    \begin{tikzcd}
            \affpsh{X}^{\text{proj}}\arrow[r, shift left=1ex, "\text{id}"{name=G}] & \affpsh{X}^{\text{loc}}\arrow[l, shift left=.5ex, "\text{id}"{name=F}]
            \arrow[phantom, from=F, to=G, , "\scriptscriptstyle\boldsymbol{\dashv}" rotate=-90]
        \end{tikzcd}
\end{center}
of the Bousfield localization to obtain a Quillen adjunction
\begin{center}
    \begin{tikzcd}
            \ch(\field)\arrow[r, shift left=1ex, "C_U"{name=G}] & \affpsh{X}^{\text{loc}}\arrow[l, shift left=.5ex, "\Gamma_U"{name=F}]
            \arrow[phantom, from=F, to=G, , "\scriptscriptstyle\boldsymbol{\dashv}" rotate=-90].
        \end{tikzcd}
\end{center}
Since $C_U$ is left Quillen, it preserves weak equivalences between cofibrant objects. But every object in $\ch(\field)$ is cofibrant, so $C_U$ preserves weak equivalences.\displaypar
The same argument works if we instead consider the site of all opens $\textnormal{Open}(X)$. In this case, for any open $V\subseteq X$ we obtain a Quillen adjunction $C_V \dashv \Gamma_V$. Note that in particular we can then take $V=X$. If $V$ is affine, this agrees with the above construction. The functors $C_V$ and $\Gamma_V$ are both strong symmetric monoidal, since the tensor product of presheaves is taken sectionwise. In particular, we obtain lax symmetric monoidal functors of $\infty$-categories $C_V: \dk^{\otimes} \rightarrow \text{Sh}_{\infty}(X)^{\otimes}$ and $\mathbb{R}\Gamma_V: \text{Sh}_{\infty}(X)^{\otimes} \rightarrow \dk^{\otimes}$.
\begin{notation}
    Let $\mathcal{O}$ be a dg operad. We then have an operad $C_X(\mathcal{O})$ in $\dgpsh{X}$. By abuse of notation, we will usually denote the operad $C_X(\mathcal{O})$ just by $\mathcal{O}$.
\end{notation}
\begin{lemma}\label{lem2}
    The functor $C_X$ preserves cofibrancy and $\Sigma$-cofibrancy of operads, as well as weak equivalences of operads. Every operad in $\dgpsh{X}$ is admissible \cite[Definition 5.1.]{PS}, and even strongly admissible \cite[Definition 5.1.]{PS} if it lies in the image of $C_X$.
\end{lemma}
\begin{proof}
The Quillen adjunction $C_X\dashv \Gamma_X$ induces adjunctions between the respective categories of symmetric collections and symmetric operads, since both are strong symmetric monoidal. The model structure on symmetric collection is transferred from the underlying model category, and hence the adjunction is again Quillen. Similarly, fibrations and weak equivalences of operads are pointwise, and hence $\Gamma_X$ preserves fibrations and trivial fibrations of operads. To see that $C_X$ preserves weak equivalences, note that these are pointwise in operads, and $C_X$ preserves weak equivalences on the underlying model categories. To see that operads in $\dgpsh{X}$ are admissible, use \cite[Theorem 5.11]{PS} and Section 8 of \cite{PS2}. Now to see that every operad in the image of $C_X$ is even strongly admissible, use \cite[Proposition 6.3]{PS} together with the fact that every operad in $\ch(\field)$ is $\Sigma$-cofibrant.
\end{proof}
This shows that we get a diagram of admissible $\Sigma$-cofibrant operads
\begin{center}
    \begin{tikzcd}
                                                                                             & C_X(C_{\ast}(\mathbb{E}^T_1)) \arrow[d, "{C_X(\phi), \simeq}"] \\
C_X(\cai) \arrow[ru, "{C_X(\psi),\simeq}"] \arrow[r, "{C_X(\theta),\simeq}"'] & C_X(\assoc)                                           
\end{tikzcd}
\end{center}
and $C_X(\cai)$ is still cofibrant.\\\par
If $X$ is a quasi-compact separated scheme over $\field$ with structure sheaf $\mathcal{O}_X$, we consider the $C_{\ast}(\mathbb{E}^T_1)$-algebra $\phi^{\ast}\mathcal{O}_X$ and choose a cofibrant replacement $\tilde{\mathcal{O}}_X \xtwoheadrightarrow{\simeq}\phi^{\ast}\mathcal{O}_X$. Applying the functor $\Phi$ constructed in Section \ref{Rectification_of_algebras_over_an_infty_operad}, we get $\Phi(\tilde{\mathcal{O}}_X)\in \alg_{\mathbb{E}_1}(\text{Sh}_{\infty}(X))$. As in the affine case, we will keep denoting this by just $\tilde{\mathcal{O}}_X$.
\begin{definition}
Let $X$ be a quasi-compact separated scheme over $\field$. We define the \textbf{Hochschild cochain complex} of $X$ to be the center
\begin{align*}
    \mathfrak{Z}(\tilde{\mathcal{O}}_X)\in \alg_{\mathbb{E}_1}(\alg_{\mathbb{E}_1}(\text{Sh}_{\infty}(X))) \simeq \alg_{\mathbb{E}_2}(\text{Sh}_{\infty}(X)).
\end{align*}
\end{definition}
%
\subsection{The Hochschild complex of a scheme is local}
In this section we will prove the following theorem, showing that the center of a scheme glues together the affine Hochschild complexes.
\begin{theorem}\label{thm5}
Let $U = \textnormal{Spec}(A)\subseteq X$ be an affine open. The map $\mathbb{R}\Gamma_U: \textnormal{Sh}_{\infty}(X) \rightarrow \dk$ is lax symmetric monoidal and hence induces a map $\mathbb{R}\Gamma_U: \alg_{\mathbb{E}_2}(\textnormal{Sh}_{\infty}(X)) \rightarrow \alg_{\mathbb{E}_2}(\dk)$. We have
\begin{align*}
    \mathbb{R}\Gamma_U(\mathfrak{Z}(\tilde{\mathcal{O}}_X)) \simeq \mathfrak{Z}(\tilde{A})\in \alg_{\mathbb{E}_2}(\dk)
\end{align*}
for any cofibrant replacement $\tilde{A}$ of $\phi^{\ast}A$.
\end{theorem}
Just like in the affine case, the center has underlying object
\begin{align*}
    \mathfrak{Z}(\tilde{\mathcal{O}}_X)(\mathfrak{a}) \simeq \text{Mor}_{\text{Mod}^{\mathbb{E}_1}_{\tilde{\mathcal{O}}_X}(\text{Sh}_{\infty}(X)\times_{\fin} \mathbb{E}_1)_{\mathfrak{a}}}(\tilde{\mathcal{O}}_X,\tilde{\mathcal{O}}_X) \in \text{Sh}_{\infty}(X),
\end{align*}
and it hence suffices to understand this endomorphism object. To this end, note that we can adapt Hinich's Rectification Theorem for modules \cite[Theorem 5.2.3]{Hin} to the local projective model structure on complexes of presheaves by using the generalized Rectification Theorem \ref{thm3}. We hence have an equivalence of $\infty$-categories
\begin{align*}
    N(\text{Mod}_{\tilde{\mathcal{O}}_X}^{C_{\ast}(\mathbb{E}^T_1)}(\dgpsh{X}^c))[W_{\text{Mod}}^{-1}] \simeq   \text{Mod}^{\mathbb{E}_1}_{\tilde{\mathcal{O}}_X}(\text{Sh}_{\infty}(X)\times_{\fin}\mathbb{E}_1)_{\mathfrak{a}}.
\end{align*}
By \cite[Theorem 1.5.1]{Hin2}, for any associative algebra object $\mathcal{O}$ in $\dgpsh{X}$, the category $\lmod_{\mathcal{O}}(\dgpsh{X})$ carries a model structure transferred from the local projective model structure on $\dgpsh{X}$. Again following the affine case, we have a Quillen equivalence
\begin{align*}
        \text{Mod}_{\tilde{\mathcal{O}}_X}^{C_{\ast}(\mathbb{E}^T_1)}(\dgpsh{X}) \cong \lmod_{U_{C_{\ast}(\mathbb{E}^T_1)}(\tilde{\mathcal{O}}_X)}(\dgpsh{X}).
\end{align*}
\begin{proposition}
    There exists a zig-zag of weak equivalences between $U_{C_{\ast}(\mathbb{E}^T_1)}(\tilde{\mathcal{O}}_X)$ and $\mathcal{O}_X\otimes \mathcal{O}_X$ in the category of associative algebras in $\dgpsh{X}$.
\end{proposition}
\begin{proof}
    We adapt the proof of Proposition \ref{prop2}. Let $\mathcal{O}_X' \xtwoheadrightarrow{\simeq} \mathcal{O}_X$ be a cofibrant resolution of $\assoc$-algebras in $\dgpsh{X}$. Then
    \begin{align*}
        (\mathcal{O}_X'\otimes \mathcal{O}_X')_x \xrightarrow{\cong} \mathcal{O}_{X,x}' \otimes \mathcal{O}_{X,x}' \xrightarrow{\simeq} \mathcal{O}_{X,x} \otimes \mathcal{O}_{X,x} \xrightarrow{\cong} (\mathcal{O}_X\otimes \mathcal{O}_X)_x,
    \end{align*}
    showing that $\mathcal{O}_X' \otimes \mathcal{O}_X'$ is weakly equivalent to $\mathcal{O}_X\otimes \mathcal{O}_X$. The rest of the argument goes through exactly like before with the following amendment: Let $\hat{\mathcal{O}}_X \xtwoheadrightarrow{\simeq} \theta^{\ast}\mathcal{O}_X'$ be a cofibrant replacement of $\cai$-algebras. To obtain the zig-zag 
    \begin{center}
        \begin{tikzcd}
                                                     & U_{\cai}(\hat{\mathcal{O}}_X) \arrow[ld, "\simeq"'] \arrow[rd, "\simeq"] &                                                               \\
U_{\cai}(\theta^{\ast}\mathcal{O}_X') &                                                                                         & U_{\cai}(\psi^{\ast}\psi_!\hat{\mathcal{O}}_X)
\end{tikzcd}
    \end{center}
    using \cite[Theorem 17.4.A, 17.4.B]{F}, we have to argue that the underlying complexes of presheaves of these $\cai$-algebras are cofibrant. To this end, recall from Lemma \ref{lem2} that $\cai$ and $\assoc$ are both strongly admissible, meaning that the forgetful functor from algebras preserves cofibrant objects. In particular, the underlying complexes of presheaves of $\mathcal{O}_X'$ and $\hat{\mathcal{O}}_X$ are both cofibrant. Now recall that $\psi_!$ is left Quillen and hence preserves cofibrancy, and the restriction of scalars functors $\theta^{\ast}$ and $\psi^{\ast}$ do not alter the underlying complex. This finishes the argument.
\end{proof}
Using \cite[Lemma 1.5.3]{Hin2}, we deduce that when equipped with the transfer model structure, the category $\lmod_{U_{\eone}(\tilde{\mathcal{O}}_X)}(\dgpsh{X})$ is Quillen equivalent to the category $\lmod_{\mathcal{O}_X\otimes \mathcal{O}_X}(\dgpsh{X})$. 
\begin{lemma}
    The category $\lmod_{\mathcal{O}_X\otimes \mathcal{O}_X}(\dgpsh{X})$ is a dg model category that is left tensored over the symmetric monoidal dg model category $\dgpsh{X}$. For $\mathcal{M},\mathcal{N}\in \lmod_{\mathcal{O}_X\otimes \mathcal{O}_X}(\dgpsh{X})$, we have a dg morphism object $\mathcal{H}om_{\mathcal{O}_X\otimes \mathcal{O}_X}(\mathcal{M},\mathcal{N})\in \dgpsh{X}$ with
    \begin{align*}
        \mathcal{H}om_{\mathcal{O}_X\otimes \mathcal{O}_X}(\mathcal{M},\mathcal{N})(V) =\map_{\mathcal{O}_V\otimes \mathcal{O}_V}(\mathcal{M}|_{V}, \mathcal{N}|_V).
    \end{align*}
\end{lemma}
\begin{proof}
Since $\lmod_{\mathcal{O}_X\otimes \mathcal{O}_X}(\dgpsh{X})$ is the category of complexes of presheaves of $\mathcal{O}_X\otimes \mathcal{O}_X$-modules, we automatically have an enrichment over $\ch(\field)$. Now note that $\lmod_{\mathcal{O}_X\otimes \mathcal{O}_X}(\dgpsh{X}))$ is tensored and powered over $\dgpsh{X}$: If $\mathcal{F}\in \dgpsh{X}$ and $\mathcal{M} \in \lmod_{\mathcal{O}_X\otimes \mathcal{O}_X}(\dgpsh{X})$, we obtain an $\mathcal{O}_X\otimes \mathcal{O}_X$-module structure on the tensor product in complexes of presheaves by 
\begin{align*}
    (\mathcal{O}_X\otimes \mathcal{O}_X)\otimes (\mathcal{F} \otimes \mathcal{M}) \cong \mathcal{F} \otimes \mathcal{O}_X\otimes \mathcal{O}_X \otimes \mathcal{M} \rightarrow \mathcal{F} \otimes \mathcal{M}.
\end{align*}
We obtain a module structure on $\mathcal{H}om(\mathcal{F},\mathcal{M})$ by the pointwise module structure
\begin{align*}
        \mathcal{O}_X\otimes \mathcal{O}_X \otimes \mathcal{H}om(\mathcal{F},\mathcal{M}) \rightarrow \mathcal{H}om(\mathcal{M},\mathcal{M}) \otimes \mathcal{H}om(\mathcal{F},\mathcal{M}) \\\xrightarrow{\text{comp}} \mathcal{H}om(\mathcal{F},\mathcal{M}).
\end{align*}
One easily checks that these indeed yield a tensoring and powering respectively. Now to obtain the tensoring and powering over $\ch(\field)$, we pre-compose these operations with the constant presheaf functor $C_X: \ch(\field) \rightarrow \dgpsh{X}$. The pushout-product axiom is checked in \cite[Lemma 1.6.3]{Hin2}, and Hinich also shows in the same section that the above actually yields a morphism object. 
\end{proof}
We can now again use Theorem \ref{thm4} to conclude that the center of $\tilde{\mathcal{O}}_X$ in the $\infty$-category $\alg_{\mathbb{E}_1}(\text{Sh}_{\infty}(X))$ is given by 
\begin{align*}
    \mathfrak{Z}(\tilde{\mathcal{O}}_X)(\mathfrak{a}) \simeq Q\mathcal{H}om_{\mathcal{O}_X\otimes\mathcal{O}_X}(\mathcal{O},\mathcal{O})
\end{align*}
for a bifibrant model $\mathcal{O}$ of $\mathcal{O}_X$ as an $\mathcal{O}_X\otimes \mathcal{O}_X$-module. The center action is given by the evaluation map
\begin{align*}
    R(Q\mathcal{H}om_{\mathcal{O}_X\otimes \mathcal{O}_X}(\mathcal{O},\mathcal{O}) \otimes \mathcal{O}) \rightarrow \mathcal{O}.
\end{align*}
\par Note that in contrast to the affine case, where $A$ was already fibrant as an $A \otimes A$-module, we need to take a \textit{bifibrant} resolution $\mathcal{O}$ of $\mathcal{O}_X$, since it is not fibrant in the local projective model structure on $\dgpsh{X}$. However, $\mathcal{O}_X$ is fibrant in the local projective model structure for the site of affine opens on $X$, and we have already seen that presheaves on this smaller site present the same $\infty$-category. We now show that we can indeed compute the restriction $\iota_{\ast}\mathfrak{Z}(\tilde{\mathcal{O}}_X)$ by just taking a cofibrant resolution of $\mathcal{O}_X$.
\begin{proposition}\label{prop4}
We have an equivalence
\begin{align*}
    \iota_{\ast}(Q\mathcal{H}om_{\mathcal{O}_X\otimes \mathcal{O}_X}(\mathcal{O},\mathcal{O})) \simeq Q\mathcal{H}om_{\iota_{\ast}(\mathcal{O}_X\otimes \mathcal{O}_X)}(\mathcal{P},\mathcal{P})
\end{align*}
for any cofibrant resolution $\mathcal{P} \xtwoheadrightarrow{\simeq} \iota_{\ast}\mathcal{O}_X$ in $\lmod_{\iota_{\ast}(\mathcal{O}_X\otimes\mathcal{O}_X)}(\affpsh{X})$. 
\end{proposition}
Before we prove the proposition, we need a few lemmas.
\begin{lemma}
We have a Quillen equivalence 
\begin{center}
    \begin{tikzcd}
           \lmod_{\iota_{\ast}(\mathcal{O}_X \otimes \mathcal{O}_X)}(\affpsh{X})\arrow[r, shift left=1ex, "\iota^{\ast}"{name=G}] & \lmod_{\mathcal{O}_X \otimes \mathcal{O}_X}(\dgpsh{X})\arrow[l, shift left=.5ex, "\iota_{\ast}"{name=F}]
            \arrow[phantom, from=F, to=G, , "\scriptscriptstyle\boldsymbol{\dashv}" rotate=-90].
        \end{tikzcd}
\end{center}
Both adjoints preserve weak equivalences.
\end{lemma}
\begin{proof}
The Quillen equivalence $\iota^{-1}\dashv \iota_{\ast}$ induces a Quillen equivalence
\begin{center}
    \begin{tikzcd}
           \lmod_{\iota_{\ast}(\mathcal{O}_X \otimes \mathcal{O}_X)}(\affpsh{X})\arrow[r, shift left=1ex, "\iota^{-1}"{name=G}] & \lmod_{\iota^{-1}\iota_{\ast}(\mathcal{O}_X \otimes \mathcal{O}_X)}(\dgpsh{X})\arrow[l, shift left=.5ex, "\iota_{\ast}"{name=F}]
            \arrow[phantom, from=F, to=G, , "\scriptscriptstyle\boldsymbol{\dashv}" rotate=-90].
        \end{tikzcd}
\end{center}
The counit yields a weak equivalence $\epsilon: \iota^{-1}\iota_{\ast}(\mathcal{O}_X\otimes \mathcal{O}_X) \xrightarrow{\simeq} \mathcal{O}_X\otimes\mathcal{O}_X$, and hence by \cite[Lemma 1.5.3]{Hin2} we get a Quillen equivalence
\begin{center}
    \begin{tikzcd}
           \lmod_{\iota^{-1}\iota_{\ast}(\mathcal{O}_X \otimes \mathcal{O}_X)}(\dgpsh{X})\arrow[r, shift left=1ex, "\epsilon^{\ast}"{name=G}] & \lmod_{\mathcal{O}_X \otimes \mathcal{O}_X}(\dgpsh{X})\arrow[l, shift left=.5ex, "\epsilon_{\ast}"{name=F}]
            \arrow[phantom, from=F, to=G, , "\scriptscriptstyle\boldsymbol{\dashv}" rotate=-90].
        \end{tikzcd}
\end{center}
composing these yields the Quillen equivalence in the statement. Clearly, $\epsilon_{\ast}$ preserves weak equivalences. But $\epsilon$ is an isomorphism on stalks, and hence $\epsilon^{\ast}$ also preserves weak equivalences.
\end{proof}
\begin{lemma}
For $\mathcal{F}\in \lmod_{\iota_{\ast}(\mathcal{O}_X\otimes \mathcal{O}_X)}(\affpsh{X})$ and $\mathcal{G} \in \lmod_{\mathcal{O}_X\otimes \mathcal{O}_X}(\dgpsh{X})$, we have
\begin{align*}
    \iota_{\ast}\mathcal{H}om_{\mathcal{O}_X\otimes \mathcal{O}_X}(\iota^{\ast}\mathcal{F},\mathcal{G}) \cong \mathcal{H}om_{\iota_{\ast}(\mathcal{O}_X\otimes \mathcal{O}_X)}(\mathcal{F},\iota_{\ast}\mathcal{G}).
\end{align*}
\end{lemma}
\begin{proof}
Let $U\in \textnormal{Aff}(X)$, and let $\iota': \textnormal{Aff}(U) \rightarrow \textnormal{Open}(U)$ be the inclusion. Then
\begin{align*}
    \map_{\mathcal{O}_U\otimes \mathcal{O}_U}(\iota^{\ast}\mathcal{F}|_U,\mathcal{G}|_U) &\cong \map_{\mathcal{O}_U\otimes \mathcal{O}_U}(\iota'^{\ast}(\mathcal{F}|_U),\mathcal{G}|_U) \\&\cong \map_{\iota'_{\ast}(\mathcal{O}_U\otimes \mathcal{O}_U)}(\mathcal{F}|_U,\iota'_{\ast}(\mathcal{G}|_U)) \\&\cong \map_{\iota_{\ast}(\mathcal{O}_X\otimes \mathcal{O}_X)|_U}(\mathcal{F}|_U,\iota_{\ast}\mathcal{G}|_U).
\end{align*}
\end{proof}
\begin{proof}[Proof of Proposition \ref{prop4}]
By definition, we have a diagram in $\lmod_{\mathcal{O}_X\otimes \mathcal{O}_X}(\dgpsh{X})$
\begin{center}
    \begin{tikzcd}
\mathcal{P}' \arrow[d, "\simeq"', tail] \arrow[r, "\simeq", two heads] & \mathcal{O}_X \\
\mathcal{O}                                                            &              
\end{tikzcd}
\end{center}
with $\mathcal{P}'$ cofibrant and $\mathcal{O}$ bifibrant. If $\mathcal{P} \xtwoheadrightarrow{\simeq} \iota_{\ast}\mathcal{O}_X$ is a cofibrant resolution, we get a weak equivalence $\iota^{\ast}\mathcal{P} \xrightarrow{\simeq} \iota^{\ast}\iota_{\ast} \mathcal{O}_X$ and $\iota^{\ast} \mathcal{P}$ is still cofibrant. We can hence solve the following lifting problem
\begin{center}
    \begin{tikzcd}
0 \arrow[rr] \arrow[d, tail]                                     &                                                            & \mathcal{P}' \arrow[d, "\simeq", two heads] \\
\iota^{\ast}\mathcal{P} \arrow[r, "\simeq"'] \arrow[rru, dashed] & \iota^{\ast}\iota_{\ast}\mathcal{O}_X \arrow[r, "\simeq"'] & \mathcal{O}_X                              
\end{tikzcd}
\end{center}
and by 2-out-of-3, the lift $\iota^{\ast}\mathcal{P} \rightarrow \mathcal{P}'$ is again a weak equivalence. Let $\iota^{\ast}\mathcal{P} \overset{\simeq}{\rightarrowtail} \mathcal{R}$ be a fibrant resolution. We can also solve the lifting problem
\begin{center}
    \begin{tikzcd}
\iota^{\ast}\mathcal{P} \arrow[d, "\simeq"', tail] \arrow[r, "\simeq"] & \mathcal{P}' \arrow[r, "\simeq", tail] & \mathcal{O} \arrow[d] \\
\mathcal{R} \arrow[rr] \arrow[rru, dashed]                             &                                        & 0                    
\end{tikzcd}
\end{center}
to obtain a weak equivalence $\mathcal{R} \xrightarrow{\simeq} \mathcal{O}$. In particular, this is a weak equivalence between bifibrant objects. Therefore,
\begin{align*}
    \iota_{\ast}Q\mathcal{H}om_{\mathcal{O}_X\otimes \mathcal{O}_X}(\mathcal{O}, \mathcal{O}) &\simeq \iota_{\ast}Q\mathcal{H}om_{\mathcal{O}_X\otimes \mathcal{O}_X}(\mathcal{R},\mathcal{R}) \\&\simeq \iota_{\ast}Q\mathcal{H}om_{\mathcal{O}_X\otimes \mathcal{O}_X}(\iota^{\ast}\mathcal{P},\mathcal{R}) \\&\simeq Q\mathcal{H}om_{\iota_{\ast}(\mathcal{O}_X\otimes \mathcal{O}_X)}(\mathcal{P},\iota_{\ast}\mathcal{R}) \\&\simeq Q\mathcal{H}om_{\iota_{\ast}(\mathcal{O}_X\otimes \mathcal{O}_X)}(\mathcal{P},  \mathcal{P}), 
\end{align*}
where in the last step we used that $\iota_{\ast}\iota^{\ast}(\mathcal{P}) \cong \mathcal{P}$ and that hence $\iota_{\ast} \mathcal{R} \leftarrow \iota_{\ast}\iota^{\ast}\mathcal{P}$ is a weak equivalence between fibrant objects. 
\end{proof}
Since $\iota_{\ast}$ is symmetric monoidal, Proposition \ref{prop4} shows that we get an induced $\mathbb{E}_2$-algebra structure on $Q\mathcal{H}om_{\iota_{\ast}(\mathcal{O}_X\otimes \mathcal{O}_X)}(\mathcal{P},\mathcal{P})$. Further, if $U\subseteq X$ is affine, then 
\begin{align*}
    \mathbb{R}\Gamma_U(\mathfrak{Z}(\tilde{\mathcal{O}}_X)) \simeq \mathbb{R}\Gamma_U(Q\mathcal{H}om_{\iota_{\ast}(\mathcal{O}_X\otimes \mathcal{O}_X)}(\mathcal{P},\mathcal{P}))
\end{align*}
as $\mathbb{E}_2$-algebras. \displaypar
In order to prove Theorem \ref{thm5}, it thus suffices to show that 
\begin{align*}
    \mathbb{R}\Gamma_U(Q\mathcal{H}om_{\iota_{\ast}(\mathcal{O}_X\otimes \mathcal{O}_X)}(\mathcal{P},\mathcal{P})) \simeq \map_{A\otimes A}(P,P)
\end{align*}
for $U = \textnormal{Spec}(A)$ and $P \xtwoheadrightarrow{\simeq} A$ a cofibrant replacement. Since we will solely work with the affine open site from now on, we will denote the restriction of a presheaf $\mathcal{F}$ on $\textnormal{Open}(X)$ to a presheaf on $\textnormal{Aff}(X)$ simply by $\mathcal{F}$ for the remainder of this section.
\begin{notation}
Let ${\diag}_X$ be the site of affine opens on $X\times_{\field} X$ of the form $W\times_{\field} W$ for $W\subseteq X$ affine open. Of course, this site is isomorphic to the affine open site on $X$, but it better conceptualizes sheaves coming from $A$-bimodules.
\end{notation}
\begin{lemma}\label{lem3}
\begin{enumerate}
    \item The functors
    \begin{align*}
        \Delta_{\ast}: \affpsh{X} \rightarrow \dgpsh{\diag_X}&, \quad \mathcal{F} \mapsto \Bigl(W\times_{\field} W \mapsto \mathcal{F}\bigl(\Delta^{-1}(W\times_{\field} W)\bigr) = \mathcal{F}(W)\Bigr),\\
        \Delta^{-1}: \dgpsh{\diag_X} \rightarrow \affpsh{X}&, \quad \mathcal{G} \mapsto \Bigl(U \mapsto \colim_{\Delta(U)\subseteq W\times_{\field} W}\mathcal{G}(W\times_{\field} W) \cong \mathcal{G}(U\times_{\field} U)\Bigr)
    \end{align*}
    are isomorphisms of categories. 
    \item We have $\Delta_{\ast}(\mathcal{O}_X \otimes \mathcal{O}_X) \cong \mathcal{O}_{X\times_{\field} X}$ and $\Delta^{-1}(\mathcal{O}_{X\times_{\field} X}) \cong \mathcal{O}_X \otimes \mathcal{O}_X$. In particular, the above isomorphism yields an isomorphism
    \begin{align*}
        \lmod_{\mathcal{O}_{X\times_{\field} X}}(\dgpsh{\diag_X}) \cong \lmod_{\mathcal{O}_X\otimes \mathcal{O}_X}(\affpsh{X}).
    \end{align*}
    \item Both $\Delta_{\ast}$ and $\Delta^{-1}$ preserve all three classes of fibrations, cofibrations and weak equivalences in the presheaf categories as well as in the left module categories.
    \item For any commutative $\field$-algebra $A$, the adjunction $\tilde{(-)} \dashv \Gamma_{\textnormal{Spec}(A)}$ between complexes of $A$-modules and complexes of presheaves of $\mathcal{O}_{\textnormal{Spec}(A)}$-modules is a Quillen adjunction. In addition, $\tilde{(-)}$ preserves acyclic fibrations.
    \item The previous statement remains true if we consider $\tilde{(-)}$ as a functor from complexes of $A\otimes A$-modules to complexes of presheaves of $\mathcal{O}_{\textnormal{Spec}(A)}$-modules on the site $\diag_{\textnormal{Spec}(A)}$.
\end{enumerate}
\end{lemma}
\begin{proof}
Statement 1 is true by construction of the functors. \\
For 2., simply note that $\mathcal{O}_{X\times_k X}(W\times_{\field} W) \cong \mathcal{O}_X(W)\otimes_{\field} \mathcal{O}_X(W) \cong (\mathcal{O}_X\otimes \mathcal{O}_X)(W)$. \\
For 3., note that we get two adjoint equivalences $\Delta^{-1}\dashv \Delta_{\ast}$ and $\Delta_{\ast} \dashv \Delta^{-1}$. Clearly both $\Delta^{-1}$ and $\Delta_{\ast}$ preserve acyclic fibrations, and hence both also preserve cofibrations. Since $\diag_X$ is isomorphic as a site to $\textnormal{Aff}(X)$, the sheaf topos $\text{Sh}({\diag_X})_{\field}$ also has enough points and hence we can check weak equivalences at stalks. But since $\diag_X$ is isomorphic to $\text{Aff}(X)$, \cite[Chapter VII.5, Corollary 4]{MLM} shows that all points in the sheaf topos on $\diag_X$ are of the form $\Delta(x)$ for $x\in X$, and $(\Delta_{\ast}\mathcal{F})_{\Delta(x)} \cong \mathcal{F}_x$ and $(\Delta^{-1}\mathcal{G})_{x} \cong \mathcal{G}_x$, proving that both $\Delta^{-1}$ and $\Delta_{\ast}$ preserve weak equivalences. \\
For 4., first note that this is indeed an adjunction. To see this, let $M$ be a complex of $A$-modules and consider a map $M\rightarrow \mathcal{F}(\textnormal{Spec}(A))$. If $U= \textnormal{Spec}(B)$ is an affine open of $X = \textnormal{Spec}(A)$, then we get a restriction map $\mathcal{F}(X) \rightarrow \mathcal{F}(\textnormal{Spec}(B))$ which is a map of $A$-modules. But $\mathcal{F}(\textnormal{Spec}(B))$ is a $B$-module, and hence we get a map $\mathcal{F}(X) \otimes_A B \rightarrow \mathcal{F}(\textnormal{Spec}(B))$. We can hence construct a map
\begin{align*}
    \tilde{M}(\textnormal{Spec}(B)) \cong M\otimes_A B \rightarrow \mathcal{F}(X)\otimes_A B \rightarrow \mathcal{F}(\textnormal{Spec}(B)).
\end{align*}
of $B$-modules. Now note that $\tilde{(-)}$ sends quasi-isomorphisms to pointwise weak equivalences: If $M \rightarrow N$ is a quasi-isomorphism of complexes of $A$-modules and $U = \textnormal{Spec}(B) \subseteq \textnormal{Spec}(A)$ is an affine open, then in particular $B$ is flat over $A$ and therefore $- \otimes_A B$ preserves quasi-isomorphisms. Hence $\tilde(M)(U) = M\otimes_A B \rightarrow N\otimes_A B = \tilde{N}(U)$ is again a quasi-isomorphism. Further, we already know that the global sections functor preserves acyclic fibrations. This shows that the above adjunction is Quillen. Now if $M \rightarrow N$ is an acyclic fibration, then so is $M\otimes_A B \rightarrow N\otimes_A B$. This finishes the proof.\\
For 5., just note that everything in the proof of 4. still works.
\end{proof}
\begin{proof}[Proof of Theorem \ref{thm5}]
We work over the affine open site. Let $\mathcal{P} \xtwoheadrightarrow{\simeq} \mathcal{O}_X$ be a cofibrant resolution in $\mathcal{O}_X\otimes \mathcal{O}_X$ modules. The $\mathcal{O}_X\otimes\mathcal{O}_X$-module $\mathcal{P}$ is bifibrant and $\mathcal{H}om_{\mathcal{O}_X\otimes \mathcal{O}_X}(-,-)$ is a right Quillen bifunctor, implying that $Q\mathcal{H}om_{\mathcal{O}_X\otimes \mathcal{O}_X}(\mathcal{P},\mathcal{P})$ is again bifibrant in $\affpsh{X}$. We hence get a weak equivalence
\begin{align*}
    \mathbb{R}\Gamma_U(Q\mathcal{H}om_{\mathcal{O}_X\otimes \mathcal{O}_X}(\mathcal{P}, \mathcal{P})) \simeq \mathcal{H}om_{\mathcal{O}_X\otimes \mathcal{O}_X}(\mathcal{P}, \mathcal{P})(U) \cong \map_{\mathcal{O}_U\otimes \mathcal{O}_U}(\mathcal{P}|_U,\mathcal{P}|_U)
\end{align*}
We then have the following chain of weak equivalences 
\begin{align*}
    \map_{\mathcal{O}_U\otimes \mathcal{O}_U}(\mathcal{P}|_U,\mathcal{P}|_U)&\cong \map_{\mathcal{O}_{U\times_{\field} U}}((\Delta_U)_{\ast}(\mathcal{P}|_U), (\Delta_U)_{\ast}(\mathcal{P}|_U)) \\&\cong \map_{\mathcal{O}_{X\times_{\field} X}|_{U\times_{\field} U}}(\Delta_{\ast}(\mathcal{P})|_{U\times_{\field} U}, \Delta_{\ast}(\mathcal{P})|_{U\times_{\field}U}).
\end{align*} 
By the above lemma $\Delta_{\ast}\mathcal{P}$ is again bifibrant, and $(\Delta_{\ast}\mathcal{P})|_{U\times_{\field}U}$ is fibrant. We can therefore use \cite[Proposition 1.7.3]{Hin2} with a choice of cofibrant resolution $\mathcal{P}'\xtwoheadrightarrow{\simeq} (\Delta_{\ast}(\mathcal{P}))|_{U\times_{\field} U}$ to get weak equivalences
\begin{align*}
    \map_{\mathcal{O}_{X\times_{\field} X}|_{U\times_{\field} U}}(\Delta_{\ast}(\mathcal{P})|_{U\times_{\field} U}, \Delta_{\ast}(\mathcal{P})|_{U\times_{\field}U}) &\xrightarrow{\simeq} \map_{\mathcal{O}_{X\times_{\field} X}|_{U\times_{\field} U}}(\mathcal{P}', \Delta_{\ast}(\mathcal{P})|_{U\times_{\field}U})\\& \xleftarrow{\simeq} \map_{\mathcal{O}_{X\times_{\field} X}|_{U\times_{\field} U}}(\mathcal{P}', \mathcal{P}').
\end{align*}
Note that $(\Delta_{\ast} \mathcal{P})|_{U\times_{\field} U} \xtwoheadrightarrow{\simeq} (\Delta_U)_{\ast}\mathcal{O}_U \cong \tilde{A}$ is again a trivial fibration, and therefore $\mathcal{P}' \xtwoheadrightarrow{\simeq} \tilde{A}$ is a cofibrant resolution in $\mathcal{O}_{U\times_{\field} U}$-modules. Now let $P \xtwoheadrightarrow{\simeq} A$ be a cofibrant resolution of $A$ as an $A^e$-module. Then $\tilde{P}$ is a cofibrant $\mathcal{O}_{U\times_{\field} U}$-module, and we hence get a weak equivalence $\tilde{P} \xrightarrow{\simeq} \mathcal{P}'$ between bifibrant objects. Therefore,
\begin{align*}
    \map_{\mathcal{O}_{U\times_{\field}U}}(\mathcal{P}',\mathcal{P}') &\xrightarrow{\simeq} \map_{\mathcal{O}_{U\times_{\field}U}}(\tilde{P},\mathcal{P}') \\&\xleftarrow{\simeq} \map_{\mathcal{O}_{U\times_{\field}U}}(\tilde{P},\tilde{P}) \\&\cong \map_{A\otimes A} (P,P).
\end{align*}
This proves that $\mathbb{R}\Gamma_U(\mathfrak{Z}(\tilde{\mathcal{O}}_X)(\mathfrak{a})) \simeq \mathfrak{Z}(\tilde{A})(\mathfrak{a})$ as complexes of $\field$-modules. Recall that $\mathbb{R}\Gamma_U$ is lax symmetric monoidal, and therefore we get an induced evaluation map 
\begin{align*}
    Q\mathcal{H}om_{\mathcal{O}_X\otimes \mathcal{O}_X}(\mathcal{P}, \mathcal{P})(U) \otimes \mathcal{P}(U) \rightarrow R(Q\mathcal{H}om_{\mathcal{O}_X\otimes \mathcal{O}_X}(\mathcal{P},\mathcal{P})(U) \otimes \mathcal{P}(U)) \rightarrow \mathcal{P}(U)
\end{align*}
But $\mathcal{P}(U) \simeq \mathcal{O}_X(U) \cong A \simeq P$, so this is in fact equivalent to the evaluation map 
\begin{align*}
    \map_{A\otimes A}(P,P) \otimes P \rightarrow P
\end{align*}
of the center of $A$. This shows that the above equivalence is indeed an equivalence of $\mathbb{E}_2$-algebras.
\end{proof}
\subsection{Recovering polydifferential operators as the center of $\mathcal{O}_X$}
Now suppose that $X$ is quasi-compact, separated, of finite type and smooth over $\field$. Let $\dgsh{X}$ denote the category of sheaves of complexes of $\field$-modules on the site $\text{Open}(X)$. Recall that in this case the Hochschild cohomology of $X$ is given by the hypercohomology of the sheaf of polydifferential operators $\mathcal{D}_{\text{poly}}(X)\in \dgsh{X}$. If $U = \text{Spec}(A)\subseteq X$ is an affine open, then 
\begin{align*}
    \mathcal{D}_{\text{poly}}(X)_n(U) &= \{f\in \text{Hom}_{\field}(A^{\otimes n},A): f \text{ is a differential operator in each factor}\}\\&\subseteq \text{Hom}_{\field}(A^{\otimes n},A).
\end{align*}
The sheaf of polydifferential operators is quasi-coherent as an $\mathcal{O}_X$-module, and therefore fibrant in the local projective model structure on affine opens. We want to show that the sheaf of polydifferential operators is indeed a model of the center of $\mathcal{O}_X$.
\begin{theorem}\label{thm6}
Let $X$ be a smooth, quasi-compact, separated scheme of finite type over $\field$. We have an equivalence
\begin{align*}
    Q\mathcal{D}_{\textnormal{poly}}(X) \simeq \mathfrak{Z}(\tilde{\mathcal{O}}_X)(\mathfrak{a})
\end{align*}
    in the $\infty$-category $\lmod_{\tilde{\mathcal{O}}_X}(\textnormal{Sh}_{\infty}(X))$ of $\tilde{\mathcal{O}}_X$-modules. 
\end{theorem}
It suffices to show this equivalence for the sites of affine opens, since they yield an equivalent $\infty$-category. Since $\iota_{\ast}\mathfrak{Z}(\tilde{\mathcal{O}}_X)(\mathfrak{a})\simeq Q\mathcal{H}om_{\mathcal{O}_X\otimes \mathcal{O}_X}(\mathcal{P},\mathcal{P})$, it thus suffices to show 
\begin{align*}
    \iota_{\ast}\mathcal{D}_{\text{poly}}(X) \simeq \mathcal{H}om_{\mathcal{O}_X\otimes \mathcal{O}_X}(\mathcal{P},\mathcal{P})
\end{align*}
as presheaves of $\mathcal{O}_X$-modules. In the following we will work in the site of affine opens and we suppress the restriction of sheaves to affine opens.
\begin{notation}
    Let $\mathcal{O}$ be an associative algebra in dg sheaves. If $\mathcal{F},\mathcal{G}$ are sheaves of left $\mathcal{O}$-modules, recall that $\mathbb{R}\mathcal{H}om_{\mathcal{O}}(\mathcal{F},\mathcal{G}) = \mathcal{H}om_{\mathcal{O}}(\mathcal{F},\mathcal{J})$ for some K-injective resolution $\mathcal{J}$ of $\mathcal{G}$ in the category of sheaves of left $\mathcal{O}$-modules. Let $\Delta: X\rightarrow X\times_{\field} X$ be the diagonal. We have already used the adjunction $\Delta^{-1}\dashv \Delta_{\ast}$ induced by this in Lemma \ref{lem3} above, but we now want to consider the full site of affine opens on $X\times_{\field} X$ instead of the smaller site $\diag_X$, and we also consider sheaves instead of presheaves. In particular, the map $\Delta^{-1}: \affsh{X\times_{\field}X} \rightarrow \affsh{X}$ is now given by the presheaf version followed by sheafification. We then have $\Delta^{-1}\Delta_{\ast} \cong \text{id}$ since $X$ is separated. In particular, $\Delta^{-1}\mathcal{O}_{X\times_{\field} X} \cong \mathcal{O}_X \overset{a}{\otimes} \mathcal{O}_X$, where we denote by $\overset{a}{\otimes}$ the tensor product of sheaves.
\end{notation}
\begin{lemma}\label{lem6}
\begin{enumerate}
    \item Let $\mathcal{P} \xtwoheadrightarrow{\simeq}\mathcal{O}_X$ be a cofibrant resolution of presheaves of $\mathcal{O}_X\otimes \mathcal{O}_X$-modules. We have a local quasi-isomorphism of complexes of presheaves 
    \begin{align*}
        \mathcal{H}om_{\mathcal{O}_X\otimes \mathcal{O}_X}(\mathcal{P},\mathcal{P}) \simeq \mathbb{R}\mathcal{H}om_{\mathcal{O}_X\overset{a}{\otimes}\mathcal{O}_X}(\mathcal{O}_X,\mathcal{O}_X)
    \end{align*}
    \item If $\mathcal{F}$ and $\mathcal{G}$ are sheaves, then $\Delta_{\ast} \mathcal{H}om_{\mathcal{O}_X\overset{a}{\otimes}\mathcal{O}_X}(\Delta^{-1}\mathcal{F},\mathcal{G}) \cong \mathcal{H}om_{\mathcal{O}_{X\times_{\field} X}}(\mathcal{F},\Delta_{\ast}\mathcal{G})$.
    \item If $\mathcal{O}_X\xrightarrow{\simeq} \mathcal{I}$ is a K-injective resolution in sheaves of $\mathcal{O}_X\overset{a}{\otimes}\mathcal{O}_X$-modules, then $\Delta_{\ast}\mathcal{O}_X \rightarrow \Delta_{\ast}\mathcal{I}$ is a K-injective resolution in $\mathcal{O}_{X\times_{\field} X}$-modules.
\end{enumerate}
\end{lemma}
\begin{proof}
For 1., let $\alpha: \mathcal{O}_X\otimes \mathcal{O}_X \xrightarrow{\simeq} \mathcal{O}_X \overset{a}{\otimes} \mathcal{O}_X$ be the unit of the sheafification adjunction. This is a weak equivalence of dg algebras in presheaves, and therefore by \cite[Lemma 1.5.3]{Hin2} induces a Quillen equivalence
\begin{center}
    \begin{tikzcd}
           \lmod_{\mathcal{O}_X \otimes \mathcal{O}_X}(\affpsh{X})\arrow[r, shift left=1ex, "\alpha^{\ast}"{name=G}] & \lmod_{\mathcal{O}_X \overset{a}{\otimes} \mathcal{O}_X}(\affpsh{X})\arrow[l, shift left=.5ex, "\alpha_{\ast}"{name=F}]
            \arrow[phantom, from=F, to=G, , "\scriptscriptstyle\boldsymbol{\dashv}" rotate=-90]
        \end{tikzcd}.
\end{center}
We therefore get the following chain of weak equivalences
\begin{align*}
\mathcal{H}om_{\mathcal{O}_X\otimes \mathcal{O}_X}(\mathcal{P},\mathcal{P}) &\simeq \mathcal{H}om_{\mathcal{O}_X\otimes \mathcal{O}_X}(\mathcal{P},\alpha_{\ast}\mathcal{O}_X) \\&\cong \mathcal{H}om_{\mathcal{O}_X\overset{a}{\otimes}\mathcal{O}_X}(\alpha^{\ast}\mathcal{P},\mathcal{O}_X) \\&\simeq \mathcal{H}om_{\mathcal{O}_X\overset{a}{\otimes} \mathcal{O}_X}(\alpha^{\ast}\mathcal{P},\mathcal{I}) \\&\cong \mathcal{H}om_{\mathcal{O}_X\overset{a}{\otimes}\mathcal{O}_X}((\alpha^{\ast}\mathcal{P})^a,\mathcal{I})  \\&\simeq \mathcal{H}om_{\mathcal{O}_X\overset{a}{\otimes} \mathcal{O}_X}(\mathcal{O}_X,\mathcal{I}) \\&= \mathbb{R}\mathcal{H}om_{\mathcal{O}_X\overset{a}{\otimes} \mathcal{O}_X}(\mathcal{O}_X,\mathcal{O}_X).
\end{align*}
The 2. statement follows directly from the definition of the presheaf hom and the adjoint properties of $\Delta^{-1}$ and $\Delta_{\ast}$. If $U$ is an affine open in $X\times_{\field} X$, then
\begin{align*}
    \Delta_{\ast}\mathcal{H}om_{\mathcal{O}_X \overset{a}{\otimes} \mathcal{O}_X}(\Delta^{-1}\mathcal{F},\mathcal{G})(U) &= \map_{(\mathcal{O}_X \overset{a}{\otimes} \mathcal{O}_X)|_{\Delta^{-1}(U)}}(\Delta^{-1}\mathcal{F}|_{\Delta^{-1}(U)}, \mathcal{G}|_{\Delta^{-1}(U)}) \\&\cong \map_{(\mathcal{O}_X \overset{a}{\otimes} \mathcal{O}_X)|_{\Delta^{-1}(U)}}((\Delta|_{\Delta^{-1}(U)})^{-1}\mathcal{F}|_{U}, \mathcal{G}|_{\Delta^{-1}(U)}) \\&\cong \map_{\mathcal{O}_{(X\times_{\field}X)|_U}}(\mathcal{F}_U, \Delta_{\ast}\mathcal{G}|_U) = \mathcal{H}om_{\mathcal{O}_{X\times_{\field} X}}(\mathcal{F},\Delta_{\ast}\mathcal{G})(U).
\end{align*}
For the 3. statement, note first that $\mathcal{O}_X$ and $\mathcal{I}$ are both fibrant in the local projective model structure, and the presheaf version of the $\Delta^{-1}\dashv \Delta_{\ast}$ adjunction is Quillen for this model structure on the affine open sites, so $\Delta_{\ast} \mathcal{O}_X \rightarrow \Delta_{\ast}\mathcal{I}$ is again a local weak equivalence. Further $\Delta^{-1}$ is exact, and therefore preserves acyclic complexes. Therefore, if $\mathcal{S}$ is an acyclic $\mathcal{O}_{X\times_{\field} X}$-module, then 
\begin{align*}
    \map_{\mathcal{O}_{X\times_k X}}(\mathcal{S},\Delta_{\ast}\mathcal{I}) \cong \map_{\mathcal{O}_X\overset{a}{\otimes}\mathcal{O}_X}(\Delta^{-1}\mathcal{S},\mathcal{I})
\end{align*}
is acyclic, proving that $\Delta_{\ast}\mathcal{I}$ is K-injective.
\end{proof}
\begin{proof}[Proof of Theorem \ref{thm6}]
By \cite[Corollary 2.9]{Y} we have a local weak equivalence 
\begin{align*}
\Delta_{\ast} \mathcal{D}_{\textnormal{poly}}(X)\simeq \mathbb{R}\mathcal{H}om_{\mathcal{O}_{X\times_{\field} X}}(\Delta_{\ast}\mathcal{O}_X,\Delta_{\ast}\mathcal{O}_X).
\end{align*}
We then get the following chain of local weak equivalences
\begin{align*}
    \Delta_{\ast} \mathbb{R}\mathcal{H}om_{\mathcal{O}_X \overset{a}{\otimes} \mathcal{O}_X}(\mathcal{O}_X,\mathcal{O}_X) &= \Delta_{\ast}\mathcal{H}om_{\mathcal{O}_X \overset{a}{\otimes} \mathcal{O}_X}(\mathcal{O}_X,\mathcal{I})\\&\cong \Delta_{\ast}\mathcal{H}om_{\mathcal{O}_X \overset{a}{\otimes} \mathcal{O}_X}(\Delta^{-1}\Delta_{\ast}\mathcal{O}_X,\mathcal{I}) \\&\cong \mathcal{H}om_{\mathcal{O}_{X\times_{\field}X}}(\Delta_{\ast}\mathcal{O}_X,\Delta_{\ast}\mathcal{I}) \\&\simeq \mathbb{R}\mathcal{H}om_{\mathcal{O}_{X\times_{\field} X}}(\Delta_{\ast}\mathcal{O}_X,\Delta_{\ast}\mathcal{O}_X) \\&\simeq \Delta_{\ast}\mathcal{D}_{\textnormal{poly}}(X)
\end{align*}
where in the second to last step we used \ref{lem6}(3.). Now note that $\Delta^{-1}$ preserves local weak equivalences, and therefore 
\begin{align*}
\mathbb{R}\mathcal{H}om_{\mathcal{O}_X\overset{a}{\otimes} \mathcal{O}_X}(\mathcal{O}_X,\mathcal{O}_X) \simeq \mathcal{D}_{\textnormal{poly}}(X).
\end{align*}
Together with \ref{lem6}(1.) this finishes the proof.
\end{proof}

Let $\mathcal{B}(\mathcal{O}_X)$ denote the $\mathcal{O}_X\otimes \mathcal{O}_X$-module presheaf $U= \text{Spec}(A)\mapsto B(A)$, where $B(A)$ denotes the bar complex of $A$. We have an acyclic fibration $\mathcal{B}(\mathcal{O}_X) \xtwoheadrightarrow{\simeq} \mathcal{O}_X$ given by multiplication. Hence for a cofibrant resolution $\mathcal{P} \xtwoheadrightarrow{\simeq} \mathcal{O}_X$ of $\mathcal{O}_X$ as an $\mathcal{O}_X\otimes \mathcal{O}_X$-module, we get a lift $\mathcal{P} \rightarrow \mathcal{B}(\mathcal{O}_X)$. We get an evaluation map
\begin{align*}
    \mathcal{D}_{\text{poly}}(X) \otimes \mathcal{B}(\mathcal{O}_X) \rightarrow \mathcal{O}_X
\end{align*}
coming from the fact that $\mathcal{D}_{\text{poly}}(X)$ is affine locally a subcomplex of the Hochschild complex. This lifts to a map
\begin{center}
    \begin{tikzcd}
Q\mathcal{D}_{\text{poly}}(X)\otimes \mathcal{P} \arrow[d] \arrow[r, dashed]       & \mathcal{P} \arrow[d, "\simeq", two heads] \\
\mathcal{D}_{\text{poly}}(X)\otimes \mathcal{B}(\mathcal{O}_X) \arrow[r] & \mathcal{O}_X                             
\end{tikzcd}
\end{center}
and it is clear from the proof of Theorem \ref{thm6} that this map corresponds to the evaluation map
\begin{align*}
    Q\mathcal{H}om_{\mathcal{O}_X\otimes \mathcal{O}_X}(\mathcal{P},\mathcal{P}) \otimes \mathcal{P} \rightarrow \mathcal{P}.
\end{align*}
It therefore makes $Q\mathcal{D}_{\text{poly}}(X)$ into a center of $\tilde{\mathcal{O}}_X$. In particular, this equips the sheaf of polydifferential operators with a new $\mathbb{E}_2$-algebra structure in the $\infty$-category of dg sheaves on $X$. 
\subsection{Comparison to the classical homotopy Gerstenhaber algebra structure on polydifferential operators}
For a separated quasi-compact smooth finite type scheme $X$ over $\field$, Tamarkin's proof of Deligne's Conjecture equips $\mathcal{D}_{\text{poly}}(X)$ with a $\ger_{\infty}$-algebra structure coming from the canonical $\braces$-algebra structure, and thus a Gerstenhaber algebra structure in the $\field$-linear derived 1-category on $X$. On the other hand, exhibiting $\mathcal{D}_{\text{poly}}(X)$ as a center of $\tilde{\mathcal{O}}_X$ equips it with an $\mathbb{E}_2$-algebra structure in $\text{Sh}_{\infty}(X)$, and therefore another Gerstenhaber algebra structure in the derived 1-category, which is just the homotopy category of $\text{Sh}_{\infty}(X)$. We want to compare these two Gerstenhaber algebra structures. \displaypar 
To this end, note that by Corollary \ref{cor7}, the multiplication in the center Gerstenhaber algebra structure is given equivalently by the convolution product or the composition product, and the bracket is induced by any filling of the appropriate square in the action category. \displaypar
By \cite{Y}, we have an isomorphism of sheaves
\begin{align*}
    \mathcal{D}_{\textnormal{poly}}(X) \cong \mathcal{H}\textnormal{om}^{\textnormal{cont}}_{\mathcal{O}_{X\times_{\field} X}}(\widehat{\mathcal{B}}(X),\mathcal{O}_X)
\end{align*}
where $\widehat{\mathcal{B}}_n(X) = \mathcal{O}_{\mathfrak{X}^{n+2}}$ is the complete bar complex with $\mathfrak{X}^n$ the formal completion of $X^n$ along the $n$-fold diagonal. From this presentation it is easy to compute the convolution and composition product. Note that for $U = \spec(A)\subseteq X$, we have
\begin{align*}
    \Gamma_U(\mathcal{H}\textnormal{om}^{\textnormal{cont}}_{\mathcal{O}_{X\times_{\field} X}}(\widehat{\mathcal{B}}(X),\mathcal{O}_X)) \cong \map_{A\otimes A}^{\textnormal{cont}}(\widehat{B}(A),A)
\end{align*}
for $\widehat{B}_n(A)$ the adic completion of $B_n(A)$ at the kernel of the multiplication map $B_n(A) \rightarrow A$. In particular, the flat resolution
\begin{align*}
    \widehat{\mathcal{B}}(X) \rightarrow \mathcal{O}_X
\end{align*}
admits a section $s: \mathcal{O}_X \rightarrow \widehat{\mathcal{B}}_0(X)$ that is glued together from the sections of the resolutions $B(A) \rightarrow A$. We can hence build a diagonal 
\begin{align*}
    \Delta: \widehat{\mathcal{B}}(X) \rightarrow \mathcal{O}_X \xrightarrow{\cong} \mathcal{O}_X \overset{a}{\otimes}_{\mathcal{O}_X} \mathcal{O}_X \xrightarrow{s\otimes s}  \widehat{\mathcal{B}}(X)\overset{a}{\otimes}_{\mathcal{O}_X}  \widehat{\mathcal{B}}(X).
\end{align*}
\begin{lemma}\label{lem4}
    The convolution product on $\mathcal{H}\textnormal{om}^{\textnormal{cont}}_{\mathcal{O}_{X\times_{\field} X}}(\widehat{\mathcal{B}}(X),\mathcal{O}_X)$ is homotopic to the cup product on $\mathcal{D}_{\textnormal{poly}}(X)$, which locally agrees with the classical cup product on Hochschild cochains. 
\end{lemma}
\begin{proof}
    The diagonal is zero on $\widehat{\mathcal{B}}_n(X)$ for $n>0$, and for $n=0$ it is given locally by the formula
    \begin{align*}
        a_0 \otimes a_1 \mapsto (1\otimes 1) \otimes_A (1 \otimes a_0a_1).
    \end{align*}
    The diagonal on $B(A)$ coming from its universal property is given for $n=0$ by 
    \begin{align*}
        a_0 \otimes a_1 \mapsto (a_0 \otimes 1) \otimes_A (1\otimes a_1),
    \end{align*}
    and for $n>0$ by 
    \begin{align*}
        a_0 \otimes a_1 \otimes \dots \otimes a_n \otimes a_{n+1} \mapsto \sum_{i = 0}^n (a_0 \otimes a_1 \otimes \dots \otimes a_i \otimes 1) \otimes_A (1 \otimes a_{i+1} \otimes \dots \otimes a_n \otimes a_{n+1}).
    \end{align*}
    Locally on $B(A)$, a homotopy between these two maps is given by 
    \begin{align*}
        H: B(A) &\rightarrow (B(A) \otimes_A B(A))[1]\\
        a_0 \otimes a_1 \otimes \dots \otimes a_n \otimes a_{n+1} &\mapsto \sum_{i=0}^{n+1} (1\otimes a_0 \otimes \dots \otimes a_{i-1} \otimes 1) \otimes_A (1\otimes a_i \otimes \dots \otimes a_{n+1}).
    \end{align*}
    Recall that the convolution product of two continuous maps $f$ and $g$ is given by 
    \begin{align*}
        \widehat{\mathcal{B}}(X) \xrightarrow{\Delta}  \widehat{\mathcal{B}}(X) \overset{a}{\otimes}_{\mathcal{O}} \widehat{\mathcal{B}}(X) \xrightarrow{f\otimes g} \mathcal{O}_X\overset{a}{\otimes}_{\mathcal{O}_X} \mathcal{O}_X \rightarrow \mathcal{O}_X.
    \end{align*}
    Locally, it suffices to consider the restriction to $B(A)$. We can then use the above homotopy $H$ to obtain a homotopy between the cup product and the above formula for the convolution product with our new diagonal. Inspecting the formula for $H$ we see that these glue together to yield a global homotopy between the global convolution product and the cup product.
\end{proof}
We already know that the local circle products coming from the $\braces$-algebra structure glue together to give a homotopy for the square
\begin{center}
    \begin{tikzcd}
\mathcal{D}_{\textnormal{poly}}(X)^{\overset{a}{\otimes} 4} \arrow[d, "(\smile \otimes \smile) \circ(\textnormal{id}\otimes \tau \otimes \textnormal{id})"'] \arrow[r, "\smile \otimes \smile"] & \mathcal{D}_{\textnormal{poly}}(X)^{\overset{a}{\otimes} 2} \arrow[d, "\smile"] \\
\mathcal{D}_{\textnormal{poly}}(X)^{\overset{a}{\otimes} 2} \arrow[r, "\smile"']                                                                                                                & \mathcal{D}_{\textnormal{poly}}(X)                                
\end{tikzcd}
\end{center}
analogous to the affine case. This implies the following:
\begin{corollary}\label{thm9}
    The Gerstenhaber structure on $\mathcal{D}_{\textnormal{poly}}(X)$ in the $\field$-linear derived 1-category coming from the center is isomorphic to the signed classical one coming from the $\braces$-algebra structure.
\end{corollary}
\begin{proof}
   By the above Lemma \ref{lem4}, the center product is given by the usual signed cup product in the derived 1-category. Again using Corollary \ref{cor7} and the fact that the global circle product induces a homotopy
   \begin{align*}
       h: \mathcal{D}_{\text{poly}}(X)^{\overset{a}{\otimes} 4} &\rightarrow \mathcal{D}_{\text{poly}}(X)[1]\\
       \mathcal{D}_{\text{poly}}(\spec(A))^{\otimes 4} \ni f_1 \otimes g_1 \otimes f_2 \otimes g_2 &\mapsto (-1)^{|f_1|+|f_2|+|f_2||g_1|-1} f_1 \smile f_2\{g_1\} \smile g_2
   \end{align*}
   in the above square, we see that the bracket determined by Corollary \ref{cor5} of the center Gerstenhaber algebra structure is equal to the classical Gerstenhaber bracket on polydifferential operators.
\end{proof}

\newpage
\setcounter{theorem}{0}
\chapter{Formal moduli problems and deformations}\label{formal_moduli_problems_and_deformations}
As discussed in Section \ref{Kontsevich_formality_morphism}, the complex of polyvector fields and the Hochschild cochain complex control deformations of Poisson algebras and associative algebras respectively. This yields a universal definition of Lie algebra structure on the (shifted) Hochschild cochain complex of an associative $\field$-algebra $A$, namely as its deformation complex over the associative operad. In this chapter, we explore the connection between the center of an operadic algebra and its deformation theory. \displaypar
We start by recalling some general theory on formal moduli problems and the deformation theory of operadic algebras. In addition to the theory of classical formal moduli problems developed by Hinich, Pridham, Lurie, D. Gaitsgory and N. Rozenblyum, there is a theory of operadic formal moduli problems, which use noncommutative test spaces. This was developed for the $\mathbb{E}_n$-operads by Lurie in \cite{DAGX}, and for Koszul operads by Calaque, Campos, and Nuiten in \cite{CCN}. \displaypar
In \cite{DAGX}, Lurie proves that the deformation problem of the right module category over an $\mathbb{E}_1$-algebra $A$ lifts to an $\mathbb{E}_2$-formal moduli problem, which is controlled by the center of $A$. This recovers a result by J. Francis, who showed that the non-unital $\mathbb{E}_2$-algebra structure on fiber of the map $\mathfrak{Z}(A) \rightarrow A$ lifts the Lie algebra structure on the operadic tangent complex of $A$. \displaypar
The goal of this chapter is to prove an analogous result for algebras over a Koszul operad $\mathscr{P}$. This is done in Section \ref{the_formal_automorphism_group}. To this end, we define an internal formal automorphism group of a $\mathscr{P}$-algebra $A$, and we show that it lifts to an operadic formal moduli problem over the Koszul dual operad $\mathscr{P}^!$. The main result of this chapter is Theorem \ref{thm21}, which shows that this operadic formal moduli problem is indeed controlled by the center of $A$. In this context, we interpret the additional $\mathbb{E}_1$-algebra structure on the center as corresponding to the composition operation in the internal infinitesimal automorphism group. \displaypar
We work with a different model of homotopy operads, namely \textbf{dg operads}\footnote{Note that we assume no constraints on 0 and 1 ary operations for general dg operads.}. We follow the conventions of \cite{LV}. For a dg operad $\mathscr{P}$ and a symmetric monoidal dg category $C$, we denote by $\alg_{\mathscr{P}}(C)$ the 1-category of $\mathscr{P}$-algebras in $C$. Similarly, if $\mathcal{C}$ is a symmetric monoidal dg $\infty$-category\footnote{A dg $\infty$-category is a presentable $\field$-linear $\infty$-category.}, we denote by $\alg_{\mathscr{P}}(\mathcal{C})$ the $\infty$-category of $\mathscr{P}$-algebras in $\mathcal{C}$. Similar to the Rectification Theorem \ref{thm3}, there also exists a rectification theorem for dg operads. If $C$ is a symmetric monoidal dg model category, we call a dg operad $\mathscr{P}$ admissible in $\mathcal{C}$ if the transfer model structure on $\alg_{\mathscr{P}}(C)$ exists. In this case, under certain technical conditions (see \cite[Theorem 4.10]{Hau}), there is an equivalence of $\infty$-categories
\begin{align*}
    \alg_{\mathscr{P}}(C)^c[W^{-1}] \xrightarrow{\simeq} \alg_{\mathscr{P}}(C[W^{-1}]) \simeq \alg_{\mathscr{P}}(N_{\text{dg}}(C^{\circ})).
\end{align*}
This Rectification Theorem holds in particular for $C= \ch(\field)$, i.e. for any dg operad $\mathscr{P}$, we have an equivalence of $\infty$-categories $\alg_{\mathscr{P}}(\ch(\field))[W^{-1}] \simeq \alg_{\mathscr{P}}(\dk)$. For more information about dg $\infty$-operads, see \cite{Hau}, \cite[Chapter 6, Section 1]{GR2}, and \cite{CHa}.\displaypar
For a dg cooperad $\mathscr{C}$, and a symmetric monoidal dg category $C$, denote by $\coalg_{\mathscr{C}}(C)$ the 1-category of $\mathscr{C}$-coalgebras, and by $\ccoalg_{\mathscr{C}}(C)$ the 1-category of conilpotent $\mathscr{C}$-coalgebras in $C$. For a symmetric monoidal dg $\infty$-category $\mathcal{C}$, denote by $\coalg_{\mathscr{C}}(\mathcal{C})$ the $\infty$-category of $\mathscr{C}$-coalgebras in $\mathcal{C}$. \displaypar
We denote by $B\mathscr{P}$ the Bar construction of an augmented dg operad $\mathscr{P}$, and by $\Omega \mathscr{C}$ the Cobar construction of a coaugmented dg cooperad $\mathscr{C}$. If $\alpha: \Omega\mathscr{C} \rightarrow \mathscr{P}$ is a twisting morphism, we denote the corresponding Bar functor on $\mathscr{P}$-algebras by $\obar_{\alpha}$, and the Cobar functor on (conilpotent) $\mathscr{C}$-coalgebras by $\cobar_{\alpha}$. If $\alpha$ is the canonical twisting morphism of a Koszul operad, we denote these functors simply by $\obar$ and $\cobar$ respectively. \displaypar
Recall that by our convention, Koszul operads are assumed to be binary quadratic, and hence satisfy $\mathscr{P}(0) = 0$ and $\mathscr{P}(1) = \field$. 
\section{Deformation problems from operadic algebras}\label{deformation_problems_from_operadic_algebras}
We start by recalling the deformation theory of algebras over a dg operad. We loosely follow the exposition in \cite{CG}. 
\begin{definition}
A \textbf{dg Artin $\field$-algebra} is an augmented commutative algebra $R\rightarrow \field$ in $\ch(\field)$ such that
\begin{compactitem}
    \item For all $n$, $H^n(R)$ is finite-dimensional.
    \item For $n>0$ and $n\ll 0$, $H^n(R) = 0$.
    \item $H^0(R)$ is an Artin local $\field$-algebra.
\end{compactitem}
We denote the $\infty$-category of dg Artin $\field$-algebras by $\alg_{\dgcomm}^{\text{Art}}$. We call a dg Artin $\field$-algebra \textbf{strict} if its augmentation ideal is nilpotent and finite-dimensional on the complex level. 
\end{definition}
Let $\mathscr{P}$ be a dg operad, and let $\phi: \mathscr{P} \rightarrow \e_A$ be a $\mathscr{P}$-algebra in $\ch(\field)$. A $\mathscr{P}$\textbf{-algebra deformation} of $\phi$ over a dg Artin $\field$-algebra $R$ is a $\mathscr{P}$-algebra structure on $R \otimes A \in \text{Mod}_R(\ch(\field))$ together with an equivalence of $\mathscr{P}$-algebras $\field \otimes_R (R\otimes A) \simeq A$ in $\ch(\field)$. This yields a functor, the \textbf{deformation functor} associated to $\phi$, 
\begin{align*}
    \widetilde{\defm}_{\phi}&: \alg_{\dgcomm}^{\text{Art}} \rightarrow \mathscr{S} \\
    R &\mapsto \map_{\text{dg}\mathcal{O}p}(\mathscr{P}, \e_A \otimes R) \times_{\map_{\text{dg}\mathcal{O}p}(\mathscr{P},\e_A)} \{\phi\}
\end{align*}
from the $\infty$-category of dg Artin $\field$-algebras to the $\infty$-category of spaces. Such a deformation functor is the prototypical example of a \textbf{formal moduli problem}. \displaypar
The general slogan, which was recognized by Deligne and Drinfeld in the 80s, is that in characteristic 0, such a formal moduli problem should be controlled by a dg Lie algebra. One way to understand this is to view a deformation functor as the formal neighborhood of a point in a moduli space. Then the corresponding dg Lie algebra is the (shifted) tangent space of that moduli space at this point. In the above example, this is the moduli space ${\mathcal{M}}^A_{\mathscr{P}}$ of $\mathscr{P}$-algebra structures on $A$, and we consider the formal neighborhood of the point $\phi$.\displaypar
If $\mathscr{P}$ is of the form $\Omega \mathscr{C}$ for a coaugmented dg cooperad $\mathcal{C}$, one can easily describe the dg Lie algebra corresponding to this deformation functor $\widetilde{\defm}_{\phi}$. To this end, consider the \textbf{convolution Lie algebra}
\begin{align*}
    \mathfrak{g}_{\mathscr{P}, A} := \text{Conv}(\bar{\mathscr{C}}, \e_A) \cong \prod_{n\geq 1} \map_{\Sigma_n}(\bar{\mathscr{C}}(n),\e_A(n))
\end{align*}
with $\bar{\mathscr{C}}$ the cokernel of the coaugmentation of $\mathscr{C}$. This is equipped with the Lie bracket coming from the convolution pre-Lie structure, see \cite[Proposition 6.4.3]{LV}. Then the theory of twisting morphisms shows (see for example \cite[Proposition 10.1.1]{LV}) that a $\mathscr{P}$-algebra structure on $A$ is equivalently given by a Maurer-Cartan element in $\mathfrak{g}_{\mathscr{P},A}$. 
\begin{definition}
Let $A\in \ch(\field)$, $\mathscr{P} = \Omega \mathscr{C}$, and $\phi \in \text{MC}(\mathfrak{g}_{\mathscr{P},A})$. The \textbf{deformation complex} of $\phi$ is the dg Lie algebra $\mathfrak{g}_{\mathscr{P},A}^{\phi}$ obtained by twisting the differential of the convolution Lie algebra by the Maurer-Cartan element $\phi$.
\end{definition}
The deformation complex of $\phi$ is the dg Lie algebra corresponding to the formal moduli problem of deforming $\phi$ in ${\mathcal{M}}^A_{\mathscr{P}}$ in the following sense.
\begin{proposition}[Proposition 3.9 \cite{CCN}]
Suppose that $R$ is a strict dg Artin $\field$-algebra with augmentation ideal $\mathfrak{m}_R$. Then 
\begin{align*}
    \widetilde{\defm}_{\phi}(R) \simeq \mathbf{MC}_{\bullet}(\mathfrak{g}_{\mathscr{P},A}^{\phi} \otimes \mathfrak{m}_R),
\end{align*}
where $\mathbf{MC}_{\bullet}$ denotes the Kan complex of Maurer-Cartan elements as defined in \cite[Definition 8.1.1]{Hin3}.
\end{proposition}
\begin{example}\label{ex2}
Consider the non-unital associative dg operad $\nuassoc$. This is a Koszul operad with Koszul dual cooperad given by the shifted non-counital coassociative dg cooperad $\nccoassoc\{1\}$. Since $\nccoassoc \{1\}(1) \simeq \field$, it is canonically coaugmented, and taking this as the cooperad $\mathcal{C}$ in the above construction yields the dg operad for non-unital homotopy associative algebras $\mathscr{P} = \Omega \nccoassoc \{1\} = \mathbb{A}_{\infty}$. Then
\begin{align*}
    \mathfrak{g}_{\mathbb{A}_{\infty},A} &\cong \prod_{n\geq 1} \map_{\Sigma_n}(\overline{\nccoassoc \{1\}}(n), \e_A(n)) \\&\cong \prod_{n\geq 2} \map_{\Sigma_n}(\field[\Sigma_n][1-n], \e_A(n)) \\&\cong \prod_{n\geq 2} \map_{\field}(A^{\otimes n}, A)[n-1].
\end{align*}
If $\phi: \nuassoc \rightarrow \e_A$ is an associative algebra structure, and $A$ is concentrated in degree 0, then as a graded $\field$-vector space this is a shifted and truncated version of the Hochschild cochain complex of $A$. The differential of the twisted complex $\mathfrak{g}_{\mathbb{A}_{\infty},A}^{\phi}$ is then precisely the differential on the (truncated, shifted) Hochschild cochain complex of $A$. The fact that the Hochschild cochain complex controls deformations of an associative algebra was recognized by Gerstenhaber in \cite{G2}, where he first developed algebraic deformation theory.
\end{example}
There is a second, closely related, deformation functor associated to a $\mathscr{P}$-algebra structure $\phi$ on $A$. Note that in the definition of a $\mathscr{P}$-algebra deformation of $\phi$, we fixed the underlying complex of the $\mathscr{P}$-algebra over the dg Artin $\field$-algebra $R$ to be $R \otimes A$. If we also allow deformations of this underlying complex, we arrive at \textbf{geometric} $\mathscr{P}$\textbf{-algebra deformations} of $\phi$. This yields a new deformation functor 
\begin{align*}
    {\defm}_{\phi}&: \alg_{\dgcomm}^{\text{Art}} \rightarrow \mathscr{S} \\
    R &\mapsto \alg_{\mathscr{P}}(\text{Mod}_R(\dk)) \times_{\alg_{\mathscr{P}}(\dk)} \{\phi\}.
\end{align*}
This is the formal moduli problem corresponding to the formal neighborhood of the moduli space $\mathcal{M}_{\mathscr{P}}$ of $\mathscr{P}$-algebras. If $\mathscr{P}$ and $A$ are connective, Hinich showed in \cite{Hin4} that the corresponding dg Lie algebra of this formal moduli problem is the operadic tangent complex of $\phi$. 
\begin{definition}
Let $\mathscr{P}$ be a dg operad and $\phi: \mathscr{P} \rightarrow \e_A$ a $\mathscr{P}$-algebra. The \textbf{operadic tangent complex} of $\phi$ is the dg Lie algebra of derived $\mathscr{P}$-derivations 
\begin{align*}
    \mathbb{T}_{\phi}^{\mathscr{P}} := \mathbb{R}\text{Der}^{\mathscr{P}}(A,A).
\end{align*}
\end{definition}
If $\mathscr{P}$ is Koszul, and $\mathscr{C} = \mathscr{P}^{\text{!`}}$ its Koszul dual cooperad, we have a cofibrant replacement of a $\mathscr{P}$-algebra $\phi$ on $A$ given by the counit of the operadic Bar-Cobar adjunction with respect to the canonical twisting morphism
\begin{align*}
    \epsilon: \cobar(\obar A) \xrightarrow{\simeq} A.
\end{align*}
In this case, we can compute the operadic tangent complex as
\begin{align*}
    \mathbb{T}_{\phi}^{\mathscr{P}} &\simeq \text{Der}^{\mathscr{P}}(\cobar(\obar A), \cobar(\obar A)) \\&\cong \map_{\field}(\obar A, \cobar(\obar A)) \\&\xrightarrow[\simeq]{\epsilon_{\ast}} \map_{\field}(\obar A,A) \\&\cong \left(\prod_{n\geq 1} \map_{\field}(\mathscr{P}^{\text{!`}}(A), A)\right)^{\phi} \\& \cong \text{Conv}(\mathscr{P}^{\text{!`}}, \e_A)^{\phi}.
\end{align*}
We readily see that this agrees with the deformation complex $\mathfrak{g}_{\mathscr{P}_{\infty},A}^{\phi}$ up to the arity 1 component of $\mathscr{P}^{\text{!`}}$. This is precisely the component of the dg Lie algebra controlling deformations of the underlying complex. In particular 
\begin{proposition}[Proposition 4.18, \cite{CG}]
    There is a homotopy fiber sequence of dg Lie algebras
    \begin{align*}
        \mathfrak{g}_{\mathscr{P}_{\infty},A}^{\phi} \rightarrow \mathbb{T}_{\phi}^{\mathscr{P}} \rightarrow \map_{\field}(A,A).
    \end{align*}
\end{proposition}
As we have seen in Example \ref{ex2}, the operadic tangent complex of an associative $\field$-algebra $A$ recovers $C^{\ast}(A,A)[1]$ up to the degree -1 component. One can easily show that this is compatible with the dg Lie algebra structures on these complexes. By Corollary \ref{cor9}, the dg Lie structure on the Hochschild complex can, at least up to non-canonical quasi-isomorphism, be obtained from the $C_{\ast}(\mathbb{E}_2)$-algebra structure on the center of $A$.  
\section{Formal moduli problems and operadic formal moduli problems}\label{formal_moduli_problems_and_operadic_formal_moduli_problems}
Given a derived scheme $X$ over $\field$, we have an associated functor of points $\alg_{\dgcomm}(\dk^{\text{cn}}) \rightarrow \mathscr{S}$ which is the restriction of the functor $\textbf{Sch}^{\text{op}}_{/\field} \rightarrow \mathscr{S}$ represented by $X$. On the other hand, given a functor $F: \alg_{\dgcomm}(\dk^{\text{cn}}) \rightarrow \mathscr{S}$, which we think of as sending $R$ to a space of families of objects parametrized by $\spec R$, we can ask whether there exists a geometric object $\mathcal{Y}$ representing it, possibly in some larger category. Finding such a representing object is commonly called a \textbf{moduli problem}, and a solution, meaning a representing object, is commonly called a moduli space for the functor. Since moduli spaces are often quite complicated, one might try to examine the formal neighborhood of a point of the moduli space instead. To this end, given a point $\eta \in F(\field)$, one studies the completion of $F$ at $\eta$, which is given by the \textbf{formal moduli problem} $\hat{F}_{\eta}$ sending a strict dg Artin $\field$-algebra $R$ to
\begin{align*}
    \hat{F}_{\eta}(R) := F(R) \times_{F(R/\mathfrak{m}_R)} \{\eta\}. 
\end{align*}
If the moduli space in question is given by algebraic structures of a certain type, then this should be interpreted as studying infinitesimal deformations of a fixed algebraic object of that type. For example, if $\mathscr{P}$ is a connective dg operad, and $A\in \ch(\field)^{\leq 0}$, then the completion of a $\field$-point $A\in \alg_{\mathscr{P}}(\dk)$ in the moduli space $\mathcal{M}_{\mathscr{P}}: R \mapsto \alg_{\mathscr{P}}(\lmod_R(\dk))$ of $\mathscr{P}$-algebra structures is precisely given by the geometric deformation functor $\defm_A$ defined in the previous section.\displaypar 
In general, (the solution to) a formal moduli problem is a certain kind of prestack over $\field$, whose geometry reflects the fact that it is supposed to model the formal completion of a sufficiently geometric stack at a point. This is made precise by Gaitsgory and Rozenblyum in \cite{GR2}, where they make the following definition for formal moduli problems parametrized by a base prestack $\mathcal{X}$.
\begin{definition}[Chapter 5, Section 1.1 \cite{GR2}]
    Let $\mathcal{X}$ be a prestack locally almost of finite type\footnote{See \cite[Chapter 2, Section 1.7]{GR1}.}. The $\infty$-category of \textbf{formal moduli problems over} $\mathcal{X}$ is the full subcategory
    \begin{align*}
        \mathbf{FMP}_{/\mathcal{X}} \subseteq (\mathbf{PreStk}_{\textnormal{laft}})_{/\mathcal{X}}
    \end{align*}
    spanned by morphisms of prestacks $\mathcal{Y} \rightarrow \mathcal{X}$ which are \textbf{inf-schematic}\footnote{A prestack locally almost of finite type is an inf-scheme if it admits deformation theory and its reduced prestack is a quasi-compact scheme. A morphism is inf-schematic if its base change by an affine yields an inf-scheme. See \cite[Chapter 2, Definition 3.1.5]{GR1}.} and induce an isomorphism on reduced prestacks $\mathcal{Y}^{\text{red}} \xrightarrow{\simeq} \mathcal{X}^{\text{red}}$.
\end{definition}
In particular, if $\mathcal{X} = \spec(\field)$ is the point, then 
\begin{align*}
    \mathbf{FMP}_{/\spec(\field)} \simeq (\mathbf{InfSch}_{\text{laft}})_{\text{nil-iso to}  \spec(\field)} \simeq (\mathbf{PreStk}_{\text{laft-def}})_{\text{nil-iso to} \spec(\field)}
\end{align*}
is equivalent to the $\infty$-category of locally almost of finite type prestacks which admit deformation theory, and whose reduced prestack is the point. Denote the category of formal moduli problems over the point simply by $\mathbf{FMP}$.
\begin{definition}
A \textbf{formal group\footnote{Note that these are different objects than what is commonly called a formal group in classical algebraic geometry.} over} $\mathcal{X}$ is a group object in the cartesian symmetric monoidal $\infty$-category $\mathbf{FMP}_{/\mathcal{X}}$. 
\end{definition}
The $\infty$-category of formal moduli problems over $\mathcal{X}$ enjoys a special stability property. Namely let $(\mathbf{FMP}_{/\mathcal{X}})_{\ast}:= \mathbf{FMP}_{\mathcal{X}//\mathcal{X}}$ denote the category of \textbf{pointed formal moduli problems over} $\mathcal{X}$. There is a loop functor 
\begin{align*}
    \Omega_{\mathcal{X}}: (\mathbf{FMP}_{/\mathcal{X}})_{\ast} &\rightarrow \mathbf{Grp}(\mathbf{FMP}_{/\mathcal{X}}), \\ \mathcal{Y} &\mapsto  \mathcal{X} \times_{\mathcal{Y}} \mathcal{X}.
\end{align*}
\begin{theorem}[Chapter 5, Theorem 1.6.4 \cite{GR2}]\label{thm20}
The functor $\Omega_{\mathcal{X}}$ is an equivalence with inverse denoted by $B_{\mathcal{X}}$.
\end{theorem}
The general slogan that any formal moduli problem in characteristic 0 should correspond to a dg Lie algebra has been rigorously formulated and proven by Pridham and Lurie, and is now called the \textbf{Lurie-Pridham correspondence}. To be precise, we have the following
\begin{theorem}[\cite{Pr}, \cite{DAGX}, \cite{GR2}]\label{thm15}
    Let $X\in {}^{<\infty}\mathbf{Sch}^{\textnormal{aff}}_{\textnormal{ft}}$ be an eventually coconnective finite type affine scheme over $\field$. There is an equivalence of $\infty$-categories
    \begin{align*}
        (\mathbf{FMP}_{/{X}})_{\ast} \simeq \alg_{\lie}(\textnormal{IndCoh}({X})).
    \end{align*}
    In particular, for ${X} = \spec(\field)$, we get an equivalence $\mathbf{FMP}\simeq \alg_{\lie}(\dk)$.
\end{theorem}
By the stability property of formal moduli problems \ref{thm20}, this can be interpreted as follows: Use the equivalence $\Omega_{{X}}$ to identify pointed formal moduli problems over ${X}$ with formal groups over ${X}$. The Lie algebra corresponding to a pointed formal moduli problem $\mathcal{Y}$ is then given by the Lie algebra of the formal group $\Omega_{{X}}(\mathcal{Y})$. In particular, to prove the theorem it suffices to construct mutually inverse functors
\begin{align*}
    \text{Lie}_{{X}}: \mathbf{Grp}(\mathbf{FMP}_{/X}) \rightarrow \alg_{\lie}(\text{IndCoh}(X))\\
    \text{exp}_X: \alg_{\lie}(\text{IndCoh}(X)) \rightarrow \mathbf{Grp}(\mathbf{FMP}_{/X})
\end{align*}
where $\text{Lie}_X$ assigns to a formal group its Lie algebra object, and $\text{exp}_X$ assigns to a Lie algebra its exponential group. This has been done in \cite{GR2}. This shows that answering the question of why formal moduli problems correspond to Lie algebras is equivalent to answering the question of why the tangent space at the identity of a Lie group has the structure of a Lie algebra.\displaypar
There are essentially two answers to this question; one geometric and one algebraic. The geometric one goes as follows: A formal space, modeled here as a formal moduli problem, is determined\footnote{For general formal moduli problems this is not quite true. See \cite[Chapter 7, Section 2.2]{GR2}.} by its (cocommutative) coalgebra of point distributions. A group structure on a formal space makes the coalgebra of distributions into a cocommutative Hopf algebra. Taking primitive elements of this cocommutative Hopf algebra then yields a Lie algebra, which is the Lie algebra associated to the formal group. One can identify the underlying object of this Lie algebra with the tangent space at the identity by noting that primitive elements in the Hopf algebra of distributions are equivalently derivations $A\rightarrow \field$ for $A$ the affine ring of the formal group, which is given by the linear dual of the coalgebra of distributions. On the other hand, after abstractly identifying a formal group with its associative algebra object in cocommutative coalgebras, one can apply the Cobar construction for the universal twisting morphism between the Lie operad and the cocommutative cooperad to get an associative algebra object in the category of Lie algebras. Then taking the delooping functor for Lie algebras, which is an equivalence of $\infty$-categories, one again obtains a Lie algebra associated to the formal group.\displaypar
These two methods of associating a Lie algebra object to a formal group agree by the following
\begin{theorem}[Chapter 6, Theorem 6.1.2 \cite{GR2}]
Let $\mathcal{C}$ be a $\field$-linear symmetric monoidal $\infty$-category. There is a canonical isomorphism of functors
\begin{align*}
    U^{\textnormal{Hopf}} \simeq \textnormal{Grp}(\textnormal{Chev}^{\textnormal{enh}}) \circ \Omega: \alg_{\lie}(\mathcal{C}) \rightarrow \alg_{\mathbb{E}_1}(\coalg^{\textnormal{coaug}}_{\dgcocomm}(\mathcal{C})),
\end{align*}
where $U^{\textnormal{Hopf}}: \alg_{\lie}(\mathcal{C}) \rightarrow \alg_{\mathbb{E}_1}(\coalg^{\textnormal{coaug}}_{\dgcocomm}(\mathcal{C})) \simeq \coalg_{\dgcocomm}(\alg^{\textnormal{aug}}_{\mathbb{E}_1}(\mathcal{C}))$ is a lift of the left adjoint of the restriction functor from augmented associative algebras to Lie algebras. 
\end{theorem}
Here, $\text{Chev}^{\text{enh}}\dashv \text{coChev}^{\text{enh}}$ is the $\infty$-categorical non-conilpotent Bar-Cobar adjunction for the Lie-operad as defined in \cite[Chapter 6, Section 4.1]{GR2}. Consequently, we also get an isomorphism between the respective right adjoints
\begin{align*}
    \text{Prim} \simeq B \circ \text{Monoid}(\text{coChev}^{\text{enh}}).
\end{align*}
This discussion highlights that the reason that the tangent space of a group structure on a formal space built from commutative algebra objects is a Lie algebra can be traced back to the fact that the Lie operad is Koszul dual to the commutative operad. In the context of noncommutative geometry, where spaces are built from algebra objects over an operad $\mathscr{P}$, one may hence expect a modified version of the correspondence between formal moduli problems and Lie algebra objects to hold, where formal moduli problems are replaced by $\mathscr{P}$-\textbf{operadic formal moduli problems}, and Lie algebras are replaced by algebras over the Koszul dual operad $\mathscr{P}^{!}$. \displaypar
Analogous to the commutative case, a $\mathscr{P}$-operadic formal moduli problem is supposed to model the completion at a point of a sufficiently geometric $\mathscr{P}$-stack. For a non-positively graded dg operad $\mathscr{P}$, define the category of $\mathscr{P}$-affine schemes to be the opposite category of the category of connective $\mathscr{P}$-algebras. Then a $\mathscr{P}$-prestack is simply a functor $X: \alg_{\mathscr{P}}(\dk^{\text{cn}}) \rightarrow \mathscr{S}$. It is a known difficulty to define $\mathscr{P}$-schemes or more generally geometric $\mathscr{P}$-stacks from such prestacks, essentially because the typical choices for Grothendieck topologies on the category of affine $\mathscr{P}$-schemes are not well-defined, see for example \cite{KR}. However, since we are mainly interested in formal neighborhoods of a point in such a would-be $\mathscr{P}$-stack, it suffices to understand affine $\mathscr{P}$-schemes together with some notion of ``infinitesimal neighborhood''. Since we already know what affine $\mathscr{P}$-schemes are supposed to be, the problem is reduced to finding a good notion of formal neighborhood. \displaypar
To this end, consider first the commutative case. If $X: \alg_{\dgcomm}(\dk^{\text{cn}}) \rightarrow \mathscr{S}$ is a scheme over $\field$ together with a $\field$-point $\eta \in X(\spec(\field))$, then a point ``infinitesimally close'' to $\eta$ is an $R$-point $\tilde{\eta}$ for $R$ an infinitesimal thickening of the point, i.e. a local dg Artin $\field$-algebra, such that $\tilde{\eta}$ agrees with $\eta$ at $\spec(\field)$. In particular, this notion only depends on the choice of local Artin algebras as the first-order objects to deform along. For the $\mathscr{P}$-operadic case, it therefore suffices to fix the first-order $\mathscr{P}$-algebras we want to deform along. This is the structure of a \textbf{deformation context}. 
\begin{definition}[Definition 1.1.3 \cite{DAGX}]
A deformation context is a presentable $\infty$-category $\Upsilon$ together with a set of objects $\{E_{\alpha}\}_{\alpha\in T}$ in the stabilization $\text{Stab}(\Upsilon)$. For a reduced dg operad $\mathcal{P}$, we consider the deformation context $(\alg_{\mathscr{P}}(\dk), \{\field\})$, where $\field$ corresponds to the spectrum object $\Omega^{\infty-n} \field = \field[n]$.
\end{definition}
There is a general definition of Artinian objects in a deformation context \cite[Definition 1.1.8]{DAGX}, but for the specific deformation context associated to a reduced dg operad $\mathscr{P}$ one can show that the following definition is equivalent \cite[Proposition 4.28]{CG}.
\begin{definition}
Let $\mathscr{P}$ be a reduced dg operad. The $\infty$-category $\alg^{\text{Art}}_{\mathscr{P}}$ of \textbf{Artin $\mathscr{P}$-algebras} is the smallest full subcategory of $\alg_{\mathscr{P}}(\dk)$ such that
\begin{itemize}
    \item $\field[n] \in \art$ for all $n\leq 0$,
    \item for any $A\in \art$ and any map $A\rightarrow \field[n]$ of $\mathscr{P}$-algebras with $n\leq 1$, the homotopy pullback $A\times_{\field[n]} 0$ is also Artinian.
\end{itemize}
\end{definition}
If $\mathscr{P}$ is the non-unital commutative operad, the corresponding Artin algebras are precisely the augmentation ideals of dg Artin $\field$-algebras. \displaypar
To define the notion of a $\mathscr{P}$-operadic formal moduli problem from this, note first that a commutative formal moduli problem $X\in \mathbf{FMP}$ is already uniquely determined by its restriction to the $\infty$-category of dg Artin $\field$-algebras. In particular, $X$ is an object of $\mathbf{PreStk}_{\text{laft}}$, and therefore uniquely determined by its restriction to eventually coconnective finite type affine schemes \cite[Chapter 2, Proposition 1.7.6]{GR1}. Further,
\begin{theorem}[Chapter 5, Proposition 1.2.2 \cite{GR2}]\label{thm14}
 Every formal moduli problem $\mathcal{Y}: ({}^{<\infty}\mathbf{Sch}^{\textnormal{aff}}_{\textnormal{ft}})^{\textnormal{op}} \rightarrow \mathscr{S}$ over $\spec(\field)$ is the left Kan extension of its restriction to the full subcategory 
    \begin{align*}
        ({}^{<\infty}\mathbf{Sch}_{\textnormal{ft}}^{\textnormal{aff}})_{\textnormal{nil-iso to} \spec(\field)}
    \end{align*}
    of eventually coconnective finite type affine schemes $S$ that satisfy $S^{\textnormal{red}} \simeq \spec \field$. Further, if $\tilde{\mathcal{Y}}$ is a presheaf on that subcategory, then its left Kan extension along the inclusion defines a formal moduli problem if and only if 
    \begin{enumerate}
        \item $\tilde{\mathcal{Y}}(\spec(\field)) \simeq \ast$,
        \item For a pushout diagram $S_1 \sqcup_S S'$ in $(\mathbf{Sch}^{\textnormal{aff}})_{\textnormal{ft}, \textnormal{nil-iso to} \spec(\field)}$, where $S\rightarrow S'$ is a square-zero extension, the resulting map 
        \begin{align*}
            \tilde{\mathcal{Y}}(S_1 \sqcup_S S') \rightarrow \tilde{\mathcal{Y}}(S_1) \times_{\tilde{\mathcal{Y}}(S)} \tilde{\mathcal{Y}}(S')
        \end{align*}
        is an isomorphism.
    \end{enumerate}
\end{theorem}
Note that the category $({}^{<\infty}\mathbf{Sch}_{\textnormal{ft}}^{\textnormal{aff}})_{\textnormal{nil-iso to} \spec(\field)}$ is the opposite category of dg Artin $\field$-algebras. This confirms the expectation that a formal neighborhood of a point in a geometric stack is uniquely determined by its values on dg Artin $\field$-algebras. In fact, the conditions on a functor $\alg_{\mathscr{C}om}^{\textnormal{Art}}(\dk) \rightarrow \mathscr{S}$ given by Theorem \ref{thm14} are the more common definition of a formal moduli problem, and such a functor is often called a Schlessinger functor. From this second description of formal moduli problems, one can read off the more general definition of a formal moduli problem in a given deformation context \cite[Definition 1.1.14]{DAGX}. For the deformation context attached to a reduced dg operad $\mathscr{P}$, we then get the following.  
\begin{definition}[Definition 2.11 \cite{CCN}]
Let $\mathscr{P}$ be a reduced dg operad. A \textbf{$\mathscr{P}$-operadic formal moduli problem} is a functor
\begin{align*}
    X: \art \rightarrow \mathscr{S}
\end{align*}
such that 
\begin{enumerate}
    \item $X(0) \simeq \ast$, where $0$ is the zero $\mathscr{P}$-algebra.
    \item For any $A\in \art$ and any map $A \rightarrow \field[n]$ with $n\leq 1$, the functor $X$ sends the pullback $A \times_{\field[n]} 0$ to a pullback in $\mathscr{S}$.
\end{enumerate}
We let $\mathbf{FMP}_{\mathscr{P}}$ denote the $\infty$-category of $\mathscr{P}$-operadic formal moduli problems.
\end{definition}
We can now state the operadic generalization of Theorem \ref{thm15}.
\begin{theorem}[Theorem 1.1 \cite{CCN}]\label{thm16}
Let $\mathscr{P}$ be a Koszul operad concentrated in non-positive degrees. Then there is an equivalence of $\infty$-categories
\begin{align*}
    \mathbf{FMP}_{\mathscr{P}} \simeq \alg_{\mathscr{P}^!}(\dk).
\end{align*}
\end{theorem}
Note that one could formally extend a $\mathscr{P}$-operadic formal moduli problem to a $\mathscr{P}$-prestack $\alg_{\mathscr{P}}(\dk^{\text{cn}}) \rightarrow \mathscr{S}$ by reversing the steps in \cite[Chapter 2, Proposition 1.7.6]{GR1} and Theorem \ref{thm14}. This would yield a more ``geometric'' description of operadic formal moduli problems. Nevertheless, the stability property and the construction of the functors $\text{exp}$ and $\text{Lie}$ in \cite{GR2} depend on the commutative and Lie operad, and do not extend to $\mathcal{P}$-prestacks. \displaypar 
Instead, to prove Theorem \ref{thm16}, Calaque-Campos-Nuiten follow the proof of Theorem \ref{thm15} given in \cite{DAGX}, which uses the algebraic interpretation of why the tangent space of a Lie group is a Lie algebra. A proof that Lurie's functor is equivalent to the one given by Gaitsgory and Rozenblyum for the category $\mathbf{FMP}$ of commutative formal moduli problems over the point is given in Appendix \ref{appendixB}. In particular, Lurie constructs a functor 
\begin{align*}
    \Psi: \alg_{\lie}(\dk) \rightarrow \mathbf{FMP}
\end{align*}
such that the cocommutative coalgebra of distributions on $\Psi(\mathfrak{g})$ is equivalent to the Chevalley-Eilenberg coalgebra $\text{Chev}^{\text{enh}}(\mathfrak{g})$. This functor should be interpreted as the Maurer-Cartan space functor, which sends a dg Lie algebra $\mathfrak{g}$ and a strict dg Artin $\field$-algebra $A$ to the Kan complex $\mathbf{MC}_{\bullet}(\mathfrak{g} \otimes A)$ as defined in \cite[Definition 8.1.1]{Hin3}. The formal space $\mathbf{MC}_{\bullet}(\mathfrak{g} \otimes -)$ is a model for the classifying space $B\exp(\mathfrak{g})$, see for example \cite[Theorem 5.2]{B}. However, in the homotopical setting, this does not yield a well-defined functor from the $\infty$-category of dg Lie algebras to the $\infty$-category of formal moduli problems. Essentially, there are two ways to circumvent this problem: One can work with slightly stricter point-set models, which in the end yield the same homotopy theory, but which make the Maurer-Cartan space functor well-behaved. This was done in \cite{Pr}. On the other hand, one can define a more suitable functor which agrees with the Maurer-Cartan space functor in the situations where it is defined. This is the approach taken in \cite{DAGX}, and it is the approach more easily adapted to the operadic setting.\displaypar
I will give a brief account of Calaque-Campos-Nuiten's proof of Theorem \ref{thm16}. For the remainder of this section, let $\mathscr{P}$ be a Koszul operad concentrated in non-positive degrees.\displaypar
Assume for now that $\mathscr{P}$ admits a Hopf unital structure\footnote{See \cite[Definition 1.11]{Saf}.}, i.e. the symmetric sequence $\mathscr{P}^{\text{un}}:= \mathscr{P} \oplus \field$ has the structure of a Hopf operad such that the inclusion $\mathscr{P} \rightarrow \mathscr{P}^{\text{un}}$ is a map of operads. Then, by \cite[Lemma 1.14]{Saf}, $\mathscr{P}$-algebras can be identified with augmented $\mathscr{P}^{\text{un}}$-algebras. Let $\mathscr{C} = \mathscr{P}^{\text{!`}}$ be the Koszul dual cooperad. Consider an affine $\mathscr{P}^{\text{un}}$-scheme $X = \spec (A)$ together with a $\field$-point given by an augmentation $\epsilon: A \rightarrow \field$. Then we expect the $\mathscr{P}^{!}$-algebra corresponding to the formal moduli problem of $X$ at $\epsilon$ to be the one corresponding to the $\mathscr{P}^{\vee}$-coalgebra of point distributions of $X$ under the Bar-Cobar duality for $\mathscr{P}^!$-algebras. In other words, one takes the completion of $A$ with respect to the augmentation, takes the augmentation ideal, and then takes the continuous dual of this $\mathscr{P}$-algebra to obtain a $\mathscr{P}^{\vee}$-coalgebra, and then applies the complete Cobar functor $\widehat{\text{Cobar}}_{\mathscr{P}^{\vee}}: \text{CoAlg}_{\mathscr{P}^{\vee}}(\ch(\field)) \rightarrow \alg_{\mathscr{P}^!}(\ch(\field))$. This agrees with taking the linear dual composed with the Bar functor $\text{Bar}: \alg_{\mathscr{P}}(\ch(\field)) \rightarrow \coalg_{\mathscr{C}}(\ch(\field))$. In particular, we deduce that an equivalence
\begin{align*}
    \mathbf{FMP}_{\mathscr{P}} \xrightarrow{\simeq} \alg_{\mathscr{P}^!}(\dk)
\end{align*}
should send a formal moduli problem of the form $\spec \bar{A} \oplus \field$ to $(\text{Bar}(\bar{A}))^{\vee}$. We take this as a blueprint even for the case of a general Koszul operad without Hopf unital structure. \displaypar
Denote by $\mathfrak{D}: (\alg_{\mathscr{P}}(\dk))^{\text{op}} \rightarrow \alg_{\mathscr{P}^!}(\dk)$ the opposite of the functor
\begin{align*}
    \alg_{\mathscr{P}}(\ch(\field)) \xrightarrow{\obar} \coalg_{\mathscr{C}}(\ch(\field)) \xrightarrow{(-)^{\vee}} (\alg_{\mathscr{P}^!\{-1\}}(\ch(\field)))^{\text{op}} \xrightarrow{[-1]} (\alg_{\mathscr{P}^!}(\ch(\field)))^{\text{op}}.
\end{align*}
The functor $\mathfrak{D}$ preserves limits, and therefore admits a left adjoint $\mathfrak{D}': \alg_{\mathscr{P}^!}(\dk) \rightarrow (\alg_{\mathscr{P}}(\dk))^{\text{op}}$. 
\begin{theorem}[Theorem 5.1 \cite{CCN}]
The adjoint functors $\mathfrak{D}' \dashv \mathfrak{D}$ above form a \textbf{deformation theory}\footnote{Specifically, they even form a \textbf{Koszul duality context}, see \cite{CaG}.} in the sense of \cite[Definition 1.3.9]{DAGX}. The subcategory $(\alg_{\mathscr{P}^!}(\dk))_0$ is given by the so-called \textbf{good $\mathscr{P}^!$-algebras}, see \cite[Section 4.2]{CCN}. 
\end{theorem}
This in particular implies that the adjunction $\mathfrak{D}'\dashv \mathfrak{D}$ restricts to an equivalence
\begin{align*}
    \alg^{\text{good}}_{\mathscr{P}^!}(\dk) \simeq (\art)^{\text{op}}.
\end{align*}
As a direct consequence, one can deduce from \cite[Theorem 1.3.12]{DAGX} that the functor
\begin{align*}
    \Psi_{\mathscr{P}}: \alg_{\mathscr{P}^!}(\dk) &\rightarrow \mathbf{FMP}_{\mathscr{P}}\\
    \mathfrak{g} &\mapsto \map_{\alg_{\mathscr{P}^!}(\dk)}(\mathfrak{D}(-),\mathfrak{g})
\end{align*}
is an equivalence of $\infty$-categories. One can interpret this as follows. There is an equivalence between Artin $\mathscr{P}$-algebras and representable $\mathscr{P}$-formal moduli problems. On the other hand, the above discussion shows that Artin $\mathscr{P}$-algebras are also equivalent to good $\mathscr{P}^!$-algebras. Because $\mathbf{FMP}_{\mathscr{P}}\subseteq \text{Fun}(\art, \mathscr{S})$, any formal moduli problem can be written as some colimit of representable ones. Lurie then proves in \cite[Proof of Theorem 1.3.12]{DAGX} that the types of colimits that produce formal moduli problems are precisely the types of colimits that produce the full category of $\mathscr{P}^!$-algebras from good ones. 
\begin{remark}
For a strict Artin $\mathscr{P}$-algebra $A$ and some conditions on $\mathscr{P}$ (see \cite[Assumption 7.6]{CCN}), we can again interpret the object $\Psi_{\mathscr{P}}(\mathfrak{g})(A)$ as the Maurer-Cartan space $\mathbf{MC}_{\bullet}(\mathfrak{g} \otimes A)$. Indeed, for $\mathfrak{g}$ a $\mathscr{C}^{\vee}$-algebra and $A$ a $\mathscr{P}$-algebra, the tensor product $\mathfrak{g} \otimes A$ inherits the structure of a shifted $L_{\infty}$-algebra.
\end{remark}
\section{The $\mathbb{E}_2$-FMP corresponding to the center}
The $\infty$-operad $\mathbb{E}^{\otimes}_n$ is Koszul self-dual up to a shift. Therefore, by the preceding section, we know that $\mathbb{E}_n$-formal moduli problems\footnote{The $\mathbb{E}_n$-operad is not a dg operad. Nevertheless, one can define $\mathbb{E}_n$-formal moduli problems analogously to the dg operadic case; see \cite[Section 4]{DAGX}.} correspond precisely to (shifted, non-unital) $\mathbb{E}_n$-algebras. Recall from Section \ref{deformation_problems_from_operadic_algebras} that to an $\mathbb{E}_1$-algebra $A$ we can associate a commutative formal moduli problem ${\text{Def}}_A$ sending a dg Artin $\field$-algebra $R$ to the space of $\mathbb{E}_1$-algebras in $\text{Mod}_R(\dk)$ that reduce to $A$ when restricted to $\text{Mod}_\field$. \displaypar
To define the formal moduli problem $\defm_A$, we used the fact that, for a commutative dg algebra $R$, the $\infty$-category $\text{Mod}_R(\dk)$ is monoidal. More generally, if $\mathcal{C}$ is a presentable symmetric monoidal $\infty$-category (see \cite[Definition 3.4.4.1]{HA}), we can view $R\in \alg_{\mathbb{E}_k}(\mathcal{C})$ as an $\mathbb{E}_1$-algebra via the functor $\mathbb{E}^{\otimes}_1 \rightarrow \mathbb{E}^{\otimes}_k$, and it therefore makes sense to consider the $\infty$-category $\lmod_R(\mathcal{C})$ of left modules over $R$. By \cite[Corollary 4.8.5.20]{HA}, we have a functor 
\begin{align*}
    \alg_{\mathbb{E}_k}(\mathcal{C}) \simeq \alg_{\mathbb{E}_{k-1}}(\alg_{\mathbb{E}_1}(\mathcal{C})) \rightarrow \alg_{\mathbb{E}_{k-1}}(\pr)
\end{align*}
sending $R\in \alg_{\mathbb{E}_k}$ to a $\mathbb{E}_{k-1}$-monoidal enhancement of $\lmod_R(\mathcal{C})$. In particular, if $R\in \alg_{\mathbb{E}_2}(\mathcal{C})$, then $\lmod_R(\mathcal{C})$ is a monoidal $\infty$-category. \displaypar
This enables us to upgrade the commutative formal moduli problem ${\text{Def}}_A$ associated to an $\mathbb{E}_1$-algebra in $\dk$ to an $\mathbb{E}_2$-formal moduli problem. For an augmented $\mathbb{E}_2$-algebra $R$, an \textbf{$R$-deformation of $A$} is an $\mathbb{E}_1$-algebra $B\in \alg_{\mathbb{E}_1}(\lmod_R(\dk))$ together with an equivalence of $\mathbb{E}_1$-algebras $\field \otimes_R B \simeq A$.  We can compose the functor $\lmod_{(-)}(\dk): \alg_{\mathbb{E}_2}(\dk) \rightarrow \alg_{\mathbb{E}_1}(\pr_\field)$ with the functor $\alg_{\mathbb{E}_1}(\pr_\field) \rightarrow \pr$ sending an $\mathbb{E}_1$-monoidal $\field$-linear $\infty$-category $\mathcal{C}$ to the $\infty$-category $\alg_{\mathbb{E}_1}(\mathcal{C})$ of $\mathbb{E}_1$-algebras to get a functor 
\begin{align*}
    \alg_{(-)}: \alg_{\mathbb{E}_2}(\dk) &\rightarrow \pr,\\
    R &\mapsto \alg_{\mathbb{E}_1}(\lmod_R)
\end{align*}
which is classified by some coCartesian fibration $\alg \rightarrow \alg_{\mathbb{E}_2}(\dk)$. This yields a left fibration
\begin{align*}
    \alg^{\text{coCart}}_{/(\field, A)} \rightarrow \alg^{\text{aug}}_{\mathbb{E}_2}(\dk)
\end{align*}
whose fiber over an $\mathbb{E}_2$-algebra $R$ is given by the space of $R$-deformations of $A$. By \cite[Corollary 4.17, Corollary 4.18]{BKP}, restricting the straightening of this left fibration to Artin $\mathbb{E}_2$-algebras yields a 1-proximate\footnote{See \cite[Definition 5.1.5]{DAGX}.} $\mathbb{E}_2$-formal moduli problem
\begin{align*}
    \mathbf{Def}_A: \alg_{\mathbb{E}_2}^{\text{Art}} \rightarrow \mathscr{S},
\end{align*}
which is an $\mathbb{E}_2$-formal moduli problem if $A$ is eventually connective \cite[Proposition 4.19]{BKP}. Recall that for connective $\mathbb{E}_1$-algebras, the dg Lie algebra corresponding to the formal moduli problem $\defm_A$ is given by the operadic tangent complex $\mathbb{T}^{\mathbb{E}_1}_A$. By \cite[Theorem 1.1]{FR}, there is a fiber sequence of Lie algebras
\begin{align*}
    A \rightarrow \mathbb{T}_A^{\mathbb{E}_1} \rightarrow \mathfrak{Z}(A)[1],
\end{align*}
which is a restriction of a fiber sequence of non-unital $\mathbb{E}_2$-algebras
\begin{align}\label{eq3}
    A[-1] \rightarrow \mathbb{T}_A^{\mathbb{E}_1}[-1] \rightarrow \mathfrak{Z}(A).
\end{align}
The non-unital $\mathbb{E}_2$-algebra of the $\mathbb{E}_2$-formal moduli problem corresponding to the 1-proximate $\mathbb{E}_2$-formal moduli problem $\mathbf{Def}_A$ can be identified with precisely this $\mathbb{E}_2$-algebra structure on $\mathbb{T}_A^{\mathbb{E}_1}[-1]$, see \cite[Proposition 4.3]{BKP}.\displaypar
Instead of deforming an $\mathbb{E}_1$-algebra $A$ directly, one can also deform its category of modules $\rmod_A(\dk)$. If $\mathcal{C} \in \lmod_{\dk}(\pr)$ is a $\field$-linear $\infty$-category and $R\in \alg_{\mathbb{E}_2}(\dk)$, \textbf{an $R$-deformation of $\mathcal{C}$} is a $\field$-linear $\infty$-category $\mathcal{D}$ left tensored over $\lmod_R(\dk)$ together with an equivalence of $\field$-linear $\infty$-categories
\begin{align*}
    \dk \otimes_{\lmod_R(\dk)} \mathcal{D} \xrightarrow{\simeq} \mathcal{C}. 
\end{align*}
There again is a coCartesian fibration\footnote{For a definition of $\text{LCat}(\field)$, see \cite[Construction 5.3.2]{DAGX}.} $\text{LCat}(\field) \rightarrow \alg_{\mathbb{E}_2}(\dk)$ which restricts to a left fibration $\text{LCat}(\field)^{\text{coCart}}_{/(\field, \mathcal{C})} \rightarrow \alg^{\text{aug}}_{\mathbb{E}_2}(\dk)$, such that the fiber over $R$ is given by the $R$-deformations of $\mathcal{C}$. By \cite[Corollary 5.3.8]{DAGX}, restricting to Artin $\mathbb{E}_2$-algebras, we get a 2-proximate $\mathbb{E}_2$-formal moduli problem (after passing to a larger universe)
\begin{align*}
    \mathbf{CatDef}_{\mathcal{C}}: \alg_{\mathbb{E}_2}^{\text{Art}} \rightarrow \mathscr{S}.
\end{align*}
By \cite[Theorem 5.2.16]{DAGX}, the corresponding $\mathbb{E}_2$-formal moduli problem is controlled by the \textbf{$\field$-linear center of $\mathcal{C}$}, which is the final object of $\text{RCat}(\field)\times_{\lmod_{\dk}(\pr)} \{\mathcal{C}\}$.\displaypar
If $A\in \alg_{\mathbb{E}_1}(\dk)$ and $\mathcal{C} = \rmod_A(\dk)$, then by \cite[Corollary 4.38]{FR}, the $\field$-linear center of $\rmod_A(\dk)$ can be identified with the center of the $\mathbb{E}_1$-algebra $A$. In particular, the $\mathbb{E}_2$-formal moduli problem given by $\mathbf{CatDef}_{\rmod_A(\dk)}$ corresponds to the $\mathbb{E}_2$-algebra $\mathfrak{Z}(A)$. The fiber sequence (\ref{eq3}) then corresponds to the morphism of 2-proximate $\mathbb{E}_2$-formal moduli problems
\begin{align*}
    \mathbf{Def}_A \rightarrow \mathbf{CatDef}_{\rmod_A(\dk)}
\end{align*}
given by the map $\alg^{\text{coCart}}_{/(\field,A)} \rightarrow \text{LCat}(\field)^{\text{coCart}}_{/(\field, \rmod_A(\dk))}$ sending an augmented $\mathbb{E}_2$-algebra $R$ together with an $R$-deformation $B\in \alg_{\mathbb{E}_1}(\lmod_R(\dk))$ of $A$ to the same augmented $\mathbb{E}_2$-algebra $R$ together with the induced $R$-deformation $\rmod_B(\dk)$ of $\rmod_A(\dk)$.\displaypar
One can interpret this in the following way: The difference between deforming an $\mathbb{E}_1$-algebra $A$ and deforming its right module category $\rmod_A(\dk)$ is that $\rmod_A(\dk)$ does not remember the distinguished generator $A$. In fact, one can identify deformations of the $\mathbb{E}_1$-algebra $A$ with deformations of the pair $(\rmod_A(\dk),A)$ as in \cite[Section 4.1]{BKP}. The fact that the center of $A$ identifies with the center of $\rmod_A(\dk)$ shows that it is Morita invariant, and therefore it makes sense that its associated $\mathbb{E}_2$-formal moduli problem does not know about the distinguished generator. \displaypar
\section{The formal automorphism group}\label{the_formal_automorphism_group}
We have seen in the previous section that the deformation problem associated to an $\mathbb{E}_1$-algebra in $\dk$ can be lifted from a commutative formal moduli problem to an $\mathbb{E}_2$-formal moduli problem, which is controlled by the $\mathbb{E}_2$-algebra given by the fiber of the center $\mathfrak{Z}(A) \rightarrow A$. In this section, we prove that this result generalizes to (sufficiently nice) Koszul operads $\mathscr{P}$. In particular, we show in Theorem \ref{thm21} that the formal group corresponding to the deformation problem of a $\mathscr{P}$-algebra $A$ lifts to a group object in $\mathscr{P}^!$-formal moduli problems, and that this formal $\mathscr{P}^!$-group is controlled by the fiber of the center of $A$. The ideas presented here were inspired by Tamarkin's paper \cite{Tam4}. \displaypar 
Throughout this section, let $\mathscr{P}$ be a Koszul operad concentrated in non-positive degrees. \displaypar 
Recall from Section \ref{deformation_problems_from_operadic_algebras} that for $A\in \alg_{\mathscr{P}}$, the deformation problem ${\defm}_A$ can be interpreted as the formal neighborhood $(\mathcal{M}_{\mathscr{P}})^{\wedge}_A$ of $A$ in the moduli space $\mathcal{M}_{\mathscr{P}}$ of $\mathscr{P}$-algebras. This moduli space is given by the prestack $\mathcal{M}_{\mathscr{P}}: \alg_{\dgcomm}(\dk^{\text{cn}}) \ni R\mapsto \alg_{\mathscr{P}}(\lmod_R(\dk))$. Following Gaitsgory-Rozenblyum's identification between formal moduli problems and dg Lie algebras, the associated dg Lie algebra of ${\defm}_A$ is the Lie algebra of the formal group $\Omega (\mathcal{M}_{\mathscr{P}})^{\wedge}_A$. This can be identified with $(\Omega_A \mathcal{M}_{\mathscr{P}})^{\wedge}_A$, the formal completion of the loop space at the constant loop at $A$. The prestack $\Omega \mathcal{M}_{\mathscr{P}}$ can be identified as the moduli space $\mathcal{A}\text{ut}_{\mathscr{P}}(A)$ of $\mathscr{P}$-algebra automorphisms of $A$,
\begin{align*}
    \mathcal{A}\text{ut}_{\mathscr{P}}(A): \alg_{\dgcomm}(\dk^{\text{cn}}) &\rightarrow \mathscr{S}\\
    R &\mapsto \text{Aut}_{\alg_{\mathscr{P}}(\lmod_R(\dk))}(R\otimes A).
\end{align*}
Consequently, the completion $(\Omega_A \mathcal{M}_{\mathscr{P}})^{\wedge}_A$ sends a dg Artin $\field$-algebra $R$ to the \textbf{formal automorphism group}
\begin{align*}
    \faut_{\mathscr{P}}(A)(R) :=  \text{Aut}_{\alg_{\mathscr{P}}(\lmod_R(\dk))}(R\otimes A) \times_{\text{Aut}_{\alg_{\mathscr{P}}(\dk)}(A)} \{\text{id}_A\}.
\end{align*}
That this is indeed the formal group corresponding to the formal moduli problem ${\defm}_A$ is shown in \cite[Theorem 0.3]{GY}.\displaypar
We can model the $\infty$-category of $\mathscr{P}$-algebras by the model category of conilpotent coalgebras over the Koszul dual cooperad $\mathscr{P}^{\text{!`}}$ with model structure transferred via the Bar-Cobar adjunction with respect to the Koszul twisting morphism. Denote the corresponding class of weak equivalences by $W_{\text{Kos}}$. Then, for a strict dg Artin $\field$-algebra $R$, 
\begin{align*}
\map_{\alg_{\mathscr{P}}(\lmod_R(\dk))}(R\otimes A, R\otimes A) &\simeq \map_{\alg_{\mathscr{P}}(\dk)}(A, R\otimes A)\\&\simeq \map_{\alg_{\mathscr{P}}(\ch(\field))[W^{-1}]}(A, R\otimes A) \\ &\simeq \map_{\ccoalg_{\mathscr{P}^{\text{!`}}}(\ch(\field))[W_{\text{Kos}}^{-1}]}(R^{\vee} \otimes \text{Bar}(A), \text{Bar}(A)),
\end{align*}
where in the last step we use the theory of twisting morphisms as in \cite[Proposition 11.3.2]{LV}. \displaypar
The category of strict dg Artin $\field$-algebras is equivalent to the opposite category of finite dimensional conilpotent dg cocommutative coalgebras by taking coaugmentation cokernels of linear duals. The $\infty$-category $\ccoalg_{\dgcocomm}(\ch(\field))[W^{-1}]$ is a left Bousfield localization of the $\infty$-category $\ccoalg_{\dgcocomm}(\ch(\field))[W_{\text{Kos}}^{-1}]$, and thus a full subcategory. In particular, $\faut_{\mathscr{P}}(A)$ is a restriction to the $\infty$-category of dg Artin $\field$-algebras of the functor
\begin{align*}
    (\text{CoAlg}^{\text{conil}}_{\dgcocomm}[W_{\text{Kos}}^{-1}])^{\text{op}} &\rightarrow \mathscr{S}\\
    S \mapsto \map_{\ccoalg_{\mathscr{P}^{\text{!`}}}[W_{\text{Kos}}^{-1}]}(S^+\otimes \text{Bar}(A),\text{Bar}(A))&\times_{\map_{\ccoalg_{\mathscr{P}^{\text{!`}}}[W_{\text{Kos}}^{-1}]}(\text{Bar}(A),\text{Bar}(A))}\{\text{id}_{\obar(A)}\},
\end{align*}
where $S^+ = \field \oplus S$ is the corresponding coaugmented cocommutative coalgebra. Now assume that $\mathscr{P}^{\text{!`}}$ admits an admissible Hopf counital structure:
\begin{definition}[Definition 1.12 \cite{Saf}]\label{def8}
Let $\mathscr{C}$ be a dg cooperad. Let $\mathscr{C}^{\text{cu}}$ be the symmetric sequence that agrees with $\mathscr{C}$ in arities $n>0$, and satisfies $\mathscr{C}^{\text{cu}}(0) = \mathscr{C}(0) \oplus \field$. A \textbf{Hopf counital structure} on $\mathscr{C}$ is a Hopf cooperad structure on the dg cooperad $\mathscr{C}^{\text{cu}}$ such that the projection $\mathscr{C}^{\text{cu}} \rightarrow \mathscr{C}$ is a morphism of cooperads and such that the unit (of the algebra structure with respect to the Hadamard tensor product) on $\mathscr{C}^{\text{cu}}(0)$ is given by the inclusion into the second factor.
\end{definition}
In this case, by \cite[Lemma 1.14]{Saf}, the category $\ccoalg_{\mathscr{P}^{\text{!`}}}(\ch(\field)) \simeq \coalg^{\text{coaug}}_{(\mathscr{P}^{\text{!`}})^{\text{cu}}}(\ch(\field))$ inherits the structure of a symmetric monoidal category. If this makes $\ccoalg_{\mathscr{P}^{\text{!`}}}(\ch(\field))$ together with the model structure transferred from the Bar-Cobar adjunction into a symmetric monoidal model category, we call the Hopf counital structure \textbf{admissible}. In this case, the equivalence of $\infty$-categories
\begin{center}
    \begin{tikzcd}
             \alg_{\mathscr{P}}(\ch(\field))[W^{-1}] \arrow[r, shift left=1ex, "\text{Bar}"{name=G}] &\ccoalg_{\mathscr{P}^{\text{!`}}}(\ch(\field))[W_{\text{Kos}}^{-1}]\arrow[l, shift left=.5ex, "\text{Cobar}"{name=F}]
            \arrow[phantom, from=F, to=G, , "\scriptscriptstyle\boldsymbol{\simeq}" rotate=0]
        \end{tikzcd}
\end{center}
makes $\alg_{\mathscr{P}}(\dk) \simeq \alg_{\mathscr{P}}(\ch(\field))[W^{-1}]$ into a symmetric monoidal $\infty$-category. In particular, the symmetric monoidal model category $(\ccoalg_{\mathscr{P}^{\text{!`}}}(\ch(\field)), W_{\text{Kos}}^{-1})$ can be taken as a model for this symmetric monoidal $\infty$-category. We denote by $\otimesC$ the tensor product on $\ccoalg_{\mathscr{P}^{\text{!`}}}(\ch(\field))[W_{\text{Kos}}^{-1}]$.
\begin{remark}
The coalgebra model category is better behaved than the algebra model category. For example, the $\infty$-category $\alg_{\mathbb{E}_1}(\alg_{\mathscr{P}}(\dk))$ is generally equivalent to $\alg_{\assoc}(\ccoalg_{\mathscr{P}^{\text{!`}}}(\ch(\field)))[W_{\text{Kos}}^{-1}]$, but not to $\alg_{\assoc}(\alg_{\mathscr{P}}(\ch(\field)))[W^{-1}]$. To see this, take for example $\mathscr{P} = \nuassoc$.
\end{remark}
In this situation, one can modify the functor $\faut_{\mathscr{P}}(A)$ to get a new functor
\begin{align*}
    &\underline{\faut}_{\mathscr{P}}(A): (\ccoalg_{\mathscr{P}^{\text{!`}}}(\ch(\field))[W_{\text{Kos}^{-1}}])^{\text{op}} \rightarrow \mathscr{S}\\
    &S \mapsto \map_{\coalg_{\mathscr{P}^{\text{!`}}}(\ch(\field))[W_{\text{Kos}}^{-1}]}(S\otimesC \text{Bar}(A),\text{Bar}(A))\times_{\map_{\coalg_{\mathscr{P}^{\text{!`}}}(\ch(\field))[W_{\text{Kos}}^{-1}]}(\text{Bar}(A),\text{Bar}(A))}\{\text{id}_{\obar(A)}\},
\end{align*}
which we understand as an ``internal'' version of the formal automorphism group.\displaypar
The following new result is the principal link between the center and the deformation theory of $A$.
\begin{proposition}
The functor $\underline{\faut}_{\mathscr{P}}(A)$ is representable, with representing object given by the centralizer of $\textnormal{id}_{\textnormal{Bar}(A)}$ in $\ccoalg_{\mathscr{P}^{\textnormal{!`}}}(\ch(\field))[W_{\textnormal{Kos}^{-1}}]$. 
\end{proposition}
\begin{proof}
Consider the right fibration
\begin{align*}
    (\ccoalg_{\mathscr{P}^{\text{!`}}}(\ch(\field))[W_{\text{Kos}}^{-1}])_{\text{Bar}(A)//\text{Bar}(A)} \rightarrow (\ccoalg_{\mathscr{P}^{\text{!`}}}(\ch(\field))[W_{\text{Kos}}^{-1}])_{\text{Bar}(A)/}
\end{align*}
classified by the representable functor $\map_{(\ccoalg_{\mathscr{P}^{\text{!`}}}(\ch(\field))[W_{\text{Kos}}^{-1}])_{\text{Bar}(A)/}}(-,\text{id}_{\text{Bar}(A)})$. Then by stability of right fibrations, the functor 
\begin{align*}
    (\ccoalg_{\mathscr{P}^{\text{!`}}}(\ch(\field))[W_{\text{Kos}}^{-1}])_{0/}\times_{(\ccoalg_{\mathscr{P}^{\text{!`}}}(\ch(\field))[W_{\text{Kos}}^{-1}])_{\text{Bar}(A)/}} (\ccoalg_{\mathscr{P}^{\text{!`}}}(\ch(\field))[W_{\text{Kos}}^{-1}])_{\text{Bar}(A)//\text{Bar}(A)} \\\rightarrow (\ccoalg_{\mathscr{P}^{\text{!`}}}(\ch(\field))[W_{\text{Kos}}^{-1}])_{0/}
\end{align*}
is again a right fibration, which is classified by the functor 
\begin{align*}
F: (\ccoalg_{\mathscr{P}^{\text{!`}}}(\ch(\field))[W_{\text{Kos}}^{-1}])_{0/}^{\text{op}} &\xrightarrow{-\otimesC \text{id}_{\text{Bar(A)}}} (\ccoalg_{\mathscr{P}^{\text{!`}}}(\ch(\field))[W_{\text{Kos}}^{-1}])_{\text{Bar}(A)/}^{\text{op}} \\ &\xrightarrow{\map_{(\ccoalg_{\mathscr{P}^{\text{!`}}}(\ch(\field))[W_{\text{Kos}}^{-1}])_{\text{Bar}(A)/}}(-,\text{id}_{\text{Bar}(A)})} \an.
\end{align*}
The domain of this right fibration is precisely the $\infty$-category $\text{Act}(\text{id}_{\text{Bar}(A)})$ from the definition of the centralizer \ref{def6}. In particular, the centralizer $\mathfrak{Z}(\text{id}_{\text{Bar}(A)})$ is final in 
\begin{align*}
    &\text{Act}(\text{id}_{\text{Bar}(A)}) \\&= (\ccoalg_{\mathscr{P}^{\text{!`}}}(\ch(\field))[W_{\text{Kos}}^{-1}])_{0/}\times_{(\ccoalg_{\mathscr{P}^{\text{!`}}}(\ch(\field))[W_{\text{Kos}}^{-1}])_{\text{Bar}(A)/}} (\ccoalg_{\mathscr{P}^{\text{!`}}}(\ch(\field))[W_{\text{Kos}}^{-1}])_{\text{Bar}(A)//\text{Bar}(A)}.
\end{align*}
By \cite[Proposition 4.4.4.5]{HTT}, a final object in the domain of a right fibration corresponds to an object representing the functor classifying it. In particular, the functor $F$ is representable, with representing object\footnote{Recall that by abuse of notation, we denote the image of the centralizer in the underlying monoidal $\infty$-category by the same symbol.} 
\begin{align*}
    \mathfrak{Z}(\text{id}_{\text{Bar}(A)}) \in \ccoalg_{\mathscr{P}^{\text{!`}}}(\ch(\field))[W_{\text{Kos}}^{-1}].
\end{align*}
Let $(u: 0 \rightarrow S)\in (\ccoalg_{\mathscr{P}^{\text{!`}}}(\ch(\field))[W_{\text{Kos}}^{-1}])_{0/}$. By \cite[Lemma 5.5.5.12]{HTT}, we can identify
\begin{align*}
&\map_{(\ccoalg_{\mathscr{P}^{\text{!`}}}(\ch(\field))[W_{\text{Kos}}^{-1}])_{\text{Bar}(A)/}}(u,\text{id}_{\text{Bar}(A)}) \\ &\simeq \map_{\coalg_{\mathscr{P}^{\text{!`}}}(\ch(\field))[W_{\text{Kos}}^{-1}]}(S\otimesC \text{Bar}(A),\text{Bar}(A))\times_{\map_{\coalg_{\mathscr{P}^{\text{!`}}}(\ch(\field))[W_{\text{Kos}}^{-1}]}(\text{Bar}(A),\text{Bar}(A))}\{\text{id}_{\obar(A)}\}.
\end{align*}
In particular, the functor $F$ is equivalent to $\underline{\faut}_{\mathscr{P}}(A)$, and hence
\begin{align*}
\underline{\faut}_{\mathscr{P}}(A) \simeq \map_{\ccoalg_{\mathscr{P}^{\text{!`}}}(\ch(\field))[W_{\text{Kos}}^{-1}]}(-, \mathfrak{Z}(\text{id}_{\text{Bar}(A)})).
\end{align*}
\end{proof}
As in the classical commutative case, for any strict Artin $\mathscr{P}^!$-algebra $R$, the shifted dual $R^{\vee}[-1]$ is a conilpotent $\mathscr{P}^{\text{!`}}$-coalgebra, and we can hence apply $\underline{\faut}_{\mathscr{P}}(A)$ to the $\infty$-category of strict Artin $\mathscr{P}^{\text{!}}$-algebras $R$ by setting
\begin{align*}
    \underline{\faut}_{\mathscr{P}}(A)(R) := \underline{\faut}_{\mathscr{P}}(A)(R^{\vee}[-1]).
\end{align*}
For a general Artin $\mathscr{P}^!$-algebra, consider the canonical resolution $\mathscr{P}^!_{\infty} = \Omega (\mathscr{P}^!)^{\text{!`}} \xrightarrow{\simeq} \mathscr{P}^!$. This yields an equivalence $\alg^{\text{Art}}_{\mathscr{P}^!} \xrightarrow{\simeq} \alg^{\text{Art}}_{\mathscr{P}_{\infty}^!}$, and since $\mathscr{P}^!_{\infty}$ is cofibrant, by \cite[Lemma 5.12]{CCN}, every Artin $\mathscr{P}^!_{\infty}$-algebra is equivalent to a strict one. For a strict Artin $\mathscr{P}^!_{\infty}$-algebra $R$, the linear dual $R^{\vee}$ is a conilpotent coalgebra over $(\mathscr{P}^!_{\infty})^{\vee} \xleftarrow{\simeq} (\mathscr{P}^!)^{\vee} = \mathscr{P}^{\text{!`}}\{-1\}$, and hence the coextension of $R^{\vee}[-1]$ along this equivalence of cooperads is a conilpotent $\mathscr{P}^{\text{!`}}$-coalgebra. 
\begin{lemma}
Let $R\in \alg^{\textnormal{Art}}_{\mathscr{P}^!}$. Then $R_s^{\vee}[-1]$ is equivalent to 
\begin{align*}
\textnormal{Bar}(\mathfrak{D}_{\phi}(R))
\end{align*}
as a $\mathscr{P}^{\text{!`}}$-coalgebra, where $\phi$ is the twisting morphism corresponding to the identity on $\mathscr{P}^!_{\infty}$. 
\end{lemma}
\begin{proof}
By \cite[Remark 7.9]{CCN}, we have
\begin{align*}
\mathfrak{D}_{\phi}(R) \simeq \mathfrak{D}_{\phi}(R_s) \simeq \text{Cobar}_{\phi^{\dag}}(R_s^{\vee})[-1]
\end{align*}
with $\phi^{\dag}: B(((\mathscr{P}^!)^{\text{!`}})^{\vee}) \rightarrow ((\mathscr{P}^!)^{\text{!`}})^{\vee}$ the canonical twisting morphism. But $((\mathscr{P}^!)^{\text{!`}})^{\vee} \cong \mathscr{P}\{-1\}$, so $B(((\mathscr{P}^!)^{\text{!`}})^{\vee}) \simeq \mathscr{P}^{\text{!`}}\{-1\}$, and hence $\phi^{\dag}$ is equivalent to the Koszul twisting morphism for $\mathscr{P}$ up to shift. Hence, $\text{Cobar}_{\phi^{\dag}}(R_s^{\vee})[-1] \simeq \text{Cobar}(R^{\vee}_s[-1])$. 
\end{proof}
We can hence restrict $\underline{\faut}_{\mathscr{P}}(A)$ along 
\begin{align*}
    \delta_{\mathscr{P}^!} := \text{Bar}(\mathfrak{D}_{\phi}(-)): \alg^{\text{Art}}_{\mathscr{P}^!} \rightarrow (\ccoalg_{\mathscr{P}^{\text{!`}}}(\ch(\field))[W_{\text{Kos}^{-1}}])^{\text{op}}
\end{align*}
to obtain a functor
\begin{align*}
    \underline{\faut}_{\mathscr{P}}(A)|_{\alg^{\text{Art}}_{\mathscr{P}^{\text{!}}}}: \alg^{\text{Art}}_{\mathscr{P}^{\text{!}}} \rightarrow \an.
\end{align*}
\begin{theorem}\label{thm21}
The functor $\underline{\faut}_{\mathscr{P}}(A)|_{\alg^{\textnormal{Art}}_{\mathscr{P}^{\textnormal{!}}}}$ is a $\mathscr{P}^!$-operadic formal moduli problem, which under the operadic Lurie-Pridham correspondence \ref{thm16} corresponds to the $\mathscr{P}$-algebra given by the centralizer of $\textnormal{id}_{A}$ in the symmetric monoidal $\infty$-category of $\mathscr{P}$-algebras.
\end{theorem}
\begin{proof}
Let $R\in \alg^{\textnormal{Art}}_{\mathscr{P}^!}$. Then
\begin{align*}
    \underline{\faut}_{\mathscr{P}}(A)(R) &\simeq \map_{\ccoalg_{\mathscr{P}^{\text{!`}}}(\ch(\field))[W_{\text{Kos}}^{-1}]}(\textnormal{Bar}(\mathfrak{D}_{\phi}(R)), \mathfrak{Z}(\text{id}_{\text{Bar}(A)})) \\&\simeq \map_{\alg_{\mathscr{P}}(\ch(\field))[W^{-1}]}(\mathfrak{D}_{\phi}(R), \textnormal{Cobar}(\mathfrak{Z}(\text{id}_{\text{Bar}(A)}))).
\end{align*}
Hence, $\underline{\faut}_{\mathscr{P}}(A)|_{\alg^{\textnormal{Art}}_{\mathscr{P}^{\textnormal{!}}}} \simeq \Psi_{\mathscr{P}^!}(\textnormal{Cobar}(\mathfrak{Z}(\text{id}_{\text{Bar}(A)})))$.
\end{proof}
The functor $\underline{\faut}_{\mathscr{P}}(A)$ has a lax monoidal structure induced by the composition in the (formal) automorphism group. In particular, for $S, T\in \ccoalg_{\mathscr{P}^{\text{!`}}}(\ch(\field))[W_{\text{Kos}^{-1}}]$, we have a canonical map
\begin{align*}
    \underline{\faut}_{\mathscr{P}}(A)(T) \times \underline{\faut}_{\mathscr{P}}(A)(S) &\rightarrow \underline{\faut}_{\mathscr{P}}(A)(T\otimesC S) 
\end{align*}
sending $f: S \otimesC \text{Bar}(A) \rightarrow \text{Bar}(A)$ and $g: T \otimesC \text{Bar}(A) \rightarrow \text{Bar}(A)$ to the composition
\begin{align*}
    T \otimesC S \otimesC \text{Bar}(A) \xrightarrow{\text{id}_T\otimesC f} T \otimesC \text{Bar}(A) \xrightarrow{g} \text{Bar}(A).
\end{align*}
Since lax monoidal presheaves correspond to $\mathbb{E}_1$-algebras in the Day convolution monoidal structure, and the Yoneda embedding is fully faithful and monoidal, this precisely corresponds to the structure of an $\mathbb{E}_1$-algebra on the representing object $\mathfrak{Z}(\text{id}_A)$. By construction, this agrees with the $\mathbb{E}_1$-algebra structure on the center $\mathfrak{Z}(A)$.
\begin{remark}
We can think of monoids with respect to the Day convolution symmetric monoidal structure on $\text{FMP}_{\mathscr{P}^!}$ as the generalization to operadic formal moduli problems of formal groups, i.e. group objects in the Cartesian symmetric monoidal category of commutative formal moduli problems. 
\end{remark}
We can summarize the situation as follows: If $\mathscr{P}^{\text{!`}}$ admits a Hopf counital structure, the internal formal automorphism group of a $\mathscr{P}$-algebra $A$ is a formal $\mathscr{P}^!$-group, whose tangent space is given by the center $\mathfrak{Z}(A)\in \alg_{\mathbb{E}_1}(\alg_{\mathscr{P}}(\dk))$. \displaypar
To get a comparison of the $\mathscr{P}^!$-formal moduli problem $\underline{\faut}_{\mathscr{P}}(A)|_{\alg^{\text{Art}}_{\mathscr{P}^{\text{!}}}}$ to the commutative formal moduli problem $\text{Def}_A$, and therefore of the center $\mathfrak{Z}(A)$ to the operadic tangent complex $\mathbb{T}_A^{\mathscr{P}}$, we need to fix a morphism of cooperads\footnote{Often one instead has a natural morphism $\nccocomm \rightarrow \mathscr{P}^{\text{!`}}\{k\}$ for some $k\in \mathbb{Z}$. For example, there is a canonical map $\nccocomm \rightarrow \nccoassoc = (\nuassoc)^{\text{!`}}\{-1\}$. One can easily trace such a shift through the constructions below.}
\begin{align*}
\varphi: \nccocomm \rightarrow \mathscr{P}^{\text{!`}},
\end{align*}
which we assume to be compatible with the respective Hopf counital structures. This yields a (derived) symmetric monoidal corestriction functor $\varphi_{\ast}: \ccoalg_{\nccocomm}(\ch(\field))[W_{\text{Kos}^{-1}}] \rightarrow \ccoalg_{\mathscr{P}^{\text{!`}}}(\ch(\field))[W_{\text{Kos}^{-1}}]$, and we can precompose $\underline{\faut}_{\mathscr{P}}(A)$ with $\varphi_{\ast}$ to get 
\begin{align*}
    (\varphi_{\ast})^{\ast}(\underline{\faut}_{\mathscr{P}}(A)): \ccoalg_{\nccocomm}(\ch(\field))[W_{\text{Kos}^{-1}}]^{\text{op}} &\rightarrow \an\\
    S \mapsto \map_{\ccoalg_{\mathscr{P}^{\text{!`}}}[W_{\text{Kos}}^{-1}]}(\varphi_{\ast}(S)\otimesC \text{Bar}(A),\text{Bar}(A))&\times_{\map_{\ccoalg_{\mathscr{P}^{\text{!`}}}[W_{\text{Kos}}^{-1}]}(\text{Bar}(A),\text{Bar}(A))}\{\text{id}_{\obar(A)}\}.
\end{align*}
This almost agrees with $\faut_{\mathscr{P}}(A)$, but the tensor product of $\mathscr{P}^{\text{!`}}$-coalgebras $\varphi_{\ast}(S) \otimesC \text{Bar}(A)$ does not recover the tensoring of $\mathscr{P}^{\text{!`}}$-coalgebras over (counital) cocommutative coalgebras. In particular, since the monoidal structure on $\ccoalg_{\mathscr{P}^{\text{!`}}}(\ch(\field))$ comes from the monoidal structure on coaugmented $(\mathscr{P}^{\text{!`}})^{\text{cu}}$-coalgebras, the underlying complex of $\varphi_{\ast}(S) \otimesC \text{Bar}(A)$ is given by $S \oplus \text{Bar}(A) \oplus (S\otimes_{\field} \text{Bar}(A))$. On the other hand, if $S^+$ is the corresponding coaugmented counital cocommutative coalgebra, the underlying complex of $S^+ \otimes \text{Bar}(A)$ is given by $\text{Bar}(A) \oplus (S\otimes \text{Bar}(A))$. Hence, there is a morphism 
\begin{align*}
    {\faut}_{\mathscr{P}}(A) \rightarrow (\varphi_{\ast})^{\ast}(\underline{\faut}_{\mathscr{P}}(A))
\end{align*}
given by extending a morphism $S^+ \otimes \text{Bar}(A) \rightarrow \text{Bar}(A)$ to $\varphi_{\ast}(S)\otimesC \text{Bar}(A)$ by sending $S$ to $0$. One can identify $(\varphi_{\ast})^{\ast}(\underline{\faut}_{\mathscr{P}}(A))$ with the formal moduli problem sending a conilpotent non-counital cocommutative coalgebra $S$ to 
\begin{align*}
\map_{\text{CoAlg}^{\text{conil,coaug}}_{(\mathscr{P}^{\text{!`}})^{\text{cu}}}[W_{\text{Kos}}^{-1}]}(\varphi_{\ast}(S)^+\otimesC \text{Bar}(A)^+,\text{Bar}^+(A))&\times_{\map_{\text{CoAlg}^{\text{conil,coaug}}_{(\mathscr{P}^{\text{!`}})^{\text{cu}}}[W_{\text{Kos}}^{-1}]}(\text{Bar}(A)^+,\text{Bar}(A)^+)}\{\text{id}_{\obar(A)^+}\}.
\end{align*}
Since $\faut_{\mathscr{P}}(A)$ is the formal group corresponding to ${\text{Def}}_A$, it is represented by the (non-counital) cocommutative coalgebra $\overline{U}(\mathbb{T}^{\mathscr{P}}_A)$. Similarly, one can show that $(\varphi_{\ast})^{\ast}(\underline{\faut}_{\mathscr{P}}(A))$ is represented by the non-counital universal enveloping algebra of the \textbf{extended operadic tangent complex}
\begin{align*}
    \mathbb{T}^{\mathscr{P},\text{ext}}_A := \text{Conv}((\mathscr{P}^{\text{!`}})^{\text{cu}},\e_A)^{\phi},
\end{align*}
with $\phi$ the Maurer-Cartan element corresponding to the $\mathscr{P}$-algebra structure on $A$. This differs from the operadic tangent complex by an extra component of $A$ corresponding to the 0 ary operations.
\displaypar
Taking duals, the morphism of cooperads $\varphi$ induces a morphism of operads $\varphi^{\vee}: \mathscr{P}^!\{-1\} \rightarrow \nucomm$. This yields an extension morphism $(\varphi^{\vee})^{\ast}: \alg_{\nucomm}(\ch(\field))[W^{-1}] \rightarrow \alg_{\mathscr{P}^!\{-1\}}(\ch(\field))[W^{-1}]$ which preserves Artin algebras. In particular, we get a morphism
\begin{align*}
    \text{Res}^{\mathscr{P}^!}_{\nucomm}:= ([-1]\circ (\varphi^{\vee})^{\ast})^{\ast}: \text{FMP}_{\mathscr{P}^!} \rightarrow \text{FMP}.
\end{align*}
\begin{proposition}\label{prop12}
Under the restriction $\textnormal{Res}^{\mathscr{P}^!}_{\nucomm}$, the $\mathscr{P}^!$-formal moduli problem $\underline{\faut}_{\mathscr{P}}(A)|_{\alg^{\textnormal{Art}}_{\mathscr{P}^{\text{!}}}}$ corresponds to the commutative formal moduli problem $(([1]\circ \varphi_{\ast})^{\ast}(\underline{\faut}_{\mathscr{P}}(A)))|_{\alg^{\textnormal{Art}}_{\nucomm}}$.
\end{proposition}
The proof of this proposition directly follows from the following observation. 
\begin{lemma}
There is a natural equivalence of functors
\begin{align*}
    \delta_{\mathscr{P}^!\{-1\}} \circ (\varphi^{\vee})^{\ast} \simeq (\varphi\{1\})_{\ast}\circ \delta_{\nucomm}: \alg^{\textnormal{Art}}_{\nucomm} \rightarrow \ccoalg_{\mathscr{P}^{\text{!`}}\{1\}}(\ch(\field))[W_{\textnormal{Kos}}^{-1}]^{\textnormal{op}}
\end{align*}
\end{lemma}
\begin{proof}
First note that the following diagram commutes
\begin{center}
\begin{tikzcd}
\alg_{\nucomm}(\ch(\field)) \arrow[r, "(\varphi^{\vee})^{\ast}"] \arrow[d]                                                                  & \alg_{\mathscr{P}^!\{-1\}}((\ch(\field)) \arrow[d]                                                         \\
\alg_{\nucomm_{\infty}}((\ch(\field)) \arrow[r, "(\Omega(\varphi^{\vee})^{\text{!`}})^{\ast}"] \arrow[d, "\mathfrak{D}_{\nucomm_{\infty}}"'] & \alg_{\mathscr{P}_{\infty}^!\{-1\}}((\ch(\field)) \arrow[d, "\mathfrak{D}_{\mathscr{P}_{\infty}^!\{-1\}}"] \\
\alg_{\lie}^{\text{op}}((\ch(\field)) \arrow[r, "((\varphi^{\vee})^!)_{\text{ext}}"']                                                                       & \alg_{\mathscr{P}\{1\}}^{\text{op}}((\ch(\field))                                                         
\end{tikzcd}
\end{center}
Further, we have an equivalence of functors
\begin{align*}
(\varphi\{1\})_{\ast} \circ \text{Bar}_{\lie} \simeq \text{Bar}_{\mathscr{P}\{1\}} \circ ((\varphi^{\vee})^!)_{\text{ext}}: \alg_{\lie}(\dk) \rightarrow \ccoalg_{\mathscr{P}^{\text{!`}}\{1\}}(\ch(\field))[W_{\text{Kos}^{-1}}]^{\textnormal{op}}.
\end{align*}
Restricting to Artin algebras, this yields the result. 
\end{proof}
\begin{corollary}
There is an equivalence of non-counital cocommutative coalgebras
\begin{align*}
    \varphi^{!}(\mathfrak{Z}(\textnormal{id}_{\textnormal{Bar}(A)})) \simeq \overline{U}(\mathbb{T}^{\mathscr{P},\textnormal{ext}}_A).
\end{align*}
\end{corollary}
\begin{proof}
The restriction of the $\mathscr{P}^!$-formal moduli problem $\underline{\faut}_{\mathscr{P}}(A)|_{\alg^{\text{Art}}_{\mathscr{P}^{\text{!}}}}$ represented by the $\mathscr{P}^{\text{!`}}$-coalgebra $\mathfrak{Z}(\text{id}_{\text{Bar}(A)})$ to a commutative formal moduli problem is given by 
\begin{align*}
     \alg^{\text{Art}}_{\nucomm} \ni R &\mapsto \map_{\ccoalg_{\mathscr{P}^{\text{!`}}}[W_{\text{Kos}}^{-1}]}(\delta_{\mathscr{P}^!}((\varphi^{\vee})^{\ast}(R)[-1]), \mathfrak{Z}(\text{id}_{\text{Bar}(A)})) \\& \simeq \map_{\ccoalg_{\mathscr{P}^{\text{!`}}}[W_{\text{Kos}}^{-1}]}(\varphi_{\ast}(\delta_{\nucomm}(R)[1]), \mathfrak{Z}(\text{id}_{\text{Bar}(A)})) \\&\simeq \map_{\ccoalg_{\nccocomm}[W_{\text{Kos}}^{-1}]}(\delta_{\nucomm}(R)[1], \varphi^{!}(\mathfrak{Z}(\text{id}_{\text{Bar}(A)}))) \\&\simeq \map_{\ccoalg_{\nccocomm\{1\}}[W_{\text{Kos}}^{-1}]}(\delta_{\nucomm}(R), \varphi^{!}(\mathfrak{Z}(\text{id}_{\text{Bar}(A)}))[-1]).
\end{align*}
In particular, this commutative formal moduli problem is represented by the shifted cocommutative coalgebra $\varphi^{!}(\mathfrak{Z}(\text{id}_{\text{Bar}(A)}))[-1]$. But by Proposition \ref{prop12}, we know that this commutative formal moduli problem is also represented by $\overline{U}(\mathbb{T}^{\mathscr{P},\textnormal{ext}}_A)[-1]$. 
\end{proof}
Since the morphism $\varphi: \nccocomm \rightarrow \mathscr{P}^{\text{!`}}$ is compatible with the Hopf counital structures, the restriction functor $\text{Res}^{\mathscr{P}^!}_{\nucomm}$ respects monoid objects with respect to the Day convolution symmetric monoidal structure on $\mathscr{P}^!$-formal moduli problems and commutative formal moduli problems respectively. Note that the Day convolution symmetric monoidal structure on commutative formal moduli problems is just the Cartesian one. In particular, 
\begin{corollary}
There is an equivalence of non-counital Hopf algebras
\begin{align*}
    \varphi^!(\mathfrak{Z}(\textnormal{Bar}(A))) \simeq \overline{U}(\mathbb{T}^{\mathscr{P}, \textnormal{ext}}_A)
\end{align*}
\end{corollary}
Finally, since there is a canonical inclusion $\mathbb{T}^{\mathscr{P}}_A \rightarrow \mathbb{T}^{\mathscr{P}, \text{ext}}_A$, we get the expected morphism of non-counital Hopf algebras
\begin{align*}
    \overline{U}(\mathbb{T}^{\mathscr{P}}_A) \rightarrow \varphi^!(\mathfrak{Z}(\text{Bar}(A))).
\end{align*}
It is a natural question whether one can recover the associative algebra structure in $\mathscr{P}^{\text{!`}}$-coalgebras on $\mathfrak{Z}(\text{Bar}(A))$ in terms of the extended operadic tangent complex $\mathbb{T}^{\mathscr{P},\text{ext}}_A$ using the Koszul property of $\mathscr{P}$. \displaypar
To this end, Calaque and Willwacher in \cite{CW} introduced a new operad $\mathbf{preLie}_{\mathscr{C}}$ for any cooperad $\mathscr{C}$ with a Hopf counital structure. This operad $\mathbf{preLie}_{\mathscr{C}}$ automatically acts on 
\begin{align*}
    \text{Conv}(\mathscr{C}^{\text{cu}},\e_A)
\end{align*}
for any complex $A$. Specifically, the operad $\mathbf{preLie}_{\mathscr{C}}$ has operations parametrized by rooted trees decorated by operations in $\mathscr{C}^{\text{cu}}$, such that the arity of the decorating operations agrees with the number of incoming edges at the vertex. Just like for the pre-Lie operad, the composition operations are given by grafting trees, and the Hopf cooperad structure is used to multiply the decorating operations at the vertices. Similarly, the action of $\mathbf{preLie}_{\mathscr{C}}$ on $\text{Conv}(\mathscr{C}^{\text{cu}},\e_A)$ is given by the pre-Lie action and using the Hopf cooperad structure to multiply the label of a vertex with the elements of the convolution complex.\displaypar
Twisting with a Maurer-Cartan element in the underlying Lie algebra does not lift the $\mathbf{preLie}_{\mathscr{C}}$-action. One therefore needs to employ the formalism of operadic twisting as described in \cite{DW}. Using this, the map $\lie \rightarrow \mathbf{preLie}_{\mathscr{C}}$ yields a twisted operad called the $\mathscr{C}$-brace operad $\mathbf{Br}_{\mathscr{C}}$, which acts on $\text{Conv}(\mathscr{C}^{\text{cu}},\e_A)^{\phi}$ for any Maurer-Cartan element $\phi$. \displaypar
The brace operad associated to $\mathscr{P}^{\text{!`}}$ should be interpreted as a homotopically well-behaved version of the Boardman-Vogt tensor product $\mathbf{Assoc}\otimes \mathscr{P}$. To this end, P. Safronov showed in \cite{Saf} that a $\mathbf{Br}_{\mathscr{C}}$-algebra structure on a complex $B$ induces the structure of an associative algebra in $\mathscr{C}$-coalgebras on the cofree conilpotent $\mathscr{C}^{\text{cu}}$-coalgebra $\mathscr{C}^{\text{cu}}(B)$ equipped with the Bar-differential. In particular, the unit is induced by $\mathscr{C}^{\text{cu}}(0) = \field$, and the product is uniquely determined by a map
\begin{align*}
    \mathscr{C}^{\text{cu}}(B) \otimes \mathscr{C}^{\text{cu}}(B) \rightarrow B,
\end{align*}
which is defined to be given by the $\mathscr{C}$-brace operations
\begin{align*}
\{-;-,\dots,-\}_{-}: \mathscr{C}^{\text{cu}}(n) \otimes B \otimes B^{\otimes n} \rightarrow B.
\end{align*}
This defines an \textbf{additivity functor} 
\begin{align*}
\text{add}: \alg_{\mathbf{Br}_{\mathscr{C}}}(\ch(\field)) \rightarrow \alg_{\assoc}(\ccoalg_{\mathscr{C}}(\ch(\field))).
\end{align*}
In particular, in the above situation, the $\mathbf{Br}_{\mathscr{P}^{\text{!`}}}$-algebra structure on $\text{Conv}((\mathscr{P}^{\text{!`}})^{\text{cu}},\e_A)^{\phi}$ yields
\begin{align*}
    \text{add}(\text{Conv}((\mathscr{P}^{\text{!`}})^{\text{cu}},\e_A)^{\phi}) \in \alg_{\mathbf{Assoc}}(\ccoalg_{\mathscr{P}^{\text{!`}}}(\ch(\field))).
\end{align*}
We expect this to be precisely the associative algebra in $\mathscr{P}^{\text{!`}}$-coalgebras structure corresponding to the internal formal automorphism group:
\begin{conjecture}\label{conj3}
    We have an equivalence of $\mathbb{E}_1$-algebras
    \begin{align*}
    \textnormal{add}(\textnormal{Conv}((\mathscr{P}^{\textnormal{!`}})^{\textnormal{cu}},\e_A)^{\phi}) \simeq \mathfrak{Z}(A) \in \alg_{\mathbb{E}_1}(\alg_{\mathscr{P}}(\dk)).
    \end{align*}
\end{conjecture}
\begin{remark}
Indeed, in \cite{MS}, V. Melani and Safronov call the $\mathscr{P}^{\text{!`}}$-brace-algebra $\text{Conv}((\mathscr{P}^{\text{!`}})^{\textnormal{cu}}, \e_A)^{\phi}$ the center of $A$.
\end{remark}
In \cite{Tam4}, Tamarkin showed a weak version of Conjecture \ref{conj3}, which holds in the underlying 1-category for $\mathscr{P} = \mathbb{P}_n$, the $(n-1)$-shifted Poisson operad. In particular, he proved the following
\begin{proposition}[Proposition 3.2 \cite{Tam4}]
Let $A\in \alg_{\mathbb{P}_n}(\ch(\field))$. The functor
\begin{align*}
    &F_{A,A}^{\textnormal{id}_A}: (\ccoalg_{\textnormal{co}\mathbb{P}_n}(\ch(\field)))^{\textnormal{op}} \rightarrow \mathbf{Set}\\ 
    &S \mapsto \hom_{\ccoalg_{\textnormal{co}\mathbb{P}_n}(\ch(\field))}(S \otimes_{\mathbb{P}_n} \textnormal{Bar}(A[n]), \textnormal{Bar}(A[n])) \times_{\hom_{\ccoalg_{\textnormal{co}\mathbb{P}_n}(\ch(\field))}(\textnormal{Bar}(A[n]), \textnormal{Bar}(A[n]))} \{\textnormal{id}_A\}
\end{align*}
is represented by $\textnormal{Bar}(\hom_{\field}(\textnormal{Bar}(A[n])^+,A)[n])$, where $\hom_{\field}(\textnormal{Bar}(A[n])^+,A)$ is equipped with the canonical $\mathbb{P}_n$-algebra structure. The induced $\assoc$-algebra structure is compatible with the $\assoc$-algebra structure on $\overline{U}(\mathbb{T}^{\mathbb{P}_n}_A)$.
\end{proposition}
\begin{example}
\begin{itemize}
    \item For the Lie operad, the brace operad $\mathbf{Br}_{\text{co}\mathscr{C}\text{om}}$ agrees with the Lie operad itself, and the additivity functor $\text{add}: \alg_{\lie}(\dk) \rightarrow \alg_{\mathbb{E}_1}(\alg_{\lie}(\dk))$ is given by the loops functor, and is thus an equivalence. Note that this is precisely the functor used in Gaitsgory-Rozenblyum's version of the Lurie-Pridham correspondence to obtain the Lie algebra associated to the Hopf algebra of distributions of a formal group.
    \item For the shifted Poisson operad $\mathbb{P}_n$, there is an identification $\mathbf{Br}_{\text{co}\mathbb{P}_n} \simeq \mathbb{P}_{n+1}$. The additivity functor $\text{add}: \alg_{\mathbb{P}_{n+1}}(\dk) \rightarrow \alg_{\mathbb{E}_1}(\alg_{\mathbb{P}_n}(\dk))$ is given by the Bar-Cobar duality for shifted Lie bialgebras, see \cite[Proposition 2.17]{Saf}. This is the Poisson version of the Dunn additivity theorem.
\end{itemize}
\end{example}
The operad $\mathbf{Braces}$ can be identified with a strictification of the $\mathscr{P}^{\text{!`}}$-brace operad for the Koszul operad $\nuassoc$. In particular, the following is a pushout of operads by \cite[Proposition 2.2]{Saf}
\begin{center}
    \begin{tikzcd}
\mathbb{A}_{\infty} \arrow[d] \arrow[r] & \mathbf{Br}_{\nccoassoc}\{1\} \arrow[d] \\
\nuassoc \arrow[r]    & \mathbf{Braces}                                          
\end{tikzcd}
\end{center}
For a cofibrant associative algebra $A$, the $\mathbf{Braces}$-algebra structure on $C^{\ast}(A,A) \simeq \text{Conv}(\coassoc\{1\}, \e_A)^{\phi}[-1]$ agrees with the $\mathbf{Br}_{\nccoassoc}$-algebra structure constructed by \cite{CW}. In particular, Safronov's additivity functor produces an $\mathbb{E}_2$-algebra directly from the $\mathbf{Braces}$-algebra structure on the Hochschild cochain complex, which in contrast to Tamarkin's morphism $\Psi_T: \ger_{\infty} \rightarrow \mathbf{Braces}$ does not depend on the choice of a Drinfeld associator. Therefore, for $\mathscr{P}$ the associative operad, Conjecture \ref{conj3} can also be interpreted as saying that the center $\mathbb{E}_2$-algebra structure on the Hochschild complex is canonically equivalent to the $\mathbf{Braces}$-algebra structure.
\section{Kontsevich formality and the center}
Recall from Chapter \ref{Kontsevich_formality_and_deformation_quantization} that the Kontsevich formality map is an $\infty$-quasi-isomorphism of $L_{\infty}$-algebras
\begin{align*}
    T^{\ast}_{\text{poly}}(A)[1] \rightarrow D^{\ast}_{\text{poly}}(A)[1]
\end{align*}
for any regular commutative $\field$-algebra $A$.\displaypar
On the other hand, Tamarkin's formality morphism, which depends on the choice of a Drinfeld associator, is an $\infty$-quasi-isomorphism of $\mathbf{Ger}_{\infty}$-algebras
\begin{align*}
    T^{\ast}_{\text{poly}}(A) \rightarrow D^{\ast}_{\text{poly}}(A),
\end{align*}
where the complex of polyvector fields forms a $\mathbf{Ger}_{\infty}$-algebra via the resolution $\mathbf{Ger}_{\infty} \rightarrow \mathbf{Ger}= \mathbb{P}_2$, and the complex of polydifferential operators forms a $\mathbf{Ger}_{\infty}$-algebra via Tamarkin's morphism $\Psi_T: \mathbf{Ger}_{\infty} \rightarrow \mathbf{Braces}$. \displaypar
By Corollary \ref{cor9}, the $\mathbf{Ger}_{\infty}$-algebra structure on $D^{\ast}_{\text{poly}}(A)$ can, at least up to homotopy, be recovered from the $\mathbb{E}_2$-algebra structure coming from the identification $D^{\ast}_{\text{poly}}(A) \simeq \mathfrak{Z}_{\mathbb{E}_1}(A)$. On the other hand, the complex of polyvector fields generalizes to arbitrary Poisson algebras $B\in \alg_{\mathbb{P}_1}(\ch(\field))$ corresponding to a Maurer-Cartan element $\psi\in\text{Conv}(\text{co}\mathbb{P}_1\{1\},\e_A)$, see \cite[Section 1.1]{Saf2}, and if $B$ is cofibrant as a commutative algebra, we have a quasi-isomorphism
\begin{align*}
    T^{\ast}_{\text{poly}}(B) \simeq \text{Conv}((\text{co}\mathbb{P}_1\{1\})^{\text{cu}}, \e_B)^{\psi}[-1]
\end{align*}
which is compatible with the homotopy $\mathbb{P}_2$-algebra structures on both sides. We can view a commutative algebra $A$ as a Poisson algebra with trivial bracket, and if $A$ is cofibrant we can identify the classical complex of polyvector fields with the above convolution complex.\displaypar
By Safronov's Poisson additivity theorem \cite[Theorem 2.22]{Saf}, the derived operadic center
\begin{align*}
    \mathfrak{Z}_{\mathbb{P}_1}(A) \in \alg_{\mathbb{E}_1}(\alg_{\mathbb{P}_1}(\dk)) \simeq \alg_{\mathbb{P}_2}(\dk)
\end{align*}
also has the structure of a $\mathbb{P}_2$-algebra. Assuming Conjecture \ref{conj3}, this $\mathbb{P}_2$-algebra agrees with the $\mathbb{P}_2$-algebra $T^{\ast}_{\text{poly}}(A)$. \displaypar
This suggests that Tamarkin's formality theorem is really a comparison between the Poisson and associative center respectively of a commutative algebra $A$. Since the Poisson center $\mathfrak{Z}_{\mathbb{P}_1}(A)$ is a $\mathbb{P}_2$-algebra, while the associative center $\mathfrak{Z}_{\mathbb{E}_1}(A)$ is a $\mathbb{E}_2$-algebra, this requires a choice of identification between $\mathbb{P}_2$ and $\mathbb{E}_2$, which is precisely the choice of a Drinfeld associator also required for Tamarkin's original formality morphism. In particular, Tamarkin's formality morphism can be viewed as comparing the $\mathbb{P}_2$-formal moduli problem of Poisson deformations of $A$ with the $\mathbb{E}_2$-formal moduli problem of associative deformations of $A$. \displaypar
On the other hand, Kontsevich's formality map, which is only a morphism of homotopy Lie algebras, can be viewed as comparing the underlying commutative formal moduli problems of Poisson and associative deformations of $A$ respectively.
\newpage
\renewcommand{\thetheorem}{\thechapter.\arabic{theorem}}
\setcounter{theorem}{0}
\appendix
\chapter{The endomorphism $\infty$-category}\label{appendixA}
In this section we show that our endomorphism $\infty$-category $\mathcal{C}_{\mathfrak{a}}^{\otimes} \times_{{\mathcal{C}_{\mathfrak{m}}}} {\mathcal{C}_{\mathfrak{m}}}_{/M}$ defined in Definition \ref{def1} is equivalent to Lurie's endomorphism $\infty$-category as defined in \cite[Definition 4.7.1.1]{HA}. In particular, this will imply that $\mathcal{C}_{\mathfrak{a}}^{\otimes} \times_{\mathcal{C}_{\mathfrak{m}}} (\mathcal{C}_{\mathfrak{m}})_{/M}$ is the underlying $\infty$-category of a monoidal $\infty$-category.\displaypar
Lurie's definition is given in terms of planar $\infty$-operads as defined in \cite[Section 4.1.3]{HA}. In particular it is given for a map $p: \mathscr{M}^{\circledast} \rightarrow \Delta^1 \times N(\mathbb{\Delta})^{\text{op}}$ exhibiting $\mathscr{M} := \mathscr{M}^{\circledast}_{0,[0]}$ as left tensored over $\mathcal{C}^{\circledast}:= \mathscr{M}^{\circledast}\times_{\Delta^1} \{1\}$ in the planar sense as in \cite[Variant 4.2.2.11]{HA}, and an element $M\in \mathscr{M}$. 
\begin{definition}[\cite{HA}, Definition 4.7.1.1]\label{def4}
Let $p: \mathscr{M}^{\circledast} \rightarrow \Delta^1 \times N(\mathbb{\Delta})^{\text{op}}$ be a map exhibiting $\mathscr{M}$ as left tensored over the planar $\infty$-operad $\mathcal{C}^{\circledast}$. An enriched morphism of $\mathscr{M}$ is a span 
\begin{align*}
    M \xleftarrow{\alpha} X \xrightarrow{\beta} N
\end{align*}
such that 
\begin{itemize}
    \item $p(\alpha)$ is the map $(0,[1]) \rightarrow (0, [0])$ given by $[0] \cong \{0\} \hookrightarrow [1]$ in $\mathbb{\Delta}$,
    \item $\beta$ is inert and $p(\beta)$ is the map $(0,[1]) \rightarrow (0, [0])$ given by $[0] \cong \{1\} \hookrightarrow [1]$ in $\mathbb{\Delta}$.
\end{itemize}
Denote by $\text{EnMor}(\mathscr{M}^{\circledast})$ the full subcategory of $\text{Fun}_{\Delta^1 \times N(\mathbb{\Delta})^{\text{op}}}(\Lambda_0^2,\mathscr{M}^{\circledast})$ spanned by the enriched morphisms. There is a pair of forgetful functors $\text{EnMor}(\mathscr{M}^{\circledast})\rightarrow \mathscr{M}$ sending a span $M \leftarrow X \rightarrow N$ to $M$ and $N$ respectively. The \textbf{endomorphism $\infty$-category} of $M\in \mathscr{M}$ is the fiber product
\begin{align*}
    \mathcal{C}[M]_{\text{Lurie}} := \{M\} \times_{\mathscr{M}} \text{EnMor}(\mathscr{M}^{\circledast}) \times_{\mathscr{M}} \{M\}.
\end{align*}
\end{definition}
Instead, we start from a coCartesian fibration of $\infty$-operads $q: \mathcal{C}^{\otimes} \rightarrow \lm$, exhibiting $\mathcal{C}_{\mathfrak{m}}$ as left tensored over the monoidal $\infty$-category $\mathcal{C}^{\otimes}_{\mathfrak{a}}$. Hence, we first need to construct from $q$ an $\infty$-category $\mathscr{M}^{\circledast}$ and a map $p: \mathscr{M}^{\circledast} \rightarrow \Delta^1 \times N(\mathbb{\Delta})^{\text{op}}$ as above.\displaypar

Construct the $\infty$-category $\mathscr{M}^{\circledast}$ as the fiber product
    \begin{center}
        \begin{tikzcd}
\mathscr{M}^{\circledast} :=\mathcal{C}^{\otimes}\times_{\lm}(N(\mathbb{\Delta}^{\text{op}})\times\Delta^1) \arrow[d] \arrow[r] & \mathcal{C}^{\otimes} \arrow[d, "q"] \\
N(\mathbb{\Delta}^{\text{op}})\times \Delta^1 \arrow[r, "\gamma"']                                 & \lm                                 
\end{tikzcd}
    \end{center}
    where $\gamma$ is defined in \cite[Remark 4.2.2.8]{HA}. In particular, $\mathscr{M}^{\circledast}$ comes equipped with a coCartesian fibration $p: \mathscr{M}^{\circledast} \rightarrow N(\mathbb{\Delta}^{\text{op}}) \times \Delta^1$. We have
    \begin{align*}
        \mathscr{M}^{\circledast}_{[0],0} = \mathcal{C}^{\otimes} \times_{\lm} \{[0],0\} \simeq \mathcal{C}_{\mathfrak{m}},
    \end{align*}
    since the functor $\text{LCut}: N(\mathbb{\Delta}^{\text{op}})\rightarrow \lm$ defined in \cite[Construction 4.2.2.6]{HA} sends $[0]$ to $(\langle 1\rangle, \{1\}) = \mathfrak{m}$. Consider the diagram
    \begin{center}
       \begin{tikzcd}
\mathscr{M}^{\circledast}\times_{\Delta^1}\{1\} \arrow[d] \arrow[r]  & \mathcal{C}^{\otimes}\times_{\lm}(N(\mathbb{\Delta}^{\text{op}})\times\Delta^1) \arrow[d] \arrow[r] & \mathcal{C}^{\otimes} \arrow[d, "q"] \\
N(\mathbb{\Delta}^{\text{op}})\times \{1\} \arrow[r, hook] \arrow[d] & N(\mathbb{\Delta}^{\text{op}})\times \Delta^1 \arrow[r, "\gamma"'] \arrow[d]                        & \lm                                  \\
\{1\} \arrow[r, hook]                                                & \Delta^1                                                                                           &                                     
\end{tikzcd}.
    \end{center}
    Then the upper right hand side square is a pullback by definition, the lower left hand side square is a pullback and the left hand side rectangle is a pullback, again by definition. By the pasting law, the upper left hand side square is a pullback, and hence, again by the pasting law, the large upper rectangle is a pullback. The lower horizontal arrow of this rectangle agrees with the map $\text{Cut}: N(\mathbb{\Delta}^{\text{op}}) \rightarrow \lm$ defined in \cite[Construction 4.1.2.9]{HA}, so 
    \begin{align*}
        \mathscr{M}^{\circledast} \times_{\Delta^1}\{1\} \simeq \mathcal{C}^{\otimes}\times_{\lm} N(\mathbb{\Delta}^{\text{op}}).
    \end{align*}
    But in the diagram
    \begin{center}
        \begin{tikzcd}
\mathcal{C}_{\mathfrak{a}}^{\otimes}\times_{\ass} N(\mathbb{\Delta}^{\text{op}}) \arrow[r] \arrow[d] & \mathcal{C}^{\otimes}_{\mathfrak{a}} \arrow[r] \arrow[d] & \mathcal{C}^{\otimes} \arrow[d, "q"] \\
N(\mathbb{\Delta}^{\text{op}}) \arrow[r, "\text{Cut}"']                                               & \ass \arrow[r, hook]                                     & \lm                                 
\end{tikzcd}
    \end{center}
    both squares are pullbacks, so the rectangle is as well, and we get
    \begin{align*}
        \mathscr{M}^{\circledast} \times_{\Delta^1} \{1\} \simeq \mathcal{C}_{\mathfrak{a}}^{\otimes} \times_{\ass} N(\mathbb{\Delta}^{\text{op}}),
    \end{align*}
    which is the $\mathbb{A}_{\infty}$-monoidal $\infty$-category corresponding to the monoidal $\infty$-category $\mathcal{C}^{\otimes}_{\mathfrak{a}}$. Call this $\mathbb{A}_{\infty}$-monoidal category $\mathcal{C}_{\mathfrak{a}}^{\circledast}$. Then $p$ exhibits ${\mathcal{C}_{\mathfrak{m}}}$ as left-tensored over $\mathcal{C}_{\mathfrak{a}}^{\circledast}$ in the planar sense. We are now able to state
\begin{proposition}
    Let $q: \mathcal{C}^{\otimes} \rightarrow \lm$ be a coCartesian fibration of $\infty$-operads. Let $p: {\mathcal{C}_{\mathfrak{m}}}^{\circledast} \rightarrow N(\mathbb{\Delta}^{\textnormal{op}})\times \Delta^1$ as constructed above. Then the $\infty$-category $\mathcal{C}_{\mathfrak{a}}[M]$ from Definition \ref{def1} is equivalent to $\mathcal{C}[M]_{\textnormal{Lurie}}$ as in Definition \ref{def4}.
\end{proposition}
\begin{proof}
    Under $\gamma: N(\mathbb{\Delta}^{\text{op}})\times \Delta^1 \rightarrow \lm$, the map
    \begin{align*}
        a: ([0],0) \rightarrow ([1],0)
    \end{align*}
    sending the point in $[0]$ to $0\in [1]$ maps to 
    \begin{align*}
        \text{LCut}(a): (\langle 2\rangle,\{2\}) &\rightarrow (\langle 1\rangle, \{1\})\\
        1 &\mapsto 1\\
        2 &\mapsto 1.
    \end{align*}
    Interpreting $(\langle 2\rangle,\{2\})$ as $(\mathfrak{a},\mathfrak{m})$ and $(\langle 1\rangle,\{1\})$ as $\mathfrak{m}$, this map corresponds to the unique element
    \begin{align*}
        \phi \in \text{Mul}_{\alm}(\{\mathfrak{a},\mathfrak{m}\},\mathfrak{m}).
    \end{align*}
    Similarly, the map
    \begin{align*}
        b: ([0],0) \rightarrow ([1],0)
    \end{align*}
    sending the point in $[0]$ to $1\in [1]$ maps to 
    \begin{align*}
    \text{LCut}(b): (\langle 2\rangle, \{2\}) &\rightarrow (\langle 1\rangle, \{1\})\\
    1 &\mapsto \ast \\
    2 &\mapsto 1.
    \end{align*}
    This map corresponds to the unique element
    \begin{align*}
        \psi \in \text{Mul}_{\alm}(\{\mathfrak{m}\},\mathfrak{m}).
    \end{align*}
    Therefore, to give an enriched morphism of ${\mathcal{C}_{\mathfrak{m}}}$ is equivalent to giving a diagram
    \begin{align*}
        M \xleftarrow{\alpha} X \xrightarrow{\beta} N
    \end{align*}
    in $\mathcal{C}^{\otimes}$ such that 
    \begin{enumerate}
        \item $q(\alpha) = \text{LCut}(a)$,
        \item $q(\beta) = \text{LCut}(b)$, and
        \item $\beta$ is inert, i.e. $q$-coCartesian.
    \end{enumerate}
    Unpacking this, $M$ and $N$ are objects in ${\mathcal{C}_{\mathfrak{m}}}$, and $X=(C,M')$ is an object in $\mathcal{C}^{\otimes}_{(\mathfrak{a},\mathfrak{m})}$, and $\mathcal{C}^{\otimes}_{(\mathfrak{a},\mathfrak{m})}\simeq \mathcal{C}_{\mathfrak{a}} \times \mathcal{C}_{\mathfrak{m}}$ by \cite[Proposition 2.1.2.12 (b)]{HA} and taking $T=(\mathfrak{a},\mathfrak{m})\in \lm_{\langle 2\rangle}$. The map $\alpha: (C,M') \rightarrow M$ and $\beta: (C,M') \rightarrow N$ are morphisms in $\mathcal{C}^{\otimes}$ lifting $\phi$ and $\psi$ respectively. Since $q$ is coCartesian, there is a $q$-coCartesian lift for $\phi$ and $X=(C,M')$, namely the map $(C,M') \rightarrow C\otimes M'$. Hence, the data of $\alpha$ is equivalent to a map $C\otimes M' \rightarrow M$ in ${\mathcal{C}_{\mathfrak{m}}}$. Similarly, there is a $q$-coCartesian lift for $\psi$ and $X=(C,M')$, namely the map $(C,M') \rightarrow M'$. Hence the data of $\beta$ is equivalent to a map $M'\rightarrow N$ in ${\mathcal{C}_{\mathfrak{m}}}$, and since $\beta$ is $q$-coCartesian as well, this map is an equivalence. Hence, the $\infty$-category $\mathcal{C}[M]_{\textnormal{Lurie}}$ is equivalent to the $\infty$-category $\mathcal{C}_{\mathfrak{a}}\times_{{\mathcal{C}_{\mathfrak{m}}}} {\mathcal{C}_{\mathfrak{m}}}_{/M}$ with objects given by pairs $(C\in \mathcal{C}_{\mathfrak{a}}, C\otimes M\rightarrow M \in \mathcal{C}_{\mathfrak{m}})$.
\end{proof}
\newpage
\setcounter{theorem}{0}
\chapter{Comparing notions of formal moduli problems}\label{appendixB}
In this section we prove that Lurie's equivalence
\begin{align*}
\Psi: \alg_{\mathbf{Lie}}(\dk) \rightarrow \mathbf{FMP}.
\end{align*}
constructed in \cite[Theorem 2.0.2]{DAGX} agrees with the equivalence
\begin{align*}
\alg_{\mathbf{Lie}}(\dk) \xrightarrow{\text{exp}} \text{Grp}(\mathbf{FMP}) \xrightarrow{B} \mathbf{FMP}
\end{align*}
constructed by Gaitsgory and Rozenblyum in \cite[Chapter 7 Section 3.3.8]{GR2}.\displaypar
Recall that $\Psi$ is given by
\begin{align*}
\mathfrak{g} \mapsto \map_{\alg_{\mathbf{Lie}}(\dk)}(\mathfrak{D}(-),\mathfrak{g}),
\end{align*}
where $\mathfrak{D}: (\alg_{\dgcomm}^{\text{aug}}(\dk))^{\text{op}} \rightarrow \text{Alg}_{\textbf{Lie}}(\dk)$ denotes the right adjoint of the cohomological Chevalley-Eilenberg complex functor.\displaypar
The category of formal moduli problems is Cartesian symmetric monoidal, and, by construction, every formal moduli problem over $\spec(\field)$ is pointed. We hence have a functor
\begin{align*}
    \Omega: \mathbf{FMP} \rightarrow \text{Grp}(\mathbf{FMP})
\end{align*}
given by $\mathcal{Y} \mapsto \spec(\field) \times_{\mathcal{Y}} \spec(\field)$. Gaitsgory and Rozenblyum show in \cite[Chapter 5, Theorem 2.3.2]{GR2} that this functor is an equivalence with inverse given by the classifying space functor $B$. \displaypar
If $\mathcal{X}$ is a prestack of locally almost finite type, one can define the category $\text{IndCoh}(\mathcal{X})$ of ind-coherent sheaves of $\mathcal{X}$. For any map $f: \mathcal{X} \rightarrow \mathcal{Y}$, there is a !-pullback
\begin{align*}
    f^!: \text{IndCoh}(\mathcal{Y}) \rightarrow \text{IndCoh}(\mathcal{X}).
\end{align*}
For a class of maps called \textbf{inf-schematic}, there also exists a well-defined continuous functor
\begin{align*}
    f_{\ast}^{\text{IndCoh}}: \text{IndCoh}(\mathcal{X}) \rightarrow \text{IndCoh}(\mathcal{Y})
\end{align*}
called the direct image functor. For a prestack of locally almost finite type $\mathcal{X}$, Gaitsgory and Rozenblyum define
\begin{align*}
    \omega_{\mathcal{X}} := p^!(\field),
\end{align*}
where $p: \mathcal{X} \rightarrow \spec(\field)$ is the unique projection. This is the monoidal unit of the $\overset{!}{\otimes}$-monoidal structure on $\text{IndCoh}(\mathcal{X})$.
\begin{definition}[Chapter 7 Section 1 \cite{GR2}] \label{def7}
    Denote by
    \begin{center}
    \begin{tikzcd}
            \text{Distr}^{\text{Cocom}^{\text{aug}}}: \mathbf{FMP} \arrow[r, shift left=1ex, ""{name=G}] &\coalg_{\dgcocomm}^{\text{coaug}}(\dk) :\text{Spec}^{\text{inf}}\arrow[l, shift left=1ex, ""{name=F}]
            \arrow[phantom, from=F, to=G, , "\scriptscriptstyle\boldsymbol{\perp}"]
        \end{tikzcd}
\end{center}
the adjoint pair of functors with $\text{Distr}^{\text{Cocom}^{\text{aug}}}$ given by $\mathcal{Y} \mapsto p_{\ast}^{\text{IndCoh}}(\omega_{\mathcal{Y}})$.
\end{definition}
In fact, $\text{Distr}^{\text{Cocom}^{\text{aug}}}$ is strict symmetric monoidal, and hence $\text{Spec}^{\text{inf}}$ is right-lax symmetric monoidal. In particular, the above adjunction upgrades to an adjunction between monoid (or group) objects in the respective symmetric monoidal categories.
\begin{center}
    \begin{tikzcd}
            \text{Grp(Distr}^{\text{Cocom}^{\text{aug}}}): \text{Grp}(\mathbf{FMP}) \arrow[r, shift left=1ex, ""{name=G}] &\text{Monoid(CoAlg}_{\dgcocomm}^{\text{coaug}}(\dk)):\text{Monoid(Spec}^{\text{inf}})\arrow[l, shift left=1ex, ""{name=F}]
            \arrow[phantom, from=F, to=G, , "\scriptscriptstyle\boldsymbol{\perp}"]
        \end{tikzcd}.
\end{center}
Neither the functor $\text{Distr}^{\text{Cocom}^{\text{aug}}}$ nor $\text{Spec}^{\text{inf}}$ is fully faithful, in contrast to the situation in classical algebraic geometry (where the spectrum functor is fully faithful). However, one can show that the counit of the adjunction is an equivalence for so-called \textbf{vector prestacks}\footnote{See \cite[Chapter 7, Section 1.4.1]{GR2}} of the form $\text{Vect}(V)$ for a $\field$-vector space $V$. A formal moduli problem on which the unit is an equivalence is called \textbf{inf-affine}.\displaypar
Denote by 
\begin{center}
    \begin{tikzcd}
            \text{Chev}^{\text{enh}}: \text{Alg}_{\text{Lie}}(\dk) \arrow[r, shift left=1ex, ""{name=G}] &\coalg_{\dgcocomm}^{\text{coaug}}(\dk):\text{coChev}^{\text{enh}}\arrow[l, shift left=1ex, ""{name=F}]
            \arrow[phantom, from=F, to=G, , "\scriptscriptstyle\boldsymbol{\perp}"]
        \end{tikzcd}
\end{center}
Gaitsgory-Rozenblyum's Bar-Cobar adjunction for the Lie operad \cite[Chapter 6, Section 4.1.4]{GR2}. The functor $\text{Chev}^{\text{enh}}$ is strict symmetric monoidal, and therefore its right adjoint is right-lax symmetric monoidal, and in particular preserves monoid objects.\displaypar
Again following Gaitsgory and Rozenblyum, we can now define the functors of taking the exponential group of a Lie algebra and taking the Lie algebra of a formal group 
\begin{definition}[Chapter 7, Section 3 \cite{GR2}]
    The exponential functor is given by
    \begin{align*}
        \text{exp}:= \text{Alg}_{\textbf{Lie}}(\dk) \xrightarrow{\Omega} \text{Grp}(\text{Alg}_{\textbf{Lie}}(\dk)) \xrightarrow{\text{Grp}(\text{Chev}^{\text{enh}})} \text{Monoid}(\coalg_{\dgcocomm}^{\text{coaug}}(\dk)) \\ \xrightarrow{\text{Monoid}(\text{Spec}^{\text{inf}})} \text{Grp}(\mathbf{FMP}).
    \end{align*}
    The Lie functor is given by
    \begin{align*}
        \text{Lie}:= \text{Grp}(\mathbf{FMP}) \xrightarrow{\text{Grp}(\text{Distr}^{\text{Cocom}^{\text{aug}}})} \text{Monoid}(\coalg_{\dgcocomm}^{\text{coaug}}(\dk))\\ \xrightarrow{\text{Monoid}(\text{coChev}^{\text{enh}})} \text{Grp}(\text{Alg}_{\textbf{Lie}}(\dk)) \xrightarrow{B} \text{Alg}_{\textbf{Lie}}(\dk).
    \end{align*}
\end{definition}
\begin{proposition}[Chapter 6, Proposition 4.3.3 \cite{GR2}]\label{prop11}
    The functor 
    \begin{align*}
        \textnormal{forget}_{\textnormal{Monoid}} \circ \textnormal{Grp}(\textnormal{Chev}^{\textnormal{enh}}) \circ \Omega: \alg_{\textnormal{Lie}}(\dk) \rightarrow \coalg_{\dgcocomm}^{\textnormal{coaug}}(\dk)
    \end{align*}
    is equivalent to the functor $\textnormal{Sym}\circ \textnormal{forget}_{\textnormal{Lie}}$.
\end{proposition}
By \cite[Chapter 6, Theorem 6.1.2]{GR2}, $\text{Grp}(\text{Chev}^{\text{enh}}) \circ \Omega$ is equivalent to the universal enveloping algebra functor $U^{\text{Hopf}}: \text{Alg}_{\text{Lie}}(\dk) \rightarrow \alg_{\mathbb{E}_1}(\coalg_{\dgcocomm}(\dk))$. Hence, the proposition just says that the underlying coalgebra of the universal enveloping algebra is the symmetric coalgebra, i.e. the cofree conilpotent commutative coalgebra.
\begin{theorem}[Chapter 7, Theorem 3.1.4 \cite{GR2}]\label{thm19}
    The functor $\textnormal{exp}: \alg_{\mathbf{Lie}}(\dk) \rightarrow \textnormal{Grp}(\mathbf{FMP})$ is an equivalence.
\end{theorem}
\begin{corollary}
    The functors $\textnormal{exp}: \alg_{\mathbf{Lie}}(\dk) \rightarrow \textnormal{Grp}(\mathbf{FMP})$ and $\textnormal{Lie}: \textnormal{Grp}(\mathbf{FMP}) \rightarrow \alg_{\mathbf{Lie}}(\dk)$ are mutually inverse.
\end{corollary}
\begin{proof}
    By the above Proposition \ref{prop11} and the definition of vector prestacks, $\text{exp}(\mathfrak{g}) \simeq \text{Vect}(\text{forget}_{\text{Lie}}(\mathfrak{g}))$ as formal moduli problems (i.e. after forgetting the group structure). This is a vector prestack, and in particular it is inf-affine by \cite[Chapter 7, Proposition 1.4.7]{GR2}, so both the unit and the counit of the adjunction in \ref{def7} are equivalences. Noting that the functor
    \begin{align*}
        B \circ \text{Monoid}(\text{coChev}^{\text{enh}})
    \end{align*}
    is a left-inverse of 
    \begin{align*}
        \text{Grp}(\text{Chev}^{\text{enh}}) \circ \Omega
    \end{align*}
    by \cite[Chapter 6, Theorem 4.4.6]{GR2}, this shows that $\text{Lie}\circ \text{exp} \simeq \text{id}_{\alg_{\textbf{Lie}}(\dk)}$.
\end{proof}
This implies that we can build an equivalence between formal moduli problems and Lie algebras by composing the exponential map and the Lie map with delooping and looping respectively:
\begin{center}
    \begin{tikzcd}
            \alg_{\textbf{Lie}}(\dk)\arrow[r, shift left=1ex, "\text{exp}"{name=G}] & \text{Grp}(\mathbf{FMP})\arrow[l, shift left=1ex, "\text{Lie}"{name=F}]  \arrow[phantom, from=F, to=G, , "\sim"] \arrow[r, shift left=1ex, "B"{name=H}] &  \mathbf{FMP}\arrow[l, shift left=1ex, "\Omega"{name=I}]
            \arrow[phantom, from=H, to=I, , "\sim"]
        \end{tikzcd}.
\end{center}
In the remainder of this chapter we will prove the following
\begin{theorem}[\cite{AK}]\label{thm18} 
    Let $\mathfrak{g}\in \alg_{\mathbf{Lie}}(\dk)$. Then 
    \begin{align*}
        \Psi(\mathfrak{g}) \simeq B\circ \textnormal{exp}(\mathfrak{g})
    \end{align*}
    as formal moduli problems.
\end{theorem}
\begin{lemma}\label{lem7}
    To prove Theorem \ref{thm18}, it suffices to show that
    \begin{align*}
        \textnormal{Grp}(\textnormal{Distr}^{\textnormal{Cocom}^{\textnormal{aug}}})(\Omega \Psi(\mathfrak{g})) \simeq U^{\textnormal{Hopf}}(\mathfrak{g}).
    \end{align*}
\end{lemma}
\begin{proof}
As a corollary of Theorem \ref{thm19}, one can show that all group objects in formal moduli problems are inf-affine when we forget the group structure (see \cite[Chapter 7, Corollary 3.2.2]{GR2}). Therefore, applying $\textnormal{Monoid}(\textnormal{Spec}^{\textnormal{inf}})$ to the left hand side of the above equation we get
\begin{align*}
    \textnormal{Monoid}(\textnormal{Spec}^{\textnormal{inf}}) \circ \textnormal{Grp}(\textnormal{Distr}^{\textnormal{Cocom}^{\textnormal{aug}}})(\Omega \Psi(\mathfrak{g})) \simeq \Omega \Psi(\mathfrak{g}).
\end{align*}
For the right hand side of the equation, we get
\begin{align*}
     \textnormal{Monoid}(\textnormal{Spec}^{\textnormal{inf}}) \circ U^{\textnormal{Hopf}}(\mathfrak{g}) \simeq \textnormal{Monoid}(\textnormal{Spec}^{\textnormal{inf}}) \circ \textnormal{Grp}(\textnormal{Chev}^{\textnormal{enh}}) \circ \Omega (\mathfrak{g}) = \textnormal{exp}(\mathfrak{g}).
\end{align*}
Then applying $B$ to both sides yields the result of the theorem.
\end{proof}
\begin{lemma}
    To prove Theorem \ref{thm18}, it further suffices to show that
    \begin{align*}
        \textnormal{Distr}^{\textnormal{Cocom}^{\textnormal{aug}}}(\Psi(\mathfrak{g})) \simeq \textnormal{Chev}^{\textnormal{enh}}(\mathfrak{g}).
    \end{align*}
\end{lemma}
\begin{proof}
    Applying $\textnormal{Bar}\circ -$ to both sides of the equation in Lemma \ref{lem7} yields
    \begin{align*}
        \textnormal{Bar}\circ \textnormal{Grp}(\textnormal{Distr}^{\textnormal{Cocom}^{\textnormal{aug}}})(\Omega \Psi(\mathfrak{g})) \simeq \textnormal{Distr}^{\textnormal{Cocom}^{\textnormal{aug}}}\circ B(\Omega \Psi(\mathfrak{g})) \simeq \textnormal{Distr}^{\textnormal{Cocom}^{\textnormal{aug}}}(\Psi(\mathfrak{g})) 
    \end{align*}
    by \cite[Chapter 7, Lemma 1.2.6]{GR2}, and
    \begin{align*}
        \textnormal{Bar}\circ U^{\textnormal{Hopf}}(\mathfrak{g}) \simeq \textnormal{Chev}^{\textnormal{enh}}(\mathfrak{g})
    \end{align*}
    by \cite[Chapter 6, Proposition 5.3.3]{GR2}. Since the monoidal structure on cocommutative coalgebras is cartesian, the Bar functor is identified with the classifying space functor. Hence, since the functors $\textnormal{Grp}(\textnormal{Distr}^{\textnormal{Cocom}^{\textnormal{aug}}})$ and $U^{\textnormal{Hopf}}$ land in  group-like elements, applying $\Omega$ recovers the original equation, and the result follows.
\end{proof}
Finally, to show $\textnormal{Distr}^{\textnormal{Cocom}^{\textnormal{aug}}}(\Psi(\mathfrak{g})) \simeq \textnormal{Chev}^{\textnormal{enh}}(\mathfrak{g})$, we use the following result of Lurie
\begin{proposition}[Proposition 2.4.31, \cite{DAGX}]
    Let $\mathfrak{g}\in \alg_{\textnormal{Lie}}(\dk)$. There is a canonical equivalence of $\infty$-categories
    \begin{align*}
        \textnormal{Rep}_{\mathfrak{g}} \xrightarrow{\simeq} \textnormal{IndCoh}(\Psi(\mathfrak{g})).
    \end{align*}
\end{proposition}
\begin{proof}[Proof of Theorem \ref{thm18}]
    Recall that $\textnormal{Distr}^{\textnormal{Cocom}^{\textnormal{aug}}}(\Psi(\mathfrak{g})) = p_{\ast}^{\text{IndCoh}}(\omega_{\Psi(\mathfrak{g})})$, where $\omega_{\Psi(\mathfrak{g})} = p^!(\field)$, and $p: \Psi(\mathfrak{g}) \rightarrow \spec \field$ is the canonical projection. Since formal moduli problems are proper, we also have that $p_{\ast}^{\text{IndCoh}} \dashv p^!$. Under the correspondence between representations of $\mathfrak{g}$ and ind-coherent sheaves on $\Psi(\mathfrak{g})$, this adjunction corresponds to the extension $\dashv$ restriction adjunction of modules induced by the map $\mathfrak{g} \rightarrow 0$:
    \begin{center}
    \begin{tikzcd}
            \field \otimes_{U^{\text{Hopf}}(\mathfrak{g})}^{\mathbb{L}} -: \textnormal{LMod}_{U^{\text{Hopf}}(\mathfrak{g})}(\dk) \arrow[r, shift left=1ex, ""{name=G}] &\dk: \text{trivial}\arrow[l, shift left=1ex, ""{name=F}]
            \arrow[phantom, from=F, to=G, , "\scriptscriptstyle\boldsymbol{\perp}"]
        \end{tikzcd}.
    \end{center}
    In particular, $\omega_{\Psi(\mathfrak{g})} = p^!(\field)$ corresponds to the trivial $\mathfrak{g}$-representation on $\field$, and therefore
    \begin{align*}
        \textnormal{Distr}(\Psi(\mathfrak{g})) = p_{\ast}^{\text{IndCoh}}(\omega_{\Psi(\mathfrak{g})}) \simeq  \field \otimes_{U^{\text{Hopf}}(\mathfrak{g})}^{\mathbb{L}} \field = \text{Chev}(\mathfrak{g})
    \end{align*}
    in $\dk$. The fact that they also agree as coalgebras then follows from the construction of the coalgebra structures on $\textnormal{Distr}^{\textnormal{Cocom}^{\textnormal{aug}}}(\Psi(\mathfrak{g}))$ and $\text{Chev}^{\text{enh}}(\mathfrak{g})$ from their respective structures of a monoidal units.
\end{proof}
\newpage
\thispagestyle{empty}
\bibliographystyle{alpha}
\inputencoding{utf8}    
\bibliography{main}
\end{document}